\pdfoutput=1
\documentclass[a4paper,reqno]{amsart}
\usepackage[foot]{amsaddr}
\usepackage[margin=3cm]{geometry}

\newcommand{\ifarticle}[2]{
    \csname@ifclassloaded\endcsname{beamer}{#2}{#1}
}

\newcommand{\ifbook}[2]{
    \csname@ifclassloaded\endcsname{amsbook}{#1}{#2}
}

    \usepackage[T1]{fontenc} 
    \usepackage{xparse} 
    \usepackage{xspace} 
    \usepackage{etoolbox} 
    \usepackage{xpatch} 
    \usepackage{pgffor} 

    \usepackage{amsmath,amssymb,amsthm} 
    \allowdisplaybreaks[1]

    \usepackage{mathtools} 
    \usepackage{mathrsfs} 
    \usepackage{stmaryrd} 
    \usepackage{bbm} 
    \usepackage{quiver} 
    \usepackage{xcolor} 
    \usepackage{scalerel} 
    \usepackage{booktabs} 
    \usepackage{nameref} 
    \usepackage[british]{babel} 
    \usepackage{csquotes} 
    \usepackage[UKenglish]{isodate} 
    \usepackage{microtype} 
    \usepackage[splitrule]{footmisc} 

    \ifarticle{
        \PassOptionsToPackage{hyphens}{url}
        \usepackage[pdfusetitle,linktoc=all,colorlinks,citecolor=cyan,linkcolor=purple,urlcolor=blue]{hyperref} 
        \urlstyle{rm}

        \usepackage[nameinlink,noabbrev,capitalize]{cleveref} 
        \usepackage{footnotebackref} 

        \usepackage[shortlabels]{enumitem} 
        \setlist{topsep=2pt,itemsep=2pt,partopsep=2pt,parsep=2pt} 

        \setcounter{tocdepth}{1}
    }{}
    \usepackage[style=alphabetic,abbreviate=false,backref=true,backrefstyle=three,maxnames=9,minalphanames=3,maxalphanames=4]{biblatex} 
    \DeclareUnicodeCharacter{0301}{\TODO[Invalid symbol.]} 

    \DeclareFieldFormat*{citetitle}{\emph{#1}} 
    \DeclareCiteCommand{\citetitle}
        {\boolfalse{citetracker}%
        \boolfalse{pagetracker}%
        \usebibmacro{prenote}}
        {\ifciteindex
            {\indexfield{indextitle}}
            {}%
        \printtext[bibhyperref]{\printfield[citetitle]{labeltitle}}}
        {\multicitedelim}
        {\usebibmacro{postnote}}
    \DeclareCiteCommand*{\citeauthor}
        {\defcounter{maxnames}{99}%
        \defcounter{minnames}{99}%
        \defcounter{uniquename}{2}%
        \boolfalse{citetracker}%
        \boolfalse{pagetracker}%
        \usebibmacro{prenote}}
        {\ifciteindex{\indexnames{labelname}}{}%
        \printnames{labelname}}
        {\multicitedelim}
        {\usebibmacro{postnote}}

    \usepackage{calc} 
    \usepackage{tabularray} 
    \usepackage{ebproof} 
    \usepackage{forest} 
    \usepackage{adjustbox} 
    \usepackage{float} 
    \usepackage{stfloats}

    \makeatletter
    \ifarticle{
        \makeatletter\def\@makefntext{\indent\@makefnmark}\makeatletter

        \xpretocmd{\@adminfootnotes}{\let\@makefntext\BHFN@OldMakefntext}{}{}
        \xpatchcmd{\@maketitle}{\let\@makefnmark\relax}{\let\@makefnmark\no@makefnmark}{}{}
        \def\no@makefnmark{}
        \renewcommand\@makefntext[1]{%
          \ifx\@makefnmark\no@makefnmark
            \BHFN@OldMakefntext{#1}%
          \else
            \renewcommand\@makefnmark{%
            \mbox{%
                \textsuperscript{%
                \normalfont
                \hyperref[\BackrefFootnoteTag]{\@thefnmark}%
                }%
            }\,%
            }%
            \BHFN@OldMakefntext{#1}%
          \fi
        }

        \LetLtxMacro{\BHFN@Old@footnotemark}{\@footnotemark}
        \renewcommand*{\@footnotemark}{%
            \refstepcounter{BackrefHyperFootnoteCounter}%
            \xdef\BackrefFootnoteTag{bhfn:\theBackrefHyperFootnoteCounter}%
            \label{\BackrefFootnoteTag}%
            \BHFN@Old@footnotemark
        }

        \makeatletter
        \def\paragraph{\@startsection{paragraph}{4}%
          \z@\z@{-\fontdimen2\font}%
          {\normalfont\bfseries}}
        \makeatother
    }{}
    \makeatother

    \ifarticle{
        \theoremstyle{plain}
        \ifbook{
            \newtheorem{theorem}{Theorem}[chapter]
        }{
            \newtheorem{theorem}{Theorem}[section]
        }
        \newtheorem{proposition}[theorem]{Proposition}
        \newtheorem{lemma}[theorem]{Lemma}
        \newtheorem{corollary}[theorem]{Corollary}

        \newtheorem*{theorem*}{Theorem}
        \newtheorem*{corollary*}{Corollary}

        \theoremstyle{definition}
        \newtheorem{definition}[theorem]{Definition}
        
        \newtheorem{example}[theorem]{Example}
        
        \newtheorem{notation}[theorem]{Notation}
        
        \newtheorem{remark}[theorem]{Remark}
        
        \Crefname{axiom}{Axiom}{Axioms}

        \newenvironment{sketch}{\proof}{\endproof}

        \Crefname{theoremenumi}{Theorem}{Theorems}
        \AtBeginEnvironment{theorem}{%
            \crefalias{enumi}{theoremenumi}%
            \setlist[enumerate,1]{
                ref={\csname thetheorem\endcsname.(\arabic*)}
            }%
            \crefalias{enumii}{theoremenumi}%
            \setlist[enumerate,2]{
                ref={\thetheorem.(\arabic*).(\alph*)}
            }%
        }

        \Crefname{propositionenumi}{Proposition}{Propositions}
        \AtBeginEnvironment{proposition}{%
            \crefalias{enumi}{propositionenumi}%
            \setlist[enumerate,1]{
                ref={\csname theproposition\endcsname.(\arabic*)}
            }%
            \crefalias{enumii}{propositionenumi}%
            \setlist[enumerate,2]{
                ref={\theproposition.(\arabic*).(\alph*)}
            }%
        }

        \Crefname{lemmaenumi}{Lemma}{Lemmas}
        \AtBeginEnvironment{lemma}{%
            \crefalias{enumi}{lemmaenumi}%
            \setlist[enumerate,1]{
                ref={\csname thelemma\endcsname.(\arabic*)}
            }%
            \crefalias{enumii}{lemmaenumi}%
            \setlist[enumerate,2]{
                ref={\thelemma.(\arabic*).(\alph*)}
            }%
        }

        \Crefname{corollaryenumi}{Corollary}{Corollaries}
        \AtBeginEnvironment{corollary}{%
            \crefalias{enumi}{corollaryenumi}%
            \setlist[enumerate,1]{
                ref={\csname thecorollary\endcsname.(\arabic*)}
            }%
            \crefalias{enumii}{corollaryenumi}%
            \setlist[enumerate,2]{
                ref={\thecorollary.(\arabic*).(\alph*)}
            }%
        }

        \Crefname{definitionenumi}{Definition}{Definitions}
        \AtBeginEnvironment{definition}{%
            \crefalias{enumi}{definitionenumi}%
            \setlist[enumerate,1]{
                ref={\csname thedefinition\endcsname.(\arabic*)}
            }%
            \crefalias{enumii}{definitionenumi}%
            \setlist[enumerate,2]{
                ref={\thedefinition.(\arabic*).(\alph*)}
            }%
        }

        \Crefname{exampleenumi}{Example}{Examples}
        \AtBeginEnvironment{example}{%
            \crefalias{enumi}{exampleenumi}%
            \setlist[enumerate,1]{
                ref={\csname theexample\endcsname.(\arabic*)}
            }%
            \crefalias{enumii}{exampleenumi}%
            \setlist[enumerate,2]{
                ref={\theexample.(\arabic*).(\alph*)}
            }%
        }

        \Crefname{axiomenumi}{Axiom}{Axioms}
        \AtBeginEnvironment{axiom}{%
            \crefalias{enumi}{axiomenumi}%
            \setlist[enumerate,1]{
                ref={\csname theaxiom\endcsname.(\arabic*)}
            }%
            \crefalias{axiomii}{axiomenumi}%
            \setlist[enumerate,2]{
                ref={\theaxiom.(\arabic*).(\alph*)}
            }%
        }

        \foreach \env in {definition,remark,example,notation,axiom}{
        \AtBeginEnvironment{\env}{%
          \pushQED{\qed}%
        }
        \AtEndEnvironment{\env}{\popQED\endexample}
        }
    }{}

    \ExplSyntaxOn
    \NewDocumentCommand{\mathcommand}{mO{0}m}
     {
      \exp_args:Nc \NewCommandCopy {khue_kept_\cs_to_str:N #1} { #1 }
      \exp_args:Nc \newcommand {khue_new_\cs_to_str:N #1}[#2]{#3}
      \DeclareDocumentCommand {#1} {}
       {
        \mode_if_math:TF
         {
          \use:c {khue_new_\cs_to_str:N #1}
         }
         {
          \use:c {khue_kept_\cs_to_str:N #1}
         }
       }
     }
    \ExplSyntaxOff

    \newenvironment{iffseq}{%
        \global\let\externaldblbackslash\\
        \[\begin{array}{cl}
        \ifundef{\internaldblbackslash}{%
            \global\let\internaldblbackslash\\%
            \gdef\\{\internaldblbackslash\cmidrule{1-1}\morecmidrules\cmidrule{1-1}}%
        }{}
    }{%
        \end{array}\]
        \global\undef\internaldblbackslash
        \global\let\\\externaldblbackslash
    }

    \newsavebox\tikzcdbox

    \mathcommand{\h}{\textup{-}}

    \newcommand{\tx}{\mathrm}
    \mathcommand{\b}{\mathbf}
    \newcommand{\s}{\mathsf}
    \newcommand{\cl}{\mathcal}
    \mathcommand{\bb}{\mathbb}
    \DeclareMathAlphabet{\bbn}{U}{bbold}{m}{n}
    \newcommand{\dc}[1]{\TextOrMath{double category\xspace#1}{\b{\bb#1}}}
    
    \mathcommand{\sf}{\mathsf}
    \mathcommand{\u}{\underline}
    \mathcommand{\o}{\overline}

    \newcommand{\TODO}[1][TODO]{\textcolor{orange}{\textup{#1}}\xspace}

    \newcommand{\flip}[1]{\text{\rotatebox[origin=c]{-180}{$#1$}}}
    
    \newcommand{\datetoday}{\date{\cleanlookdateon\today}}

    \newcommand{\defeq}{\mathrel{:=}}
    \mathcommand{\d}{\mathbin{;}}
    \mathcommand{\c}{\circ}
    \newcommand{\ph}[1][]{{({-}_{#1})}}
    
    \newcommand{\iso}{\cong}
    
    \newcommand{\xto}{\xrightarrow}
    \newcommand{\xfrom}{\xleftarrow}
    \newcommand{\tto}{\Rightarrow}
    
    \newcommand{\xtto}{\xRightarrow}
    
    \newcommand{\ffto}{\hookrightarrow}
    
    \newcommand*\cocolon{%
            \nobreak
            \mskip6mu plus1mu
            \mathpunct{}%
            \nonscript
            \mkern-\thinmuskip
            {:}%
            \mskip2mu
            \relax
    }

    \makeatletter
    \def\slashedarrowfill@#1#2#3#4#5{%
    $\m@th\thickmuskip0mu\medmuskip\thickmuskip\thinmuskip\thickmuskip
    \relax#5#1\mkern-7mu%
    \cleaders\hbox{$#5\mkern-2mu#2\mkern-2mu$}\hfill
    \mathclap{#3}\mathclap{#2}%
    \cleaders\hbox{$#5\mkern-2mu#2\mkern-2mu$}\hfill
    \mkern-7mu#4$%
    }
    \def\rightslashedarrowfill@{%
    \slashedarrowfill@\relbar\relbar\mapstochar\rightarrow}
    \newcommand\xslashedrightarrow[2][]{%
    \ext@arrow 0055{\rightslashedarrowfill@}{#1}{#2}}
    \def\leftslashedarrowfill@{%
    \slashedarrowfill@\leftarrow\relbar\mapsfromchar\relbar}
    \newcommand\xslashedleftarrow[2][]{%
    \ext@arrow 0055{\leftslashedarrowfill@}{#1}{#2}}
    \makeatother

    \newcommand{\xlto}{\xslashedrightarrow}
    \newcommand{\lto}{\xlto{}}
    \newcommand{\xlfrom}{\xslashedleftarrow}
    \newcommand{\lfrom}{\xlfrom{}}

    \newcommand{\inv}{^{-1}}
    \newcommand{\op}{{}^\tx{op}}
    \newcommand{\co}{{}^\tx{co}}
    \newcommand{\coop}{{}^\tx{co\,op}}
    \newcommand{\rev}{{}^\tx{rev}}

    \newcommand{\tp}[1]{\langle#1\rangle}
    \newcommand{\unit}{{\tp{}}}
    
    \newcommand{\Kleene}{^{\star}}

    \newcommand{\rf}{\mathbin{\blacktriangleleft}}
    
    \newcommand{\rx}{\mathbin{\blacktriangleright}}
    \newcommand{\plx}{\mathbin{{\rhd}\mathclap{\mspace{-17.5mu}\cdot}}}
    \newcommand{\prx}{\mathbin{{\blacktriangleright}\mathclap{\mspace{-17.5mu}\textcolor{white}{\cdot}}}}
    \newcommand{\plf}{\mathbin{\flip{{\rhd}\mathclap{\mspace{-17.5mu}\cdot}}}}
    \newcommand{\prf}{\mathbin{\flip{{\blacktriangleright}\mathclap{\mspace{-17.5mu}\textcolor{white}{\cdot}}}}}

    \newcommand{\adj}{\dashv}
    
    \newcommand{\radj}[1]{\mathrel{\adj_{#1}}}
    \newcommand{\rcadj}[1]{\mathrel{\vdash_{#1}}}

    \newcommand{\ob}[1]{|#1|}

    \DeclareFontFamily{U}{min}{}
    \DeclareFontShape{U}{min}{m}{n}{<-> udmj30}{}
    \newcommand{\yo}{\!\text{\usefont{U}{min}{m}{n}\symbol{'210}}\!}

    \mathcommand{\comma}{\downarrow}

    \newsavebox{\whitecircstar}\sbox{\whitecircstar}{\kern.075em\tikz{\node[draw, circle,line width=.36pt, inner sep=0]{$*$};}\kern.075em}
    \newcommand{\ostar}{\mathbin{\scalerel*{\usebox{\whitecircstar}}{\odot}}}
    \newsavebox{\blackcircstar}\sbox{\blackcircstar}{\kern.075em\tikz{\node[fill, circle, line width=.36pt, inner sep=0, text=white]{$*$};}\kern.075em}
    \newcommand{\bulletstar}{\mathbin{\scalerel*{\usebox{\blackcircstar}}{\odot}}}

    \newcommand{\skt}{\olessthan}
    
    \makeatletter
    \def\widebreve{\mathpalette\wide@breve}
    \def\wide@breve#1#2{\sbox\z@{$#1#2$}%
         \mathop{\vbox{\m@th\ialign{##\crcr
    \kern0.08em\brevefill#1{0.8\wd\z@}\crcr\noalign{\nointerlineskip}%
                        $\hss#1#2\hss$\crcr}}}\limits}
    \def\brevefill#1#2{$\m@th\sbox\tw@{$#1($}%
      \hss\resizebox{#2}{\wd\tw@}{\rotatebox[origin=c]{90}{\upshape(}}\hss$}
    \makeatother

    \NewDocumentCommand{\jrule}{om}{%
        \IfNoValueTF{#1}
            {\textsc{#2}}
            {$#1$-\textsc{#2}}%
    }

    \newcommand{\Set}{{\b{Set}}}
    
    \newcommand{\V}{{\bb V}} 
    \newcommand{\Cat}{\b{Cat}}
    
    \newcommand{\VCat}{{\V\h\Cat}}
    
    \newcommand{\RAdj}{\b{RAdj}}
    \newcommand{\Mnd}{\b{Mnd}}
    \newcommand{\RMnd}{\b{RMnd}}
    \newcommand{\Mon}{\b{Mon}}
    \newcommand{\Alg}{\b{Alg}}
    
    \newcommand{\Opalg}{\b{Opalg}}

    \newcommand{\ff}{fully faithful}
    
    \newcommand{\ffness}{full faithfulness}

    \newcommand{\ioo}{identity-on-objects}

    \newcommand{\EM}{Eilenberg--Moore}

    \newcommand{\eg}{e.g.\@\xspace}
    \newcommand{\ie}{i.e.\@\xspace}
    
    \newcommand{\cf}{cf.\@\xspace}
    
    \newcommand{\aka}{a.k.a.\@\xspace}
    \NewDocumentCommand{\etc}{t.}{etc.\@\xspace}
    \NewDocumentCommand{\ibid}{t.}{ibid.\@\xspace}
    \NewDocumentCommand{\loccit}{t.}{loc.\ cit.\@\xspace}

    \newcommand{\vd}{virtual double}
    \newcommand{\vdc}{\vd{} category}
    \newcommand{\vdcs}{\vd{} categories}
    
    \newcommand{\ve}{virtual equipment}

    \newcommand{\lfp}{locally finitely presentable}

\makeatletter

\define@key{beamerframe}{c}[true]{
    \beamer@frametopskip=0pt plus 1fill\relax%
    \beamer@framebottomskip=0pt plus 1fill\relax%
}

\patchcmd{\beamer@sectionintoc}{\vfill}{\vskip\itemsep}{}{}

\makeatother

\ifarticle{}{

  \colorlet{colour-bg}{black!85} 
  \definecolor{colour-primary}{HTML}{cc80ff} 
  \colorlet{colour-text}{black!10} 
  \colorlet{colour-subtle}{black!40} 
  \colorlet{colour-block-bg}{black!80} 
  \definecolor{colour-warning-bg}{HTML}{ffea80} 
  \definecolor{colour-warning-primary}{HTML}{e08152} 

  \hypersetup{linktoc=all,colorlinks, citecolor={colour-primary}, linkcolor={colour-primary}, urlcolor={colour-primary}} 
  \urlstyle{sf}

  \usepackage{changepage} 

  \usepackage{ragged2e} 

  \setbeamertemplate{section in toc}[sections numbered]

  \apptocmd{\frame}{}{\justifying}{}

  \usefonttheme[onlymath]{serif}

  \beamertemplatenavigationsymbolsempty
  \setbeamertemplate{footline}{
    \hfill%
    \usebeamercolor[fg]{page number in head/foot}%
    \usebeamerfont{page number in head/foot}%
    \setbeamertemplate{page number in head/foot}[pagenumber]%
    \usebeamertemplate*{page number in head/foot}\kern.2cm\vskip.2cm%
  }

  \setbeamertemplate{frametitle}[default][center]
  \addtobeamertemplate{frametitle}{\vspace*{1ex}}{\vspace*{1ex}}

  \setbeamertemplate{itemize item}{$\bullet$}

  \setbeamertemplate{theorems}[numbered]
  \newtheorem{proposition}[theorem]{\translate{Proposition}}

  \addtobeamertemplate{block begin}
  {}
  {\vspace{0ex} 
  \begin{adjustwidth}{1em}{1em} 
  \tikzcdset{background color=colour-block-bg}
  }
  \addtobeamertemplate{block end}{\end{adjustwidth}%
  \vspace{1ex}} 
  {}
  \renewenvironment<>{block}[1]{%
      \begin{actionenv}#2%
        \def\insertblocktitle{\begin{adjustwidth}{.8em}{.8em}#1\end{adjustwidth}}%
        \par%
        \usebeamertemplate{block begin}}
      {\par%
        \usebeamertemplate{block end}%
      \end{actionenv}}

  \addtobeamertemplate{block example begin}
  {}
  {\vspace{0ex} 
  \begin{adjustwidth}{1em}{1em} 
  \tikzcdset{background color=colour-block-bg}
  }
  \addtobeamertemplate{block example end}{\end{adjustwidth}%
  \vspace{1ex}} 
  {}
  \renewenvironment<>{exampleblock}[1]{%
      \begin{actionenv}#2%
          \def\insertblocktitle{\begin{adjustwidth}{.8em}{.8em}#1\end{adjustwidth}}%
          \par%
          \only<presentation>{
            \setbeamercolor{local structure}{parent=example text}}%
          \usebeamertemplate{block example begin}}
        {\par%
          \usebeamertemplate{block example end}%
        \end{actionenv}}

    \setbeamercolor{background canvas}{bg=colour-bg}
    \tikzcdset{background color=colour-bg}

    \setbeamercolor{block title}{fg=colour-primary,bg=colour-block-bg}
    \setbeamercolor{block title example}{fg=colour-primary,bg=colour-block-bg}
    \setbeamercolor{block body}{bg=colour-block-bg}
    \setbeamercolor{block body example}{bg=colour-block-bg}

    \setbeamercolor{title}{fg=colour-primary}
    \setbeamercolor{frametitle}{fg=colour-primary}
    \setbeamercolor{alerted text}{fg=colour-primary}

    \setbeamercolor{normal text}{fg=colour-text}
    \setbeamercolor{item}{fg=colour-text}
    \setbeamercolor{section in toc}{fg=colour-text}

    \setbeamercolor{page number in head/foot}{fg=colour-subtle}

    \setbeamercolor*{bibliography entry author}{fg=colour-text}
    \setbeamercolor*{bibliography entry note}{fg=colour-text}

}

\IfFileExists{tangle.sty}{\usepackage{tangle}}{}
\usepackage{bbm}
\usepackage{wasysym}
\usepackage{suffix}

\newcommand{\X}{\bb X}
\newcommand{\tX}{\u\X}
\renewcommand{\Cat}{\dc{Cat}}

\newcommand{\K}{\cl K}
\newcommand{\I}{\s I}
\newcommand{\Icat}{\star}
\newcommand{\tensor}{\otimes}
\newcommand{\tensorI}{I}
\newcommand{\RCmnd}{\b{RCmnd}}
\newcommand{\lh}[1]{\llbracket #1 \rrbracket}
\mathcommand{\L}{\textup{\sagittarius}}
\newcommand{\LMnd}{\L\Mnd}
\newcommand{\LRMnd}{\L\RMnd}

\newcommand{\jAE}{j \colon A \to E}

\newcommand{\jadj}{\radj j}
\newcommand{\ljr}{\ell \jadj r}

\newcommand{\ljrp}{\ell' \jadj r'}
\newcommand{\cp}[1]{\mathbin{\smallsmile_{#1}}}
\newcommand{\pc}[1]{\mathbin{\smallfrown_{#1}}}
\newcommand{\Act}{\b{Act}}
\newcommand{\VcoCat}{\V\co\h\Cat}

\newcommand{\Res}{\b{Res}}
\newcommand{\Kl}{\b{Kl}}

\newcommand{\aop}{{\rtimes}}
\newcommand{\eaop}{\aop}
\newcommand{\oop}{{\ltimes}}

\newcommand{\wc}{\ostar}
\newcommand{\wl}{\bulletstar}

\newcommand{\limcell}{\mu}
\newcommand{\lextcell}{\pi}
\newcommand{\lliftcell}{\eta}
\renewcommand{\defeq}{\mathrel{:=}}
\newcommand{\opcart}{\tx{opcart}}
\newcommand{\odotl}{\mathbin{{\odot}_L}}
\newcommand{\odotr}{\mathbin{{\odot}_R}}
\newcommand{\cart}{\tx{cart}}

\renewcommand{\EM}{\relax\ifmmode\b{EM}\else{}Eilenberg\nobreakdash--Moore\fi}

\newcommand{\rfP}[3]{\mathcal{P}{#1}(#2, #3)}
\newcommand{\reP}[3]{\mathcal{Q}{#1}(#2, #3)}
\mathcommand{\P}{\cl P}

\makeatletter
\@ifpackageloaded{tangle}{
\tgColours{F5A3A3,F5CCA3,F5F5A3,CCF5A3,A3F5A3,A3F5CC,A3F5F5,A3CCF5,A3A3F5,CCA3F5,F5A3F5,F5A3CC}

\newcommand{\tgRegionLabel}[2]{\begin{tikzpicture}[scale=0.8, every node/.style={transform shape}]
	\fill[color=\tgColour#1] (0,0) rectangle (1,1);
    \node at (0.5,0.5) {$#2$};
\end{tikzpicture}}
}{}
\makeatother

\newcounter{eqstep}
\newcounter{eqsubstep}[eqstep]
\newcommand{\nexteqstep}{\stepcounter{eqstep}}
\newcommand{\eqstepref}[1]{(\hyperlink{eqstep:#1}{#1})}

\makeatletter
\newcommand{\raisedtarget}[1]{\Hy@raisedlink{\hypertarget{#1}{}}}
\makeatother

\newcommand{\tangleeq}{\refstepcounter{eqsubstep}\raisedtarget{eqstep:\theeqstep.\theeqsubstep}{\mathclap{\overset{(\theeqstep.\theeqsubstep)}{=}}}}

\WithSuffix\newcommand\tangleeq*{\mathclap{=}}

\newcommand{\tangleeql}{\refstepcounter{eqsubstep}\raisedtarget{eqstep:\theeqstep.\theeqsubstep}{\mathllap{\overset{(\theeqstep.\theeqsubstep)}{=}}}}

\WithSuffix\newcommand\tangleeql*{\mathllap{=}}

\newenvironment{tangleeqs}{%
    \mathcommand{\=}{\tangleeq}
    \global\let\externaldblbackslash\\
    \csname gather*\endcsname
    \ifundef{\internaldblbackslash}{%
        \global\let\internaldblbackslash\\%
        \gdef\\{\internaldblbackslash\tangleeql}%
    }{}
    \nexteqstep
}{%
    \csname endgather*\endcsname
    \global\undef\internaldblbackslash
    \global\let\\\externaldblbackslash
}

\newenvironment{tangleeqs*}{%
    \mathcommand{\=}{\tangleeq*}
    \global\let\externaldblbackslash\\
    \csname gather*\endcsname
    \ifundef{\internaldblbackslash}{%
        \global\let\internaldblbackslash\\%
        \gdef\\{\internaldblbackslash\tangleeql*}%
    }{}
}{%
    \csname endgather*\endcsname
    \global\undef\internaldblbackslash
    \global\let\\\externaldblbackslash
}

\usepackage{sepfootnotes}
\newsymbolfootnotes{symbol}
\bibliography{references}

\newcommand{\M}{{\b M}}
\newcommand{\A}{{\b A}}

\title{The formal theory of relative monads}
\author{Nathanael Arkor}
\address{Department of Mathematics and Statistics, Faculty of Science, Masaryk University, Czech Republic}
\author{Dylan McDermott}
\address{Department of Computer Science, Reykjavik University, Iceland}
\datetoday

\keywords{Relative monad, relative adjunction, virtual double category, virtual equipment, skew-multicategory, skew-monoidal category, formal category theory, enriched category theory}
\subjclass[2020]{18D70,18D65,18C15,18C20,18A40,18D60,18D20,18N10,18M65,18M50}

\hypersetup{pdfauthor={Nathanael Arkor, Dylan McDermott}}

\begin{document}

\begin{abstract}
	We develop the theory of relative monads and relative adjunctions in a virtual equipment, extending the theory of monads and adjunctions in a 2-category. The theory of relative comonads and relative coadjunctions follows by duality. While some aspects of the theory behave analogously to the non-relative setting, others require new insights. In particular, the universal properties that define the algebra object and the opalgebra object for a monad in a virtual equipment are stronger than the classical notions of algebra object and opalgebra object for a monad in a 2-category. Inter alia, we prove a number of representation theorems for relative monads, establishing the unity of several concepts in the literature, including the devices of Walters, the $j$-monads of Diers, and the relative monads of Altenkirch, Chapman, and Uustalu. A motivating setting is the virtual equipment $\VCat$ of categories enriched in a monoidal category~$\V$, though many of our results are new even for $\V = \Set$.
\end{abstract}

\maketitle

\setcounter{tocdepth}{1}
\tableofcontents

\section{Introduction}

The definition of a monad, being 2-diagrammatic in nature -- expressed purely in terms of categories, functors, natural transformations, and equations therebetween -- may be internalised in any 2-category~\cite{benabou1967introduction}, and much of the theory of ordinary monads on categories continues to hold in this context~\cite{street1972formal,auderset1974adjonctions}. This permits a unified treatment of monads on ordinary categories, enriched categories, internal categories, and so on.

A monad on a category is in particular a structured \emph{endo}functor. It is natural to ask whether this restriction might be relaxed, permitting monads whose domains may be distinct from their codomains. This is precisely the notion of relative monad~\cite{altenkirch2010monads}. Given a fixed functor $\jAE$, a \emph{$j$-relative monad} comprises a functor $t \colon A \to E$ equipped with natural transformations -- the \emph{unit} $\eta \colon j \tto t$ and the \emph{extension operator} $\dag \colon E(j, t) \tto E(t, t)$ -- subject to laws expressing unitality and associativity. Much of the theory of monads extends, with appropriate modifications, to the context of relative monads.

Herein, we develop the theory of relative monads in a two-dimensional setting, analogous to the theory of monads in a 2-category. However, unlike the definition of a monad, the definition of a relative monad is not 2-diagrammatic: the extension operator $\dag \colon E(j, t) \tto E(t, t)$ involves a transformation between homs and cannot be captured by the structure of a 2-category. It is therefore necessary to work in a context for formal category theory, which axiomatises the structure of such transformations. In particular, we work within the context of a \ve{}~\cite{cruttwell2010unified}. While we work throughout at this level of generality, many of our results are new even in the classical setting of relative monads in $\Cat$. For instance, the following results are likely to be of interest even to readers who are not concerned with the formal aspects of the theory.

\begin{itemize}
	\item Relative monads are always monoids, permitting one to drop the left extension existence assumptions of \cite{altenkirch2010monads,altenkirch2015monads}, provided one is willing to work with skew-multicategories rather than skew-monoidal categories (\cref{relative-monads-are-tight-monoids,Xj1-is-skew-monoidal}).
	\item Relative adjunctions may be presented by means of a unit and a counit, in addition to the classical isomorphism of hom-sets (\cref{reformulations-of-relative-adjunction}).
	\item Left relative adjoints may be computed by (pointwise) left lifts (\cref{left-adjoint-is-left-lift}), and right relative adjoints by (pointwise) left extensions (\cref{right-adjoint-vs-absolute-left-extension}).
	\item Relative monads and relative adjunctions may be composed with suitable relative adjunctions (\cref{relative-adjunction-composition,relative-monad-relative-adjunction-composition}), recovering several known constructions of relative monads and relative adjunctions.
	\item In addition to forming initial and terminal resolutions, the Kleisli and \EM{} categories for a relative monad satisfy stronger universal properties with respect to morphisms of relative adjunctions (\cref{algebra-objects-induce-j-monadic-resolutions,opalgebra-objects-induce-j-opmonadic-resolutions}).
	\item Relative monads embed faithfully into categories of slices and coslices via their Kleisli and \EM{} constructions (\cref{relative-monads-form-full-subcategory-of-slices,relative-monads-form-full-subcategory-of-coslices}).
	\item The Kleisli categories for arbitrary relative monads may be constructed from Kleisli categories for trivial relative monads (\cref{coincidence-of-opalgebra-objects}).
\end{itemize}

As part of our development, we prove a number of representation theorems for relative monads (\cref{relative-monads-are-tight-monoids,relative-monads-are-loose-relative-monads,relative-monads-are-loose-monads,j-dense-rmnd,Xj1-is-skew-monoidal}). In doing so, we unify several concepts that have arisen in the category theory literature, such as the \emph{devices} of \textcite{walters1969alternative,walters1970categorical}, the \emph{$j$-monads} of \textcite{diers1975jmonades}, and the \emph{relative monads} and \emph{skew monoids} of \textcite{altenkirch2010monads,altenkirch2015monads} (\cref{examples-of-enriched-relative-monads}).

\subsection{Outline of the paper}

In \cref{virtual-double-categories} we recall the definition of \ve{}~\cite{cruttwell2010unified}, and in \cref{formal-category-theory} develop some basic category theory in this setting, such as the theory of weighted limits and colimits, pointwise extensions and lifts, and \ffness{} and density.

In \cref{skew-multicategorical-hom-categories}, we introduce relative monads (\cref{relative-monad}).
We motivate the definition by identifying relative monads with monoids in a skew-multicategory structure on the hom-categories of a \ve{} (\cref{relative-monads-are-tight-monoids}), which we introduce in \cref{skew-multicategorical-hom}. We furthermore establish a number of equivalent definitions of relative monad (\cref{relative-monads-are-loose-relative-monads,relative-monads-are-loose-monads,Xj1-is-skew-monoidal}), recovering notions of monad-like structures that have arisen in the literature.
In \cref{relative-adjunctions}, we introduce relative adjunctions (\cref{relative-adjunction}), giving several equivalent characterisations akin to those for (non-relative) adjunctions (\cref{reformulations-of-relative-adjunction}), establish their limit and colimit preservation properties (\cref{left-adjoints-preserve-colimits,right-adjoints-preserve-limits}), and explain their relation to relative monads.
In \cref{algebras-and-opalgebras}, we introduce algebras (\cref{algebra}) and opalgebras (\cref{opalgebra}) for relative monads as left- and right-actions of monoids in skew-multicategories, and consider universal algebras (\cref{algebra-object}) and opalgebras (\cref{opalgebra-object}), which generalise the notions of algebra object (\aka \EM{} object) and opalgebra object (\aka Kleisli object) for a monad in a 2-category. In particular, we prove that every algebra object forms a terminal resolution (\cref{algebra-object-is-j-monadic}), and that every opalgebra object forms an initial resolution (\cref{opalgebra-object-is-j-opmonadic}).
In \cref{duality}, we briefly discuss the dual theory of relative comonads and relative coadjunctions.

Finally, in \cref{relative-monads-in-VCat}, we consider the special case of relative monads and relative adjunctions in the \ve{} $\VCat$ of categories enriched in a monoidal category $\V$.
In particular, we show that the definition of relative monad in that setting may be simplified (\cref{enriched-relative-monad}), and construct (co)algebra objects (\cref{VCat-admits-algebra-objects,VCat-admits-coalgebra-objects}) and (co)opalgebra objects (\cref{VCat-admits-opalgebra-objects,VCat-admits-coopalgebra-objects}). Previous notions of enriched relative monad and relative adjunction in the literature are recovered as special cases.

\subsection{Deferrals}

It is worth highlighting some aspects of the formal theory of relative monads we have chosen not to pursue in this paper.

First, in this paper, we study 1-categories of relative monads -- namely, the 1-category of $j$-relative monads for a fixed root $\jAE$ -- and do not consider the two-dimensional structure formed by relative monads with different roots. This is in contrast to the seminal paper of \citeauthor{street1972formal} on the formal theory of monads~\cite{street1972formal}. There are two reasons for this choice. The first is that we are motivated by applications for which the root $j$ is fixed; and the second is that, contrary to morphisms of monads, the appropriate definition of morphism between arbitrary relative monads is nonevident.

Second, we do not consider the relationship between relative monads and non-relative monads, or, more generally, between relative monads with different roots, as studied by \textcite{walters1970categorical} and \textcite{altenkirch2010monads,altenkirch2015monads}. While this is an essential aspect of the theory of relative monads, it has been omitted from the present paper for reasons of space.

Third, though we focus herein only on enriched relative monads, there are several examples of structures resembling relative monads that we expect may be seen as relative monads in particular equipments, such as the \emph{strong relative monads} of \textcite{uustalu2010strong}; the \emph{enriched abstract clones} of \textcite[Definition~1.1]{fiore2017concrete}; and the \emph{relative monads} of \cite[Definition~2.1]{lobbia2023distributive}.

These aspects, and others, shall be developed in forthcoming work.

\subsection{Related work}

The study of relative monads in a formal setting has been previously proposed. \textcite{maillard2019principles,arkor2022monadic} independently defined relative monads in a representable \ve{} (a \emph{proarrow equipment} in the sense of \citeauthor{wood1982abstract}~\cite{wood1982abstract,wood1985proarrows}): their definition coincides with ours in that setting. However, our treatment is more general, and addresses several deficiencies with these previous approaches: we give a more detailed comparison throughout.

A different approach was proposed by \textcite{lobbia2023distributive}, who defined a notion of relative monad in any 2-category, generalising the \emph{extension systems} in a 2-category defined by \textcite{marmolejo2010monads}. While it is possible to capture relative monads for ordinary and internal categories in this setting, it is not possible to capture relative monads for enriched categories, and therefore is inadequate for our purposes.

Our motivation lies in the theory of relative monads in a virtual equipment, which in particular subsumes the theory of monads in a pseudo double category with companions and conjoints. Most of our results appear to be new even in the latter setting. We note that the theory of monads in a pseudo double category studied by \textcite{fiore2011monads,fiore2012double} is a theory of \emph{loose-monads} (\cref{loose-monad}) -- rather than of \emph{tight-monads} (\cref{tight-monad}), which is our concern -- and so is orthogonal to our development. We intend to explore the relation between the two notions more thoroughly in future work.

\subsection{Acknowledgements}

The authors thank John Bourke, Gabriele Lobbia, and Tarmo Uustalu for discussions about relative monads and skew-multicategories; Christian Williams for introducing the authors to string diagrams for double categories, which simplify many of the proofs; Paul Blain Levy and Morgan Rogers for pointing out an oversight in the definition of associative-normal left-skew-multicategory; and Dot Twocubes for pointing out an oversight in the definition of unitality therefor. The paper has benefitted from comments by Marcelo Fiore, Richard Garner, and Martin Hyland on an earlier development of the theory~\cite{arkor2022monadic}. The authors are also deeply grateful to the anonymous reviewer for their careful reading of the paper.
The second author was supported by Icelandic Research Fund grant \textnumero\,228684-052.

\section{Virtual equipments}
\label{virtual-double-categories}

There are many flavours of category theory -- enriched, internal, indexed and fibred, and so on -- each of which admits much of the same theory as ordinary category theory, such as the study of limits and colimits, adjunctions and monads, presheaves, pointwise extensions, and so on. To avoid the repetition inherent in proving the same theorems in each setting -- for instance, that every adjunction induces a monad, or that left adjoints preserve colimits -- it is desirable to work in a general context in which (1) these theorems may be proven, and for which (2) each of these flavours of category theory is merely an example. This is the study of \emph{formal category theory}~\cite{gray1974formal}.

A fundamental question then arises: what is an appropriate setting for formal category theory? In other words: what structure of categories is fundamental to their study? An evident choice is the 2-categorical structure possessed by categories, functors, and natural transformations, and early attempts to study formal category theory took place in the setting of 2-categories equipped with various property-like structure~\cite{gray1974formal,street1974fibrations,street1974elementary}. This setting is apt for studying some kinds of categorical structure, in particular monads and adjunctions~\cite{street1972two,street1972formal}, which are essentially 2-categorical in nature. However, it was clear from the beginning that this setting was not expressive enough to capture many fundamental concepts in enriched category theory.

The shortcoming with 2-categories as a setting for formal category theory is the absence of a notion of \emph{hom} (such as hom-sets for ordinary categories, or hom-objects for enriched categories), which are crucial in defining concepts such as weighted limits and colimits, presheaves, pointwise extensions, and (crucially for our purposes) relative monads and relative adjunctions. While in some settings (notably for internal categories), homs may be captured faithfully using comma objects, justifying the use of 2-categories in these cases, this is not possible for enriched categories. Instead, homs must be provided as extra structure on a 2-category: this was the central insight of \textcite{street1978yoneda}, who introduced \emph{Yoneda structures} as a setting for formal category theory that captures enriched categories in addition to internal categories. A Yoneda structure axiomatises the presheaf construction together with the existence of nerves for suitably small functors. However, a shortcoming of the notion of Yoneda structure is that there are flavours of category theory that do not admit a presheaf construction: for instance, $\V$-enriched category theory for non-closed monoidal categories $\V$.

Shortly following the paper of \citeauthor{street1978yoneda}, \textcite{wood1982abstract} introduced \emph{proarrow equipments} as a simplification of Yoneda structures. Proarrow equipments axiomatise the structure of distributors (also called profunctors or (bi)modules), rather than the presheaf construction. A distributor from $A$ to $B$, denoted $A \lto B$, is simply a functor $B\op \times A \to \Set$. Distributors capture the structure of the hom-sets of a category: for every locally small category $A$, the Yoneda embedding forms a distributor $A({-}_1, {-}_2) \colon A\op \times A \to \Set$ (in fact, this forms the identity distributor on $A$). Every Yoneda structure induces a proarrow equipment by considering a distributor to be a 1-cell into a presheaf object, and in this sense proarrow equipments generalise Yoneda structures. Furthermore, since the existence of a presheaf construction is not required, proarrow equipments capture more general bases of enrichment than Yoneda structures.

However, the setting of proarrow equipments is not quite general enough to capture $\V$-enriched category theory for arbitrary $\V$. In particular, to compose distributors requires sufficient colimits in $\V$, which may not exist in general. This motivated \textcite{cruttwell2010unified} to introduce \emph{\ve{}s}, which are a generalisation of proarrow equipments that do not require the existence of composite distributors. In contrast to previous approaches, virtual equipments are general enough to capture enriched category for arbitrary bases of enrichment. For this reason, we view it as the appropriate setting in which to develop formal category theory, and it is the setting in which we work.

Our main example is the virtual equipment $\VCat$ of categories enriched in a monoidal category $\V$, which we discuss in \cref{relative-monads-in-VCat}. In particular, \cref{relative-monads-in-VCat} serves as a case study explaining how the general theory we present may be instantiated in concrete examples.

\subsection{Virtual double categories}

A \ve{} is in particular a \vdc{}, so we begin by recalling the definition and introducing the notation we shall use. A \vdc{} is a generalisation of a pseudo double category whose morphisms in one axis (the \emph{loose} axis) do not necessarily have composites, and whose morphisms in the other axis (the \emph{tight} axis) compose strictly. We shall employ a string diagram notation for \vdcs{} and equipments, which aids the readability of diagrammatic proofs. Our notation is based on that of \textcite{myers2020yoneda,myers2016string}, though we have made some alterations. For the convenience of readers unfamiliar with string diagrams, we generally present definitions in terms both of pasting diagrams and of string diagrams, but use either as convenient in proofs.

\begin{definition}[{\cites[61]{burroni1971tcategories}[Definition~1]{leinster2002generalized}[Definition~2.1]{cruttwell2010unified}}]
    A \emph{\vdc{}} comprises the following data.
    \begin{enumerate}
        \item A category of \emph{objects} and \emph{tight-cells}.
		We will occasionally elide object names where unimportant in pasting diagrams, denoting each (potentially distinct) object by a point ($\hspace{.3em}\cdot\hspace{.3em}$).
		In string diagrammatic notation, we denote an object by a region, such as the following.
        \begin{center}
			\tgRegionLabel{6}{A}
			\tgRegionLabel{10}{B}
			\tgRegionLabel{2}{C}
			\tgRegionLabel{0}{D}
			\tgRegionLabel{4}{E}
		\end{center}
		In practice, we elide the object names in string diagrams, which may be inferred from context.
		To aid readability, we will often colour regions, using a different colour for each object.
		The colours are not essential for interpreting the string diagrams.

		We denote a tight-cell $f$ from an object $A$ to an object $B$ by an arrow $f \colon A \to B$; denote the composition of tight-cells $f \colon A \to B$ and $g \colon B \to C$ both by $(f \d g) \colon A \to C$ and by $g f \colon A \to C$; and denote the identity of an object $A$ by $1_A \colon A \to A$, or simply by $=$ in pasting diagrams. In string diagrammatic notation, we denote a tight-cell $f \colon A \to B$ by a horizontal line decorated with an arrow. (The purpose of the arrow will be explained in \cref{cartesian-cell}.)
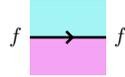
		\[
		\begin{tangle}{(1,1)}
			\tgBorderA{(0,0)}{\tgColour6}{\tgColour6}{\tgColour10}{\tgColour10}
			\tgArrow{(0,0)}{0}
			\tgAxisLabel{(1,0.5)}{west}{f}
			\tgAxisLabel{(0,0.5)}{east}{f}
		\end{tangle}
		\]
		Composition of tight-cells $(f \d g)$ is denoted by vertical conjunction.
		
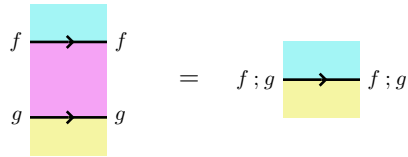
\[
		\begin{tangle}{(1,2)}
			\tgBorderA{(0,0)}{\tgColour6}{\tgColour6}{\tgColour10}{\tgColour10}
			\tgArrow{(0,0)}{0}
			\tgBorderA{(0,1)}{\tgColour10}{\tgColour10}{\tgColour2}{\tgColour2}
			\tgArrow{(0,1)}{0}
			\tgAxisLabel{(1,0.5)}{west}{f}
			\tgAxisLabel{(0,0.5)}{east}{f}
			\tgAxisLabel{(1,1.5)}{west}{g}
			\tgAxisLabel{(0,1.5)}{east}{g}
		\end{tangle}
		\quad = \quad\;\,
		\begin{tangle}{(1,1)}
			\tgBorderA{(0,0)}{\tgColour6}{\tgColour6}{\tgColour2}{\tgColour2}
			\tgArrow{(0,0)}{0}
			\tgAxisLabel{(1,0.5)}{west}{f \d g}
			\tgAxisLabel{(0,0.5)}{east}{f \d g}
		\end{tangle}
		\]
		Identity tight-cells are implicit in string diagrams.
        \item For each pair of objects $A$ and $B$, a class of \emph{loose-cells} from $A$ to $B$. We denote a loose-cell $p$ from $A$ to $B$ by an arrow with a vertical stroke $p \colon A \lto B$. In string diagrammatic notation, we denote such a loose-cell by a vertical line, as follows. Note that the region on the right corresponds to the object $A$, while the region on the left corresponds to the object $B$. One may therefore be inclined to instead write $p \colon B \lfrom A$, though we shall generally not do so here.
        
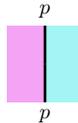
\[
		\begin{tangle}{(1,1)}
			\tgBorderA{(0,0)}{\tgColour10}{\tgColour6}{\tgColour6}{\tgColour10}
			\tgAxisLabel{(0.5,1)}{north}{p}
			\tgAxisLabel{(0.5,0)}{south}{p}
		\end{tangle}
		\]
        \item For each chain of loose-cells $p_1, \ldots, p_n$ ($n \geq 0$) and compatible tight-cells $f_0, f_n$ and loose-cell $q$ (together forming a \emph{frame}), a class of 2-cells with the given frame.
		
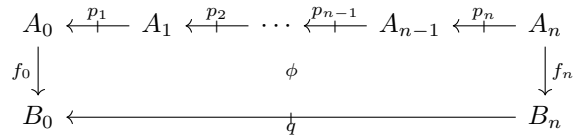
\[\begin{tikzcd}
			{A_0} & {A_1} & \cdots & {A_{n - 1}} & {A_n} \\
			{B_0} &&&& {B_n}
			\arrow["q", "\shortmid"{marking}, from=2-5, to=2-1]
			\arrow["{p_n}"', "\shortmid"{marking}, from=1-5, to=1-4]
			\arrow["{p_1}"', "\shortmid"{marking}, from=1-2, to=1-1]
			\arrow["{p_{n - 1}}"', "\shortmid"{marking}, from=1-4, to=1-3]
			\arrow["{p_2}"', "\shortmid"{marking}, from=1-3, to=1-2]
			\arrow[""{name=0, anchor=center, inner sep=0}, "{f_n}", from=1-5, to=2-5]
			\arrow[""{name=1, anchor=center, inner sep=0}, "{f_0}"', from=1-1, to=2-1]
			\arrow["\phi"{description}, draw=none, from=0, to=1]
		\end{tikzcd}\]
		In string diagrammatic notation, we denote such a 2-cell by a bubble as follows.
		
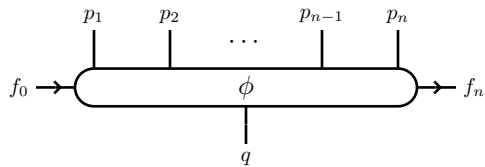
\[
		\begin{tangle}{(7,3)}[trim x,trim y]
			\tgBlank{(0,0)}{white}
			\tgBorderA{(1,0)}{white}{white}{white}{white}
			\tgBorder{(1,0)}{1}{0}{1}{0}
			\tgBorderA{(2,0)}{white}{white}{white}{white}
			\tgBorder{(2,0)}{1}{0}{1}{0}
			\tgBlank{(3,0)}{white}
			\tgBorderA{(4,0)}{white}{white}{white}{white}
			\tgBorder{(4,0)}{1}{0}{1}{0}
			\tgBorderA{(5,0)}{white}{white}{white}{white}
			\tgBorder{(5,0)}{1}{0}{1}{0}
			\tgBlank{(6,0)}{white}
			\tgBorderA{(0,1)}{white}{white}{white}{white}
			\tgBorder{(0,1)}{0}{1}{0}{1}
			\tgBorderA{(1,1)}{white}{white}{white}{white}
			\tgBorder{(1,1)}{1}{1}{0}{1}
			\tgBorderA{(2,1)}{white}{white}{white}{white}
			\tgBorder{(2,1)}{1}{1}{0}{1}
			\tgBorderA{(3,1)}{white}{white}{white}{white}
			\tgBorder{(3,1)}{0}{1}{1}{1}
			\tgBorderA{(4,1)}{white}{white}{white}{white}
			\tgBorder{(4,1)}{1}{1}{0}{1}
			\tgBorderA{(5,1)}{white}{white}{white}{white}
			\tgBorder{(5,1)}{1}{1}{0}{1}
			\tgBorderA{(6,1)}{white}{white}{white}{white}
			\tgBorder{(6,1)}{0}{1}{0}{1}
			\tgBlank{(0,2)}{white}
			\tgBlank{(1,2)}{white}
			\tgBlank{(2,2)}{white}
			\tgBorderA{(3,2)}{white}{white}{white}{white}
			\tgBorder{(3,2)}{1}{0}{1}{0}
			\tgBlank{(4,2)}{white}
			\tgBlank{(5,2)}{white}
			\tgBlank{(6,2)}{white}
			\tgCell[(4,0)]{(3,1)}{\phi}
			\tgArrow{(0.5,1)}{0}
			\tgArrow{(5.5,1)}{0}
			\tgAxisLabel{(1.5,0.75)}{south}{p_1}
			\tgAxisLabel{(2.5,0.75)}{south}{p_2}
			\tgAxisLabel{(4.5,0.75)}{south}{p_{n - 1}}
			\tgAxisLabel{(5.5,0.75)}{south}{p_n}
			\tgAxisLabel{(0.75,1.5)}{east}{f_0}
			\tgAxisLabel{(6.25,1.5)}{west}{f_n}
			\tgAxisLabel{(3.5,2.25)}{north}{q}
			\node at (3.5,.9) {$\cdots$};
		\end{tangle}
		\]
		Observe that both our pasting diagrams and our string diagrams are written from right-to-left (matching nondiagrammatic composition order of loose-cells).

		When the tight-cells in the frame of the 2-cell $\phi$ are identities, we say that $\phi$ is \emph{globular} and denote it by $\phi \colon p_1, \ldots, p_n \tto q$. We shall have no need to denote more general 2-cells nondiagrammatically.
		\[\begin{tikzcd}
			\cdot & \cdot & \cdots & \cdot & \cdot \\
			\cdot &&&& \cdot
			\arrow[""{name=0, anchor=center, inner sep=0}, Rightarrow, no head, from=2-5, to=1-5]
			\arrow["q", "\shortmid"{marking}, from=2-5, to=2-1]
			\arrow["{p_n}"', "\shortmid"{marking}, from=1-5, to=1-4]
			\arrow[""{name=1, anchor=center, inner sep=0}, Rightarrow, no head, from=2-1, to=1-1]
			\arrow["{p_1}"', "\shortmid"{marking}, from=1-2, to=1-1]
			\arrow["{p_{n - 1}}"', "\shortmid"{marking}, from=1-4, to=1-3]
			\arrow["{p_2}"', "\shortmid"{marking}, from=1-3, to=1-2]
			\arrow["\phi"{description}, draw=none, from=0, to=1]
		\end{tikzcd}\]
		In pasting diagrammatic notation, we denote a nullary 2-cell by a square of the following form.
		\[\begin{tikzcd}
			A & A \\
			B & {B'}
			\arrow["p", "\shortmid"{marking}, from=2-2, to=2-1]
			\arrow[""{name=0, anchor=center, inner sep=0}, "{f'}", from=1-2, to=2-2]
			\arrow[""{name=1, anchor=center, inner sep=0}, "f"', from=1-1, to=2-1]
			\arrow[Rightarrow, no head, from=1-2, to=1-1]
			\arrow["\phi"{description}, draw=none, from=0, to=1]
		\end{tikzcd}\]
        \item For every configuration of 2-cells of the following shape,
		\[\begin{tikzcd}
			\cdot & \cdots & \cdot & \cdots & \cdot & \cdots & \cdot \\
			\cdot && \cdot & \cdots & \cdot && \cdot \\
			\cdot &&&&&& \cdot
			\arrow["\shortmid"{marking}, from=2-7, to=2-5]
			\arrow[""{name=0, anchor=center, inner sep=0}, "{f'}", from=1-7, to=2-7]
			\arrow[""{name=1, anchor=center, inner sep=0}, from=1-5, to=2-5]
			\arrow["{p_{m_n}^n}"', "\shortmid"{marking}, from=1-7, to=1-6]
			\arrow["{p_{1}^n}"', "\shortmid"{marking}, from=1-6, to=1-5]
			\arrow["{p_{m_1}^1}"', "\shortmid"{marking}, from=1-3, to=1-2]
			\arrow["{p_1^1}"', "\shortmid"{marking}, from=1-2, to=1-1]
			\arrow["\shortmid"{marking}, from=2-3, to=2-1]
			\arrow[""{name=2, anchor=center, inner sep=0}, from=1-3, to=2-3]
			\arrow[""{name=3, anchor=center, inner sep=0}, "f"', from=1-1, to=2-1]
			\arrow[""{name=4, anchor=center, inner sep=0}, "{g'}", from=2-7, to=3-7]
			\arrow["q", "\shortmid"{marking}, from=3-7, to=3-1]
			\arrow[""{name=5, anchor=center, inner sep=0}, "g"', from=2-1, to=3-1]
			\arrow["{\phi_n}"{description}, draw=none, from=0, to=1]
			\arrow["{\phi_1}"{description}, draw=none, from=2, to=3]
			\arrow["\psi"{description}, draw=none, from=4, to=5]
		\end{tikzcd}\]
		a 2-cell,
		\[\begin{tikzcd}
			\cdot & \cdots & \cdot \\
			\cdot && \cdot
			\arrow["{p_{m_n}^n}"', "\shortmid"{marking}, from=1-3, to=1-2]
			\arrow["q", "\shortmid"{marking}, from=2-3, to=2-1]
			\arrow[""{name=0, anchor=center, inner sep=0}, "{f' \d g'}", from=1-3, to=2-3]
			\arrow[""{name=1, anchor=center, inner sep=0}, "{f \d g}"', from=1-1, to=2-1]
			\arrow["{p_1^1}"', "\shortmid"{marking}, from=1-2, to=1-1]
			\arrow["{(\phi_1, \ldots, \phi_n) \d \psi}"{description}, draw=none, from=0, to=1]
		\end{tikzcd}\]
		the \emph{composite}.
		In string diagrammatic notation, composition of loose-cells is given by conjunction.
		
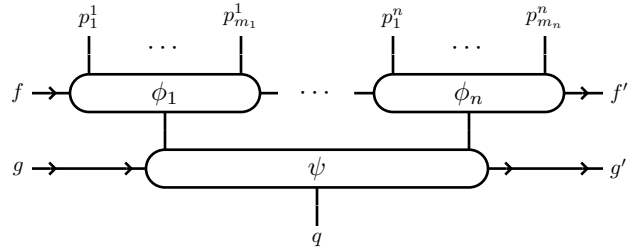
\[
		\begin{tangle}{(9,4)}[trim x,trim y]
			\tgBlank{(0,0)}{white}
			\tgBorderA{(1,0)}{white}{white}{white}{white}
			\tgBorder{(1,0)}{1}{0}{1}{0}
			\tgBlank{(2,0)}{white}
			\tgBorderA{(3,0)}{white}{white}{white}{white}
			\tgBorder{(3,0)}{1}{0}{1}{0}
			\tgBlank{(4,0)}{white}
			\tgBorderA{(5,0)}{white}{white}{white}{white}
			\tgBorder{(5,0)}{1}{0}{1}{0}
			\tgBlank{(6,0)}{white}
			\tgBorderA{(7,0)}{white}{white}{white}{white}
			\tgBorder{(7,0)}{1}{0}{1}{0}
			\tgBlank{(8,0)}{white}
			\tgBorderA{(0,1)}{white}{white}{white}{white}
			\tgBorder{(0,1)}{0}{1}{0}{1}
			\tgBorderA{(1,1)}{white}{white}{white}{white}
			\tgBorder{(1,1)}{1}{1}{0}{1}
			\tgBorderA{(2,1)}{white}{white}{white}{white}
			\tgBorder{(2,1)}{0}{1}{1}{1}
			\tgBorderA{(3,1)}{white}{white}{white}{white}
			\tgBorder{(3,1)}{1}{1}{0}{1}
			\tgBlank{(4,1)}{white}
			\tgBorderA{(5,1)}{white}{white}{white}{white}
			\tgBorder{(5,1)}{1}{1}{0}{1}
			\tgBorderA{(6,1)}{white}{white}{white}{white}
			\tgBorder{(6,1)}{0}{1}{1}{1}
			\tgBorderA{(7,1)}{white}{white}{white}{white}
			\tgBorder{(7,1)}{1}{1}{0}{1}
			\tgBorderA{(8,1)}{white}{white}{white}{white}
			\tgBorder{(8,1)}{0}{1}{0}{1}
			\tgBorderA{(0,2)}{white}{white}{white}{white}
			\tgBorder{(0,2)}{0}{1}{0}{1}
			\tgBorderA{(1,2)}{white}{white}{white}{white}
			\tgBorder{(1,2)}{0}{1}{0}{1}
			\tgBorderA{(2,2)}{white}{white}{white}{white}
			\tgBorder{(2,2)}{1}{1}{0}{1}
			\tgBorderA{(3,2)}{white}{white}{white}{white}
			\tgBorder{(3,2)}{0}{1}{0}{1}
			\tgBorderA{(4,2)}{white}{white}{white}{white}
			\tgBorder{(4,2)}{0}{1}{1}{1}
			\tgBorderA{(5,2)}{white}{white}{white}{white}
			\tgBorder{(5,2)}{0}{1}{0}{1}
			\tgBorderA{(6,2)}{white}{white}{white}{white}
			\tgBorder{(6,2)}{1}{1}{0}{1}
			\tgBorderA{(7,2)}{white}{white}{white}{white}
			\tgBorder{(7,2)}{0}{1}{0}{1}
			\tgBorderA{(8,2)}{white}{white}{white}{white}
			\tgBorder{(8,2)}{0}{1}{0}{1}
			\tgBlank{(0,3)}{white}
			\tgBlank{(1,3)}{white}
			\tgBlank{(2,3)}{white}
			\tgBlank{(3,3)}{white}
			\tgBorderA{(4,3)}{white}{white}{white}{white}
			\tgBorder{(4,3)}{1}{0}{1}{0}
			\tgBlank{(5,3)}{white}
			\tgBlank{(6,3)}{white}
			\tgBlank{(7,3)}{white}
			\tgBlank{(8,3)}{white}
			\tgCell[(2,0)]{(2,1)}{\phi_1}
			\tgCell[(4,0)]{(4,2)}{\psi}
			\tgCell[(2,0)]{(6,1)}{\phi_n}
			\tgArrow{(0.5,1)}{0}
			\tgArrow{(0.5,2)}{0}
			\tgArrow{(1.5,2)}{0}
			\tgArrow{(7.5,1)}{0}
			\tgArrow{(6.5,2)}{0}
			\tgArrow{(7.5,2)}{0}
			\tgAxisLabel{(1.5,0.75)}{south}{p_1^1}
			\tgAxisLabel{(3.5,0.75)}{south}{p_{m_1}^1}
			\tgAxisLabel{(5.5,0.75)}{south}{p_{1}^n}
			\tgAxisLabel{(7.5,0.75)}{south}{p_{m_n}^n}
			\tgAxisLabel{(0.75,1.5)}{east}{f}
			\tgAxisLabel{(8.25,1.5)}{west}{f'}
			\tgAxisLabel{(0.75,2.5)}{east}{g}
			\tgAxisLabel{(8.25,2.5)}{west}{g'}
			\tgAxisLabel{(4.5,3.25)}{north}{q}
			\node at (2.5,.9) {$\cdots$};
			\node at (6.5,.9) {$\cdots$};
			\node at (4.5,1.5) {$\cdots$};
		\end{tangle}
		\]
        \item For each loose-cell $p \colon A' \lto A$, a 2-cell $1_p \colon p \tto p$, the \emph{identity} of $p$, denoted simply as $=$ in pasting diagrams.
        \[\begin{tikzcd}
        	A & {A'} \\
        	A & {A'}
        	\arrow[""{name=0, anchor=center, inner sep=0}, Rightarrow, no head, from=1-1, to=2-1]
        	\arrow["p"', "\shortmid"{marking}, from=1-2, to=1-1]
        	\arrow[""{name=1, anchor=center, inner sep=0}, Rightarrow, no head, from=1-2, to=2-2]
        	\arrow["p", "\shortmid"{marking}, from=2-2, to=2-1]
        	\arrow["{=}"{description}, draw=none, from=1, to=0]
        \end{tikzcd}\]
		Identity 2-cells are implicit in string diagrams.
    \end{enumerate}
	We shall sometimes denote by
	\[\begin{tikzcd}
		\cdot & \cdots & \cdot \\
		\cdot & \cdots & \cdot
		\arrow["\shortmid"{marking}, from=1-3, to=1-2]
		\arrow["\shortmid"{marking}, from=1-2, to=1-1]
		\arrow["\shortmid"{marking}, from=2-3, to=2-2]
		\arrow["\shortmid"{marking}, from=2-2, to=2-1]
		\arrow[""{name=0, anchor=center, inner sep=0}, Rightarrow, no head, from=1-3, to=2-3]
		\arrow[""{name=1, anchor=center, inner sep=0}, Rightarrow, no head, from=1-1, to=2-1]
		\arrow["{=}"{description}, draw=none, from=0, to=1]
	\end{tikzcd}\]
	the juxtaposition of identity 2-cells
	\[\begin{tikzcd}
		\cdot & \cdot & \cdot & \cdot \\
		\cdot & \cdot & \cdot & \cdot
		\arrow["\shortmid"{marking}, from=1-4, to=1-3]
		\arrow["\shortmid"{marking}, from=2-2, to=2-1]
		\arrow[""{name=0, anchor=center, inner sep=0}, Rightarrow, no head, from=1-4, to=2-4]
		\arrow[""{name=1, anchor=center, inner sep=0}, Rightarrow, no head, from=1-1, to=2-1]
		\arrow["\shortmid"{marking}, from=1-2, to=1-1]
		\arrow["\shortmid"{marking}, from=2-4, to=2-3]
		\arrow["\cdots"{description}, draw=none, from=1-3, to=1-2]
		\arrow["\cdots"{description}, draw=none, from=2-3, to=2-2]
		\arrow[""{name=2, anchor=center, inner sep=0}, Rightarrow, no head, from=1-3, to=2-3]
		\arrow[""{name=3, anchor=center, inner sep=0}, Rightarrow, no head, from=1-2, to=2-2]
		\arrow["{=}"{description}, draw=none, from=0, to=2]
		\arrow["{=}"{description}, draw=none, from=3, to=1]
	\end{tikzcd}\]
    Composition of 2-cells is required to be associative and unital in the evident manner~\cite[Definition~2.1]{cruttwell2010unified}: associativity is implicit in pasting diagrams, and associativity and unitality are implicit in string diagrams.

    For a \vdc{} $\X$, we denote by $\X\lh{A, B}$ the category whose objects are loose-cells $p \colon A \lto B$ and whose morphisms are globular 2-cells $p \tto q$.
\end{definition}

\begin{remark}
	Our string diagram notation should be viewed formally as the dual of our pasting diagram notation: in particular, our string diagrams adhere strictly to a grid, which permits their direct interpretation as pasting diagrams without appeal to a topological argument as is traditional~\cite{joyal1991geometry,myers2020yoneda}.
\end{remark}

\begin{remark}
	\label{name-after-objects}
	We shall prefer to name \vdcs{} after their objects, rather than their loose-cells as is also common (for instance, in \cite{cruttwell2010unified}). This promotes the viewpoint that, just as there are often canonical notions of homomorphism and transformation of two-dimensional structures, so too is there often a canonical notion of (bi)module.
\end{remark}

Our motivating example of a \vdc{} will be the \vdc{} of $\V$-enriched categories, in which the tight-cells are $\V$-functors, the loose-cells are $\V$-distributors, and the 2-cells are $\V$-natural transformations. We defer an explicit definition to \cref{VCat}.

\subsubsection{Composites}

While loose-cells do not admit composites in general, a given \vdc{} may admit some composites, which are characterised by a universal property, analogous to the characterisation of tensors in multicategories.

\begin{definition}[{\cite[Definition~5.1]{cruttwell2010unified}}]
	\label{opcartesian}
    A 2-cell
	\[\begin{tikzcd}
		\cdot & \cdots & \cdot \\
		\cdot && \cdot
		\arrow["{q_m}"', "\shortmid"{marking}, from=1-3, to=1-2]
		\arrow["{q_1}"', "\shortmid"{marking}, from=1-2, to=1-1]
		\arrow["q", "\shortmid"{marking}, from=2-3, to=2-1]
		\arrow[""{name=0, anchor=center, inner sep=0}, Rightarrow, no head, from=1-3, to=2-3]
		\arrow[""{name=1, anchor=center, inner sep=0}, Rightarrow, no head, from=1-1, to=2-1]
		\arrow["\opcart"{description}, draw=none, from=0, to=1]
	\end{tikzcd}\]
    in a \vdc{} is \emph{opcartesian} if any 2-cell
	\[\begin{tikzcd}
		\cdot & \cdots & \cdot & \cdots & \cdot & \cdots & \cdot \\
		\cdot &&&&&& \cdot
		\arrow["{r_n}"', "\shortmid"{marking}, from=1-7, to=1-6]
		\arrow["{r_1}"', "\shortmid"{marking}, from=1-6, to=1-5]
		\arrow["{q_m}"', "\shortmid"{marking}, from=1-5, to=1-4]
		\arrow["{q_1}"', "\shortmid"{marking}, from=1-4, to=1-3]
		\arrow["{p_l}"', "\shortmid"{marking}, from=1-3, to=1-2]
		\arrow["{p_1}"', "\shortmid"{marking}, from=1-2, to=1-1]
		\arrow[""{name=0, anchor=center, inner sep=0}, "g", from=1-7, to=2-7]
		\arrow[""{name=1, anchor=center, inner sep=0}, "f"', from=1-1, to=2-1]
		\arrow["s", "\shortmid"{marking}, from=2-7, to=2-1]
		\arrow["\phi"{description}, draw=none, from=0, to=1]
	\end{tikzcd}\]
    factors uniquely therethrough:
	\[\begin{tikzcd}
		\cdot & \cdots & \cdot & \cdots & \cdot & \cdots & \cdot \\
		\cdot & \cdots & \cdot && \cdot & \cdots & \cdot \\
		\cdot &&&&&& \cdot
		\arrow["{r_n}"', "\shortmid"{marking}, from=1-7, to=1-6]
		\arrow["{r_1}"', "\shortmid"{marking}, from=1-6, to=1-5]
		\arrow["{q_m}"', "\shortmid"{marking}, from=1-5, to=1-4]
		\arrow["{q_1}"', "\shortmid"{marking}, from=1-4, to=1-3]
		\arrow["{p_l}"', "\shortmid"{marking}, from=1-3, to=1-2]
		\arrow["{p_1}"', "\shortmid"{marking}, from=1-2, to=1-1]
		\arrow["s", "\shortmid"{marking}, from=3-7, to=3-1]
		\arrow["q"{description}, from=2-5, to=2-3]
		\arrow[""{name=0, anchor=center, inner sep=0}, Rightarrow, no head, from=1-5, to=2-5]
		\arrow[""{name=1, anchor=center, inner sep=0}, "g", from=2-7, to=3-7]
		\arrow[""{name=2, anchor=center, inner sep=0}, "f"', from=2-1, to=3-1]
		\arrow[""{name=3, anchor=center, inner sep=0}, Rightarrow, no head, from=1-7, to=2-7]
		\arrow[""{name=4, anchor=center, inner sep=0}, Rightarrow, no head, from=1-1, to=2-1]
		\arrow["{r_n}"{description}, from=2-7, to=2-6]
		\arrow["{r_1}"{description}, from=2-6, to=2-5]
		\arrow["{p_l}"{description}, from=2-3, to=2-2]
		\arrow["{p_1}"{description}, from=2-2, to=2-1]
		\arrow[""{name=5, anchor=center, inner sep=0}, Rightarrow, no head, from=1-3, to=2-3]
		\arrow["{=}"{description}, draw=none, from=3, to=0]
		\arrow["{=}"{description}, draw=none, from=5, to=4]
		\arrow["\opcart"{description}, draw=none, from=0, to=5]
		\arrow["\check\phi"{description}, draw=none, from=1, to=2]
	\end{tikzcd}\]
    In this case, we call $q$ the \emph{(loose-) composite} $q_1 \odot \cdots \odot q_m$ of $q_1, \ldots, q_m$ (note that we write composites in the same order they appear in our string diagrams). To aid readability, we shall often elide the distinction between $\phi$ and $\check\phi$. In string diagrammatic notation, we denote the opcartesian 2-cell above by horizontal conjunction.
	\[
	\begin{tangle}{(3,1)}
		\tgBorderA{(0,0)}{white}{white}{white}{white}
		\tgBorder{(0,0)}{1}{0}{1}{0}
		\tgBlank{(1,0)}{white}
		\tgBorderA{(2,0)}{white}{white}{white}{white}
		\tgBorder{(2,0)}{1}{0}{1}{0}
		\tgAxisLabel{(0.5,1)}{north}{q_1}
		\tgAxisLabel{(0.5,0)}{south}{q_1}
		\tgAxisLabel{(2.5,1)}{north}{q_m}
		\tgAxisLabel{(2.5,0)}{south}{q_m}
		\node at (1.5,.5) {$\cdots$};
	\end{tangle}
	\]
	We denote by $\phi_1, \ldots, \phi_m$ a 2-cell of the following form (assuming the composite $q_1 \odot \cdots \odot q_m$ exists).
	\[\begin{tikzcd}
		\cdot & \cdot & \cdots & \cdot & \cdot \\
		\cdot & \cdot & \cdots & \cdot & \cdot \\
		\cdot &&&& \cdot
		\arrow["{q_1 \odot \cdots \odot q_m}", "\shortmid"{marking}, from=3-5, to=3-1]
		\arrow[""{name=0, anchor=center, inner sep=0}, Rightarrow, no head, from=2-5, to=3-5]
		\arrow[""{name=1, anchor=center, inner sep=0}, Rightarrow, no head, from=2-1, to=3-1]
		\arrow[""{name=2, anchor=center, inner sep=0}, from=1-5, to=2-5]
		\arrow["{q_1}"{description}, from=2-2, to=2-1]
		\arrow["\shortmid"{marking}, from=2-3, to=2-2]
		\arrow["\shortmid"{marking}, from=2-4, to=2-3]
		\arrow["{q_m}"{description}, from=2-5, to=2-4]
		\arrow[""{name=3, anchor=center, inner sep=0}, from=1-1, to=2-1]
		\arrow["\shortmid"{marking}, from=1-5, to=1-4]
		\arrow["\shortmid"{marking}, from=1-4, to=1-3]
		\arrow["\shortmid"{marking}, from=1-3, to=1-2]
		\arrow["\shortmid"{marking}, from=1-2, to=1-1]
		\arrow[""{name=4, anchor=center, inner sep=0}, from=1-4, to=2-4]
		\arrow[""{name=5, anchor=center, inner sep=0}, from=1-2, to=2-2]
		\arrow["\opcart"{description}, draw=none, from=0, to=1]
		\arrow["{\phi_m}"{description}, draw=none, from=2, to=4]
		\arrow["{\phi_1}"{description}, draw=none, from=5, to=3]
		\arrow["\cdots"{description}, draw=none, from=4, to=5]
	\end{tikzcd}\]
	When $m = 0$, we call $q \colon A \lto A$ the \emph{loose-identity} and denote it by $A(1, 1)$, or simply by $=\!\!\!|\!\!\!=$ in pasting diagrams. Identity loose-cells are implicit in string diagrams.

	We denote a nullary loose-cell with loose-identity codomain by $\phi \colon f \tto g$.
	\[\begin{tikzcd}
		\cdot & \cdot \\
		\cdot & \cdot
		\arrow["\shortmid"{marking}, Rightarrow, no head, from=2-2, to=2-1]
		\arrow[""{name=0, anchor=center, inner sep=0}, "g", from=1-2, to=2-2]
		\arrow[""{name=1, anchor=center, inner sep=0}, "f"', from=1-1, to=2-1]
		\arrow[Rightarrow, no head, from=1-2, to=1-1]
		\arrow["\phi"{description}, draw=none, from=0, to=1]
	\end{tikzcd}\]
    A \vdc{} is \emph{representable} when it admits all loose-composites (including loose-identities).
\end{definition}

Loose-composites are unique up to isomorphism and are essentially associative and unital. Representable \vdcs{} are equivalent to pseudo double categories~\cite[Theorem~5.2]{cruttwell2010unified}.
As an intuition for opcartesian 2-cells, observe that it does not make sense to ask whether a non-unary 2-cell in a \vdc{} is invertible, since 2-cells have unary codomain. Opcartesian 2-cells act as a universal unary approximant for a chain of loose-cells, and thus behave much as an invertible 2-cell would (in particular, a unary globular 2-cell is opcartesian if and only if it is invertible).

Due to our string diagram notation for opcartesian 2-cells, we may draw string diagrams that have multiple loose-cells at the bottom, but only when these loose-cells have a composite;
a pasting diagram corresponding to a string diagram having multiple loose-cells at the bottom has an opcartesian 2-cell at the bottom.

In some cases, we shall require a property weaker than opcartesianness (though stronger than the \emph{weakly opcartesian} 2-cells of \cite[Remark~5.8]{cruttwell2010unified}), reminiscent of left-representability of multicategories~\cite{bourke2018skew}.

\begin{definition}
	A 2-cell in a \vdc{} is \emph{left-opcartesian} if it satisfies the universal property of an opcartesian 2-cell (\cref{opcartesian}) only in the special case where $l = 0$ and $f$ is the identity. In this case, we call $q$ the \emph{left-composite} $q_1 \odotl \cdots \odotl q_m$ of $q_1, \ldots, q_m$.
    Dually, a 2-cell is \emph{right-opcartesian} if it satisfies the universal property only in the special case where $n = 0$ and $g$ is the identity. In this case, we call $q$ the \emph{right-composite} $q_1 \odotr \cdots \odotr q_m$.
\end{definition}

Left-composites are unique up to isomorphism and are left-associative in the sense that there is a canonical isomorphism,
\[q_1 \odotl q_2 \odotl \cdots \odotl q_m \iso (\cdots (q_1 \odotl q_2) \odotl \cdots) \odotl q_m\]
together with a canonical reassociating 2-cell.
\[q_1 \odotl q_2 \odotl \cdots \odotl q_m \tto q_1 \odotl (\cdots \odotl (q_{m - 1} \odotl q_m) \cdots)\]
Dual statements hold for right-composites. We reserve no notation for nullary left- or right-composites, since we shall not make use of them here (all loose-identities we consider have the full universal property of \cref{opcartesian}). Observe that loose-composites are, in particular, left- and right-composites (the converse is not true in general).

When a \vdc{} admits loose-identities, the tight-cells form a 2-category.
\begin{definition}[{\cite[Proposition~6.1]{cruttwell2010unified}}]
	Let $\X$ be a \vdc{} admitting loose-identities. Denote by $\tX$ the \emph{tight 2-category} associated to $\X$, having
	\begin{enumerate}
		\item objects: those of $\X$;
		\item 1-cells: tight-cells in $\X$;
		\item 2-cells $\phi \colon f \tto g$: nullary 2-cells with loose-identity codomain in $\X$ as follows.
		\[\begin{tikzcd}
			A & A \\
			B & B
			\arrow[""{name=0, anchor=center, inner sep=0}, "g", from=1-2, to=2-2]
			\arrow[""{name=1, anchor=center, inner sep=0}, "f"', from=1-1, to=2-1]
			\arrow["\shortmid"{marking}, Rightarrow, no head, from=2-2, to=2-1]
			\arrow[Rightarrow, no head, from=1-2, to=1-1]
			\arrow["\phi"{description}, draw=none, from=0, to=1]
		\end{tikzcd}\]
	\end{enumerate}
	Given objects $A$ and $B$ in such an $\X$, we denote by $\X[A, B]$ the hom-category $\tX(A, B)$.
\end{definition}

Identities and composition of 2-cells in $\tX$ are given by composition of 2-cells in $\X$ as follows.
\[\begin{tikzcd}
	A & A \\
	B & B \\
	B & B
	\arrow[Rightarrow, no head, from=1-1, to=1-2]
	\arrow[Rightarrow, no head, from=2-1, to=2-2]
	\arrow[""{name=0, anchor=center, inner sep=0}, "f"', from=1-1, to=2-1]
	\arrow[""{name=1, anchor=center, inner sep=0}, "f", from=1-2, to=2-2]
	\arrow[""{name=2, anchor=center, inner sep=0}, Rightarrow, no head, from=2-1, to=3-1]
	\arrow[""{name=3, anchor=center, inner sep=0}, Rightarrow, no head, from=2-2, to=3-2]
	\arrow["\shortmid"{marking}, Rightarrow, no head, from=3-1, to=3-2]
	\arrow["{=}"{description}, draw=none, from=0, to=1]
	\arrow["\opcart"{description}, draw=none, from=2, to=3]
\end{tikzcd}
\hspace{4em}
\begin{tikzcd}
	A & A & A \\
	B & B & B \\
	B && B
	\arrow[""{name=0, anchor=center, inner sep=0}, "h", from=1-3, to=2-3]
	\arrow[""{name=1, anchor=center, inner sep=0}, "g"{description}, from=1-2, to=2-2]
	\arrow[""{name=2, anchor=center, inner sep=0}, "f"', from=1-1, to=2-1]
	\arrow["\shortmid"{marking}, Rightarrow, no head, from=2-3, to=2-2]
	\arrow["\shortmid"{marking}, Rightarrow, no head, from=2-2, to=2-1]
	\arrow[Rightarrow, no head, from=1-3, to=1-2]
	\arrow[Rightarrow, no head, from=1-2, to=1-1]
	\arrow["\shortmid"{marking}, Rightarrow, no head, from=3-3, to=3-1]
	\arrow[""{name=3, anchor=center, inner sep=0}, Rightarrow, no head, from=2-3, to=3-3]
	\arrow[""{name=4, anchor=center, inner sep=0}, Rightarrow, no head, from=2-1, to=3-1]
	\arrow["\opcart"{description}, draw=none, from=3, to=4]
	\arrow["\psi"{description}, draw=none, from=0, to=1]
	\arrow["\phi"{description}, draw=none, from=1, to=2]
\end{tikzcd}
\]

For instance, the tight 2-category $\u{\Cat}$ associated to the \vdc{} $\Cat$ is the usual 2-category of categories, functors, and natural transformations.

\subsection{Virtual equipments}

A crucial property of the \vdcs{} with which we shall be concerned is the ability to restrict loose-cells along adjacent tight-cells.

\begin{definition}[{\cite[Definition~7.1]{cruttwell2010unified}}]
	\label{cartesian-cell}
    A 2-cell
	\[\begin{tikzcd}
		\cdot & \cdot \\
		\cdot & \cdot
		\arrow["p"', "\shortmid"{marking}, from=1-2, to=1-1]
		\arrow["q", "\shortmid"{marking}, from=2-2, to=2-1]
		\arrow[""{name=0, anchor=center, inner sep=0}, "g", from=1-2, to=2-2]
		\arrow[""{name=1, anchor=center, inner sep=0}, "f"', from=1-1, to=2-1]
		\arrow["\cart"{description}, draw=none, from=0, to=1]
	\end{tikzcd}\]
    in a \vdc{} is \emph{cartesian} if any 2-cell
	\[\begin{tikzcd}
		\cdot & \cdots & \cdot \\
		\cdot && \cdot
		\arrow["q", "\shortmid"{marking}, from=2-3, to=2-1]
		\arrow[""{name=0, anchor=center, inner sep=0}, "{g' \d g}", from=1-3, to=2-3]
		\arrow[""{name=1, anchor=center, inner sep=0}, "{f' \d f}"', from=1-1, to=2-1]
		\arrow["{r_n}"', "\shortmid"{marking}, from=1-3, to=1-2]
		\arrow["{r_1}"', "\shortmid"{marking}, from=1-2, to=1-1]
		\arrow["\phi"{description}, draw=none, from=0, to=1]
	\end{tikzcd}\]
    factors uniquely therethrough:
	\[\begin{tikzcd}
		\cdot & \cdots & \cdot \\
		\cdot && \cdot \\
		\cdot && \cdot
		\arrow["q", "\shortmid"{marking}, from=3-3, to=3-1]
		\arrow["{r_n}"', "\shortmid"{marking}, from=1-3, to=1-2]
		\arrow[""{name=0, anchor=center, inner sep=0}, "g", from=2-3, to=3-3]
		\arrow[""{name=1, anchor=center, inner sep=0}, "f"', from=2-1, to=3-1]
		\arrow["p"{description}, from=2-3, to=2-1]
		\arrow[""{name=2, anchor=center, inner sep=0}, "{g'}", from=1-3, to=2-3]
		\arrow[""{name=3, anchor=center, inner sep=0}, "{f'}"', from=1-1, to=2-1]
		\arrow["{r_1}"', "\shortmid"{marking}, from=1-2, to=1-1]
		\arrow["\cart"{description}, draw=none, from=0, to=1]
		\arrow["\hat\phi"{description}, draw=none, from=2, to=3]
	\end{tikzcd}\]
    In this case, we call $p$ the \emph{restriction} $q(f, g)$.
    If $q$ is a loose-identity $A(1, 1)$, we denote $p = A(1, 1)(f, g)$ simply by $A(f, g)$.
    We denote the factorisation of a 2-cell
	
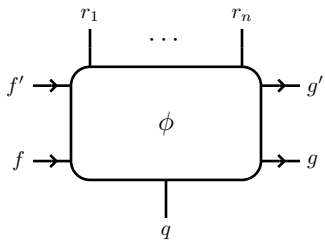
\[
	\begin{tangle}{(5,4)}[trim x,trim y]
		\tgBlank{(0,0)}{white}
		\tgBorderA{(1,0)}{white}{white}{white}{white}
		\tgBorder{(1,0)}{1}{0}{1}{0}
		\tgBlank{(2,0)}{white}
		\tgBorderA{(3,0)}{white}{white}{white}{white}
		\tgBorder{(3,0)}{1}{0}{1}{0}
		\tgBlank{(4,0)}{white}
		\tgBorderA{(0,1)}{white}{white}{white}{white}
		\tgBorder{(0,1)}{0}{1}{0}{1}
		\tgBorderA{(1,1)}{white}{white}{white}{white}
		\tgBorder{(1,1)}{1}{1}{1}{1}
		\tgBorderA{(2,1)}{white}{white}{white}{white}
		\tgBorder{(2,1)}{0}{1}{0}{1}
		\tgBorderA{(3,1)}{white}{white}{white}{white}
		\tgBorder{(3,1)}{1}{1}{1}{1}
		\tgBorderA{(4,1)}{white}{white}{white}{white}
		\tgBorder{(4,1)}{0}{1}{0}{1}
		\tgBorderA{(0,2)}{white}{white}{white}{white}
		\tgBorder{(0,2)}{0}{1}{0}{1}
		\tgBorderA{(1,2)}{white}{white}{white}{white}
		\tgBorder{(1,2)}{1}{1}{0}{1}
		\tgBorderA{(2,2)}{white}{white}{white}{white}
		\tgBorder{(2,2)}{0}{1}{1}{1}
		\tgBorderA{(3,2)}{white}{white}{white}{white}
		\tgBorder{(3,2)}{1}{1}{0}{1}
		\tgBorderA{(4,2)}{white}{white}{white}{white}
		\tgBorder{(4,2)}{0}{1}{0}{1}
		\tgBlank{(0,3)}{white}
		\tgBlank{(1,3)}{white}
		\tgBorderA{(2,3)}{white}{white}{white}{white}
		\tgBorder{(2,3)}{1}{0}{1}{0}
		\tgBlank{(3,3)}{white}
		\tgBlank{(4,3)}{white}
		\tgCell[(2,1)]{(2,1.5)}{\phi}
		\tgArrow{(0.5,1)}{0}
		\tgArrow{(0.5,2)}{0}
		\tgArrow{(3.5,1)}{0}
		\tgArrow{(3.5,2)}{0}
		\tgAxisLabel{(1.5,0.75)}{south}{r_1}
		\tgAxisLabel{(3.5,0.75)}{south}{r_n}
		\tgAxisLabel{(0.75,1.5)}{east}{f'}
		\tgAxisLabel{(4.25,1.5)}{west}{g'}
		\tgAxisLabel{(0.75,2.5)}{east}{f}
		\tgAxisLabel{(4.25,2.5)}{west}{g}
		\tgAxisLabel{(2.5,3.25)}{north}{q}
		\node at (2.5,.9) {$\cdots$};
	\end{tangle}
	\]
	through a cartesian 2-cell by
	
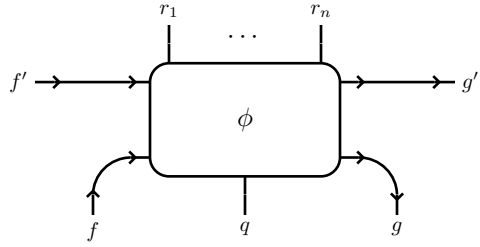
\[
	\begin{tangle}{(7,4)}[trim x,trim y]
		\tgBlank{(0,0)}{white}
		\tgBlank{(1,0)}{white}
		\tgBorderA{(2,0)}{white}{white}{white}{white}
		\tgBorder{(2,0)}{1}{0}{1}{0}
		\tgBlank{(3,0)}{white}
		\tgBorderA{(4,0)}{white}{white}{white}{white}
		\tgBorder{(4,0)}{1}{0}{1}{0}
		\tgBlank{(5,0)}{white}
		\tgBlank{(6,0)}{white}
		\tgBorderA{(0,1)}{white}{white}{white}{white}
		\tgBorder{(0,1)}{0}{1}{0}{1}
		\tgBorderA{(1,1)}{white}{white}{white}{white}
		\tgBorder{(1,1)}{0}{1}{0}{1}
		\tgBorderA{(2,1)}{white}{white}{white}{white}
		\tgBorder{(2,1)}{1}{1}{1}{1}
		\tgBorderA{(3,1)}{white}{white}{white}{white}
		\tgBorder{(3,1)}{0}{1}{0}{1}
		\tgBorderA{(4,1)}{white}{white}{white}{white}
		\tgBorder{(4,1)}{1}{1}{1}{1}
		\tgBorderA{(5,1)}{white}{white}{white}{white}
		\tgBorder{(5,1)}{0}{1}{0}{1}
		\tgBorderA{(6,1)}{white}{white}{white}{white}
		\tgBorder{(6,1)}{0}{1}{0}{1}
		\tgBlank{(0,2)}{white}
		\tgBorderC{(1,2)}{3}{white}{white}
		\tgBorderA{(2,2)}{white}{white}{white}{white}
		\tgBorder{(2,2)}{1}{1}{0}{1}
		\tgBorderA{(3,2)}{white}{white}{white}{white}
		\tgBorder{(3,2)}{0}{1}{1}{1}
		\tgBorderA{(4,2)}{white}{white}{white}{white}
		\tgBorder{(4,2)}{1}{1}{0}{1}
		\tgBorderC{(5,2)}{2}{white}{white}
		\tgBlank{(6,2)}{white}
		\tgBlank{(0,3)}{white}
		\tgBorderA{(1,3)}{white}{white}{white}{white}
		\tgBorder{(1,3)}{1}{0}{1}{0}
		\tgBlank{(2,3)}{white}
		\tgBorderA{(3,3)}{white}{white}{white}{white}
		\tgBorder{(3,3)}{1}{0}{1}{0}
		\tgBlank{(4,3)}{white}
		\tgBorderA{(5,3)}{white}{white}{white}{white}
		\tgBorder{(5,3)}{1}{0}{1}{0}
		\tgBlank{(6,3)}{white}
		\tgCell[(2,1)]{(3,1.5)}{\phi}
		\tgArrow{(1.5,1)}{0}
		\tgArrow{(1.5,2)}{0}
		\tgArrow{(4.5,1)}{0}
		\tgArrow{(4.5,2)}{0}
		\tgArrow{(0.5,1)}{0}
		\tgArrow{(5.5,1)}{0}
		\tgArrow{(5,2.5)}{3}
		\tgArrow{(1,2.5)}{1}
		\tgAxisLabel{(2.5,0.75)}{south}{r_1}
		\tgAxisLabel{(4.5,0.75)}{south}{r_n}
		\tgAxisLabel{(0.75,1.5)}{east}{f'}
		\tgAxisLabel{(6.25,1.5)}{west}{g'}
		\tgAxisLabel{(1.5,3.25)}{north}{f}
		\tgAxisLabel{(3.5,3.25)}{north}{q}
		\tgAxisLabel{(5.5,3.25)}{north}{g}
		\node at (3.5,.9) {$\cdots$};
	\end{tangle}
	\]
    and consequently elide the distinction between $\phi$ and $\hat\phi$.
	We denote the cartesian 2-cell by
	
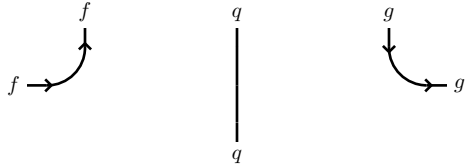
\[
	\begin{tangle}{(7,3)}[trim x,trim y]
		\tgBlank{(0,0)}{white}
		\tgBorderA{(1,0)}{white}{white}{white}{white}
		\tgBorder{(1,0)}{1}{0}{1}{0}
		\tgBlank{(2,0)}{white}
		\tgBorderA{(3,0)}{white}{white}{white}{white}
		\tgBorder{(3,0)}{1}{0}{1}{0}
		\tgBlank{(4,0)}{white}
		\tgBorderA{(5,0)}{white}{white}{white}{white}
		\tgBorder{(5,0)}{1}{0}{1}{0}
		\tgBlank{(6,0)}{white}
		\tgBorderA{(0,1)}{white}{white}{white}{white}
		\tgBorder{(0,1)}{0}{1}{0}{1}
		\tgBorderC{(1,1)}{1}{white}{white}
		\tgBlank{(2,1)}{white}
		\tgBorderA{(3,1)}{white}{white}{white}{white}
		\tgBorder{(3,1)}{1}{0}{1}{0}
		\tgBlank{(4,1)}{white}
		\tgBorderC{(5,1)}{0}{white}{white}
		\tgBorderA{(6,1)}{white}{white}{white}{white}
		\tgBorder{(6,1)}{0}{1}{0}{1}
		\tgBlank{(0,2)}{white}
		\tgBlank{(1,2)}{white}
		\tgBlank{(2,2)}{white}
		\tgBorderA{(3,2)}{white}{white}{white}{white}
		\tgBorder{(3,2)}{1}{0}{1}{0}
		\tgBlank{(4,2)}{white}
		\tgBlank{(5,2)}{white}
		\tgBlank{(6,2)}{white}
		\tgArrow{(1,0.5)}{1}
		\tgArrow{(0.5,1)}{0}
		\tgArrow{(5,0.5)}{3}
		\tgArrow{(5.5,1)}{0}
		\tgAxisLabel{(1.5,0.75)}{south}{f}
		\tgAxisLabel{(3.5,0.75)}{south}{q}
		\tgAxisLabel{(5.5,0.75)}{south}{g}
		\tgAxisLabel{(0.75,1.5)}{east}{f}
		\tgAxisLabel{(6.25,1.5)}{west}{g}
		\tgAxisLabel{(3.5,2.25)}{north}{q}
	\end{tangle}
	\]
\end{definition}

Restrictions are unique up to isomorphism and are pseudofunctorial: given a loose-cell $p \colon B \lto C$ and tight-cells $f \colon D \to C$ and $g \colon A \to B$, each pair of 2-cells $\phi \colon f' \tto f$ and $\gamma \colon g \tto g'$ induces a 2-cell $p(\phi, \gamma) \colon p(f, g) \tto p(f', g')$, assuming both restrictions exist.

Our string diagram notation is justified by the fact that the restriction $q(f, g)$ is the composite $C(f, 1) \odot q \odot B(1, g)$ when $\X$ admits the loose-identities $B(1, 1)$ and $C(1, 1)$ (\cite[Theorem~7.16]{cruttwell2010unified}). When $\X$ does not admit loose-identities, we use the same string diagram notation: however, in this case, the strings associated to $C(f, 1)$ and $B(1, g)$ do not represent isolated loose-cells, and must be read together with $q$ as the restriction $q(f, g)$.

As a special case of the string diagram notation for restrictions, if a tight-cell $f \colon A \to B$ admits the restriction $B(1, f) \colon A \lto B$ (called the \emph{companion} of $f$), then it may be bent down; while if it admits the restriction $B(f, 1) \colon B \lto A$ (called the \emph{conjoint} of $f$), then it may be bent up. Note that the existence of the companion and conjoint for $f$ is predicated on the existence of a loose-identity for the codomain $B$.
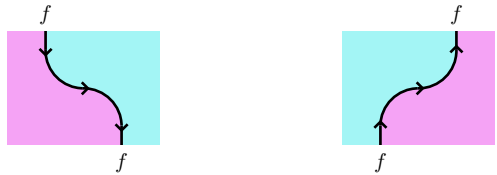
\[
\begin{tangle}{(2,3)}[trim y]
	\tgBorderA{(0,0)}{\tgColour10}{\tgColour6}{\tgColour6}{\tgColour10}
	\tgBlank{(1,0)}{\tgColour6}
	\tgBorderC{(0,1)}{0}{\tgColour10}{\tgColour6}
	\tgBorderC{(1,1)}{2}{\tgColour6}{\tgColour10}
	\tgBlank{(0,2)}{\tgColour10}
	\tgBorderA{(1,2)}{\tgColour10}{\tgColour6}{\tgColour6}{\tgColour10}
	\tgArrow{(0.5,1)}{0}
	\tgArrow{(0,0.5)}{3}
	\tgArrow{(1,1.5)}{3}
	\tgAxisLabel{(0.5,0.75)}{south}{f}
	\tgAxisLabel{(1.5,2.25)}{north}{f}
\end{tangle}
\hspace{4em}
\begin{tangle}{(2,3)}[trim y]
	\tgBlank{(0,0)}{\tgColour6}
	\tgBorderA{(1,0)}{\tgColour6}{\tgColour10}{\tgColour10}{\tgColour6}
	\tgBorderC{(0,1)}{3}{\tgColour6}{\tgColour10}
	\tgBorderC{(1,1)}{1}{\tgColour10}{\tgColour6}
	\tgBorderA{(0,2)}{\tgColour6}{\tgColour10}{\tgColour10}{\tgColour6}
	\tgBlank{(1,2)}{\tgColour10}
	\tgArrow{(0.5,1)}{0}
	\tgArrow{(1,0.5)}{1}
	\tgArrow{(0,1.5)}{1}
	\tgAxisLabel{(1.5,0.75)}{south}{f}
	\tgAxisLabel{(0.5,2.25)}{north}{f}
\end{tangle}
\]
This explains the use of the arrow notation: in a \vdc{} with companions and conjoints, lines annotated with arrows in a string diagram may be bent (they may point right, up, or down, but not left). Note that we label companions and conjoints in a string diagram by their tight-cell -- \eg{} by $f$, rather than by $B(1, f)$ or $B(f, 1)$ -- to aid readability.
The universal property of restriction ensures that the following \emph{zig-zag laws} hold~\cites[618]{cruttwell2010unified}[\S2]{myers2020yoneda}.
\[
\begin{tangle}{(3,2)}[trim x]
	\tgBorderA{(0,0)}{\tgColour6}{\tgColour6}{\tgColour10}{\tgColour10}
	\tgBorderC{(1,0)}{2}{\tgColour6}{\tgColour10}
	\tgBlank{(2,0)}{\tgColour6}
	\tgBlank{(0,1)}{\tgColour10}
	\tgBorderC{(1,1)}{0}{\tgColour10}{\tgColour6}
	\tgBorderA{(2,1)}{\tgColour6}{\tgColour6}{\tgColour10}{\tgColour10}
	\tgArrow{(1,0.5)}{3}
	\tgArrow{(0.5,0)}{0}
	\tgArrow{(1.5,1)}{0}
	\tgAxisLabel{(0.75,0.5)}{east}{f}
	\tgAxisLabel{(2.25,1.5)}{west}{f}
\end{tangle}
\quad = \quad
\begin{tangle}{(3,1)}[trim x]
	\tgBorderA{(0,0)}{\tgColour6}{\tgColour6}{\tgColour10}{\tgColour10}
	\tgBorderA{(1,0)}{\tgColour6}{\tgColour6}{\tgColour10}{\tgColour10}
	\tgBorderA{(2,0)}{\tgColour6}{\tgColour6}{\tgColour10}{\tgColour10}
	\tgArrow{(0.5,0)}{0}
	\tgArrow{(1.5,0)}{0}
	\tgAxisLabel{(0.75,0.5)}{east}{f}
	\tgAxisLabel{(2.25,0.5)}{west}{f}
\end{tangle}
\quad = \quad
\begin{tangle}{(3,2)}[trim x]
	\tgBlank{(0,0)}{\tgColour6}
	\tgBorderC{(1,0)}{3}{\tgColour6}{\tgColour10}
	\tgBorderA{(2,0)}{\tgColour6}{\tgColour6}{\tgColour10}{\tgColour10}
	\tgBorderA{(0,1)}{\tgColour6}{\tgColour6}{\tgColour10}{\tgColour10}
	\tgBorderC{(1,1)}{1}{\tgColour10}{\tgColour6}
	\tgBlank{(2,1)}{\tgColour10}
	\tgArrow{(1,0.5)}{1}
	\tgArrow{(1.5,0)}{0}
	\tgArrow{(0.5,1)}{0}
	\tgAxisLabel{(2.25,0.5)}{west}{f}
	\tgAxisLabel{(0.75,1.5)}{east}{f}
\end{tangle}
\]
From this observation, we may easily deduce that, given parallel tight-cells $f, g \colon A \to B$, the following are in natural bijection (\cf{}~\cite[Corollary~7.22]{cruttwell2010unified}), and hence that restriction is \ff{}.
\begin{align*}
	f & \tto g & B(1, f) & \tto B(1, g) & B(g, 1) & \tto B(f, 1)
\end{align*}
Consequently, $f \iso g$ if and only if $B(1, f) \iso B(1, g)$ if and only if $B(g, 1) \iso B(f, 1)$.
This permits us to view tight-cells as special loose-cells by taking either companions or conjoints: this process may be thought of as \emph{loosening} a tight-cell (\cf{}~\cite{lack2012enhanced}).

It will be helpful to observe explicitly how composition of tight-cells interacts with restriction in string diagrams. Given tight-cells as below,
\[A \xto f B \xto g C \xfrom h D \xfrom i E\]
and assuming that the restriction $C(h i, g f) \colon A \lto E$ exists, the loose-cell $C(h i, g f)$ is denoted by the following string diagram. Notice in particular that, due to their variance, $f$ and $g$ appear below in nondiagrammatic order, while $h$ and $i$ appear in diagrammatic order.
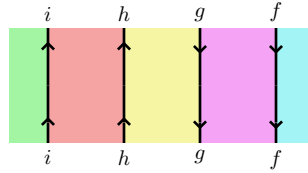
\[
\begin{tangle}{(4,3)}[trim y]
	\tgBorderA{(0,0)}{\tgColour4}{\tgColour0}{\tgColour0}{\tgColour4}
	\tgBorderA{(1,0)}{\tgColour0}{\tgColour2}{\tgColour2}{\tgColour0}
	\tgBorderA{(2,0)}{\tgColour2}{\tgColour10}{\tgColour10}{\tgColour2}
	\tgBorderA{(3,0)}{\tgColour10}{\tgColour6}{\tgColour6}{\tgColour10}
	\tgBorderA{(0,1)}{\tgColour4}{\tgColour0}{\tgColour0}{\tgColour4}
	\tgBorderA{(1,1)}{\tgColour0}{\tgColour2}{\tgColour2}{\tgColour0}
	\tgBorderA{(2,1)}{\tgColour2}{\tgColour10}{\tgColour10}{\tgColour2}
	\tgBorderA{(3,1)}{\tgColour10}{\tgColour6}{\tgColour6}{\tgColour10}
	\tgBorderA{(0,2)}{\tgColour4}{\tgColour0}{\tgColour0}{\tgColour4}
	\tgBorderA{(1,2)}{\tgColour0}{\tgColour2}{\tgColour2}{\tgColour0}
	\tgBorderA{(2,2)}{\tgColour2}{\tgColour10}{\tgColour10}{\tgColour2}
	\tgBorderA{(3,2)}{\tgColour10}{\tgColour6}{\tgColour6}{\tgColour10}
	\tgArrow{(3,0.5)}{3}
	\tgArrow{(3,1.5)}{3}
	\tgArrow{(2,0.5)}{3}
	\tgArrow{(2,1.5)}{3}
	\tgArrow{(1,0.5)}{1}
	\tgArrow{(1,1.5)}{1}
	\tgArrow{(0,0.5)}{1}
	\tgArrow{(0,1.5)}{1}
	\tgAxisLabel{(0.5,0.75)}{south}{i}
	\tgAxisLabel{(1.5,0.75)}{south}{h}
	\tgAxisLabel{(2.5,0.75)}{south}{g}
	\tgAxisLabel{(3.5,0.75)}{south}{f}
	\tgAxisLabel{(0.5,2.25)}{north}{i}
	\tgAxisLabel{(1.5,2.25)}{north}{h}
	\tgAxisLabel{(2.5,2.25)}{north}{g}
	\tgAxisLabel{(3.5,2.25)}{north}{f}
\end{tangle}
\]

\begin{notation}
	Let $f \colon A \to B$ be a tight-cell. Assuming $f$ admits the restrictions $B(1, f)$ and $B(f, 1)$, denote by $\cp f$ and $\pc f$ the following 2-cells.
	\[
	\cp f \defeq
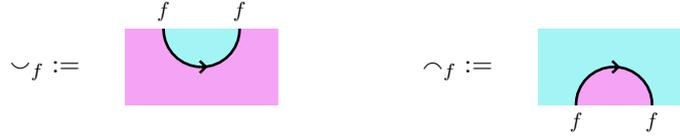
	\begin{tangle}{(2,1)}
		\tgBorderC{(0,0)}{0}{\tgColour10}{\tgColour6}
		\tgBorderC{(1,0)}{1}{\tgColour10}{\tgColour6}
		\tgArrow{(0.5,0)}{0}
		\tgAxisLabel{(0.5,0)}{south}{f}
		\tgAxisLabel{(1.5,0)}{south}{f}
	\end{tangle}
	\hspace{4em}
	\pc f \defeq
	\begin{tangle}{(2,1)}
		\tgBorderC{(0,0)}{3}{\tgColour6}{\tgColour10}
		\tgBorderC{(1,0)}{2}{\tgColour6}{\tgColour10}
		\tgArrow{(0.5,0)}{0}
		\tgAxisLabel{(0.5,1)}{north}{f}
		\tgAxisLabel{(1.5,1)}{north}{f}
	\end{tangle}
	\]
	Furthermore, given a 2-cell
	\[\begin{tikzcd}
		{A_0} & {A_1} & \cdots & {A_{n - 1}} & {A_n} \\
		{B_0} &&&& {B_n}
		\arrow["{C(g, h)}", "\shortmid"{marking}, from=2-5, to=2-1]
		\arrow["{p_n}"', "\shortmid"{marking}, from=1-5, to=1-4]
		\arrow["{p_1}"', "\shortmid"{marking}, from=1-2, to=1-1]
		\arrow["{p_{n - 1}}"', "\shortmid"{marking}, from=1-4, to=1-3]
		\arrow["{p_2}"', "\shortmid"{marking}, from=1-3, to=1-2]
		\arrow[""{name=0, anchor=center, inner sep=0}, "{f_n}", from=1-5, to=2-5]
		\arrow[""{name=1, anchor=center, inner sep=0}, "{f_0}"', from=1-1, to=2-1]
		\arrow["\phi"{description}, draw=none, from=0, to=1]
	\end{tikzcd}\]
	and a tight-cell $k \colon C \to D$, denote by
	\[\begin{tikzcd}
		{A_0} & {A_1} & \cdots & {A_{n - 1}} & {A_n} \\
		{B_0} &&&& {B_n}
		\arrow["{D(kg, kh)}", "\shortmid"{marking}, from=2-5, to=2-1]
		\arrow["{p_n}"', "\shortmid"{marking}, from=1-5, to=1-4]
		\arrow["{p_1}"', "\shortmid"{marking}, from=1-2, to=1-1]
		\arrow["{p_{n - 1}}"', "\shortmid"{marking}, from=1-4, to=1-3]
		\arrow["{p_2}"', "\shortmid"{marking}, from=1-3, to=1-2]
		\arrow[""{name=0, anchor=center, inner sep=0}, "{f_n}", from=1-5, to=2-5]
		\arrow[""{name=1, anchor=center, inner sep=0}, "{f_0}"', from=1-1, to=2-1]
		\arrow["{\phi \d k}"{description}, draw=none, from=0, to=1]
	\end{tikzcd}\]
	the 2-cell defined by the following diagram.
	
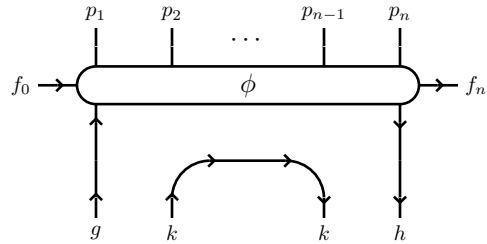
\[
	\begin{tangle}{(7,4)}[trim x,trim y]
		\tgBlank{(0,0)}{white}
		\tgBorderA{(1,0)}{white}{white}{white}{white}
		\tgBorder{(1,0)}{1}{0}{1}{0}
		\tgBorderA{(2,0)}{white}{white}{white}{white}
		\tgBorder{(2,0)}{1}{0}{1}{0}
		\tgBlank{(3,0)}{white}
		\tgBorderA{(4,0)}{white}{white}{white}{white}
		\tgBorder{(4,0)}{1}{0}{1}{0}
		\tgBorderA{(5,0)}{white}{white}{white}{white}
		\tgBorder{(5,0)}{1}{0}{1}{0}
		\tgBlank{(6,0)}{white}
		\tgBorderA{(0,1)}{white}{white}{white}{white}
		\tgBorder{(0,1)}{0}{1}{0}{1}
		\tgBorderA{(1,1)}{white}{white}{white}{white}
		\tgBorder{(1,1)}{1}{1}{1}{1}
		\tgBorderA{(2,1)}{white}{white}{white}{white}
		\tgBorder{(2,1)}{1}{1}{0}{1}
		\tgBorderA{(3,1)}{white}{white}{white}{white}
		\tgBorder{(3,1)}{0}{1}{0}{1}
		\tgBorderA{(4,1)}{white}{white}{white}{white}
		\tgBorder{(4,1)}{1}{1}{0}{1}
		\tgBorderA{(5,1)}{white}{white}{white}{white}
		\tgBorder{(5,1)}{1}{1}{1}{1}
		\tgBorderA{(6,1)}{white}{white}{white}{white}
		\tgBorder{(6,1)}{0}{1}{0}{1}
		\tgBlank{(0,2)}{white}
		\tgBorderA{(1,2)}{white}{white}{white}{white}
		\tgBorder{(1,2)}{1}{0}{1}{0}
		\tgBorderC{(2,2)}{3}{white}{white}
		\tgBorderA{(3,2)}{white}{white}{white}{white}
		\tgBorder{(3,2)}{0}{1}{0}{1}
		\tgBorderC{(4,2)}{2}{white}{white}
		\tgBorderA{(5,2)}{white}{white}{white}{white}
		\tgBorder{(5,2)}{1}{0}{1}{0}
		\tgBlank{(6,2)}{white}
		\tgBlank{(0,3)}{white}
		\tgBorderA{(1,3)}{white}{white}{white}{white}
		\tgBorder{(1,3)}{1}{0}{1}{0}
		\tgBorderA{(2,3)}{white}{white}{white}{white}
		\tgBorder{(2,3)}{1}{0}{1}{0}
		\tgBlank{(3,3)}{white}
		\tgBorderA{(4,3)}{white}{white}{white}{white}
		\tgBorder{(4,3)}{1}{0}{1}{0}
		\tgBorderA{(5,3)}{white}{white}{white}{white}
		\tgBorder{(5,3)}{1}{0}{1}{0}
		\tgBlank{(6,3)}{white}
		\tgCell[(4,0)]{(3,1)}{\phi}
		\tgArrow{(0.5,1)}{0}
		\tgArrow{(5.5,1)}{0}
		\tgArrow{(1,1.5)}{1}
		\tgArrow{(5,1.5)}{3}
		\tgArrow{(1,2.5)}{1}
		\tgArrow{(2,2.5)}{1}
		\tgArrow{(2.5,2)}{0}
		\tgArrow{(3.5,2)}{0}
		\tgArrow{(5,2.5)}{3}
		\tgArrow{(4,2.5)}{3}
		\tgAxisLabel{(1.5,0.75)}{south}{p_1}
		\tgAxisLabel{(2.5,0.75)}{south}{p_2}
		\tgAxisLabel{(4.5,0.75)}{south}{p_{n - 1}}
		\tgAxisLabel{(5.5,0.75)}{south}{p_n}
		\tgAxisLabel{(0.75,1.5)}{east}{f_0}
		\tgAxisLabel{(6.25,1.5)}{west}{f_n}
		\tgAxisLabel{(1.5,3.25)}{north}{g}
		\tgAxisLabel{(2.5,3.25)}{north}{k}
		\tgAxisLabel{(4.5,3.25)}{north}{k}
		\tgAxisLabel{(5.5,3.25)}{north}{h}
		\node at (3.5,.9) {$\cdots$};
	\end{tangle}
	\]
\end{notation}

\begin{notation}
     Given a 2-cell $\phi \colon p_1, \ldots, p_n \tto q$ and tight-cells $f$ and $g$ with appropriate (co)domain, we denote by $\phi(f, g) \colon p_1(f, 1), \ldots, p_n(1, g) \tto q(f, g)$ the 2-cell given by pasting the conjoint of $f$ and the companion of $g$.
\end{notation}

The existence of companions and conjoints typically holds for \vdcs{} of category-like structures (\cf{}~\cite[Examples~7.7]{cruttwell2010unified}). For instance, in $\Cat$, the companion of a tight-cell $f \colon A \to B$ is the distributor $B({-}_1, f {-}_2) \colon A \lto B$ given by postcomposition by $f$, while the conjoint of $f$ is the distributor $B(f {-}_1, {-}_2) \colon B \lto A$ given by precomposition by $f$. The importance of this structure motivates the following definition.

\begin{definition}[{\cite[Definition~7.6]{cruttwell2010unified}}]
    A \emph{virtual equipment} (or simply \emph{equipment}) is a \vdc{} that admits all loose-identities and restrictions.
\end{definition}

It will be useful to have terminology for those loose-cells induced by tight-cells via restriction.

\begin{definition}
	\label{representable-and-corepresentable}
    Let $j \colon A_0 \to B$ and $i \colon A_n \to B$ be tight-cells, and consider a chain $p_1 \colon A_1 \lto A_0, \ldots, p_n \colon A_n \lto A_{n - 1}$ of loose-cells. If $B(j, i)$ exists and forms the loose-composite of the chain, then we say that $p_1, \ldots, p_n$ is \emph{$j$-represented by $i$}, is \emph{$i$-corepresented by $j$}, is \emph{$j$-representable}, and is \emph{$i$-corepresentable}. We omit the prefixes $j$-{} and $i$-{} when the respective tight-cells are the identity.
\end{definition}

A loose-cell is thus representable precisely when it is the companion of a tight-cell, and is corepresentable precisely when it is the conjoint of a tight-cell.

\subsection{Duality}

\begin{definition}
    The \emph{dual} $\X\co$ of a \vdc{} $\X$ is the \vdc{} with the same objects and tight-cells as $\X$, whose loose-cells $A \lto B$ are the loose-cells $B \lto A$ in $\X$, and whose 2-cells with frame
	\[\begin{tikzcd}
		{A_0} & \cdots & {A_n} \\
		{B_0} && {B_n}
		\arrow["q", "\shortmid"{marking}, from=2-3, to=2-1]
		\arrow["{p_n}"', "\shortmid"{marking}, from=1-3, to=1-2]
		\arrow["{f_n}", from=1-3, to=2-3]
		\arrow["{f_0}"', from=1-1, to=2-1]
		\arrow["{p_1}"', "\shortmid"{marking}, from=1-2, to=1-1]
	\end{tikzcd}\]
    are the 2-cells
	\[\begin{tikzcd}
		{A_n} & \cdots & {A_0} \\
		{B_n} && {B_0}
		\arrow["q", "\shortmid"{marking}, from=2-3, to=2-1]
		\arrow["{p_1}"', "\shortmid"{marking}, from=1-3, to=1-2]
		\arrow[""{name=0, anchor=center, inner sep=0}, "{f_0}", from=1-3, to=2-3]
		\arrow[""{name=1, anchor=center, inner sep=0}, "{f_n}"', from=1-1, to=2-1]
		\arrow["{p_n}"', "\shortmid"{marking}, from=1-2, to=1-1]
		\arrow["\phi"{description}, draw=none, from=0, to=1]
	\end{tikzcd}\]
    in $\X$. String diagramatically, $\X\co$ arises from $\X$ by horizontal reflection and reversing arrows.
\end{definition}

\begin{remark}
  A \vdc{} has only one notion of dual, which combines the two notions of duality for a 2-category or bicategory. In particular, the duality acts as $\ph\co$ on the tight-cells, and as $\ph\op$ on the loose-cells; the two dualities coincide for the 2-cells. The tight 2-category of $\X\co$ is $(\tX)\co$, where the latter is formed in the usual way by reversing 2-cells.
\end{remark}

\begin{lemma}
    \label{X-equipment-iff-Xco-equipment}
    Let $\X$ be a \vdc{}. Then $\X\co$ is an equipment if and only if $\X$ is an equipment. Furthermore, $\X\co$ admits a loose-composite of a chain $p_n, \ldots, p_1$ if and only if $\X$ admits a loose-composite of $p_1, \ldots, p_n$.
\end{lemma}

\begin{proof}
  Loose-identities $A(1, 1)$ in $\X\co$ are loose-identities $A(1, 1)$ in $\X$ (and conversely). Restrictions $p(g, f)$ in $\X\co$ are restrictions $p(f, g)$ in $\X$ (and conversely). Loose-composites $p_n \odot \cdots \odot p_1$ in $\X\co$ are loose-composites $p_1 \odot \cdots \odot p_n$ in $\X$. They satisfy their respective universal properties by definition.
\end{proof}

\subsection{Monads and adjunctions}
\label{monads-and-adjunctions}

Since a \vdc{} has two kinds of morphism -- tight-cells and loose-cells -- there are two kinds of monad one may consider in a \vdc{}, assuming the existence of sufficient loose-identities. Monads formed from tight-cells (which we simply call \emph{monads}, or \emph{tight-monads} to disambiguate), and their generalisation to relative monads, will be of primary interest throughout the paper. However, monads formed from loose-cells (which we call \emph{loose-monads}) are of secondary interest in certain representation theorems, and it will be useful to introduce them here.

\begin{definition}
	\label{tight-monad}
	Let $\X$ be a \vdc{} admitting loose-identities and let $A$ be an object. A \emph{monad} on $A$ is a monad on $A$ in the tight 2-category $\tX$~\cite[Definition~5.4.1]{benabou1967introduction}. Denote by $\Mnd(A)$ the category of monads on $A$, and by $U_A \colon \Mnd(A) \to \X[A, A]$ the forgetful functor.
\end{definition}

In the above definition, the only loose-identity we require is $A(1, 1)$ since, if one expands the definition of a monad on $A$, this is the only one that is used.
We assume all loose-identities exist for simplicity, so that we can state the definition in terms of the tight 2-category.
A similar consideration applies to several of the definitions below.

\begin{definition}
	\label{loose-monad}
	A \emph{loose-monad} (\cite[\S2.6]{leinster1999generalized}) in a \vdc{} comprises
	\begin{enumerate}
		\item an object $A$, the \emph{base};
		\item a loose-cell $t \colon A \lto A$, the \emph{underlying loose-cell};
		\item a 2-cell $\mu \colon t, t \tto t$, the \emph{multiplication};
		\item a 2-cell $\eta \colon \tto t$, the \emph{unit},
	\end{enumerate}
	satisfying the following equations.
	\begin{align*}
		(\mu, 1_t) \d \mu & = (1_t, \mu) \d \mu &
		(\eta, 1_t) \d \mu & = 1_t &
		(1_t, \eta) \d \mu & = 1_t
	\end{align*}
	A \emph{loose-monad on $A$} is a loose-monad with base $A$. A morphism of loose-monads\footnotemark{} on $A$ from $(t, \mu, \eta)$ to $(t', \mu', \eta')$ is a 2-cell $\tau \colon t \tto t'$ satisfying the following equations.
	\begin{align*}
		\eta \d \tau & = \eta' &
		\mu \d \tau & = (\tau, \tau) \d \mu'
	\end{align*}
	\footnotetext{The loose-monad morphisms we consider are a special case of those of \cite[\S2.6]{leinster1999generalized} and \cite[Definition~8.3]{cruttwell2010unified}, which permit morphisms between loose-monads with different bases.}
	Loose-monads on $A$ and their morphisms form a category $\LMnd(A)$. Denote by ${U_A \colon \LMnd(A) \to \X\lh{A, A}}$ the faithful functor sending $(t, \mu, \eta) \mapsto t$.
\end{definition}

In $\Cat$, tight-monads are the classical notion of monad on a category. Loose-monads are monads in the (virtual) bicategory of distributors, which are known to correspond to bijective-on-objects functors~\cite[6.22]{justesen1968bikategorien}.

Just as with monads in a 2-category, we have corresponding notions of adjunction for (tight) monads and loose-monads.

\begin{definition}
	\label{tight-adjunction}
	Let $\X$ be a \vdc{} admitting loose-identities. An \emph{adjunction} in $\X$ is an adjunction in the tight 2-category $\tX$~\cite[762]{maranda1965formal}.
\end{definition}

\begin{definition}[{\cite[Definition~5.31]{shulman2013enriched}}]
	\label{loose-adjunction}
	A \emph{loose-adjunction} in a \vdc{} comprises
	\begin{enumerate}
		\item an object $A$, the \emph{base};
		\item an object $C$, the \emph{apex}, admitting a loose-identity;
		\item a loose-cell $\ell \colon A \lto C$, the \emph{left loose-adjoint};
		\item a loose-cell $r \colon C \lto A$, the \emph{right loose-adjoint}, admitting a composite $r \odot \ell \colon A \lto A$;
		\item a 2-cell $\eta \colon \tto r \odot \ell$, the \emph{unit};
		\item a 2-cell $\varepsilon \colon \ell, r \tto C(1, 1)$, the \emph{counit},
	\end{enumerate}
	satisfying the following (\emph{left} and \emph{right}) zig-zag laws.
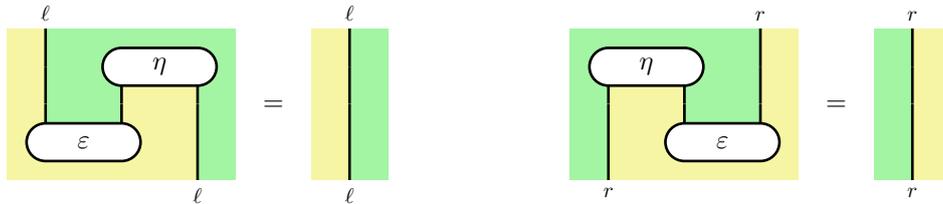
	\begin{tangleeqs*}
	\begin{tangle}{(3,2)}
		\tgBorderA{(0,0)}{\tgColour2}{\tgColour4}{\tgColour4}{\tgColour2}
		\tgBorderA{(1,0)}{\tgColour4}{\tgColour4}{\tgColour2}{\tgColour4}
		\tgBorderA{(2,0)}{\tgColour4}{\tgColour4}{\tgColour4}{\tgColour2}
		\tgBorderA{(0,1)}{\tgColour2}{\tgColour4}{\tgColour2}{\tgColour2}
		\tgBorderA{(1,1)}{\tgColour4}{\tgColour2}{\tgColour2}{\tgColour2}
		\tgBorderA{(2,1)}{\tgColour2}{\tgColour4}{\tgColour4}{\tgColour2}
		\tgCell[(1,0)]{(0.5,1)}{\varepsilon}
		\tgCell[(1,0)]{(1.5,0)}{\eta}
		\tgAxisLabel{(0.5,0)}{south}{\ell}
		\tgAxisLabel{(2.5,2)}{north}{\ell}
	\end{tangle}
	\=
	\begin{tangle}{(1,2)}
		\tgBorderA{(0,0)}{\tgColour2}{\tgColour4}{\tgColour4}{\tgColour2}
		\tgBorderA{(0,1)}{\tgColour2}{\tgColour4}{\tgColour4}{\tgColour2}
		\tgAxisLabel{(0.5,0)}{south}{\ell}
		\tgAxisLabel{(0.5,2)}{north}{\ell}
	\end{tangle}
	\hspace{4em}
	\begin{tangle}{(3,2)}
		\tgBorderA{(0,0)}{\tgColour4}{\tgColour4}{\tgColour2}{\tgColour4}
		\tgBorderA{(1,0)}{\tgColour4}{\tgColour4}{\tgColour4}{\tgColour2}
		\tgBorderA{(2,0)}{\tgColour4}{\tgColour2}{\tgColour2}{\tgColour4}
		\tgBorderA{(0,1)}{\tgColour4}{\tgColour2}{\tgColour2}{\tgColour4}
		\tgBorderA{(1,1)}{\tgColour2}{\tgColour4}{\tgColour2}{\tgColour2}
		\tgBorderA{(2,1)}{\tgColour4}{\tgColour2}{\tgColour2}{\tgColour2}
		\tgCell[(1,0)]{(0.5,0)}{\eta}
		\tgCell[(1,0)]{(1.5,1)}{\varepsilon}
		\tgAxisLabel{(2.5,0)}{south}{r}
		\tgAxisLabel{(0.5,2)}{north}{r}
	\end{tangle}
	\=
	\begin{tangle}{(1,2)}
		\tgBorderA{(0,0)}{\tgColour4}{\tgColour2}{\tgColour2}{\tgColour4}
		\tgBorderA{(0,1)}{\tgColour4}{\tgColour2}{\tgColour2}{\tgColour4}
		\tgAxisLabel{(0.5,0)}{south}{r}
		\tgAxisLabel{(0.5,2)}{north}{r}
	\end{tangle}
	\end{tangleeqs*}
\end{definition}

The motivating example of a loose-adjunction is the relationship between the representable and corepresentable loose-cells induced by a tight-cell.

\begin{lemma}
	\label{loose-adjunction-induced-by-tight-cell}
	Let $\ell \colon A \to C$ be a tight-cell in a \ve{}. Then $C(1, \ell) \adj C(\ell, 1)$.
\end{lemma}

\begin{proof}
	First, observe that $C(\ell, \ell) \iso C(\ell, 1) \odot C(1, \ell)$. The unit is given by $\pc \ell$ and the counit is given by $\cp \ell$. The zig-zag laws follow from those for restriction.
\end{proof}

\begin{remark}
	\label{representables-vs-maps}
	Following \cref{loose-adjunction-induced-by-tight-cell}, representable loose-cells in an equipment are left loose-adjoints (often simply called \emph{maps}).
	However, the converse is not generally true: for instance, the left adjoint distributors $A \lto E$ are equivalent not to functors $A \to E$, but to functors from $A$ to the cocompletion of $E$ under absolute colimits (\cf{}~\cite[\S6]{kelly2005notes}). The distinction between representables and maps is a crucial aspect of the insufficiency of 2-categories as a setting for formal category theory (\cf{}~\cref{formal-mw-monads}).
\end{remark}

As expected, loose-adjunctions induce loose-monads.

\begin{lemma}
	\label{loose-adjunction-induces-loose-monad}
	Every loose-adjunction $\ell \adj r$ induces a loose-monad.
\end{lemma}

\begin{proof}
	Let $(\ell, r, \eta, \varepsilon)$ be a loose-adjunction. We define a loose-monad by \[(r \odot \ell, r \odot \varepsilon \odot \ell, \eta)\]
	The unit laws follows from the zig-zag laws for a loose-adjunction, while the associativity law follows from associativity of composition of 2-cells in the \ve{}.
\end{proof}

As a consequence of \cref{loose-adjunction-induced-by-tight-cell,loose-adjunction-induces-loose-monad}, we have that every tight-cell $\ell \colon A \to C$ in a \ve{} induces a loose-monad $C(\ell, \ell)$ whose unit is given by $\pc \ell \colon \tto C(\ell, \ell)$ and whose counit is given by $\cp \ell \colon C(1, \ell), C(\ell, 1) \tto C(1, 1)$ (\cf{}~\cite[Lemma~8.4]{cruttwell2010unified}).

\begin{definition}
	\label{co-representable-resolution}
	Let $\ell$ be a tight-cell. A loose-monad $T$ is \emph{induced by $\ell$} if there exists a tight-cell $\ell$ such that $C(1, \ell) \adj C(\ell, 1)$ induces $T$ via \cref{loose-adjunction-induces-loose-monad}.
\end{definition}

Monads and adjunctions in an equipment induce loose-monads and loose-adjunctions: this will be discussed in the greater generality of relative monads and relative adjunctions in \cref{skew-multicategorical-hom-categories,relative-adjunctions}. Note that, while (tight) comonads may also be defined in a \vdc{} with loose-identities, loose-comonads may not without the assumption of binary loose-composites, since their definition involves 2-cells with non-unary codomain (\cf{}~\cref{duality}).

We conclude with the following observation (\cf{}~\cite[Lemma~1.1.1]{johnstone2002sketches1}), which is useful for establishing that a canonical 2-cell is an isomorphism in the presence of a non-canonical isomorphism.

\begin{lemma}
	\label{invertible-unit-if-noncanonical-isomorphism}
	The unit $\eta \colon \tto r \odot \ell$ of a loose-adjunction $\ell \adj r$ is opcartesian if and only if $A$ admits a loose-identity and $r \odot \ell \iso A(1, 1)$.
\end{lemma}

\begin{proof}
	Suppose that $A$ admits a loose-identity and $r \odot \ell \iso A(1, 1)$. The loose-monad structure on $r \odot \ell$ transfers to a loose-monad structure on $A(1, 1)$. By an Eckmann--Hilton argument (\cf{}~\cite[Theorem~1.12]{eckman1961structure}), the set $\X\lh{A, A}(A(1, 1), A(1, 1))$ is equipped with commutative monoid structure. Thus the unit of the transferred loose-monad, which is a one-sided inverse for the multiplication, is a two-sided inverse. By transferring the loose-monad structure back to $r \odot \ell$, we deduce that the unit is invertible. The converse is trivial.
\end{proof}

\section{Formal category theory}
\label{formal-category-theory}

We shall now introduce some basic concepts and results, well known in ordinary category theory, in the formal setting of equipments. Many of these results are known in the context of Yoneda structures~\cite{street1978yoneda} or proarrow equipments~\cite{wood1982abstract,wood1985proarrows}, but have not yet been generalised to the context of \vdcs{}. The reader interested primarily in relative monads is recommended to skip directly to \cref{skew-multicategorical-hom-categories} and refer back to this section for definitions and lemmas where necessary. The remainder of the paper may be read as if it applied only to enriched categories: our terminology has been chosen to align with the standard terminology in $\VCat$, as will be established in \cref{formal-category-theory-in-VCat}.

We work in the context of an arbitrary virtual double category $\X$.
When we discuss right lifts and extensions in \cref{right-lifts}, we do not assume that $\X$ admits restrictions or loose-identities, but assume the existence of both of these when discussing (co)limits, from \cref{weighted-colimits} onwards.

\subsection{Right lifts and extensions}
\label{right-lifts}

Two fundamental structures in equipments are \emph{right lifts} and \emph{right extensions}, which generalise the usual notions of right lift and right extension in a bicategory. In this paper, we are primarily interested in right lifts, which will be used to define weighted colimits and (pointwise) left extensions. However, we will also make some use of right extensions.
Right extensions are dual to right lifts, in the sense that a right lift $q \rf p$ in $\X$ is precisely a right extension $p \rx q$ in $\X\co$.
We spell out both definitions explicitly for convenience.

\begin{definition}[{\cite[Definition~9.1.2]{riehl2022elements}}]
	\label{right-lift}
    Let $p \colon Y \lto Z$ and $q \colon X \lto Z$ be loose-cells.
    A loose-cell $q \rf p \colon X \lto Y$, equipped with a 2-cell $\varpi \colon p, (q \rf p) \tto q$, is the \emph{right lift} of $q$ through $p$ when every 2-cell as on the left below factors uniquely as a diagram of the form on the right below (where $n \geq 0$). We call $\varpi$ the \emph{counit} of the right lift.
	\[
	\begin{tikzcd}[row sep=5em]
		Z & Y & \cdots & X \\
		Z &&& X
		\arrow["{r_n}"', "\shortmid"{marking}, from=1-4, to=1-3]
		\arrow["{r_1}"', "\shortmid"{marking}, from=1-3, to=1-2]
		\arrow["p"', "\shortmid"{marking}, from=1-2, to=1-1]
		\arrow[""{name=0, anchor=center, inner sep=0}, Rightarrow, no head, from=1-4, to=2-4]
		\arrow[""{name=1, anchor=center, inner sep=0}, Rightarrow, no head, from=1-1, to=2-1]
		\arrow["q", "\shortmid"{marking}, from=2-4, to=2-1]
		\arrow["\phi"{description}, draw=none, from=0, to=1]
	\end{tikzcd}
	\hspace{1em} = \hspace{1em}
	\begin{tikzcd}
		Z & Y & \cdots & X \\
		Z & Y && X \\
		Z &&& X
		\arrow["{r_n}"', "\shortmid"{marking}, from=1-4, to=1-3]
		\arrow["{r_1}"', "\shortmid"{marking}, from=1-3, to=1-2]
		\arrow["p"', "\shortmid"{marking}, from=1-2, to=1-1]
		\arrow["q", "\shortmid"{marking}, from=3-4, to=3-1]
		\arrow[""{name=0, anchor=center, inner sep=0}, Rightarrow, no head, from=1-1, to=2-1]
		\arrow[""{name=1, anchor=center, inner sep=0}, Rightarrow, no head, from=2-1, to=3-1]
		\arrow["p"{description}, from=2-2, to=2-1]
		\arrow[""{name=2, anchor=center, inner sep=0}, Rightarrow, no head, from=1-2, to=2-2]
		\arrow[""{name=3, anchor=center, inner sep=0}, Rightarrow, no head, from=1-4, to=2-4]
		\arrow[""{name=4, anchor=center, inner sep=0}, Rightarrow, no head, from=2-4, to=3-4]
		\arrow["{q \rf p}"{description}, from=2-4, to=2-2]
		\arrow["\varpi"{description}, draw=none, from=1, to=4]
		\arrow["\lambda\phi"{description}, draw=none, from=3, to=2]
		\arrow["{=}"{description}, draw=none, from=2, to=0]
	\end{tikzcd}
	\]
\end{definition}

\begin{definition}[{\cite[Definition~9.1.2]{riehl2022elements}}]
    \label{right-extension}
    Let $p \colon X \lto Y$ and $q \colon X \lto Z$ be loose-cells. A loose-cell $p \rx q \colon Y \lto Z$ equipped with a 2-cell $\varpi \colon (p \rx q), p \tto q$ is the \emph{right extension} of $q$ along $p$ when every 2-cell as on the left below factors uniquely as a diagram of the form on the right below (where $n \geq 0$). We call $\varpi$ the \emph{counit} of the right extension.
	\[
	\begin{tikzcd}[row sep=5em]
		Z & \cdots & Y & X \\
		Z &&& X
		\arrow["p"', "\shortmid"{marking}, from=1-4, to=1-3]
		\arrow["{r_n}"', "\shortmid"{marking}, from=1-3, to=1-2]
		\arrow["{r_1}"', "\shortmid"{marking}, from=1-2, to=1-1]
		\arrow[""{name=0, anchor=center, inner sep=0}, Rightarrow, no head, from=1-4, to=2-4]
		\arrow[""{name=1, anchor=center, inner sep=0}, Rightarrow, no head, from=1-1, to=2-1]
		\arrow["q", "\shortmid"{marking}, from=2-4, to=2-1]
		\arrow["\phi"{description}, draw=none, from=0, to=1]
	\end{tikzcd}
	\hspace{1em} = \hspace{1em}
	\begin{tikzcd}
		Z & \cdots & Y & X \\
		Z && Y & X \\
		Z &&& X
		\arrow["p"', "\shortmid"{marking}, from=1-4, to=1-3]
		\arrow["{r_n}"', "\shortmid"{marking}, from=1-3, to=1-2]
		\arrow["{r_1}"', "\shortmid"{marking}, from=1-2, to=1-1]
		\arrow["q", "\shortmid"{marking}, from=3-4, to=3-1]
		\arrow[""{name=0, anchor=center, inner sep=0}, Rightarrow, no head, from=1-1, to=2-1]
		\arrow[""{name=1, anchor=center, inner sep=0}, Rightarrow, no head, from=2-1, to=3-1]
		\arrow["{p \rx q}"{description}, from=2-3, to=2-1]
		\arrow[""{name=2, anchor=center, inner sep=0}, Rightarrow, no head, from=1-4, to=2-4]
		\arrow[""{name=3, anchor=center, inner sep=0}, Rightarrow, no head, from=2-4, to=3-4]
		\arrow["p"{description}, from=2-4, to=2-3]
		\arrow[""{name=4, anchor=center, inner sep=0}, Rightarrow, no head, from=1-3, to=2-3]
		\arrow["\varpi"{description}, draw=none, from=1, to=3]
		\arrow["{=}"{description}, draw=none, from=2, to=4]
		\arrow["\lambda\phi"{description}, draw=none, from=4, to=0]
	\end{tikzcd}
\]
\end{definition}

A useful intuition is that when $\X$ is the delooping of a monoidal category (so that loose-cells in $\X$ correspond to the objects of a monoidal category), a right lift corresponds to a right-hom, while a right extension corresponds to a left-hom.

\begin{example}[{\cite[Lemma~9.1.4]{riehl2022elements}}]
  \label{lift-companion}
  Suppose that $\X$ is an equipment, and consider a loose-cell $q$ and tight-cells $x, y$ as follows.
\[\begin{tikzcd}[sep=large]
	{Y'} & X \\
	Y & {X'}
	\arrow["y"', from=1-1, to=2-1]
	\arrow["x"', from=2-2, to=1-2]
	\arrow["q"', "\shortmid"{marking}, from=1-2, to=2-1]
	\arrow["{X(x, 1) \rx q}", "\shortmid"{marking}, from=2-2, to=2-1]
	\arrow["{q \rf Y(1, y)}"', "\shortmid"{marking}, from=1-2, to=1-1]
\end{tikzcd}\]
  The right extension $X(x, 1) \rx q$ and right lift $q \rf Y(1, y)$ both exist, and are given by restrictions
  \[
    X(x, 1) \rx q \iso q(1, x)
    \hspace{4em}
    q \rf Y(1, y) \iso q(y, 1)
  \]
  with the counits $q(1, x), X(x, 1) \tto q$ and $Y(1, y), q(y, 1) \tto q$ given by bending the tight-cells.
\[
\begin{tangle}{(3,1)}
	\tgBorderA{(0,0)}{white}{white}{white}{white}
	\tgBorder{(0,0)}{1}{0}{1}{0}
	\tgBorderC{(1,0)}{0}{white}{white}
	\tgBorderC{(2,0)}{1}{white}{white}
	\tgArrow{(1.5,0)}{0}
	\tgAxisLabel{(0.5,0)}{south}{q}
	\tgAxisLabel{(1.5,0)}{south}{x}
	\tgAxisLabel{(2.5,0)}{south}{x}
	\tgAxisLabel{(0.5,1)}{north}{q}
\end{tangle}
\hspace{4em}
\begin{tangle}{(3,1)}
	\tgBorderC{(0,0)}{0}{white}{white}
	\tgBorderC{(1,0)}{1}{white}{white}
	\tgBorderA{(2,0)}{white}{white}{white}{white}
	\tgBorder{(2,0)}{1}{0}{1}{0}
	\tgArrow{(0.5,0)}{0}
	\tgAxisLabel{(0.5,0)}{south}{y}
	\tgAxisLabel{(1.5,0)}{south}{y}
	\tgAxisLabel{(2.5,0)}{south}{q}
	\tgAxisLabel{(2.5,1)}{north}{q}
\end{tangle}
\]
  Hence, in an equipment every restriction $q(y, x)$ can be written using right lifts and extensions in two ways:
  \[
    (X(x, 1) \rx q) \rf Y(1, y)
    ~\iso~q(1, x)(y, 1)
    ~\iso~q(y, x)
    ~\iso~q(y, 1)(1, x)
    ~\iso~X(x, 1) \rx (q \rf Y(1, y))
  \]
\end{example}

In the last line of the above example, the right lift through $Y(1, y)$ and right extension along $X(x, 1)$ commute with one another. This is an instance of the following general result about the commutativity of right lifts with right extensions.

\begin{lemma}\label{lifts-commute-with-extensions}
  Let $p$, $p'$ and $q$ be loose-cells such that the right extension $p' \rx q$ and right lift $q \rf p$ both exist.
\[\begin{tikzcd}
	Y & X \\
	Z & {Y'}
	\arrow["{p'}", "\shortmid"{marking}, from=1-2, to=2-2]
	\arrow["{q \rf p}"', "\shortmid"{marking}, from=1-2, to=1-1]
	\arrow["{p' \rx q}", "\shortmid"{marking}, from=2-2, to=2-1]
	\arrow["p"', "\shortmid"{marking}, from=1-1, to=2-1]
	\arrow["q"', "\shortmid"{marking}, from=1-2, to=2-1]
\end{tikzcd}\]
  Then the right lift $(p' \rx q) \rf p$ exists exactly when the right extension $p' \rx (q \rf p)$ exists, in which case they are isomorphic.
  \[
    (p' \rx q) \rf p \iso p' \rx (q \rf p)
  \]
\end{lemma}
\begin{proof}
  Immediate from the fact that 2-cells with the following three frames are in bijection.
  \[
\begin{tikzcd}[column sep=1.6em]
	Z & Y & \cdots & {Y'} \\
	Z &&& {Y'}
	\arrow["{r_n}"', "\shortmid"{marking}, from=1-4, to=1-3]
	\arrow["{r_1}"', "\shortmid"{marking}, from=1-3, to=1-2]
	\arrow["p"', "\shortmid"{marking}, from=1-2, to=1-1]
	\arrow["{p' \rx q}", "\shortmid"{marking}, from=2-4, to=2-1]
	\arrow[Rightarrow, no head, from=1-1, to=2-1]
	\arrow[Rightarrow, no head, from=1-4, to=2-4]
\end{tikzcd}
~~
\begin{tikzcd}[column sep=1.6em]
	Z & Y & \cdots & {Y'} & X \\
	Z &&&& X
	\arrow["{r_n}"', "\shortmid"{marking}, from=1-4, to=1-3]
	\arrow["{r_1}"', "\shortmid"{marking}, from=1-3, to=1-2]
	\arrow["p"', "\shortmid"{marking}, from=1-2, to=1-1]
	\arrow["{p'}"', "\shortmid"{marking}, from=1-5, to=1-4]
	\arrow["q", "\shortmid"{marking}, from=2-5, to=2-1]
	\arrow[Rightarrow, no head, from=1-1, to=2-1]
	\arrow[Rightarrow, no head, from=1-5, to=2-5]
\end{tikzcd}
~~
\begin{tikzcd}[column sep=1.6em]
	Y & \cdots & {Y'} & X \\
	Y &&& X
	\arrow["{r_n}"', "\shortmid"{marking}, from=1-3, to=1-2]
	\arrow["{r_1}"', "\shortmid"{marking}, from=1-2, to=1-1]
	\arrow["{p'}"', "\shortmid"{marking}, from=1-4, to=1-3]
	\arrow["{q \rf p}", "\shortmid"{marking}, from=2-4, to=2-1]
	\arrow[Rightarrow, no head, from=1-4, to=2-4]
	\arrow[Rightarrow, no head, from=1-1, to=2-1]
\end{tikzcd}
\]
\end{proof}

Each of the following two results has one statement for right lifts and a dual statement for right extensions.
We give the proofs only for right lifts; for right extensions the result follows by duality.
First, right lifts through (left-) composites can be curried as follows.

\begin{lemma}[{\cf{}~\cite[Definition~9.1.6]{riehl2022elements}}]
    \label{currying-right-lift}
    Let $q \colon X \lto Z$ be a loose-cell.
    \begin{enumerate}
      \item If $Y' \xlto{p'} Y \xlto{p} Z$ are loose-cells such that the left-composite $p \odotl p' \colon Y' \lto Z$ and right lift $q \rf p \colon X \lto Y$ exist, then, if either side of the following exists, so does the other, in which case they are isomorphic.
        \[q \rf (p \odotl p') \iso (q \rf p) \rf p'\]
      \item If $X \xlto{p} Y \xlto{p'} Y'$ are loose-cells such that the right-composite $p' \odotr p \colon X \lto Y'$ and right extension $p \rx q \colon Y \lto Z$ exist, then, if either side of the following exists, so does the other, in which case they are isomorphic.
        \[(p' \odotr p) \rx q \iso p' \rx (p \rx q)\]
    \end{enumerate}
\end{lemma}

\begin{proof}
  For (1), there are bijections between 2-cells with the following three frames.
\[
\begin{tikzcd}[column sep=1.23em]
	Z &&& {Y'} & \cdots & X \\
	Z &&&&& X
	\arrow["{r_n}"', "\shortmid"{marking}, from=1-6, to=1-5]
	\arrow["{r_1}"', "\shortmid"{marking}, from=1-5, to=1-4]
	\arrow[Rightarrow, no head, from=1-6, to=2-6]
	\arrow["q", "\shortmid"{marking}, from=2-6, to=2-1]
	\arrow[Rightarrow, no head, from=1-1, to=2-1]
	\arrow["{p \odotl p'}"', "\shortmid"{marking}, from=1-4, to=1-1]
\end{tikzcd}
~~
\begin{tikzcd}[column sep=1.23em]
	Z && Y && {Y'} & \cdots & X \\
	Z &&&&&& X
	\arrow["{r_n}"', "\shortmid"{marking}, from=1-7, to=1-6]
	\arrow["{r_1}"', "\shortmid"{marking}, from=1-6, to=1-5]
	\arrow["{p'}"', "\shortmid"{marking}, from=1-5, to=1-3]
	\arrow[Rightarrow, no head, from=1-7, to=2-7]
	\arrow["q", "\shortmid"{marking}, from=2-7, to=2-1]
	\arrow["p"', "\shortmid"{marking}, from=1-3, to=1-1]
	\arrow[Rightarrow, no head, from=1-1, to=2-1]
\end{tikzcd}
~~
\begin{tikzcd}[column sep=1.23em]
	Y && {Y'} & \cdots & X \\
	Y &&&& X
	\arrow["{r_n}"', "\shortmid"{marking}, from=1-5, to=1-4]
	\arrow["{r_1}"', "\shortmid"{marking}, from=1-4, to=1-3]
	\arrow["{p'}"', "\shortmid"{marking}, from=1-3, to=1-1]
	\arrow[Rightarrow, no head, from=1-5, to=2-5]
	\arrow["{q \rf p}", "\shortmid"{marking}, from=2-5, to=2-1]
	\arrow[Rightarrow, no head, from=1-1, to=2-1]
\end{tikzcd}
\]
  The universal property of $q \rf (p \odotl p')$ is therefore equivalent to that of $(q \rf p) \rf p'$.
\end{proof}

Second, in an equipment, restrictions preserve right lifts as follows.

\begin{lemma}
    \label{right-lift-and-restriction}
    Assume that $\X$ is a virtual equipment, let $p$ and $q$ be loose-cells such that the right lift $q \rf p$ exists, and let $x$ and $y$ be tight-cells.
\[
\begin{tikzcd}
	{Y'} && {X'} \\
	Y && X \\
	& Z
	\arrow["{(q \rf p)(y, x)}"', "\shortmid"{marking}, from=1-3, to=1-1]
	\arrow["y"', from=1-1, to=2-1]
	\arrow["x", from=1-3, to=2-3]
	\arrow["q", "\shortmid"{marking}, from=2-3, to=3-2]
	\arrow["p"', "\shortmid"{marking}, from=2-1, to=3-2]
	\arrow["{q\rf p}"', "\shortmid"{marking}, from=2-3, to=2-1]
\end{tikzcd}
\hspace{4em}
\begin{tikzcd}[row sep=5em]
	{Y'} && {X'} \\
	& Z
	\arrow["{q(1, x)}", "\shortmid"{marking}, from=1-3, to=2-2]
	\arrow["{p(1, y)}"', "\shortmid"{marking}, from=1-1, to=2-2]
	\arrow["{q(1, x)\rf p(1, y)}"', "\shortmid"{marking}, from=1-3, to=1-1]
\end{tikzcd}
\]
    Then the restriction $(q \rf p)(y, x)$ forms the right lift of $q(1, x)$ through $p(1, y)$.
  \[
    (q \rf p)(y, x) \iso q(1, x) \rf p(1, y)
  \]
  The counit is the following 2-cell.
\[\begin{tikzcd}[column sep=huge]
	Z & {Y'} & {X'} \\
	Z & Y & {X'} \\
	Z && {X'}
	\arrow["{p(1, y)}"', "\shortmid"{marking}, from=1-2, to=1-1]
	\arrow["{(q \rf p)(y, x)}"', "\shortmid"{marking}, from=1-3, to=1-2]
	\arrow["{q(1, x)}", "\shortmid"{marking}, from=3-3, to=3-1]
	\arrow["{(q\rf p)(1, x)}"{description}, from=2-3, to=2-2]
	\arrow[""{name=0, anchor=center, inner sep=0}, "y"{description}, from=1-2, to=2-2]
	\arrow[""{name=1, anchor=center, inner sep=0}, Rightarrow, no head, from=1-3, to=2-3]
	\arrow[""{name=2, anchor=center, inner sep=0}, Rightarrow, no head, from=1-1, to=2-1]
	\arrow["p"{description}, from=2-2, to=2-1]
	\arrow[""{name=3, anchor=center, inner sep=0}, Rightarrow, no head, from=2-1, to=3-1]
	\arrow[""{name=4, anchor=center, inner sep=0}, Rightarrow, no head, from=2-3, to=3-3]
	\arrow["{\varpi(1, x)}"{description}, draw=none, from=4, to=3]
	\arrow["\cart"{description}, draw=none, from=1, to=0]
	\arrow["\cart"{description}, draw=none, from=0, to=2]
\end{tikzcd}\]
    Dually, given loose-cells $p$ and $q$ such that the right extension $p \rx q$ exists, we have
    \[
      (p \rx q)(z, y) \iso p(y, 1) \rx q(z, 1)
    \]
\end{lemma}

\begin{proof}
  We have the following isomorphisms, from which one can calculate the counit of the right lift $q(1, x) \rf p(1, y)$.
  \begin{align*}
    (q \rf p)(y, x)
    &\iso\tag{\cref{lift-companion}}
    X(x, 1) \rx ((q \rf p) \rf Y(1, y))
    \\&\iso\tag{\cref{currying-right-lift}}
    X(x, 1) \rx (q \rf (p \odot Y(1, y)))
    \\&\iso\tag{\cite[Theorem~7.16]{cruttwell2010unified}}
    X(x, 1) \rx (q \rf p(1, y))
    \\&\iso\tag{\cref{lifts-commute-with-extensions}}
    (X(x, 1) \rx q) \rf p(1, y)
    \\&\iso\tag{\cref{lift-companion}}
    q(1, x) \rf p(1, y)
  \end{align*}
\end{proof}

\begin{remark}
  While the statement of \cref{right-lift-and-restriction} does not mention loose-identities, we must assume they exist because the proof makes essential use of the loose-cells $X(x, 1)$ and $Y(1, y)$.
\end{remark}

\subsection{Weighted colimits and weighted limits}
\label{weighted-colimits}

We use right lifts to define the notion of \emph{weighted colimit} in an equipment $\X$, and, dually, use right extensions to define the notion of \emph{weighted limit}. The
definitions involve restrictions and loose-identities, so henceforth we assume their existence in $\X$. While, in enriched category theory, weights are often taken to be presheaves~\cite[(3.5)]{kelly1982basic}, in a formal context the appropriate notion of weight is a loose-cell (\cf{}~\cites[\S4]{street1978yoneda}[\S2]{wood1982abstract})\footnotemark{}.
\footnotetext{Note that we follow modern practice in using the term \emph{weighted (co)limit}. Older texts such as \cite{street1978yoneda,wood1982abstract,kelly1982basic} instead use the term \emph{indexed (co)limit}.}

\begin{definition}
	\label{weighted-colimit}
    Let $p \colon Y \lto Z$ be a loose-cell and $f \colon Z \to X$ be a tight-cell. A \emph{$p$-weighted cocone}\footnotemark{} (or simply \emph{$p$-cocone}) for $f$ is a pair $(c, \gamma)$ of a tight-cell $c \colon Y \to X$ and a 2-cell $\gamma \colon p \tto X(f, c)$.
	\footnotetext{Weighted (co)cones are often called \emph{(co)cylinders} (\cf{}~\cite[\S3.1]{kelly1982basic}). However, this term is misleading, as the term \emph{cylinder} suggests a symmetric notion, as it is used, for instance, in \cite{freyd1972categories} or \cite{garner2015isbell}. In contrast, the term \emph{weighted (co)cone} is consistent with the terminology \emph{weighted (co)limit}.}
    A cocone $(p \wc f, \lambda)$ is \emph{colimiting} (alternatively the \emph{$p$-weighted colimit}, or simply \emph{$p$-colimit}) of $f$ when the 2-cell
\[\begin{tikzcd}[column sep=6em]
	Z & Y & X \\
	Z & Y & X \\
	Z && X
	\arrow["{X(p \wc f, 1)}"', "\shortmid"{marking}, from=1-3, to=1-2]
	\arrow["{X(p \wc f, 1)}"{description}, from=2-3, to=2-2]
	\arrow[""{name=0, anchor=center, inner sep=0}, Rightarrow, no head, from=1-3, to=2-3]
	\arrow[""{name=1, anchor=center, inner sep=0}, Rightarrow, no head, from=1-2, to=2-2]
	\arrow["p"', "\shortmid"{marking}, from=1-2, to=1-1]
	\arrow["{X(f, p \wc f)}"{description}, from=2-2, to=2-1]
	\arrow[""{name=2, anchor=center, inner sep=0}, Rightarrow, no head, from=1-1, to=2-1]
	\arrow["{X(f, 1)}", "\shortmid"{marking}, from=3-3, to=3-1]
	\arrow[""{name=3, anchor=center, inner sep=0}, Rightarrow, no head, from=2-3, to=3-3]
	\arrow[""{name=4, anchor=center, inner sep=0}, Rightarrow, no head, from=2-1, to=3-1]
	\arrow["{\cp{p \wc f}(f, 1)}"{description}, draw=none, from=3, to=4]
	\arrow["{=}"{description}, draw=none, from=0, to=1]
	\arrow["\lambda"{description}, draw=none, from=1, to=2]
\end{tikzcd}\]
	exhibits $X(p \wc f, 1)$ as the right lift $X(f, 1) \rf p$.
	A tight-cell $g \colon X \to X'$ \emph{preserves} the colimit $p \wc f$ when the cocone $(((p \wc f) \d g), (\lambda \d g))$ is the $p$-colimit of $(f \d g) \colon Z \to X'$.
\end{definition}

A weighted limit in $\X$ is a weighted colimit in $\X\co$. For convenience, we spell out the definition.

\begin{definition}
  \label{weighted-limit}
  Let $p \colon Z \lto Y$ be a loose-cell and $f \colon Z \to X$ be a tight-cell. A \emph{$p$-weighted cone} (or simply \emph{$p$-cone}) for $f$ is a pair $(c, \gamma)$ of a tight-cell $c \colon Y \to X$ and a 2-cell $\gamma \colon p \tto X(c, f)$. A cone $(p \wl f, \limcell)$ is \emph{limiting} (alternatively the \emph{$p$-weighted limit}, or simply \emph{$p$-limit}) of $f$ when the 2-cell
\[\begin{tikzcd}[column sep=6em]
	X & Y & Z \\
	X & Y & Z \\
	X && Z
	\arrow["{X(p \wl f, f)}"{description}, from=2-3, to=2-2]
	\arrow["{X(1, p \wl f)}"{description}, from=2-2, to=2-1]
	\arrow["{X(1, f)}", "\shortmid"{marking}, from=3-3, to=3-1]
	\arrow[""{name=0, anchor=center, inner sep=0}, Rightarrow, no head, from=2-1, to=3-1]
	\arrow[""{name=1, anchor=center, inner sep=0}, Rightarrow, no head, from=2-3, to=3-3]
	\arrow["p"', "\shortmid"{marking}, from=1-3, to=1-2]
	\arrow["{X(1, p \wl f)}"', "\shortmid"{marking}, from=1-2, to=1-1]
	\arrow[""{name=2, anchor=center, inner sep=0}, Rightarrow, no head, from=1-3, to=2-3]
	\arrow[""{name=3, anchor=center, inner sep=0}, Rightarrow, no head, from=1-1, to=2-1]
	\arrow[""{name=4, anchor=center, inner sep=0}, Rightarrow, no head, from=1-2, to=2-2]
	\arrow["{\cp{p \wl f}(1, f)}"{description}, draw=none, from=1, to=0]
	\arrow["\limcell"{description}, draw=none, from=2, to=4]
	\arrow["{=}"{description}, draw=none, from=4, to=3]
\end{tikzcd}\]
  exhibits $X(1, p \wl f)$ as the right extension $p \rx X(1, f)$.
  A tight-cell $g \colon X \to X'$ \emph{preserves} the limit $p \wl f$ when the cone $((p \wl f \d g), (\limcell \d g))$ is the $p$-limit of $(f \d g) \colon Z \to X'$.
\end{definition}

\begin{remark}
	The weighted (co)limits of \cref{weighted-colimit,weighted-limit} are equivalent to those of \cite[Definition~9.5.1]{riehl2022elements}.
\end{remark}

\begin{example}[{\cite[Lemma~9.5.3]{riehl2022elements}}]
  \label{colimit-companion}
  It follows from \cref{lift-companion} that companion-weighted colimits and conjoint-weighted limits are both given by composition of tight-cells:
  \[
    X(1, x) \wc f \iso (x \d f) \iso X(x, 1) \wl f
	\qedhere
  \]
\end{example}

The results relating right lifts and extensions above induce similar results for weighted (co)limits.
Again, we give proofs only for colimits; the proofs for limits are dual.
First, (co)limits interact nicely with composites.
\begin{lemma}[{\cf{}~\cite[Lemma~9.5.4]{riehl2022elements}}]
    \label{currying-colimit}
    Let $f \colon Z \to X$ be a tight-cell.
    \begin{enumerate}
      \item If $Y' \xlto{p'} Y \xlto{p} Z$ are loose-cells such that the left-composite $p \odotl p' \colon Y' \lto Z$ and colimit $p \wc f \colon Y \to X$ exist, then, if either side of the following exists, so does the other, in which case they are isomorphic.
        \[(p \odotl p') \wc f \iso p' \wc (p \wc f)\]
      \item If $Z \xlto{p} Y \xlto{p'} Y'$ are loose-cells such that the right-composite $p' \odotr p \colon Z \lto Y'$ and limit $p \wl f \colon Y \to X$ exist, then, if either side of the following exists, so does the other, in which case they are isomorphic.
        \[(p' \odotr p) \wl f \iso p' \wl (p \wl f)\]
    \end{enumerate}
\end{lemma}
\begin{proof}
  Immediate from \cref{currying-right-lift}.
\end{proof}

Second, weighted (co)limits interact with restriction as follows.
\begin{lemma}
    \label{colimit-and-restriction}
    Let $p \colon Y \lto Z$ be a loose-cell and $f \colon Z \to X$ be a tight-cell such that the colimit $p \wc f \colon Y \to X$ exists.
    For each $x \colon X' \to X$ and $y \colon Y' \to Y$, we have a right lift
    \[
      X((p \wc f)y, x) \iso X(f, x) \rf p(1, y)
    \]
    with counit
\[\begin{tikzcd}[column sep=6em]
	Z & {Y'} & {X'} \\
	Z & {Y'} & {X'} \\
	Z && {X'}
	\arrow["{X((p \wc f)y, x)}"', "\shortmid"{marking}, from=1-3, to=1-2]
	\arrow["{X((p \wc f)y, x)}"{description}, from=2-3, to=2-2]
	\arrow[""{name=0, anchor=center, inner sep=0}, Rightarrow, no head, from=1-3, to=2-3]
	\arrow[""{name=1, anchor=center, inner sep=0}, Rightarrow, no head, from=1-2, to=2-2]
	\arrow["{p(1, y)}"', "\shortmid"{marking}, from=1-2, to=1-1]
	\arrow["{X(f, (p \wc f)y)}"{description}, from=2-2, to=2-1]
	\arrow[""{name=2, anchor=center, inner sep=0}, Rightarrow, no head, from=1-1, to=2-1]
	\arrow["{X(f, x)}", "\shortmid"{marking}, from=3-3, to=3-1]
	\arrow[""{name=3, anchor=center, inner sep=0}, Rightarrow, no head, from=2-3, to=3-3]
	\arrow[""{name=4, anchor=center, inner sep=0}, Rightarrow, no head, from=2-1, to=3-1]
	\arrow["{\cp{(p \wc f)y}(f, x)}"{description}, draw=none, from=3, to=4]
	\arrow["{=}"{description}, draw=none, from=0, to=1]
	\arrow["{\lambda(1, y)}"{description}, draw=none, from=1, to=2]
\end{tikzcd}\]
    In particular, for each tight-cell $y \colon Y' \to Y$, the 2-cell
    \[
      \lambda(1, y) \colon p(1, y) \tto X(f, (p \wc f)y)
    \]
    forms the colimiting $p(1, y)$-cocone of $f$:
    \[
      (y \d (p \wc f)) \iso p(1, y) \wc f
    \]
    Dually, weighted limits $p \wl f$ satisfy the following.
    \[
      X(x, (p \wl f)y) \iso p(y, 1) \rx X(x, f)
      \hspace{4em}
      (y \d (p \wl f)) \iso p(y, 1) \wl f
    \]
\end{lemma}
\begin{proof}
  Using \cref{right-lift-and-restriction}, we have
  \[
    X((p \wc f)y, x) \iso X(p \wc f, 1)(y, x)
    \iso (X(f, 1) \rf p)(y, x)
    \iso X(f, x) \rf p(1, y)
  \]
  from which we may calculate the counit of the right lift.
  The second part, which is \cite[Lemma~9.5.5]{riehl2022elements}, follows by taking $x = 1_X$.
\end{proof}

We also observe the following interaction between weighted colimits and weighted limits.
\begin{lemma}[{\cf{}~\cite[Proposition~8.5]{shulman2013enriched}}]
	Let $p \colon Y \lto Z$ be a loose-cell, and $f \colon Z \to X$ and $g \colon Y \to X$ be tight-cells. Supposing that the $p$-colimit $p \wc f$ and $p$-limit $p \wl g$ exist, there is a natural bijection of 2-cells in the tight 2-category $\tX$:
	\begin{iffseq}
		p \wc f \tto g \\
		f \tto p \wl g
	\end{iffseq}
\end{lemma}

\begin{proof}
  We have $X(p \wc f, g) \iso X(f, g) \rf p$ and $X(f, p \wl g) \iso p \rx X(f, g)$ by \cref{colimit-and-restriction}, so there are natural bijections as follows.
  \begin{iffseq}
	p \wc f \tto g \\
	p \tto X(f, g) \\
	f \tto p \wl g
  \end{iffseq}\\[-4.5ex]
\end{proof}

\subsection{Pointwise left extensions}

We specialise the definition of weighted colimit to obtain a definition
of (pointwise) \emph{left extension}.
(We will not need to consider the dual notion of right extension explicitly here.)
In enriched category theory, there are two notions of left extension: nonpointwise extensions, which are defined by a 2-categorical universal property; and pointwise extensions, which are typically defined by a universal property involving presheaves. In a formal context, the nonpointwise notion is appropriate for loose-cells (\cf{}~\cref{right-extension}), whereas the pointwise notion is appropriate for tight-cells. Concretely, in $\Cat$, it is generally appropriate only to consider nonpointwise extensions of \emph{distributors}, and consequently only to consider pointwise extensions of functors. We shall therefore drop the qualifiers \emph{pointwise} and \emph{nonpointwise} except for emphasis.

\begin{definition}
    \label{left-extension}
    Let $j \colon Z \to Y$ and $f \colon Z \to X$ be tight-cells.
    A tight-cell $j \plx f \colon Y \to X$ equipped with a 2-cell
    $\lextcell \colon f \tto j \d (j \plx f)$ is the \emph{left extension} of $f$ along
    $j$ when the 2-cell
    \[
        Y(j, 1) \xtto{\pc{j \plx f}(j, 1)}
        X((j \plx f)j, j \plx f) \xtto{X(\lextcell, j \plx f)}
        X(f, j \plx f)
    \]
    exhibits $j \plx f$ as the $Y(j, 1)$-colimit of $f$.
\end{definition}

As with right lifts and weighted colimits, left extensions interact nicely with identities and composites, and with restriction.

\begin{lemma}
    Let $f \colon Z \to X$ be a tight-cell.
    \begin{enumerate}
      \item The left extension $1_Z \plx f$ exists and is isomorphic to $f$.
      \item \label{currying-pointwise-extension} If $Z \xto{j} Y \xto{j'} Y'$ are tight-cells such that the left extension $j \plx f \colon Y \to X$ exists, then, if either side of the following exists, so does the other, in which case they are isomorphic.
        \[(j'j) \plx f \iso j' \plx (j \plx f)\]
    \end{enumerate}
\end{lemma}
\begin{proof}
  Immediate from \cref{currying-colimit}, since $Y(j, 1) \odot Y'(j', 1) \iso Y'(j'j, 1)$.
\end{proof}

\begin{lemma}
    \label{pointwise-extension-and-restriction}
    Let $j \colon Z \to Y$ and $f \colon Z \to X$ be tight-cells such that the left extension $j \plx f \colon Y \to X$ exists.
    For each $x \colon X' \to X$ and $y \colon Y' \to Y$, the following 2-cell exhibits $X((j \plx f)y, x)$ as the right lift $X(f, x) \rf Y(j, y)$.
\[\begin{tikzcd}[column sep=9em]
	Z & {Y'} & {X'} \\
	Z & {Y'} \\
	Z & {Y'} & {X'} \\
	Z && {X'}
	\arrow[""{name=0, anchor=center, inner sep=0}, "{X((j \plx f)y, x)}"', "\shortmid"{marking}, from=1-3, to=1-2]
	\arrow[""{name=1, anchor=center, inner sep=0}, "{X((j \plx f)y, x)}"{description}, from=3-3, to=3-2]
	\arrow["{Y(j, y)}"', "\shortmid"{marking}, from=1-2, to=1-1]
	\arrow["{X(f, (j \plx f)y)}"{description}, from=3-2, to=3-1]
	\arrow["{X(f, x)}", "\shortmid"{marking}, from=4-3, to=4-1]
	\arrow[""{name=2, anchor=center, inner sep=0}, Rightarrow, no head, from=3-3, to=4-3]
	\arrow[""{name=3, anchor=center, inner sep=0}, Rightarrow, no head, from=3-1, to=4-1]
	\arrow[""{name=4, anchor=center, inner sep=0}, Rightarrow, no head, from=1-2, to=2-2]
	\arrow[""{name=5, anchor=center, inner sep=0}, Rightarrow, no head, from=1-1, to=2-1]
	\arrow[""{name=6, anchor=center, inner sep=0}, Rightarrow, no head, from=2-2, to=3-2]
	\arrow[""{name=7, anchor=center, inner sep=0}, Rightarrow, no head, from=2-1, to=3-1]
	\arrow[Rightarrow, no head, from=1-3, to=3-3]
	\arrow["{X((j \plx f) j, (j \plx f) y)}"{description}, from=2-2, to=2-1]
	\arrow["{\cp{(j \plx f)y}(f, x)}"{description}, draw=none, from=2, to=3]
	\arrow["{\pc{j \plx f}(j, y)}"{description}, draw=none, from=4, to=5]
	\arrow["{=}"{description}, draw=none, from=1, to=0]
	\arrow["{X(\lextcell, (j \plx f) y)}"{description}, draw=none, from=6, to=7]
\end{tikzcd}\]
    In particular, for each $y \colon Y' \to Y$, the 2-cell
    \[
        Y(j, y) \xtto{\pc {j \plx f}(j, y)}
        X((j \plx f)j, (j \plx f)y) \xtto{X(\lextcell, (j \plx f)y)}
        X(f, (j \plx f)y)
    \]
    exhibits $(y \d (j \plx f)) \colon Y' \to X$ as the colimit $Y(j, y) \wc f$.
\end{lemma}

\begin{proof}
  Immediate from \cref{colimit-and-restriction}.
\end{proof}

Pointwise extensions in particular satisfy the 2-categorical universal property of nonpointwise extensions.

\begin{lemma}
    \label{hom-cell-via-left-extension}
    Let $j \colon Z \to Y$ and $f \colon Z \to X$ be tight-cells such that the left extension $j \plx f \colon Y \to X$ exists.
    The universal property induces a bijection of 2-cells:
    \begin{iffseq}
        r_1, \dots, r_n \tto X((j \plx f)y, x) \\
        Y(j, y), r_1, \dots, r_n \tto X(f, x)
    \end{iffseq}
    In particular, there is a natural bijection of 2-cells in the tight 2-category $\tX$,
    \begin{iffseq}
        j \plx f \tto x \\
        f \tto j \d x
    \end{iffseq}
    so that $\lextcell \colon f \tto j \d (j \plx f)$ exhibits $j \plx f$ as the (nonpointwise) left extension of $f$ along $j$.
\end{lemma}

\begin{proof}
    By \cref{pointwise-extension-and-restriction},
    $X((j \plx f)y, x)$ is the right lift
    $X(f, x) \rf Y(j, y)$.
    The first bijection is immediate from the universal property of this right lift.
    The second bijection follows by taking $y = 1_Y$ and $n = 0$, since we have natural bijections
      \begin{iffseq}
        j \plx f \tto x \\
        {~} \tto X(j \plx f, x) \\
        Y(j, 1) \tto X(f, x) \\
        f \tto j \d x
      \end{iffseq}
	using the first bijection, and the universal properties of the restrictions.
\end{proof}

\subsection{Density and absolute colimits}

Restriction is \ff{}, so that 2-cells $E(1, f) \tto E(1, g)$ between representable loose-cells are in bijection with 2-cells $f \tto g$ between tight-cells. We shall often desire a similar property with respect to 2-cells between $j$-representable loose-cells, \ie{} a bijection between 2-cells $E(j, f) \tto E(j, g)$ and 2-cells $f \tto g$. This holds provided that the tight-cell $\jAE$ is \emph{dense} in the following sense.

\begin{definition}
    \label{density}
    A tight-cell $\jAE$ is \emph{dense} when the identity 2-cell
    $1_j \colon j \tto j$ exhibits $1_E \colon E \to E$ as the left extension $j \plx j$.
\end{definition}

\begin{lemma}
    \label{density-implies-j*-is-ff}
    A tight-cell $\jAE$ is dense if and only if the assignment of a 2-cell $\phi \colon r_1, \ldots, r_n \tto E(g, h)$ to the 2-cell
	\[\begin{tikzcd}
		A && \cdot & \cdots & \cdot \\
		A && \cdot && \cdot \\
		A &&&& \cdot
		\arrow[""{name=0, anchor=center, inner sep=0}, Rightarrow, no head, from=1-1, to=2-1]
		\arrow["{E(j, g)}"', "\shortmid"{marking}, from=1-3, to=1-1]
		\arrow[""{name=1, anchor=center, inner sep=0}, Rightarrow, no head, from=1-3, to=2-3]
		\arrow["{r_1}"', "\shortmid"{marking}, from=1-4, to=1-3]
		\arrow["{r_n}"', "\shortmid"{marking}, from=1-5, to=1-4]
		\arrow[""{name=2, anchor=center, inner sep=0}, Rightarrow, no head, from=1-5, to=2-5]
		\arrow[""{name=3, anchor=center, inner sep=0}, Rightarrow, no head, from=2-1, to=3-1]
		\arrow["{E(j, g)}"{description}, from=2-3, to=2-1]
		\arrow["{E(g, h)}"{description}, from=2-5, to=2-3]
		\arrow[""{name=4, anchor=center, inner sep=0}, Rightarrow, no head, from=2-5, to=3-5]
		\arrow["{E(j, h)}", "\shortmid"{marking}, from=3-5, to=3-1]
		\arrow["{=}"{description}, draw=none, from=1, to=0]
		\arrow["\phi"{description}, draw=none, from=2, to=1]
		\arrow["{\cp g(j, h)}"{description}, draw=none, from=4, to=3]
	\end{tikzcd}\]
	induces a bijection of 2-cells:
    \begin{iffseq}
        r_1, \dots, r_n \tto E(g, h) \\
        E(j, g), r_1, \dots, r_n \tto E(j, h)
    \end{iffseq}
    In particular, if $j$ is dense, then there is a bijection of 2-cells:
    \begin{iffseq}
        g \tto h \\
        E(j, g) \tto E(j, h)
    \end{iffseq}
\end{lemma}

\begin{proof}
  The first bijection states that the above 2-cell exhibits $E(g, h)$ as the right lift $E(j, h) \rf E(j, g)$.
  Hence, if $j$ is dense, the first bijection follows from \cref{hom-cell-via-left-extension}.
  For the converse, by taking $g = h = 1_E$ we obtain a canonical isomorphism $E(1, 1) \iso E(j, 1) \rf E(j, 1)$, and hence density of $j$.

    The second bijection follows from the first by taking $n = 0$ and using the following bijection.
	\begin{iffseq}
		g \tto h \\
		{~} \tto E(g, h)
	\end{iffseq}\\[-4.5ex]
\end{proof}

\emph{Absoluteness} with respect to a tight-cell $\jAE$ is a well-behavedness condition for colimits that permits the calculation of a weighted colimit via loose-composition.

\begin{definition}
	\label{j-absolute}
    Let $p \colon Y \lto Z$ be a loose-cell and let $\jAE$ and $f \colon Z \to E$ be
    tight-cells.
    The colimit $(p \wc f, \lambda)$ is \emph{$j$-absolute} if
    the 2-cell
	\[\begin{tikzcd}[column sep=huge]
		A & Z & Y \\
		A & Z & Y \\
		A && Y
		\arrow["{E(j, f)}"', "\shortmid"{marking}, from=1-2, to=1-1]
		\arrow["p"', "\shortmid"{marking}, from=1-3, to=1-2]
		\arrow["{E(j, f)}"{description}, from=2-2, to=2-1]
		\arrow["{E(f, p \wc f)}"{description}, from=2-3, to=2-2]
		\arrow[""{name=0, anchor=center, inner sep=0}, Rightarrow, no head, from=3-3, to=2-3]
		\arrow[""{name=1, anchor=center, inner sep=0}, Rightarrow, no head, from=3-1, to=2-1]
		\arrow[""{name=2, anchor=center, inner sep=0}, Rightarrow, no head, from=2-1, to=1-1]
		\arrow[""{name=3, anchor=center, inner sep=0}, Rightarrow, no head, from=2-2, to=1-2]
		\arrow[""{name=4, anchor=center, inner sep=0}, Rightarrow, no head, from=2-3, to=1-3]
		\arrow["{E(j, p \wc f)}", "\shortmid"{marking}, from=3-3, to=3-1]
		\arrow["\lambda"{description}, draw=none, from=4, to=3]
		\arrow["{\cp f(j, p \wc f)}"{description}, draw=none, from=0, to=1]
		\arrow["{=}"{description}, draw=none, from=3, to=2]
	\end{tikzcd}\]
    is left-opcartesian. A colimit with codomain $E$ is \emph{absolute} if it is $1_E$-absolute.
\end{definition}

In other words, $p \wc f$ is $j$-absolute just when the left-composite
$E(j, f) \odotl p$ exists, and the canonical 2-cell $E(j, f) \odotl p \tto E(j, p \wc f)$ is an isomorphism. In particular, a left extension $g \plx f$ is $j$-absolute just when the left composite $E(j, f) \odotl C(g, 1)$ exists, and the canonical 2-cell $E(j, f) \odotl C(g, 1) \tto E(j, g \plx f)$ is an isomorphism.

\begin{lemma}
    \label{j-plx-preserves-j-absolute-colimits}
    Each left extension $j \plx f$ preserves $j$-absolute colimits.
\end{lemma}

\begin{proof}
    Let $p \wc g$ be a $j$-absolute colimit. We have
    \begin{align*}
      p \wc (g \d (j \plx f))
      &\iso \tag{\cref{pointwise-extension-and-restriction}}
      p \wc (E(j, g) \wc f)
      \\&\iso \tag{\cref{currying-colimit}}
      (E(j, g) \odotl p) \wc f
      \\&\iso \tag{$p \wc g$ is $j$-absolute}
      E(j, p \wc g) \wc f
      \\&\iso \tag{\cref{pointwise-extension-and-restriction}}
      (p \wc g) \d (j \plx f)
    \end{align*}
    A simple calculation shows that the 2-cell $p \tto E((j \plx f) g, (j \plx f) (p \wc g))$ induced by these isomorphisms is the canonical one.
\end{proof}

In general, to show that a tight-cell $f$ forms a $j$-absolute colimit, we
must show both that it forms the colimit and that a particular 2-cell
is left-opcartesian.
When $j$ is dense, it is enough to establish the existence of a
left-opcartesian 2-cell, which then implies that $f$ forms a colimit, as we show in the following
lemma.

\begin{lemma}\label{absolute-implies-colimit}
  Let $p \colon Y \lto Z$ be a loose-cell and let $\jAE$ and $f \colon Z \to E$ be tight-cells.
  If $j$ is dense, then a tight-cell $f' \colon Y \to E$ forms the $j$-absolute $p$-colimit of $f$ if and only if there is an isomorphism:
  \[
     E(j, f) \odotl p \iso E(j, f')
  \]
\end{lemma}
\begin{proof}
  The only if direction is trivial.
  For the other direction, assume there is such an isomorphism.
  Then $f'$ forms the colimit $p \wc f$ because
  \begin{align*}
    f'
    &\iso E(j, f') \wc j \tag{\cref{pointwise-extension-and-restriction}, using density of $j$}
    \\&\iso (E(j, f) \odotl p) \wc j \tag{assumption}
    \\&\iso p \wc (E(j, f) \wc j) \tag{\cref{currying-colimit}}
    \\&\iso p \wc f \tag{\cref{pointwise-extension-and-restriction}, using density of $j$}
  \end{align*}
  The universal 2-cell $\lambda \colon p \tto E(f, f')$ is the unique 2-cell such that the left-opcartesian 2-cell ${E(j, f), p \tto E(j, f')}$ witnessing the isomorphism above is equal to that in the definition of $j$-absoluteness (\cref{j-absolute}).
  Hence this colimit is $j$-absolute.
\end{proof}

\subsection{Pointwise left lifts}
We define a notion of \emph{pointwise} left lift.
This notion has not explicitly appeared in the literature previously.
It is the appropriate pointwise notion of left lift, in the same sense that \cref{left-extension} is the appropriate pointwise notion of left extension.
Pointwise left lifts are closely related to relative adjunctions, as we explain in \cref{relative-adjunctions}.

\begin{definition}
    \label{left-lift}
    Let $j \colon Z \to X$ and $f \colon Y \to X$ be tight-cells. A tight-cell $j \plf f \colon Z \to Y$ equipped with a 2-cell $\lliftcell \colon j \tto (j \plf f) \d f$ is the \emph{left lift} of $j$ through $f$ when the following 2-cell exhibits $Y(j \plf f, 1)$ as the right extension $X(f, 1) \rx X(j, 1)$.
	\[\begin{tikzcd}[column sep=6em]
		Z & Y & X \\
		Z & Y \\
		Z & Y & X \\
		Z && X
		\arrow["{X(f, 1)}"', "\shortmid"{marking}, from=1-3, to=1-2]
		\arrow["{Y(j \plf f, 1)}"', "\shortmid"{marking}, from=1-2, to=1-1]
		\arrow["{X(f(j \plf f), f)}"{description}, from=2-2, to=2-1]
		\arrow[""{name=0, anchor=center, inner sep=0}, Rightarrow, no head, from=1-2, to=2-2]
		\arrow[""{name=1, anchor=center, inner sep=0}, Rightarrow, no head, from=1-1, to=2-1]
		\arrow[Rightarrow, no head, from=1-3, to=3-3]
		\arrow[""{name=2, anchor=center, inner sep=0}, Rightarrow, no head, from=2-2, to=3-2]
		\arrow["{X(f, 1)}"{description}, from=3-3, to=3-2]
		\arrow["{X(j, f)}"{description}, from=3-2, to=3-1]
		\arrow[""{name=3, anchor=center, inner sep=0}, Rightarrow, no head, from=2-1, to=3-1]
		\arrow[""{name=4, anchor=center, inner sep=0}, Rightarrow, no head, from=3-3, to=4-3]
		\arrow[""{name=5, anchor=center, inner sep=0}, Rightarrow, no head, from=3-1, to=4-1]
		\arrow["{X(j, 1)}", "\shortmid"{marking}, from=4-3, to=4-1]
		\arrow["{\pc f(j \plf f, 1)}"{description}, draw=none, from=0, to=1]
		\arrow["{X(\lliftcell, f)}"{description}, draw=none, from=2, to=3]
		\arrow["{\cp f(j, 1)}"{description}, draw=none, from=4, to=5]
	\end{tikzcd}\]
\end{definition}

We give the definition in the above form for ease of comparison with the definition of left extension (\cref{left-extension}).
However, since $X(f, 1) \odot Y(j \plf f, 1) \iso X(f(j \plf f), 1)$, the 2-cell above can be written equivalently as the following composite.
\[
  Y(j \plf f, 1), X(f, 1) \xtto{\opcart} X(f(j \plf f), 1) \xtto{X(\lliftcell, 1)} X(j, 1)
\]

\begin{lemma}
    \label{pointwise-left-lift-is-hom}
    Let $j \colon Z \to X$ and $f \colon Y \to X$ be tight-cells and suppose that the left lift $j \plf f \colon Z \to Y$ exists.
    Then the 2-cell
    \[
      Y(j \plf f, 1) \xtto{\pc f(j \plx f, 1)} X(f(j \plf f), f) \xtto{X(\lliftcell, f)} X(j, f)
    \]
    is an isomorphism.
\end{lemma}

\begin{proof}
  By \cref{lift-companion}, we have isomorphisms
  $
    Y(j \plf f, 1)
    \iso
    E(f, 1) \rx E(j, 1)
    \iso
    E(j, f)
  $,
  which compose to the required 2-cell.
\end{proof}

In \cref{hom-cell-via-left-extension}, we observed that every pointwise left extension in $\X$ was in particular a nonpointwise left extension in the tight 2-category $\tX$. The following shows that pointwise left lifts in $\X$ satisfy an analogous, but stronger, universal property.

\begin{proposition}
    \label{pointwise-left-lift-is-absolute-left-lift}
    Let $j \colon Z \to X$ and $f \colon Y \to X$ be tight-cells and suppose the left lift $j \plf f \colon Z \to Y$ exists. Then $j \plf f$ is an absolute (nonpointwise) left lift in the tight 2-category $\tX$.
\end{proposition}

\begin{proof}
  We show that, for every $z \colon Z' \to Z$, the tight-cell $(z \d (j \plf f)) \colon Z' \to Y$ equipped with the 2-cell
  $(z \d \eta) \colon (z \d j) \tto (z \d (j \plf f) \d f)$
  is the left lift of $(z \d j)$ through $f$ in $\tX$.
  For each $y \colon Y' \to Y$ we have
  \[
    Y((j \plf f)z, y)
    \iso Y(j \plf f, 1)(z, y)
    \iso X(j, f)(z, y)
    \iso X(jz, fy)
  \]
  by \cref{pointwise-left-lift-is-hom}.
  Hence there are bijections
  \begin{iffseq}
      z \d (j \plf f) \tto y \\
      z \d j \tto y \d f
  \end{iffseq}
  that send the identity on $z \d (j \plf f)$ to $(z \d \eta)$, as required.
\end{proof}

We summarise the relationship between pointwise left and right extensions and lifts in the table below; pointwise right extensions and lifts in $\X$ are defined to be pointwise left extensions and lifts in $\X\co$.

\begin{center}
	\begin{tblr}{c|ccc}
		Pointwise \underline{\hspace{0.5cm}} & \SetCell[c=3]{c}{are characterised by \underline{\hspace{0.5cm}}.} \\
		\hline
		left extensions & $X(j \plx f, 1)$ & $\iso$ & $X(f, 1) \rf Y(j, 1)$ \\
		left lifts & $Y(j \plf f, 1)$ & $\iso$ & $X(f, 1) \rx X(j, 1)$ \\
		right extensions & $X(1, j \prx f)$ & $\iso$ & $Y(1, j) \rx X(1, f)$ \\
		right lifts & $Y(1, j \prf f)$ & $\iso$ & $X(1, j) \rf X(1, f)$
	\end{tblr}
\end{center}

Note that pointwise left extensions and pointwise left lifts in $\X$, like pointwise right lifts and pointwise right extensions in $\X$, are characterised in terms of \emph{right} lifts and extensions in $\X$.

\subsection{Full faithfulness}

\begin{definition}
    \label{full-faithfulness}
    A tight-cell $\jAE$ is \emph{\ff{}} when any of the following equivalent conditions holds.
	\begin{enumerate}
		\item The 2-cell $1_j \colon j \tto j$ is cartesian~\cite[Definition~4.12]{koudenburg2020augmented}.
		\item The 2-cell $\pc j \colon \tto E(j, j)$ is opcartesian.
		\item The 2-cell $A(1, 1) \tto E(j, j)$ induced by $\pc j$ is invertible. \qedhere
	\end{enumerate}
\end{definition}

In fact, condition (3) may be weakened to asking for there to be \emph{any} isomorphism of loose-cells.

\begin{corollary}
	\label{ff-iff-invertible}
	A tight-cell $\jAE$ is \ff{} if and only if $A(1, 1) \iso E(j, j)$.
\end{corollary}

\begin{proof}
	The postcomposition 2-cell $\pc j \colon {} \tto E(j, j)$ is the unit of the loose-adjunction $E(1, j) \adj E(j, 1)$ and is hence opcartesian if and only if $A(1, 1) \iso E(j, j)$ by \cref{invertible-unit-if-noncanonical-isomorphism}.
\end{proof}

The following lemma is useful in dealing with left extensions along \ff{} tight-cells.

\begin{lemma}
    \label{ff-implies-lx-unit-is-invertible}
    Let $\jAE$ be a tight-cell. If $j$ is \ff{}, then for every tight-cell $f \colon A \to X$ for which the left extension $(j \plx f, \pi)$ exists, the 2-cell $\lextcell \colon f \tto j \d (j \plx f)$ is invertible.
\end{lemma}

\begin{proof}
    The 2-cell $\lextcell$ is equal to the composite of the following isomorphisms.
    \begin{align*}
      f
      &~\iso~A(1, 1) \wc f \tag{\cref{currying-colimit}}
      \\
      &~\iso~E(j, j) \wc f \tag{$j$ is \ff{}}
      \\
      &~\iso~j \d (j \plx f) \tag{\cref{pointwise-extension-and-restriction}}
    \end{align*}
\end{proof}

\section{Relative monads}
\label{skew-multicategorical-hom-categories}

With the preliminaries out of the way, we begin by introducing the definition of a relative monad, and of a morphism of relative monads, in a \ve{}.

\begin{definition}
    \label{relative-monad}
    Let $\X$ be an equipment. A \emph{relative monad} in $\X$ comprises
    \begin{enumerate}
        \item a tight-cell $\jAE$, the \emph{root};
        \item a tight-cell $t \colon A \to E$, the \emph{carrier} or \emph{underlying tight-cell};
        \item a 2-cell $\dag \colon E(j, t) \tto E(t, t)$, the \emph{extension operator};
        \item a 2-cell $\eta \colon j \tto t$, the \emph{unit},
    \end{enumerate}
    satisfying the following equations.
    \[
	\begin{tikzcd}[column sep=large]
		A & A \\
		A & A \\
		A & A
		\arrow["{E(j, t)}"', "\shortmid"{marking}, from=1-2, to=1-1]
		\arrow["{E(t, t)}"{description}, from=2-2, to=2-1]
		\arrow[""{name=0, anchor=center, inner sep=0}, Rightarrow, no head, from=1-2, to=2-2]
		\arrow[""{name=1, anchor=center, inner sep=0}, Rightarrow, no head, from=1-1, to=2-1]
		\arrow[""{name=2, anchor=center, inner sep=0}, Rightarrow, no head, from=2-2, to=3-2]
		\arrow[""{name=3, anchor=center, inner sep=0}, Rightarrow, no head, from=2-1, to=3-1]
		\arrow["{E(j, t)}", "\shortmid"{marking}, from=3-2, to=3-1]
		\arrow["\dag"{description}, draw=none, from=0, to=1]
		\arrow["{E(\eta, t)}"{description}, draw=none, from=2, to=3]
	\end{tikzcd}
    \quad = \quad
	\begin{tikzcd}
		A & A \\
		A & A
		\arrow["{E(j, t)}"', "\shortmid"{marking}, from=1-2, to=1-1]
		\arrow["{E(j, t)}", from=2-2, to=2-1]
		\arrow[""{name=0, anchor=center, inner sep=0}, Rightarrow, no head, from=1-2, to=2-2]
		\arrow[""{name=1, anchor=center, inner sep=0}, Rightarrow, no head, from=1-1, to=2-1]
		\arrow["{=}"{description}, draw=none, from=0, to=1]
	\end{tikzcd}
	\hspace{4em}
	\begin{tangle}{(2,4)}[trim y]
		\tgBorderA{(0,0)}{\tgColour6}{\tgColour4}{\tgColour4}{\tgColour6}
		\tgBorderA{(1,0)}{\tgColour4}{\tgColour6}{\tgColour6}{\tgColour4}
		\tgBorderA{(0,1)}{\tgColour6}{\tgColour4}{\tgColour4}{\tgColour6}
		\tgBorder{(0,1)}{0}{1}{0}{0}
		\tgBorderA{(1,1)}{\tgColour4}{\tgColour6}{\tgColour6}{\tgColour4}
		\tgBorder{(1,1)}{0}{0}{0}{1}
		\tgBorderA{(0,2)}{\tgColour6}{\tgColour4}{\tgColour4}{\tgColour6}
		\tgBorderA{(1,2)}{\tgColour4}{\tgColour6}{\tgColour6}{\tgColour4}
		\tgBorderA{(0,3)}{\tgColour6}{\tgColour4}{\tgColour4}{\tgColour6}
		\tgBorderA{(1,3)}{\tgColour4}{\tgColour6}{\tgColour6}{\tgColour4}
		\tgCell[(1,0)]{(0.5,1)}{\dag}
		\tgCell{(0,2)}{\eta}
		\tgArrow{(0,1.5)}{1}
		\tgArrow{(0,2.5)}{1}
		\tgArrow{(0,0.5)}{1}
		\tgArrow{(1,0.5)}{3}
		\tgArrow{(1,1.5)}{3}
		\tgArrow{(1,2.5)}{3}
		\tgAxisLabel{(0.5,0.75)}{south}{j}
		\tgAxisLabel{(1.5,0.75)}{south}{t}
		\tgAxisLabel{(0.5,3.25)}{north}{j}
		\tgAxisLabel{(1.5,3.25)}{north}{t}
	\end{tangle}
    \tangleeq*
	\begin{tangle}{(2,4)}[trim y]
		\tgBorderA{(0,0)}{\tgColour6}{\tgColour4}{\tgColour4}{\tgColour6}
		\tgBorderA{(1,0)}{\tgColour4}{\tgColour6}{\tgColour6}{\tgColour4}
		\tgBorderA{(0,1)}{\tgColour6}{\tgColour4}{\tgColour4}{\tgColour6}
		\tgBorderA{(1,1)}{\tgColour4}{\tgColour6}{\tgColour6}{\tgColour4}
		\tgBorderA{(0,2)}{\tgColour6}{\tgColour4}{\tgColour4}{\tgColour6}
		\tgBorderA{(1,2)}{\tgColour4}{\tgColour6}{\tgColour6}{\tgColour4}
		\tgBorderA{(0,3)}{\tgColour6}{\tgColour4}{\tgColour4}{\tgColour6}
		\tgBorderA{(1,3)}{\tgColour4}{\tgColour6}{\tgColour6}{\tgColour4}
		\tgArrow{(0,1.5)}{1}
		\tgArrow{(0,2.5)}{1}
		\tgArrow{(0,0.5)}{1}
		\tgArrow{(1,0.5)}{3}
		\tgArrow{(1,1.5)}{3}
		\tgArrow{(1,2.5)}{3}
		\tgAxisLabel{(0.5,0.75)}{south}{j}
		\tgAxisLabel{(1.5,0.75)}{south}{t}
		\tgAxisLabel{(0.5,3.25)}{north}{j}
		\tgAxisLabel{(1.5,3.25)}{north}{t}
	\end{tangle}
    \]
    \[
	\begin{tikzcd}[column sep=large]
		A & A \\
		A & A \\
		A & A \\
		A & A
		\arrow["{E(j, t)}"{description}, from=3-2, to=3-1]
		\arrow["{E(t, t)}", "\shortmid"{marking}, from=4-2, to=4-1]
		\arrow[""{name=0, anchor=center, inner sep=0}, Rightarrow, no head, from=3-2, to=4-2]
		\arrow[""{name=1, anchor=center, inner sep=0}, Rightarrow, no head, from=3-1, to=4-1]
		\arrow["{E(j, j)}"{description}, from=2-2, to=2-1]
		\arrow[""{name=2, anchor=center, inner sep=0}, Rightarrow, no head, from=2-2, to=3-2]
		\arrow[""{name=3, anchor=center, inner sep=0}, Rightarrow, no head, from=2-1, to=3-1]
		\arrow[""{name=4, anchor=center, inner sep=0}, Rightarrow, no head, from=1-2, to=2-2]
		\arrow[Rightarrow, no head, from=1-2, to=1-1]
		\arrow[""{name=5, anchor=center, inner sep=0}, Rightarrow, no head, from=1-1, to=2-1]
		\arrow["{\pc j}"{description}, draw=none, from=4, to=5]
		\arrow["{E(j, \eta)}"{description}, draw=none, from=2, to=3]
		\arrow["\dag"{description}, draw=none, from=0, to=1]
	\end{tikzcd}
    \quad = \quad
	\begin{tikzcd}
		A & A \\
		A & A
		\arrow["{E(t, t)}", "\shortmid"{marking}, from=2-2, to=2-1]
		\arrow[""{name=0, anchor=center, inner sep=0}, Rightarrow, no head, from=1-2, to=2-2]
		\arrow[Rightarrow, no head, from=1-2, to=1-1]
		\arrow[""{name=1, anchor=center, inner sep=0}, Rightarrow, no head, from=1-1, to=2-1]
		\arrow["{\pc t}"{description}, draw=none, from=0, to=1]
	\end{tikzcd}
	\hspace{4em}
	\begin{tangle}{(2,4)}[trim y]
		\tgBlank{(0,0)}{\tgColour6}
		\tgBlank{(1,0)}{\tgColour6}
		\tgBorderA{(0,1)}{\tgColour6}{\tgColour6}{\tgColour4}{\tgColour6}
		\tgBorderA{(1,1)}{\tgColour6}{\tgColour6}{\tgColour6}{\tgColour4}
		\tgBorderA{(0,2)}{\tgColour6}{\tgColour4}{\tgColour4}{\tgColour6}
		\tgBorder{(0,2)}{0}{1}{0}{0}
		\tgBorderA{(1,2)}{\tgColour4}{\tgColour6}{\tgColour6}{\tgColour4}
		\tgBorder{(1,2)}{0}{0}{0}{1}
		\tgBorderA{(0,3)}{\tgColour6}{\tgColour4}{\tgColour4}{\tgColour6}
		\tgBorderA{(1,3)}{\tgColour4}{\tgColour6}{\tgColour6}{\tgColour4}
		\tgCell[(1,0)]{(0.5,1)}{\eta}
		\tgCell[(1,0)]{(0.5,2)}{\dag}
		\tgArrow{(0,1.5)}{1}
		\tgArrow{(0,2.5)}{1}
		\tgArrow{(1,1.5)}{3}
		\tgArrow{(1,2.5)}{3}
		\tgAxisLabel{(0.5,3.25)}{north}{t}
		\tgAxisLabel{(1.5,3.25)}{north}{t}
	\end{tangle}
    \tangleeq*
	\begin{tangle}{(2,4)}[trim y]
		\tgBlank{(0,0)}{\tgColour6}
		\tgBlank{(1,0)}{\tgColour6}
		\tgBorderC{(0,1)}{3}{\tgColour6}{\tgColour4}
		\tgBorderC{(1,1)}{2}{\tgColour6}{\tgColour4}
		\tgBorderA{(0,2)}{\tgColour6}{\tgColour4}{\tgColour4}{\tgColour6}
		\tgBorderA{(1,2)}{\tgColour4}{\tgColour6}{\tgColour6}{\tgColour4}
		\tgBorderA{(0,3)}{\tgColour6}{\tgColour4}{\tgColour4}{\tgColour6}
		\tgBorderA{(1,3)}{\tgColour4}{\tgColour6}{\tgColour6}{\tgColour4}
		\tgArrow{(0.5,1)}{0}
		\tgArrow{(1,1.5)}{3}
		\tgArrow{(1,2.5)}{3}
		\tgArrow{(0,2.5)}{1}
		\tgArrow{(0,1.5)}{1}
		\tgAxisLabel{(0.5,3.25)}{north}{t}
		\tgAxisLabel{(1.5,3.25)}{north}{t}
	\end{tangle}
    \]
    \[
	\begin{tikzcd}[column sep=large]
		A & A & A \\
		A & A & A \\
		A && A
		\arrow["{E(j, t)}"', "\shortmid"{marking}, from=1-3, to=1-2]
		\arrow["{E(j, t)}"', "\shortmid"{marking}, from=1-2, to=1-1]
		\arrow["{E(t, t)}"{description}, from=2-3, to=2-2]
		\arrow["{E(t, t)}"{description}, from=2-2, to=2-1]
		\arrow[""{name=0, anchor=center, inner sep=0}, Rightarrow, no head, from=1-3, to=2-3]
		\arrow[""{name=1, anchor=center, inner sep=0}, Rightarrow, no head, from=1-2, to=2-2]
		\arrow[""{name=2, anchor=center, inner sep=0}, Rightarrow, no head, from=1-1, to=2-1]
		\arrow["{E(t, t)}", "\shortmid"{marking}, from=3-3, to=3-1]
		\arrow[""{name=3, anchor=center, inner sep=0}, Rightarrow, no head, from=2-3, to=3-3]
		\arrow[""{name=4, anchor=center, inner sep=0}, Rightarrow, no head, from=2-1, to=3-1]
		\arrow["\dag"{description}, draw=none, from=0, to=1]
		\arrow["\dag"{description}, draw=none, from=1, to=2]
		\arrow["{\cp t(t, t)}"{description}, draw=none, from=3, to=4]
	\end{tikzcd}
    \quad = \quad
	\begin{tikzcd}[column sep=large]
		A & A & A \\
		A & A & A \\
		A && A \\
		A && A
		\arrow["{E(j, t)}"', "\shortmid"{marking}, from=1-2, to=1-1]
		\arrow["{E(t, t)}"{description}, from=2-3, to=2-2]
		\arrow["{E(j, t)}"{description}, from=2-2, to=2-1]
		\arrow[""{name=0, anchor=center, inner sep=0}, Rightarrow, no head, from=1-2, to=2-2]
		\arrow[""{name=1, anchor=center, inner sep=0}, Rightarrow, no head, from=1-1, to=2-1]
		\arrow["{E(j, t)}"{description}, from=3-3, to=3-1]
		\arrow[""{name=2, anchor=center, inner sep=0}, Rightarrow, no head, from=2-3, to=3-3]
		\arrow[""{name=3, anchor=center, inner sep=0}, Rightarrow, no head, from=2-1, to=3-1]
		\arrow["{E(t, t)}", "\shortmid"{marking}, from=4-3, to=4-1]
		\arrow[""{name=4, anchor=center, inner sep=0}, Rightarrow, no head, from=3-3, to=4-3]
		\arrow[""{name=5, anchor=center, inner sep=0}, Rightarrow, no head, from=3-1, to=4-1]
		\arrow["{E(j, t)}"', "\shortmid"{marking}, from=1-3, to=1-2]
		\arrow[""{name=6, anchor=center, inner sep=0}, Rightarrow, no head, from=1-3, to=2-3]
		\arrow["\dag"{description}, draw=none, from=6, to=0]
		\arrow["{\cp t(j, t)}"{description}, draw=none, from=2, to=3]
		\arrow["\dag"{description}, draw=none, from=4, to=5]
		\arrow["{=}"{description}, draw=none, from=0, to=1]
	\end{tikzcd}
    \]
    \[
	\begin{tangle}{(4,3)}[trim y=.25]
		\tgBorderA{(0,0)}{\tgColour6}{\tgColour4}{\tgColour4}{\tgColour6}
		\tgBorderA{(1,0)}{\tgColour4}{\tgColour6}{\tgColour6}{\tgColour4}
		\tgBorderA{(2,0)}{\tgColour6}{\tgColour4}{\tgColour4}{\tgColour6}
		\tgBorderA{(3,0)}{\tgColour4}{\tgColour6}{\tgColour6}{\tgColour4}
		\tgBorderA{(0,1)}{\tgColour6}{\tgColour4}{\tgColour4}{\tgColour6}
		\tgBorder{(0,1)}{0}{1}{0}{0}
		\tgBorderA{(1,1)}{\tgColour4}{\tgColour6}{\tgColour4}{\tgColour4}
		\tgBorder{(1,1)}{0}{0}{0}{1}
		\tgBorderA{(2,1)}{\tgColour6}{\tgColour4}{\tgColour4}{\tgColour4}
		\tgBorder{(2,1)}{0}{1}{0}{0}
		\tgBorderA{(3,1)}{\tgColour4}{\tgColour6}{\tgColour6}{\tgColour4}
		\tgBorder{(3,1)}{0}{0}{0}{1}
		\tgBorderA{(0,2)}{\tgColour6}{\tgColour4}{\tgColour4}{\tgColour6}
		\tgBlank{(1,2)}{\tgColour4}
		\tgBlank{(2,2)}{\tgColour4}
		\tgBorderA{(3,2)}{\tgColour4}{\tgColour6}{\tgColour6}{\tgColour4}
		\tgCell[(1,0)]{(0.5,1)}{\dag}
		\tgCell[(1,0)]{(2.5,1)}{\dag}
		\tgArrow{(1.5,1)}{0}
		\tgArrow{(0,1.5)}{1}
		\tgArrow{(3,1.5)}{3}
		\tgArrow{(0,0.5)}{1}
		\tgArrow{(2,0.5)}{1}
		\tgArrow{(1,0.5)}{3}
		\tgArrow{(3,0.5)}{3}
		\tgAxisLabel{(0.5,0.25)}{south}{j}
		\tgAxisLabel{(1.5,0.25)}{south}{t}
		\tgAxisLabel{(2.5,0.25)}{south}{j}
		\tgAxisLabel{(3.5,0.25)}{south}{t}
		\tgAxisLabel{(0.5,2.75)}{north}{t}
		\tgAxisLabel{(3.5,2.75)}{north}{t}
	\end{tangle}
    \tangleeq*
	\begin{tangle}{(4,4)}[trim y]
		\tgBorderA{(0,0)}{\tgColour6}{\tgColour4}{\tgColour4}{\tgColour6}
		\tgBorderA{(1,0)}{\tgColour4}{\tgColour6}{\tgColour6}{\tgColour4}
		\tgBorderA{(2,0)}{\tgColour6}{\tgColour4}{\tgColour4}{\tgColour6}
		\tgBorderA{(3,0)}{\tgColour4}{\tgColour6}{\tgColour6}{\tgColour4}
		\tgBorderA{(0,1)}{\tgColour6}{\tgColour4}{\tgColour4}{\tgColour6}
		\tgBorderA{(1,1)}{\tgColour4}{\tgColour6}{\tgColour4}{\tgColour4}
		\tgBorderA{(2,1)}{\tgColour6}{\tgColour4}{\tgColour6}{\tgColour4}
		\tgBorderA{(3,1)}{\tgColour4}{\tgColour6}{\tgColour6}{\tgColour6}
		\tgBorderA{(0,2)}{\tgColour6}{\tgColour4}{\tgColour4}{\tgColour6}
		\tgBorder{(0,2)}{0}{1}{0}{0}
		\tgBorderA{(1,2)}{\tgColour4}{\tgColour4}{\tgColour4}{\tgColour4}
		\tgBorder{(1,2)}{0}{1}{0}{1}
		\tgBorderA{(2,2)}{\tgColour4}{\tgColour6}{\tgColour6}{\tgColour4}
		\tgBorder{(2,2)}{0}{0}{0}{1}
		\tgBlank{(3,2)}{\tgColour6}
		\tgBorderA{(0,3)}{\tgColour6}{\tgColour4}{\tgColour4}{\tgColour6}
		\tgBlank{(1,3)}{\tgColour4}
		\tgBorderA{(2,3)}{\tgColour4}{\tgColour6}{\tgColour6}{\tgColour4}
		\tgBlank{(3,3)}{\tgColour6}
		\tgCell[(2,0)]{(1,2)}{\dag}
		\tgArrow{(0,1.5)}{1}
		\tgArrow{(2,1.5)}{3}
		\tgArrow{(0,2.5)}{1}
		\tgArrow{(0,0.5)}{1}
		\tgArrow{(2,0.5)}{1}
		\tgCell[(2,0)]{(2,1)}{\dag}
		\tgArrow{(1,0.5)}{3}
		\tgArrow{(3,0.5)}{3}
		\tgArrow{(2,2.5)}{3}
		\tgAxisLabel{(0.5,0.75)}{south}{j}
		\tgAxisLabel{(1.5,0.75)}{south}{t}
		\tgAxisLabel{(2.5,0.75)}{south}{j}
		\tgAxisLabel{(3.5,0.75)}{south}{t}
		\tgAxisLabel{(0.5,3.25)}{north}{t}
		\tgAxisLabel{(2.5,3.25)}{north}{t}
	\end{tangle}
    \]
    A \emph{$j$-relative monad} (alternatively \emph{monad on $j$}, \emph{monad relative to $j$}, or simply \emph{$j$-monad}) is a relative monad with root $j$. A \emph{morphism} of $j$-monads from $(t, \dag, \eta)$ to $(t', \dag', \eta')$ is a 2-cell $\tau \colon t \tto t'$ satisfying the following equations.
	\[
	\begin{tikzcd}
		& j \\
		t && {t'}
		\arrow["\tau"', from=2-1, to=2-3]
		\arrow["\eta"', from=1-2, to=2-1]
		\arrow["{\eta'}", from=1-2, to=2-3]
	\end{tikzcd}
	\hspace{4em}
	\begin{tangle}{(4,1)}[trim x]
		\tgBorderA{(0,0)}{\tgColour6}{\tgColour6}{\tgColour4}{\tgColour4}
		\tgBorderA{(1,0)}{\tgColour6}{\tgColour6}{\tgColour4}{\tgColour4}
		\tgBorderA{(2,0)}{\tgColour6}{\tgColour6}{\tgColour4}{\tgColour4}
		\tgBorderA{(3,0)}{\tgColour6}{\tgColour6}{\tgColour4}{\tgColour4}
		\tgCell{(1,0)}{\eta}
		\tgCell{(2,0)}{\tau}
		\tgArrow{(1.5,0)}{0}
		\tgArrow{(0.5,0)}{0}
		\tgArrow{(2.5,0)}{0}
		\tgAxisLabel{(0.75,0.5)}{east}{j}
		\tgAxisLabel{(3.25,0.5)}{west}{t'}
	\end{tangle}
    =
	\begin{tangle}{(3,1)}[trim x]
		\tgBorderA{(0,0)}{\tgColour6}{\tgColour6}{\tgColour4}{\tgColour4}
		\tgBorderA{(1,0)}{\tgColour6}{\tgColour6}{\tgColour4}{\tgColour4}
		\tgBorderA{(2,0)}{\tgColour6}{\tgColour6}{\tgColour4}{\tgColour4}
		\tgCell{(1,0)}{\eta'}
		\tgArrow{(0.5,0)}{0}
		\tgArrow{(1.5,0)}{0}
		\tgAxisLabel{(0.75,0.5)}{east}{j}
		\tgAxisLabel{(2.25,0.5)}{west}{t'}
	\end{tangle}
	\]
	\[
	\begin{tikzcd}
		{E(j, t)} && {E(j, t')} \\
		{E(t, t)} && {E(t', t')} \\
		& {E(t, t')}
		\arrow["\dag"', from=1-1, to=2-1]
		\arrow["{E(t, \tau)}"', from=2-1, to=3-2]
		\arrow["{E(j, \tau)}", from=1-1, to=1-3]
		\arrow["{\dag'}", from=1-3, to=2-3]
		\arrow["{E(\tau, t')}", from=2-3, to=3-2]
	\end{tikzcd}
	\hspace{4em}
	\begin{tangle}{(2,5)}[trim y]
		\tgBorderA{(0,0)}{\tgColour6}{\tgColour4}{\tgColour4}{\tgColour6}
		\tgBorderA{(1,0)}{\tgColour4}{\tgColour6}{\tgColour6}{\tgColour4}
		\tgBorderA{(0,1)}{\tgColour6}{\tgColour4}{\tgColour4}{\tgColour6}
		\tgBorderA{(1,1)}{\tgColour4}{\tgColour6}{\tgColour6}{\tgColour4}
		\tgBorderA{(0,2)}{\tgColour6}{\tgColour4}{\tgColour4}{\tgColour6}
		\tgBorder{(0,2)}{0}{1}{0}{0}
		\tgBorderA{(1,2)}{\tgColour4}{\tgColour6}{\tgColour6}{\tgColour4}
		\tgBorder{(1,2)}{0}{0}{0}{1}
		\tgBorderA{(0,3)}{\tgColour6}{\tgColour4}{\tgColour4}{\tgColour6}
		\tgBorderA{(1,3)}{\tgColour4}{\tgColour6}{\tgColour6}{\tgColour4}
		\tgBorderA{(0,4)}{\tgColour6}{\tgColour4}{\tgColour4}{\tgColour6}
		\tgBorderA{(1,4)}{\tgColour4}{\tgColour6}{\tgColour6}{\tgColour4}
		\tgCell[(1,0)]{(0.5,2)}{\dag}
		\tgArrow{(0,2.5)}{1}
		\tgArrow{(0,1.5)}{1}
		\tgArrow{(1,1.5)}{3}
		\tgCell{(1,3)}{\tau}
		\tgArrow{(0,3.5)}{1}
		\tgArrow{(1,2.5)}{3}
		\tgArrow{(1,3.5)}{3}
		\tgArrow{(0,0.5)}{1}
		\tgArrow{(1,0.5)}{3}
		\tgAxisLabel{(0.5,0.75)}{south}{j}
		\tgAxisLabel{(1.5,0.75)}{south}{t}
		\tgAxisLabel{(0.5,4.25)}{north}{t}
		\tgAxisLabel{(1.5,4.25)}{north}{t'}
	\end{tangle}
    \tangleeq*
	\begin{tangle}{(2,5)}[trim y]
		\tgBorderA{(0,0)}{\tgColour6}{\tgColour4}{\tgColour4}{\tgColour6}
		\tgBorderA{(1,0)}{\tgColour4}{\tgColour6}{\tgColour6}{\tgColour4}
		\tgBorderA{(0,1)}{\tgColour6}{\tgColour4}{\tgColour4}{\tgColour6}
		\tgBorderA{(1,1)}{\tgColour4}{\tgColour6}{\tgColour6}{\tgColour4}
		\tgBorderA{(0,2)}{\tgColour6}{\tgColour4}{\tgColour4}{\tgColour6}
		\tgBorder{(0,2)}{0}{1}{0}{0}
		\tgBorderA{(1,2)}{\tgColour4}{\tgColour6}{\tgColour6}{\tgColour4}
		\tgBorder{(1,2)}{0}{0}{0}{1}
		\tgBorderA{(0,3)}{\tgColour6}{\tgColour4}{\tgColour4}{\tgColour6}
		\tgBorderA{(1,3)}{\tgColour4}{\tgColour6}{\tgColour6}{\tgColour4}
		\tgBorderA{(0,4)}{\tgColour6}{\tgColour4}{\tgColour4}{\tgColour6}
		\tgBorderA{(1,4)}{\tgColour4}{\tgColour6}{\tgColour6}{\tgColour4}
		\tgCell[(1,0)]{(0.5,2)}{\dag'}
		\tgArrow{(0,1.5)}{1}
		\tgArrow{(1,1.5)}{3}
		\tgArrow{(1,2.5)}{3}
		\tgCell{(0,3)}{\tau}
		\tgArrow{(0,2.5)}{1}
		\tgArrow{(1,3.5)}{3}
		\tgArrow{(0,3.5)}{1}
		\tgCell{(1,1)}{\tau}
		\tgArrow{(1,0.5)}{3}
		\tgArrow{(0,0.5)}{1}
		\tgAxisLabel{(0.5,0.75)}{south}{j}
		\tgAxisLabel{(1.5,0.75)}{south}{t}
		\tgAxisLabel{(0.5,4.25)}{north}{t}
		\tgAxisLabel{(1.5,4.25)}{north}{t'}
	\end{tangle}
	\]
    $j$-monads and their morphisms form a category $\RMnd(j)$.
    Denote by $U_j \colon \RMnd(j) \to \X[A, E]$ the faithful functor sending each $j$-monad $(t, \dag, \eta)$ to its carrier $t$.
\end{definition}

\begin{remark}
	Our definition of relative monad coincides with that of \cite[Definition~3.5.1]{maillard2019principles} and \cite[Definition~5.2.6]{arkor2022monadic} in a representable equipment.
\end{remark}

\begin{example}
	For any tight-cell $\jAE$, the triple \[(j, 1_j, 1_{E(j, j)})\] forms a $j$-monad, the \emph{trivial $j$-monad}. It will follow from \cref{unit-in-skew-multicategory-is-initial-monoid} that the trivial $j$-monad is initial in $\RMnd(j)$.
\end{example}

Relative monads were introduced as a generalisation of monads from endofunctors to arbitrary functors~\cite{altenkirch2010monads}. However, despite this motivation, the classical definition of relative monad does not immediately appear monad-esque. We therefore proceed by justifying \cref{relative-monad} from an alternative perspective.

For an object $A$ in a 2-category $\K$, the hom-category $\K(A, A)$ is canonically equipped with the structure of a strict monoidal category, whose tensor product is given by composition of endo-1-cells, and whose unit is given by the identity 1-cell on $A$. A monoid in $\K(A, A)$ is precisely a monad on~$A$.

More generally, we may consider monads in a \vdc{} $\X$. In this context there are two notions of monad: loose-monads and tight-monads (\cref{monads-and-adjunctions}). For an object $A$ in $\X$, we may consider both loose-monads and tight-monads on $A$ as monoids. As with monads in 2-categories, endo-tight-cells on $A$ form a strict monoidal category $\X[A, A]$ (assuming $A$ admits a loose-identity), and a monoid therein is precisely a tight-monad in the sense of \cref{tight-monad}. However, since loose-cells do not admit composites in general, endo-loose-cells on $A$ form not a monoidal category, but a multicategory $\X\lh{A, A}$, the objects of which are loose-cells $A \lto A$, and the multimorphisms of which are 2-cells with the following frame.
\[\begin{tikzcd}
	A & A & \cdots & A & A \\
	A &&&& A
	\arrow["{p_n}"', "\shortmid"{marking}, from=1-5, to=1-4]
	\arrow["{p_{n - 1}}"', "\shortmid"{marking}, from=1-4, to=1-3]
	\arrow["{p_2}"', "\shortmid"{marking}, from=1-3, to=1-2]
	\arrow["{p_1}"', "\shortmid"{marking}, from=1-2, to=1-1]
	\arrow["q", "\shortmid"{marking}, from=2-5, to=2-1]
	\arrow[Rightarrow, no head, from=1-5, to=2-5]
	\arrow[Rightarrow, no head, from=1-1, to=2-1]
\end{tikzcd}\]
A monoid in $\X\lh{A, A}$ is then precisely a loose-monad in the sense of \cref{loose-monad}. Furthermore, when $\X$ is an equipment, the monoidal category $\X[A, A]$ forms a full sub-multicategory of $\X\lh{A, A}$, each tight-monad $(t, \mu, \eta)$ being represented by a loose-monad $(A(1, t), A(1, \mu), A(1, \eta))$.

We should like to generalise this situation to relative monads by
considering arbitrary hom-categories, for which the tight- and loose-cells may have a domain different to their codomain. On the face of it, such a proposition makes little sense, since it is not possible to form a chain of two loose-cells $p, q \colon A \lto E$ unless $A = E$. However, supposing we were given a loose-cell $j^* \colon E \lto A$, we could form a chain of loose-cells
\[A \xlto q E \xlto{j^*} A \xlto p E\]
which acts as a form of composition \emph{relative to} $j^*$.

For this composition to be associative and unital in an appropriate sense, we cannot simply take any 1-cell $j^* \colon E \lto A$ relative to which to compose. However, it is enough to assume that $j^*$ is the right adjoint of a loose-adjunction $j_* \adj j^*$. In particular, in the context of an equipment $\X$, we may take $j^* \defeq E(j, 1)$ to be the conjoint of a tight-cell $\jAE$, which is right-adjoint to the companion $j_* \defeq E(1, j)$ by \cref{loose-adjunction-induced-by-tight-cell}.
We may then define a notion of multimorphism between loose-cells $A \lto E$, given by 2-cells with the following frame.
\[\begin{tikzcd}
	E & A & E & A & \cdots & E & A & E & A \\
	E &&&&&&&& A
	\arrow["{p_n}"', "\shortmid"{marking}, from=1-9, to=1-8]
	\arrow["{j^*}"', "\shortmid"{marking}, from=1-8, to=1-7]
	\arrow["{p_{n - 1}}"', "\shortmid"{marking}, from=1-7, to=1-6]
	\arrow["{j^*}"', "\shortmid"{marking}, from=1-6, to=1-5]
	\arrow["{j^*}"', "\shortmid"{marking}, from=1-5, to=1-4]
	\arrow["q", "\shortmid"{marking}, from=2-9, to=2-1]
	\arrow[Rightarrow, no head, from=1-9, to=2-9]
	\arrow[Rightarrow, no head, from=1-1, to=2-1]
	\arrow["{p_1}"', "\shortmid"{marking}, from=1-2, to=1-1]
	\arrow["{p_2}"', "\shortmid"{marking}, from=1-4, to=1-3]
	\arrow["{j^*}"', "\shortmid"{marking}, from=1-3, to=1-2]
\end{tikzcd}\]
Though this does not quite suffice to define an appropriate multicategory structure on the loose-cells $A \lto E$ of $\X$, it does form a weaker notion of multicategory that generalises the skew-monoidal categories of \textcite{szlachanyi2012skew} in the same way that multicategories generalise monoidal categories (\cf{}~\cref{skew-multicategories-as-generalised-multicategories}). Thereafter, it is natural to consider monoids in this generalised multicategory as a notion of \emph{$j$-relative} monad. By restricting to the monoids that are representable in a sense analogous to that of tight-monads above, we shall show that this recovers \cref{relative-monad}.

\subsection{Associative-normal left-skew-multicategories}

We begin by defining the generalised notion of multicategory required to describe the skew composition described above: these generalised multicategories are similar to multicategories, but in which multimorphisms may additionally have \emph{markers} in their domain, denoted by $\bullet$, which represent the unit in a left-skew-monoidal category.
Below, we write $\bullet^{m}$ as an abbreviation for $\underbrace{\bullet, \dots, \bullet}_m$ (where $m \geq 0$).

\begin{definition}
	\label{alpha-normal-left-skew-multicategory}
	An \emph{associative-normal left-skew-multicategory} $\M$ comprises
	\begin{enumerate}
        \item a class $\ob\M$ of \emph{objects};
        \item a class $\M(X_1, \ldots, X_n; Y)$ of \emph{multimorphisms} for each $n > 0$, $X_1, \ldots, X_n \in \ob\M + \{ \bullet \}$ and $Y \in \ob\M$;
        \item an \emph{identity} multimorphism $1_X \in \M(X; X)$ for each $X \in \ob\M$;
        \item for each multimorphism $g \colon \bullet^{m_0}, Y_1, \bullet^{m_1}, \ldots, \bullet^{m_{n - 1}}, Y_n, \bullet^{m_n} \to Z$ where $Y_1, \ldots, Y_n, Z \in \ob\M$ and $n, m_i \geq 0$, and multimorphisms $f_1 \colon \vec{X_1} \to Y_1$, \ldots, $f_n \colon \vec{X_n} \to Y_n$ where $\vec{X_i} \in (\ob M + \{ \bullet \})\Kleene$, a \emph{composite} multimorphism: \[\bullet^{m_0}, \vec{X_1}, \bullet^{m_1}, \ldots, \bullet^{m_{n - 1}}, \vec{X_n}, \bullet^{m_n} \xto{(f_1, \ldots, f_n) \d g} Z\]
        \item a \emph{left-unitor} function \[\lambda_{(\vec X; Y), k} \colon \M(X_1, \ldots, X_k, X_{k + 1}, \ldots, X_n; Y) \to \M(X_1, \ldots, X_k, \bullet, X_{k + 1}, \ldots, X_n; Y)\] for each $0 \leq k < n$;
        \item a \emph{right-unitor} function \[\rho_{(\vec X; Y),k} \colon \M(X_1, \ldots, X_k, \bullet, X_{k + 1}, \ldots, X_n; Y) \to \M(X_1, \ldots, X_k, X_{k + 1}, \ldots, X_n; Y)\] for each $0 < k \leq n$,
    \end{enumerate}
	such that composition is associative and unital; that the left- and right-unitors cohere with pre- and postcomposition; that the left- and right-unitors cohere with themselves\symbolquicknote{These conditions were mistakenly omitted in the published version of the paper.}; and that the right-unitor is a retraction of the left-unitor, in the following sense.
	\begin{align*}
		(f_1, \ldots, \lambda_{(\vec{X_i}; Y_i), k} f_i, \ldots, f_n) \d g & = \tag{$0 \leq k < \ob{\vec{X_i}}$} \\ & \hspace{-4em} \lambda_{(\vec{\vec X}; Z), (\sum_{0 \leq j < i} m_j) + (\sum_{0 < j < i} \ob{\vec{X_j}}) + k}((f_1, \ldots, f_n) \d g) \\
		(f_1, \ldots, \rho_{(\vec{X_i}; Y_i), k} f_i, \ldots, f_n) \d g & = \tag{$0 < k \leq \ob{\vec{X_i}}$} \\ & \hspace{-4em} \rho_{(\vec{\vec X}; Z), (\sum_{0 \leq j < i} m_j) + (\sum_{0 < j < i} \ob{\vec{X_j}}) + k}((f_1, \ldots, f_n) \d g) \\
    (f_1, \ldots, f_n) \d (\lambda_{(\vec Y; Z), (\sum_{0 \leq j < k} m_j) + k + \ell} g) & = \tag{$0 \leq k \leq n$, $0 \leq \ell \leq m_k$, $k < n \lor \ell < m_k$} \\ & \hspace{-4em} \lambda_{(\vec{\vec X}; Z), (\sum_{0 \leq j < k} m_j) + (\sum_{0 < j \leq k} \ob{\vec{X_j}}) + \ell}((f_1, \ldots, f_n) \d g) \\
		(f_1, \ldots, f_n) \d (\rho_{(\vec Y; Z), (\sum_{0 \leq j < k} m_j) + k + \ell} g) & = \tag{$0 \leq k \leq n$, $0 \leq \ell < m_k$, $0 < k \lor 0 < \ell$} \\ & \hspace{-4em} \rho_{(\vec{\vec X}; Z), (\sum_{0 \leq j < k} m_j) + (\sum_{0 < j \leq k} \ob{\vec{X_j}}) + \ell}((f_1, \ldots, f_n) \d g) \\
    \lambda_{(X_1, \ldots, X_i, \bullet, X_{i + 1}, \ldots, X_n; Y), j + 1}(\lambda_{(\vec X; Y), i}(f)) & = \lambda_{(X_1, \ldots, X_j, \bullet, X_{j + 1}, \ldots, X_n; Y), i}(\lambda_{(\vec X; Y), j}(f)) \tag{$0 \leq i \leq j < n$} \\
		\rho_{(\vec X; Y), j}(\rho_{(X_1, \ldots, X_j, \bullet, X_{j + 1}, \ldots, X_n; Y), i}(f)) & = \rho_{(\vec X; Y), i}(\rho_{(X_1, \ldots, X_i, \bullet, X_{i + 1}, \ldots, X_n; Y), j + 1}(f)) \tag{$0 < i \leq j \leq n$} \\
		\lambda_{(\vec X; Y), k} \d \rho_{(\vec X; Y),k} & = 1_{\M(\vec X; Y)} \tag{$0 < k < n$}
	\end{align*}
	Above, $\vec{\vec X}$ is shorthand for the domain of a multimorphism:
	\[\bullet^{m_0}, \vec{X_1}, \bullet^{m_1}, \ldots, \bullet^{m_{n - 1}}, \vec{X_n}, \bullet^{m_n}\]

	$\M$ is \emph{left-normal} when $\lambda$ is invertible; and is \emph{right-normal} when $\rho$ is invertible.

	A \emph{functor} between associative-normal left-skew-multicategories is a homomorphism of associative-normal left-skew-multicategories.
\end{definition}

\begin{remark}
	\label{skew-multicategories-as-generalised-multicategories}
	Associative-normal left-skew-multicategories are part of a larger story, which we briefly outline. The construction of the free left-skew-monoidal category described in \cite{bourke2018free} extends to a virtual double monad $\bb S$ on $\Cat$ via convolution in the usual way (\cf{}~\cites[\S11]{street2013skew}[Theorem~7.3]{fiore2018relative}). Normalised $\bb S$-monoids in the sense of \textcite[Definition~8.3]{cruttwell2010unified} might then naturally be called \emph{left-skew-multicategories}. The construction $\bb S$ of the free left-skew-monoidal category restricts to give constructions $\bb S_N$ of free \emph{strict partially-normal} left-skew-monoidal categories (\cf{}~\cites[\S1]{lack2012skew}[Definition 3.1]{uustalu2020eilenberg}), where some subset $N \subseteq \{ \alpha, \lambda, \rho \}$ of the structural transformations for associativity, and left- and right-unitality of a left-skew-monoidal category are taken to be identities. Correspondingly, normalised $\bb S_N$-monoids give notions of \emph{$N$-normal left-skew-multicategories}: in particular,
	\begin{itemize}
		\item $\varnothing$-normal left-skew-multicategories are left-skew-multicategories in the aforementioned sense;
		\item $\{ \alpha \}$-normal left-skew-multicategories are the associative-normal left-skew-multicategories of \cref{alpha-normal-left-skew-multicategory};
		\item $\{ \alpha, \rho \}$-normal left-skew-multicategories are the \emph{skew-multicategories} of \cite[Definition~4.2]{bourke2018skew} (\cf{} \cite[\S3, Alternative perspective 2]{bourke2018skew});
		\item $\{ \alpha, \lambda, \rho \}$-normal left-skew-multicategories are multicategories~\cite[106]{lambek1969deductive}.
	\end{itemize}

	For $N \subseteq N' \subseteq \{ \alpha, \lambda, \rho \}$, each $N$-normal left-skew-multicategory $\M$ has an underlying wide $N'$-normal left-skew-multicategory $\M_{N' \setminus N}$, given by restricting to the \emph{$(N' \setminus N)$-normal} multimorphisms. In particular, every associative-normal left-skew-multicategory has an underlying $\{ \alpha, \rho \}$-normal left-skew-multicategory, which permits us the later use of the theory of left-representability developed in \cite{bourke2018skew}.
\end{remark}

The associative-normal left-skew-multicategories with which we are concerned satisfy an additional representability property: namely, the existence of a nullary tensor product.

\begin{definition}
	\label{unital-skew-multicategory}
    A \emph{unit} in an associative-normal left-skew-multicategory $\M$ comprises an object $J \in \M$ and a multimorphism $\bullet \to J$ such that, for all objects $X_1, \ldots, X_n, Y \in \M$ and $0 \leq k \leq n$,
	\begin{enumerate}
		\item the function \[\M(X_1, \ldots, X_k, J, X_{k + 1}, \ldots, X_n; Y) \to \M(X_1, \ldots, X_k, \bullet, X_{k + 1}, \ldots, X_n; Y)\] induced by precomposition with $\bullet \to J$ is a bijection;
		\item the following diagram of sets commutes\symbolquicknote{This condition was mistakenly omitted in the published version of the paper.}, where the vertical functions are those induced by precomposition with $\bullet \to J$.
		\[\begin{tikzcd}[column sep=huge]
			{\M(X_1, \ldots, X_k, J, X_{k + 1}, \ldots, X_n; Y)} & {\M(X_1, \ldots, X_k, \bullet, J, X_{k + 1}, \ldots, X_n; Y)} \\
			{\M(X_1, \ldots, X_k, \bullet, X_{k + 1}, \ldots, X_n; Y)} & {\M(X_1, \ldots, X_k, \bullet, \bullet, X_{k + 1}, \ldots, X_n; Y)}
			\arrow["{\lambda_{(\vec X; Y), k}}", from=1-1, to=1-2]
			\arrow[from=1-1, to=2-1]
			\arrow[from=1-2, to=2-2]
			\arrow["{\rho_{(\vec X; Y), k + 1}}", from=2-2, to=2-1]
		\end{tikzcd}\]
	\end{enumerate}
	An associative-normal left-skew-multicategory with a unit is called \emph{unital}.
\end{definition}

We now make precise the earlier intuition by showing that each loose-adjunction $j_* \adj j^*$ in a \vdc{} induces skew-multicategorical structure on the corresponding hom-category.

\begin{theorem}
    \label{skew-multicategorical-hom}
    Let $\X$ be a \vdc{} with a loose-adjunction $j_* \adj j^* \colon E \lto A$. The loose-cells $A \lto E$ in $\X$ together with 2-cells of the form $p_1, j^*, p_2, j^*, \ldots, j^*, p_n \tto q$ form a multicategory, which extends to a unital associative-normal left-skew-multicategory $\X\lh{j_* \adj j^*}$.
\end{theorem}

\begin{proof}
    We define a multicategory $\X\lh{j_* \adj j^*}$ as follows. The class of objects is given by those of $\X\lh{A, E}$. The multimorphisms $p_1, \ldots, p_n \to q$ for $n > 0$ are 2-cells $p_1, j^*, p_2, j^*, \ldots, j^*, p_n \tto q$.
	\[\begin{tikzcd}
		E & A & E & A & \cdots & E & A & E & A \\
		E &&&&&&&& A
		\arrow["{p_n}"', "\shortmid"{marking}, from=1-9, to=1-8]
		\arrow["{j^*}"', "\shortmid"{marking}, from=1-8, to=1-7]
		\arrow["{p_{n - 1}}"', "\shortmid"{marking}, from=1-7, to=1-6]
		\arrow["{j^*}"', "\shortmid"{marking}, from=1-6, to=1-5]
		\arrow["{p_1}"', "\shortmid"{marking}, from=1-2, to=1-1]
		\arrow["q", "\shortmid"{marking}, from=2-9, to=2-1]
		\arrow[""{name=0, anchor=center, inner sep=0}, Rightarrow, no head, from=1-9, to=2-9]
		\arrow[""{name=1, anchor=center, inner sep=0}, Rightarrow, no head, from=1-1, to=2-1]
		\arrow["{j^*}"', "\shortmid"{marking}, from=1-5, to=1-4]
		\arrow["{p_2}"', "\shortmid"{marking}, from=1-4, to=1-3]
		\arrow["{j^*}"', "\shortmid"{marking}, from=1-3, to=1-2]
		\arrow["\phi"{description}, draw=none, from=0, to=1]
	\end{tikzcd}\]
    There are no nullary multimorphisms. The identity multimorphism on $p$ is given by the identity 2-cell $1_p$. Composition is given by the following pasting,
	\[\begin{tikzcd}
		E & \cdots & A & E & A & E & \cdots & A \\
		E && A & E & A & E && A \\
		E &&&&&&& A
		\arrow["{p_n}", "\shortmid"{marking}, from=2-8, to=2-6]
		\arrow["{p_1}", "\shortmid"{marking}, from=2-3, to=2-1]
		\arrow["q", "\shortmid"{marking}, from=3-8, to=3-1]
		\arrow[""{name=0, anchor=center, inner sep=0}, Rightarrow, no head, from=2-8, to=3-8]
		\arrow["{p_{n, m_n}}"', "\shortmid"{marking}, from=1-8, to=1-7]
		\arrow[""{name=1, anchor=center, inner sep=0}, Rightarrow, no head, from=1-8, to=2-8]
		\arrow[""{name=2, anchor=center, inner sep=0}, "{j^*}", "\shortmid"{marking}, from=2-6, to=2-5]
		\arrow[Rightarrow, no head, from=1-5, to=2-5]
		\arrow["{p_{1, m_1}}"', "\shortmid"{marking}, from=1-3, to=1-2]
		\arrow[""{name=3, anchor=center, inner sep=0}, Rightarrow, no head, from=1-3, to=2-3]
		\arrow[""{name=4, anchor=center, inner sep=0}, "{j^*}"', "\shortmid"{marking}, from=1-4, to=1-3]
		\arrow[""{name=5, anchor=center, inner sep=0}, "{j^*}", "\shortmid"{marking}, from=2-4, to=2-3]
		\arrow[Rightarrow, no head, from=1-4, to=2-4]
		\arrow["\cdots"{description}, draw=none, from=2-5, to=2-4]
		\arrow["\cdots"{description}, draw=none, from=1-5, to=1-4]
		\arrow[""{name=6, anchor=center, inner sep=0}, "{j^*}"', "\shortmid"{marking}, from=1-6, to=1-5]
		\arrow["{p_{1, 1}}"', "\shortmid"{marking}, from=1-2, to=1-1]
		\arrow[""{name=7, anchor=center, inner sep=0}, Rightarrow, no head, from=2-1, to=3-1]
		\arrow[""{name=8, anchor=center, inner sep=0}, Rightarrow, no head, from=1-1, to=2-1]
		\arrow[""{name=9, anchor=center, inner sep=0}, Rightarrow, no head, from=1-6, to=2-6]
		\arrow["{p_{n, 1}}"', "\shortmid"{marking}, from=1-7, to=1-6]
		\arrow["{=}"{description}, draw=none, from=4, to=5]
		\arrow["{=}"{description}, draw=none, from=6, to=2]
		\arrow["{\phi_1}"{description}, draw=none, from=3, to=8]
		\arrow["{\phi_n}"{description}, draw=none, from=1, to=9]
		\arrow["\psi"{description}, draw=none, from=0, to=7]
	\end{tikzcd}\]
    associativity and unitality being inherited from that of composition of 2-cells in $\X$.

	Next, we derive an associative-normal left-skew-multicategory structure on $\X\lh{j_* \adj j^*}$. We define a function \[[1_{\X\lh{A, E}}, \ph \mapsto j_*] \colon \X\lh{A, E} + \{ \bullet \} \to \X\lh{A, E}\] sending the marker $\bullet$ to the loose-cell $j_*$. This defines a multicategory with objects ${\X\lh{A, E} + \{ \bullet \}}$.

	We define a family
	\[\lambda_{(p_1, \ldots, p_n; q), k} \colon \X\lh{j_* \adj j^*}(p_1, \ldots, p_n; q) \to \X\lh{j_* \adj j^*}(p_1, \ldots, \bullet, \ldots, p_n; q)\]
	by pasting the counit of the loose-adjunction, and a family
	\[\rho_{(p_1, \ldots, p_n; q), k} \colon \X\lh{j_* \adj j^*}(p_1, \ldots, \bullet, \ldots, p_n; q) \to \X\lh{j_* \adj j^*}(p_1, \ldots, p_n; q)\]
	by pasting the unit of the loose-adjunction. That these cohere with composition, together with the compatibility conditions between $\lambda$ and itself, and between $\rho$ and itself, follows from associativity of composition in $\X$.
	The compatibility condition between $\lambda$ and $\rho$ follows from the right zig-zag law associated to the loose-adjunction. The left-adjoint loose-cell $j_*$ provides a unit for $\X\lh{j_* \adj j^*}$, the first condition of \cref{unital-skew-multicategory} holding by definition, and the second condition following from the left zig-zag law associated to the loose-adjunction.
\end{proof}

\begin{remark}
	\label{skew-multicategory-via-skew-warping}
	\Cref{skew-multicategorical-hom} generalises the construction of \cite[\S7]{lack2014monads} from bicategories and skew-monoidal categories to \vdcs{} and skew-multicategories. In fact, the construction of the associative-normal left-skew-multicategory $\X\lh{j_* \adj j^*}$ fits into a more general context, which we briefly outline. Recall from \cite[\S3]{lack2012skew} that skew-monoidal structures on a category $C$ are often induced from monoidal structures on $C$ by tensoring with a skew-warping $T$, via $a \skt b \defeq a \otimes T b$. The construction in \cite[\S7]{lack2014monads} is essentially a categorification of this idea, where the monoidal category $C$ is replaced by a bicategory $\K$. In this context, the tensor $\otimes$ becomes composition $\odot$ of loose-cells, and the skew-warping is given by $j^* \odot \ph$. However, although \citeauthor{lack2014monads} do consider a bicategorical notion of skew-warping in \S4 \ibid{}, they do not exhibit the skew-monoidal structure on $\K[A, E]$ as an instance of this construction. To do so would require the consideration of a notion of \emph{relative skew-warping}, to capture the skew-monoidal structure induced by a single right-adjoint 1-cell $j^* \colon E \lto A$, rather than by a family of right-adjoint 1-cells indexed by the objects of $\K$.

	We have chosen to follow \textcite{lack2014monads} in giving an explicit construction of the skew-multicategorical structure in \cref{skew-multicategorical-hom}, since a formalisation of the approach outlined above would require a further generalisation of the theory \ibid{} to \vdcs{}.
\end{remark}

\begin{definition}
	Let $\X$ be an equipment with a tight-cell $\jAE$. Denote by $\X\lh j$ the skew-multicategory $\X\lh{E(1, j) \adj E(j, 1)}$.
\end{definition}

\begin{definition}
    Given a skew-multicategory $\M$, denote by $\M_1$ the category of unary multimorphisms, \ie{} the category whose objects are those of $\M$ and whose hom-set $\M_1(X, Y) \defeq \M(X; Y)$.
\end{definition}

In particular $\X\lh{j}_1 = \X\lh{A, E}$. With the intention of obtaining the classical definition of relative monad, we examine the monoids in the skew-multicategory $\X\lh{j}$: it will turn out that $j$-relative monads are equivalent to monoids whose underlying loose-cell is representable.

\begin{definition}
	\label{monoid-in-skew-multicategory}
    Let $\M$ be an associative-normal left-skew-multicategory. A \emph{monoid} in $\M$ comprises
    \begin{enumerate}
        \item an object $M \in \M$, the \emph{carrier};
        \item a multimorphism $m \colon M, M \to M$, the \emph{multiplication};
        \item a multimorphism $u \colon \bullet \to M$, the \emph{unit},
    \end{enumerate}
    satisfying the following equations.
	\begin{align*}
		(u, 1_M) \d m & = \lambda_{(M; M), 0}(1_M) &
		\rho_{(M; M), 1}((1_M, u) \d m) & = 1_M &
		(m, 1_M) \d m & = (1_M, m) \d m
	\end{align*}
    A \emph{monoid homomorphism} from $(M, m, u)$ to $(M', m', u')$ is a unary multimorphism $f \colon M \to M'$ satisfying the following equations.
	\begin{align*}
		u \d f & = u' &
		m \d f & = (f, f) \d m'
	\end{align*}
    Monoids in $\M$ and their homomorphisms form a category $\Mon(\M)$ functorial in $\M$. Denote by $U_\M \colon \Mon(\M) \to \M_1$ the faithful functor sending each monoid $(M, m, u)$ to its carrier $M$.
\end{definition}

Any unital skew-multicategory contains an initial monoid.

\begin{proposition}
	\label{unit-in-skew-multicategory-is-initial-monoid}
    Let $\M$ be a unital skew-multicategory. The unit $J$ forms a monoid, which is initial amongst monoids in $\M$.
\end{proposition}

\begin{proof}
	Unitality of $\M$ gives a multimorphism $u \colon \bullet \to J$, and induces from $\lambda_{(J; J), 0}(1_J) \colon \bullet, J \to J$ a multimorphism $m \colon J, J \to J$. The left unit law and associativity law follow from the first condition for unitality (\cref{unital-skew-multicategory}). The right unit law follows from the two conditions for unitality.

    Given any monoid $(M', m', u')$ in $\M$, the unit $u' \colon \bullet \to M$ induces a multimorphism $J \to M$ which forms a monoid homomorphism: the unit law follows from the unitality bijection; while the multiplication law follows from unitality and the unit laws for $M'$.
\end{proof}

\subsection{Relative monads as monoids in a skew-multicategory}
\label{relative-monads-as-monoids}

Before relating relative monads to monoids, we first introduce the slightly more general notion of \emph{loose relative monad}, which stands in a similar relation to the notion of relative monad that loose-monads stand in relation to (tight) monads (\cref{monads-and-adjunctions}), and will be used to simplify some later proofs.

\begin{definition}
	\label{loose-relative-monad}
	For a tight-cell $\jAE$, denote by $\LRMnd(j) \defeq \Mon(\X\lh j)$ the category of \emph{loose $j$-relative monads}, and denote by $U_j \colon \LRMnd(j) \to \X\lh{A, E}$ the forgetful functor.
\end{definition}

Every loose $j$-relative monad induces a loose-monad on its domain by restricting along $j$. As a consequence, loose-monads relative to identity tight-cells are simply loose-monads.

\begin{lemma}
	\label{loose-relative-monad-to-loose-monad}
	Restriction along $j$ induces a functor $\ph(j, 1) \colon \LRMnd(j) \to \LMnd(A)$ commuting with the forgetful functors.
	\[\begin{tikzcd}
		{\LRMnd(j)} & {\LMnd(A)} \\
		{\X\lh{A, E}} & {\X\lh{A, A}}
		\arrow["{U_j}"', from=1-1, to=2-1]
		\arrow["{\ph(j, 1)}"', from=2-1, to=2-2]
		\arrow["{U_A}", from=1-2, to=2-2]
		\arrow["{\ph(j, 1)}", from=1-1, to=1-2]
	\end{tikzcd}\]
	Furthermore, when $j$ is the identity, this functor is an isomorphism.
\end{lemma}

\begin{proof}
	There is a functor of associative-normal left-skew-multicategories $\ph(j, 1) \colon \X\lh j \to \X\lh{A; A}$ sending each loose-cell $p \colon A \lto E$ to $p(j, 1) \colon A \lto A$ and each multimorphism \[\bullet^{m_0}, p_1, \bullet^{m_1}, \ldots, \bullet^{m_{n - 1}}, p_n, \bullet^{m_n} \tto q\] to the pasting of $E(j, 1)$ with the 2-cell $p_1, E(j, p_2), \ldots, E(j, p_n) \tto q$ given by precomposing $\pc j$ for each $\bullet$. By functoriality of $\Mon$, there is therefore a functor $\Mon(\ph(j, 1)) \colon \Mon(\X\lh j) \to \Mon(\X\lh{A; A})$ commuting with the forgetful functors, which by definition is a functor $\LRMnd(j) \to \LMnd(A)$.

	When $j$ is the identity, $\ph(j, 1)$ is invertible, since restriction is pseudofunctorial.
\end{proof}

While loose relative monads are of interest in their own right, herein we shall be interested in restricting to those monoids in $\X\lh j$ whose underlying loose-cells are representable.

\begin{definition}
    Let $\X$ be a \ve{} with a tight-cell $\jAE$. Define $\X[j]$ to be the (unital) full associative-normal sub-left-skew-multicategory of $\X\lh j$ spanned by the representable loose-cells.
\end{definition}

We shall unwrap the definition of a monoid in $\X[j]$ to compare it with the classical definition of a relative monad. Explicitly, a monoid in $\X[j]$ comprises
\begin{enumerate}
	\item a tight-cell $t \colon A \to E$;
	\item a 2-cell $\mu \colon E(1, t), E(j, 1), E(1, t) \tto E(1, t)$;
	\item a 2-cell $E(1, \eta) \colon E(1, j) \tto E(1, t)$,
\end{enumerate}
satisfying the following equations.
\[
\begin{tikzcd}[column sep=large]
	E & A & E & A \\
	E & A & E & A \\
	E &&& A
	\arrow["{E(1, t)}"{description}, from=2-4, to=2-3]
	\arrow["{E(j, 1)}"{description}, from=2-3, to=2-2]
	\arrow["{E(1, t)}"{description}, from=2-2, to=2-1]
	\arrow["{E(1, t)}", "\shortmid"{marking}, from=3-4, to=3-1]
	\arrow[""{name=0, anchor=center, inner sep=0}, Rightarrow, no head, from=2-4, to=3-4]
	\arrow[""{name=1, anchor=center, inner sep=0}, Rightarrow, no head, from=2-1, to=3-1]
	\arrow[""{name=2, anchor=center, inner sep=0}, Rightarrow, no head, from=1-2, to=2-2]
	\arrow[""{name=3, anchor=center, inner sep=0}, Rightarrow, no head, from=1-1, to=2-1]
	\arrow["{E(1, j)}"', "\shortmid"{marking}, from=1-2, to=1-1]
	\arrow["{E(1, t)}"', "\shortmid"{marking}, from=1-4, to=1-3]
	\arrow["{E(j, 1)}"', "\shortmid"{marking}, from=1-3, to=1-2]
	\arrow[""{name=4, anchor=center, inner sep=0}, Rightarrow, no head, from=1-4, to=2-4]
	\arrow[""{name=5, anchor=center, inner sep=0}, Rightarrow, no head, from=1-3, to=2-3]
	\arrow["\mu"{description}, draw=none, from=0, to=1]
	\arrow["\eta"{description}, draw=none, from=2, to=3]
	\arrow["{=}"{description}, draw=none, from=4, to=5]
	\arrow["{=}"{description}, draw=none, from=5, to=2]
\end{tikzcd}
\quad = \quad
\begin{tikzcd}[column sep=large]
	E & A & E & A \\
	E && E & A \\
	E &&& A
	\arrow["{E(1, t)}"{description}, from=2-4, to=2-3]
	\arrow[""{name=0, anchor=center, inner sep=0}, Rightarrow, no head, from=1-1, to=2-1]
	\arrow["{E(1, j)}"', "\shortmid"{marking}, from=1-2, to=1-1]
	\arrow["{E(1, t)}"', "\shortmid"{marking}, from=1-4, to=1-3]
	\arrow["{E(j, 1)}"', "\shortmid"{marking}, from=1-3, to=1-2]
	\arrow[""{name=1, anchor=center, inner sep=0}, Rightarrow, no head, from=1-4, to=2-4]
	\arrow[""{name=2, anchor=center, inner sep=0}, Rightarrow, no head, from=1-3, to=2-3]
	\arrow["\shortmid"{marking}, Rightarrow, no head, from=2-3, to=2-1]
	\arrow["{E(1, t)}", "\shortmid"{marking}, from=3-4, to=3-1]
	\arrow[""{name=3, anchor=center, inner sep=0}, Rightarrow, no head, from=2-4, to=3-4]
	\arrow[""{name=4, anchor=center, inner sep=0}, Rightarrow, no head, from=2-1, to=3-1]
	\arrow["{=}"{description}, draw=none, from=1, to=2]
	\arrow["{\cp j}"{description}, draw=none, from=2, to=0]
	\arrow["\opcart"{description}, draw=none, from=3, to=4]
\end{tikzcd}
\]\[
\begin{tangle}{(3,4)}[trim y]
	\tgBorderA{(0,0)}{\tgColour4}{\tgColour6}{\tgColour6}{\tgColour4}
	\tgBorderA{(1,0)}{\tgColour6}{\tgColour4}{\tgColour4}{\tgColour6}
	\tgBorderA{(2,0)}{\tgColour4}{\tgColour6}{\tgColour6}{\tgColour4}
	\tgBorderA{(0,1)}{\tgColour4}{\tgColour6}{\tgColour6}{\tgColour4}
	\tgBorderA{(1,1)}{\tgColour6}{\tgColour4}{\tgColour4}{\tgColour6}
	\tgBorderA{(2,1)}{\tgColour4}{\tgColour6}{\tgColour6}{\tgColour4}
	\tgBorderA{(0,2)}{\tgColour4}{\tgColour6}{\tgColour4}{\tgColour4}
	\tgBorderA{(1,2)}{\tgColour6}{\tgColour4}{\tgColour6}{\tgColour4}
	\tgBorderA{(2,2)}{\tgColour4}{\tgColour6}{\tgColour6}{\tgColour6}
	\tgBlank{(0,3)}{\tgColour4}
	\tgBorderA{(1,3)}{\tgColour4}{\tgColour6}{\tgColour6}{\tgColour4}
	\tgBlank{(2,3)}{\tgColour6}
	\tgCell[(2,0)]{(1,2)}{\mu}
	\tgArrow{(1,1.5)}{1}
	\tgArrow{(0,1.5)}{3}
	\tgArrow{(2,1.5)}{3}
	\tgCell{(0,1)}{\eta}
	\tgArrow{(1,2.5)}{3}
	\tgArrow{(0,0.5)}{3}
	\tgArrow{(2,0.5)}{3}
	\tgArrow{(1,0.5)}{1}
	\tgAxisLabel{(0.5,0.75)}{south}{j}
	\tgAxisLabel{(1.5,0.75)}{south}{j}
	\tgAxisLabel{(2.5,0.75)}{south}{t}
	\tgAxisLabel{(1.5,3.25)}{north}{t}
\end{tangle}
\tangleeq*
\begin{tangle}{(3,4)}[trim y]
	\tgBorderA{(0,0)}{\tgColour4}{\tgColour6}{\tgColour6}{\tgColour4}
	\tgBorderA{(1,0)}{\tgColour6}{\tgColour4}{\tgColour4}{\tgColour6}
	\tgBorderA{(2,0)}{\tgColour4}{\tgColour6}{\tgColour6}{\tgColour4}
	\tgBorderA{(0,1)}{\tgColour4}{\tgColour6}{\tgColour6}{\tgColour4}
	\tgBorderA{(1,1)}{\tgColour6}{\tgColour4}{\tgColour4}{\tgColour6}
	\tgBorderA{(2,1)}{\tgColour4}{\tgColour6}{\tgColour6}{\tgColour4}
	\tgBorderC{(0,2)}{0}{\tgColour4}{\tgColour6}
	\tgBorderC{(1,2)}{1}{\tgColour4}{\tgColour6}
	\tgBorderA{(2,2)}{\tgColour4}{\tgColour6}{\tgColour6}{\tgColour4}
	\tgBlank{(0,3)}{\tgColour4}
	\tgBlank{(1,3)}{\tgColour4}
	\tgBorderA{(2,3)}{\tgColour4}{\tgColour6}{\tgColour6}{\tgColour4}
	\tgArrow{(0.5,2)}{0}
	\tgArrow{(0,1.5)}{3}
	\tgArrow{(1,1.5)}{1}
	\tgArrow{(2,1.5)}{3}
	\tgArrow{(2,2.5)}{3}
	\tgArrow{(1,0.5)}{1}
	\tgArrow{(0,0.5)}{3}
	\tgArrow{(2,0.5)}{3}
	\tgAxisLabel{(0.5,0.75)}{south}{j}
	\tgAxisLabel{(1.5,0.75)}{south}{j}
	\tgAxisLabel{(2.5,0.75)}{south}{t}
	\tgAxisLabel{(2.5,3.25)}{north}{t}
\end{tangle}
\]
\[
\begin{tikzcd}[column sep=large]
	E & A & A \\
	E & A & A \\
	E & A & A \\
	E && A
	\arrow["{E(1, t)}"{description}, from=3-2, to=3-1]
	\arrow["{E(1, t)}", "\shortmid"{marking}, from=4-3, to=4-1]
	\arrow[""{name=0, anchor=center, inner sep=0}, Rightarrow, no head, from=3-3, to=4-3]
	\arrow[""{name=1, anchor=center, inner sep=0}, Rightarrow, no head, from=3-1, to=4-1]
	\arrow[""{name=2, anchor=center, inner sep=0}, Rightarrow, no head, from=2-2, to=3-2]
	\arrow[""{name=3, anchor=center, inner sep=0}, Rightarrow, no head, from=2-1, to=3-1]
	\arrow["{E(1, t)}"{description}, from=2-2, to=2-1]
	\arrow["{E(j, t)}"{description}, from=3-3, to=3-2]
	\arrow[Rightarrow, no head, from=1-3, to=1-2]
	\arrow["{E(j, j)}"{description}, from=2-3, to=2-2]
	\arrow[""{name=4, anchor=center, inner sep=0}, Rightarrow, no head, from=2-3, to=3-3]
	\arrow[""{name=5, anchor=center, inner sep=0}, Rightarrow, no head, from=1-3, to=2-3]
	\arrow[""{name=6, anchor=center, inner sep=0}, Rightarrow, no head, from=1-2, to=2-2]
	\arrow[""{name=7, anchor=center, inner sep=0}, Rightarrow, no head, from=1-1, to=2-1]
	\arrow["{E(1, t)}"', "\shortmid"{marking}, from=1-2, to=1-1]
	\arrow["\mu"{description}, draw=none, from=0, to=1]
	\arrow["{=}"{description}, draw=none, from=2, to=3]
	\arrow["{E(j, \eta)}"{description}, draw=none, from=4, to=2]
	\arrow["{=}"{description}, draw=none, from=6, to=7]
	\arrow["{\pc j}"{description}, draw=none, from=5, to=6]
\end{tikzcd}
\quad = \quad
\begin{tikzcd}
	E & A \\
	E & A
	\arrow["{E(1, t)}", "\shortmid"{marking}, no head, from=2-2, to=2-1]
	\arrow[""{name=0, anchor=center, inner sep=0}, Rightarrow, no head, from=1-2, to=2-2]
	\arrow[""{name=1, anchor=center, inner sep=0}, Rightarrow, no head, from=1-1, to=2-1]
	\arrow["{E(1, t)}"', "\shortmid"{marking}, from=1-2, to=1-1]
	\arrow["{=}"{description}, draw=none, from=0, to=1]
\end{tikzcd}
\hspace{4em}
\begin{tangle}{(3,4)}[trim y]
	\tgBorderA{(0,0)}{\tgColour4}{\tgColour6}{\tgColour6}{\tgColour4}
	\tgBlank{(1,0)}{\tgColour6}
	\tgBlank{(2,0)}{\tgColour6}
	\tgBorderA{(0,1)}{\tgColour4}{\tgColour6}{\tgColour6}{\tgColour4}
	\tgBorderA{(1,1)}{\tgColour6}{\tgColour6}{\tgColour4}{\tgColour6}
	\tgBorderA{(2,1)}{\tgColour6}{\tgColour6}{\tgColour6}{\tgColour4}
	\tgBorderA{(0,2)}{\tgColour4}{\tgColour6}{\tgColour4}{\tgColour4}
	\tgBorderA{(1,2)}{\tgColour6}{\tgColour4}{\tgColour6}{\tgColour4}
	\tgBorderA{(2,2)}{\tgColour4}{\tgColour6}{\tgColour6}{\tgColour6}
	\tgBlank{(0,3)}{\tgColour4}
	\tgBorderA{(1,3)}{\tgColour4}{\tgColour6}{\tgColour6}{\tgColour4}
	\tgBlank{(2,3)}{\tgColour6}
	\tgCell[(2,0)]{(1,2)}{\mu}
	\tgArrow{(1,1.5)}{1}
	\tgArrow{(2,1.5)}{3}
	\tgArrow{(0,1.5)}{3}
	\tgCell[(1,0)]{(1.5,1)}{\eta}
	\tgArrow{(0,0.5)}{3}
	\tgArrow{(1,2.5)}{3}
	\tgAxisLabel{(0.5,0.75)}{south}{t}
	\tgAxisLabel{(1.5,3.25)}{north}{t}
\end{tangle}
\tangleeq*
\begin{tangle}{(1,4)}[trim y]
	\tgBorderA{(0,0)}{\tgColour4}{\tgColour6}{\tgColour6}{\tgColour4}
	\tgBorderA{(0,1)}{\tgColour4}{\tgColour6}{\tgColour6}{\tgColour4}
	\tgBorderA{(0,2)}{\tgColour4}{\tgColour6}{\tgColour6}{\tgColour4}
	\tgBorderA{(0,3)}{\tgColour4}{\tgColour6}{\tgColour6}{\tgColour4}
	\tgArrow{(0,1.5)}{3}
	\tgArrow{(0,0.5)}{3}
	\tgArrow{(0,2.5)}{3}
	\tgAxisLabel{(0.5,0.75)}{south}{t}
	\tgAxisLabel{(0.5,3.25)}{north}{t}
\end{tangle}
\]
\[
\begin{tikzcd}[column sep=1.54em]
	E & A & E & A && E && A \\
	E &&& A && E && A \\
	E &&&&&&& A
	\arrow["{E(1, t)}"{description}, from=2-8, to=2-6]
	\arrow["{E(j, 1)}"{description}, from=2-6, to=2-4]
	\arrow["{E(1, t)}"{description}, from=2-4, to=2-1]
	\arrow["{E(1, t)}", "\shortmid"{marking}, from=3-8, to=3-1]
	\arrow[""{name=0, anchor=center, inner sep=0}, Rightarrow, no head, from=2-8, to=3-8]
	\arrow[""{name=1, anchor=center, inner sep=0}, Rightarrow, no head, from=2-1, to=3-1]
	\arrow["{E(1, t)}"', "\shortmid"{marking}, from=1-8, to=1-6]
	\arrow["{E(j, 1)}"', "\shortmid"{marking}, from=1-6, to=1-4]
	\arrow["{E(1, t)}"', "\shortmid"{marking}, from=1-4, to=1-3]
	\arrow["{E(j, 1)}"', "\shortmid"{marking}, from=1-3, to=1-2]
	\arrow["{E(1, t)}"', "\shortmid"{marking}, from=1-2, to=1-1]
	\arrow[""{name=2, anchor=center, inner sep=0}, Rightarrow, no head, from=1-1, to=2-1]
	\arrow[""{name=3, anchor=center, inner sep=0}, Rightarrow, no head, from=1-8, to=2-8]
	\arrow[""{name=4, anchor=center, inner sep=0}, Rightarrow, no head, from=1-4, to=2-4]
	\arrow[""{name=5, anchor=center, inner sep=0}, Rightarrow, no head, from=1-6, to=2-6]
	\arrow["\mu"{description}, draw=none, from=0, to=1]
	\arrow["\mu"{description}, draw=none, from=4, to=2]
	\arrow["{=}"{description}, draw=none, from=3, to=5]
	\arrow["{=}"{description}, draw=none, from=5, to=4]
\end{tikzcd}
=
\begin{tikzcd}[column sep=1.54em]
	E && A && E & A & E & A \\
	E && A && E &&& A \\
	E &&&&&&& A
	\arrow["{E(1, t)}"{description}, from=2-8, to=2-5]
	\arrow["{E(j, 1)}"{description}, from=2-5, to=2-3]
	\arrow["{E(1, t)}"{description}, from=2-3, to=2-1]
	\arrow["{E(1, t)}", "\shortmid"{marking}, from=3-8, to=3-1]
	\arrow[""{name=0, anchor=center, inner sep=0}, Rightarrow, no head, from=2-8, to=3-8]
	\arrow[""{name=1, anchor=center, inner sep=0}, Rightarrow, no head, from=2-1, to=3-1]
	\arrow["{E(1, t)}"', "\shortmid"{marking}, from=1-8, to=1-7]
	\arrow["{E(j, 1)}"', "\shortmid"{marking}, from=1-7, to=1-6]
	\arrow["{E(1, t)}"', "\shortmid"{marking}, from=1-6, to=1-5]
	\arrow["{E(j, 1)}"', "\shortmid"{marking}, from=1-5, to=1-3]
	\arrow["{E(1, t)}"', "\shortmid"{marking}, from=1-3, to=1-1]
	\arrow[""{name=2, anchor=center, inner sep=0}, Rightarrow, no head, from=1-1, to=2-1]
	\arrow[""{name=3, anchor=center, inner sep=0}, Rightarrow, no head, from=1-5, to=2-5]
	\arrow[""{name=4, anchor=center, inner sep=0}, Rightarrow, no head, from=1-8, to=2-8]
	\arrow[""{name=5, anchor=center, inner sep=0}, Rightarrow, no head, from=1-3, to=2-3]
	\arrow["\mu"{description}, draw=none, from=0, to=1]
	\arrow["\mu"{description}, draw=none, from=4, to=3]
	\arrow["{=}"{description}, draw=none, from=3, to=5]
	\arrow["{=}"{description}, draw=none, from=5, to=2]
\end{tikzcd}
\]\[
\begin{tangle}{(6,4)}[trim y]
	\tgBorderA{(0,0)}{\tgColour4}{\tgColour6}{\tgColour6}{\tgColour4}
	\tgBorderA{(1,0)}{\tgColour6}{\tgColour4}{\tgColour4}{\tgColour6}
	\tgBorderA{(2,0)}{\tgColour4}{\tgColour6}{\tgColour6}{\tgColour4}
	\tgBorderA{(3,0)}{\tgColour6}{\tgColour4}{\tgColour4}{\tgColour6}
	\tgBlank{(4,0)}{\tgColour4}
	\tgBorderA{(5,0)}{\tgColour4}{\tgColour6}{\tgColour6}{\tgColour4}
	\tgBorderA{(0,1)}{\tgColour4}{\tgColour6}{\tgColour4}{\tgColour4}
	\tgBorderA{(1,1)}{\tgColour6}{\tgColour4}{\tgColour6}{\tgColour4}
	\tgBorderA{(2,1)}{\tgColour4}{\tgColour6}{\tgColour6}{\tgColour6}
	\tgBorderA{(3,1)}{\tgColour6}{\tgColour4}{\tgColour4}{\tgColour6}
	\tgBlank{(4,1)}{\tgColour4}
	\tgBorderA{(5,1)}{\tgColour4}{\tgColour6}{\tgColour6}{\tgColour4}
	\tgBlank{(0,2)}{\tgColour4}
	\tgBorderA{(1,2)}{\tgColour4}{\tgColour6}{\tgColour4}{\tgColour4}
	\tgBorderA{(2,2)}{\tgColour6}{\tgColour6}{\tgColour4}{\tgColour4}
	\tgBorderA{(3,2)}{\tgColour6}{\tgColour4}{\tgColour6}{\tgColour4}
	\tgBorderA{(4,2)}{\tgColour4}{\tgColour4}{\tgColour6}{\tgColour6}
	\tgBorderA{(5,2)}{\tgColour4}{\tgColour6}{\tgColour6}{\tgColour6}
	\tgBlank{(0,3)}{\tgColour4}
	\tgBlank{(1,3)}{\tgColour4}
	\tgBlank{(2,3)}{\tgColour4}
	\tgBorderA{(3,3)}{\tgColour4}{\tgColour6}{\tgColour6}{\tgColour4}
	\tgBlank{(4,3)}{\tgColour6}
	\tgBlank{(5,3)}{\tgColour6}
	\tgCell[(2,0)]{(1,1)}{\mu}
	\tgCell[(4,0)]{(3,2)}{\mu}
	\tgArrow{(3,1.5)}{1}
	\tgArrow{(5,1.5)}{3}
	\tgArrow{(1,1.5)}{3}
	\tgArrow{(1,0.5)}{1}
	\tgArrow{(2,0.5)}{3}
	\tgArrow{(0,0.5)}{3}
	\tgArrow{(5,0.5)}{3}
	\tgArrow{(3,0.5)}{1}
	\tgArrow{(3,2.5)}{3}
	\tgAxisLabel{(0.5,0.75)}{south}{t}
	\tgAxisLabel{(1.5,0.75)}{south}{j}
	\tgAxisLabel{(2.5,0.75)}{south}{t}
	\tgAxisLabel{(3.5,0.75)}{south}{j}
	\tgAxisLabel{(5.5,0.75)}{south}{t}
	\tgAxisLabel{(3.5,3.25)}{north}{t}
\end{tangle}
\tangleeq*
\begin{tangle}{(6,4)}[trim y]
	\tgBorderA{(0,0)}{\tgColour4}{\tgColour6}{\tgColour6}{\tgColour4}
	\tgBlank{(1,0)}{\tgColour6}
	\tgBorderA{(2,0)}{\tgColour6}{\tgColour4}{\tgColour4}{\tgColour6}
	\tgBorderA{(3,0)}{\tgColour4}{\tgColour6}{\tgColour6}{\tgColour4}
	\tgBorderA{(4,0)}{\tgColour6}{\tgColour4}{\tgColour4}{\tgColour6}
	\tgBorderA{(5,0)}{\tgColour4}{\tgColour6}{\tgColour6}{\tgColour4}
	\tgBorderA{(0,1)}{\tgColour4}{\tgColour6}{\tgColour6}{\tgColour4}
	\tgBlank{(1,1)}{\tgColour6}
	\tgBorderA{(2,1)}{\tgColour6}{\tgColour4}{\tgColour4}{\tgColour6}
	\tgBorderA{(3,1)}{\tgColour4}{\tgColour6}{\tgColour4}{\tgColour4}
	\tgBorderA{(4,1)}{\tgColour6}{\tgColour4}{\tgColour6}{\tgColour4}
	\tgBorderA{(5,1)}{\tgColour4}{\tgColour6}{\tgColour6}{\tgColour6}
	\tgBorderA{(0,2)}{\tgColour4}{\tgColour6}{\tgColour4}{\tgColour4}
	\tgBorderA{(1,2)}{\tgColour6}{\tgColour6}{\tgColour4}{\tgColour4}
	\tgBorderA{(2,2)}{\tgColour6}{\tgColour4}{\tgColour6}{\tgColour4}
	\tgBorderA{(3,2)}{\tgColour4}{\tgColour4}{\tgColour6}{\tgColour6}
	\tgBorderA{(4,2)}{\tgColour4}{\tgColour6}{\tgColour6}{\tgColour6}
	\tgBlank{(5,2)}{\tgColour6}
	\tgBlank{(0,3)}{\tgColour4}
	\tgBlank{(1,3)}{\tgColour4}
	\tgBorderA{(2,3)}{\tgColour4}{\tgColour6}{\tgColour6}{\tgColour4}
	\tgBlank{(3,3)}{\tgColour6}
	\tgBlank{(4,3)}{\tgColour6}
	\tgBlank{(5,3)}{\tgColour6}
	\tgCell[(4,0)]{(2,2)}{\mu}
	\tgArrow{(2,1.5)}{1}
	\tgArrow{(4,1.5)}{3}
	\tgArrow{(0,1.5)}{3}
	\tgCell[(2,0)]{(4,1)}{\mu}
	\tgArrow{(2,0.5)}{1}
	\tgArrow{(4,0.5)}{1}
	\tgArrow{(0,0.5)}{3}
	\tgArrow{(3,0.5)}{3}
	\tgArrow{(5,0.5)}{3}
	\tgArrow{(2,2.5)}{3}
	\tgAxisLabel{(0.5,0.75)}{south}{t}
	\tgAxisLabel{(2.5,0.75)}{south}{j}
	\tgAxisLabel{(3.5,0.75)}{south}{t}
	\tgAxisLabel{(4.5,0.75)}{south}{j}
	\tgAxisLabel{(5.5,0.75)}{south}{t}
	\tgAxisLabel{(2.5,3.25)}{north}{t}
\end{tangle}
\]
A monoid homomorphism is a 2-cell $E(1, \tau) \colon E(1, t) \tto E(1, t')$ satisfying the following equations.
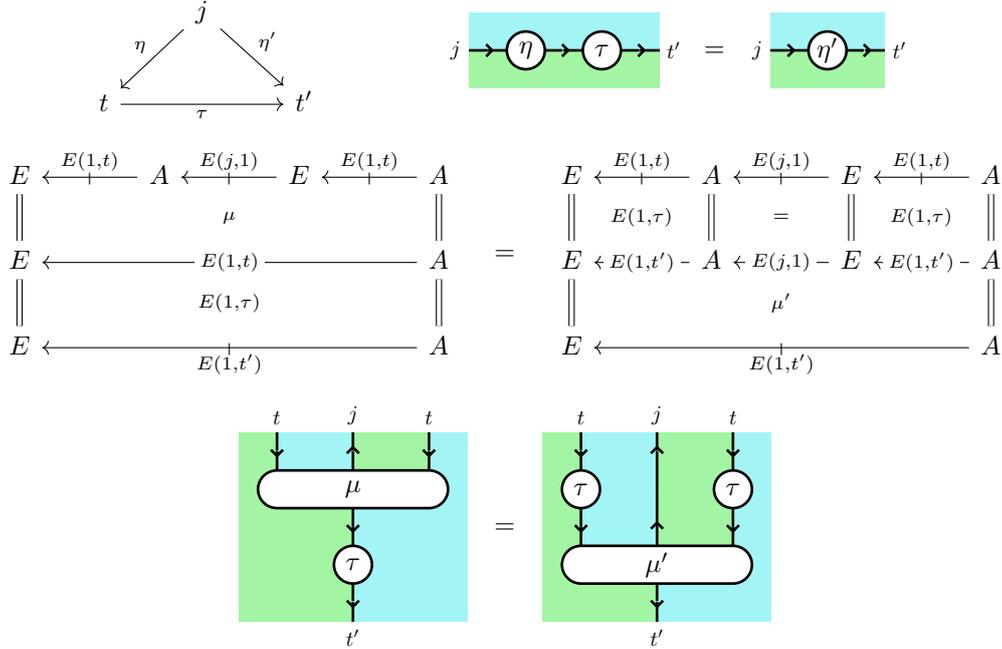
\[
\begin{tikzcd}
	& j \\
	t && {t'}
	\arrow["\tau"', from=2-1, to=2-3]
	\arrow["\eta"', from=1-2, to=2-1]
	\arrow["{\eta'}", from=1-2, to=2-3]
\end{tikzcd}
\hspace{4em}
\begin{tangle}{(4,1)}[trim x]
	\tgBorderA{(0,0)}{\tgColour6}{\tgColour6}{\tgColour4}{\tgColour4}
	\tgBorderA{(1,0)}{\tgColour6}{\tgColour6}{\tgColour4}{\tgColour4}
	\tgBorderA{(2,0)}{\tgColour6}{\tgColour6}{\tgColour4}{\tgColour4}
	\tgBorderA{(3,0)}{\tgColour6}{\tgColour6}{\tgColour4}{\tgColour4}
	\tgCell{(1,0)}{\eta}
	\tgCell{(2,0)}{\tau}
	\tgArrow{(1.5,0)}{0}
	\tgArrow{(0.5,0)}{0}
	\tgArrow{(2.5,0)}{0}
	\tgAxisLabel{(0.75,0.5)}{east}{j}
	\tgAxisLabel{(3.25,0.5)}{west}{t'}
\end{tangle}
=
\begin{tangle}{(3,1)}[trim x]
	\tgBorderA{(0,0)}{\tgColour6}{\tgColour6}{\tgColour4}{\tgColour4}
	\tgBorderA{(1,0)}{\tgColour6}{\tgColour6}{\tgColour4}{\tgColour4}
	\tgBorderA{(2,0)}{\tgColour6}{\tgColour6}{\tgColour4}{\tgColour4}
	\tgCell{(1,0)}{\eta'}
	\tgArrow{(0.5,0)}{0}
	\tgArrow{(1.5,0)}{0}
	\tgAxisLabel{(0.75,0.5)}{east}{j}
	\tgAxisLabel{(2.25,0.5)}{west}{t'}
\end{tangle}
\]
\[
\begin{tikzcd}[column sep=large]
	E & A & E & A \\
	E &&& A \\
	E &&& A
	\arrow["{E(1, t)}"', "\shortmid"{marking}, from=1-4, to=1-3]
	\arrow["{E(j, 1)}"', "\shortmid"{marking}, from=1-3, to=1-2]
	\arrow["{E(1, t)}"', "\shortmid"{marking}, from=1-2, to=1-1]
	\arrow["{E(1, t)}"{description}, from=2-4, to=2-1]
	\arrow[""{name=0, anchor=center, inner sep=0}, Rightarrow, no head, from=1-4, to=2-4]
	\arrow[""{name=1, anchor=center, inner sep=0}, Rightarrow, no head, from=1-1, to=2-1]
	\arrow["{E(1, t')}", "\shortmid"{marking}, from=3-4, to=3-1]
	\arrow[""{name=2, anchor=center, inner sep=0}, Rightarrow, no head, from=2-4, to=3-4]
	\arrow[""{name=3, anchor=center, inner sep=0}, Rightarrow, no head, from=2-1, to=3-1]
	\arrow["\mu"{description}, draw=none, from=0, to=1]
	\arrow["{E(1, \tau)}"{description}, draw=none, from=2, to=3]
\end{tikzcd}
\quad = \quad
\begin{tikzcd}[column sep=large]
	E & A & E & A \\
	E & A & E & A \\
	E &&& A
	\arrow["{E(1, t')}"{description}, from=2-4, to=2-3]
	\arrow["{E(j, 1)}"{description}, from=2-3, to=2-2]
	\arrow["{E(1, t')}"{description}, from=2-2, to=2-1]
	\arrow["{E(1, t')}", "\shortmid"{marking}, from=3-4, to=3-1]
	\arrow[""{name=0, anchor=center, inner sep=0}, Rightarrow, no head, from=2-4, to=3-4]
	\arrow[""{name=1, anchor=center, inner sep=0}, Rightarrow, no head, from=2-1, to=3-1]
	\arrow["{E(1, t)}"', "\shortmid"{marking}, from=1-4, to=1-3]
	\arrow["{E(j, 1)}"', "\shortmid"{marking}, from=1-3, to=1-2]
	\arrow["{E(1, t)}"', "\shortmid"{marking}, from=1-2, to=1-1]
	\arrow[""{name=2, anchor=center, inner sep=0}, Rightarrow, no head, from=1-2, to=2-2]
	\arrow[""{name=3, anchor=center, inner sep=0}, Rightarrow, no head, from=1-1, to=2-1]
	\arrow[""{name=4, anchor=center, inner sep=0}, Rightarrow, no head, from=1-4, to=2-4]
	\arrow[""{name=5, anchor=center, inner sep=0}, Rightarrow, no head, from=1-3, to=2-3]
	\arrow["{\mu'}"{description}, draw=none, from=0, to=1]
	\arrow["{E(1, \tau)}"{description}, draw=none, from=2, to=3]
	\arrow["{E(1, \tau)}"{description}, draw=none, from=4, to=5]
	\arrow["{=}"{description}, draw=none, from=5, to=2]
\end{tikzcd}
\]\[
\begin{tangle}{(3,4)}[trim y]
	\tgBorderA{(0,0)}{\tgColour4}{\tgColour6}{\tgColour6}{\tgColour4}
	\tgBorderA{(1,0)}{\tgColour6}{\tgColour4}{\tgColour4}{\tgColour6}
	\tgBorderA{(2,0)}{\tgColour4}{\tgColour6}{\tgColour6}{\tgColour4}
	\tgBorderA{(0,1)}{\tgColour4}{\tgColour6}{\tgColour4}{\tgColour4}
	\tgBorderA{(1,1)}{\tgColour6}{\tgColour4}{\tgColour6}{\tgColour4}
	\tgBorderA{(2,1)}{\tgColour4}{\tgColour6}{\tgColour6}{\tgColour6}
	\tgBlank{(0,2)}{\tgColour4}
	\tgBorderA{(1,2)}{\tgColour4}{\tgColour6}{\tgColour6}{\tgColour4}
	\tgBlank{(2,2)}{\tgColour6}
	\tgBlank{(0,3)}{\tgColour4}
	\tgBorderA{(1,3)}{\tgColour4}{\tgColour6}{\tgColour6}{\tgColour4}
	\tgBlank{(2,3)}{\tgColour6}
	\tgCell[(2,0)]{(1,1)}{\mu}
	\tgCell{(1,2)}{\tau}
	\tgArrow{(1,1.5)}{3}
	\tgArrow{(1,0.5)}{1}
	\tgArrow{(0,0.5)}{3}
	\tgArrow{(2,0.5)}{3}
	\tgArrow{(1,2.5)}{3}
	\tgAxisLabel{(0.5,0.75)}{south}{t}
	\tgAxisLabel{(1.5,0.75)}{south}{j}
	\tgAxisLabel{(2.5,0.75)}{south}{t}
	\tgAxisLabel{(1.5,3.25)}{north}{t'}
\end{tangle}
\tangleeq*
\begin{tangle}{(3,4)}[trim y]
	\tgBorderA{(0,0)}{\tgColour4}{\tgColour6}{\tgColour6}{\tgColour4}
	\tgBorderA{(1,0)}{\tgColour6}{\tgColour4}{\tgColour4}{\tgColour6}
	\tgBorderA{(2,0)}{\tgColour4}{\tgColour6}{\tgColour6}{\tgColour4}
	\tgBorderA{(0,1)}{\tgColour4}{\tgColour6}{\tgColour6}{\tgColour4}
	\tgBorderA{(1,1)}{\tgColour6}{\tgColour4}{\tgColour4}{\tgColour6}
	\tgBorderA{(2,1)}{\tgColour4}{\tgColour6}{\tgColour6}{\tgColour4}
	\tgBorderA{(0,2)}{\tgColour4}{\tgColour6}{\tgColour4}{\tgColour4}
	\tgBorderA{(1,2)}{\tgColour6}{\tgColour4}{\tgColour6}{\tgColour4}
	\tgBorderA{(2,2)}{\tgColour4}{\tgColour6}{\tgColour6}{\tgColour6}
	\tgBlank{(0,3)}{\tgColour4}
	\tgBorderA{(1,3)}{\tgColour4}{\tgColour6}{\tgColour6}{\tgColour4}
	\tgBlank{(2,3)}{\tgColour6}
	\tgCell[(2,0)]{(1,2)}{\mu'}
	\tgCell{(0,1)}{\tau}
	\tgCell{(2,1)}{\tau}
	\tgArrow{(1,1.5)}{1}
	\tgArrow{(0,1.5)}{3}
	\tgArrow{(2,1.5)}{3}
	\tgArrow{(0,0.5)}{3}
	\tgArrow{(1,0.5)}{1}
	\tgArrow{(2,0.5)}{3}
	\tgArrow{(1,2.5)}{3}
	\tgAxisLabel{(0.5,0.75)}{south}{t}
	\tgAxisLabel{(1.5,0.75)}{south}{j}
	\tgAxisLabel{(2.5,0.75)}{south}{t}
	\tgAxisLabel{(1.5,3.25)}{north}{t'}
\end{tangle}
\]

We may now exhibit the $j$-monads of \cref{relative-monad} as monoids in $\X[j]$: conceptually, the equivalence arises from transposing the loose-cell $E(1, t)$ in the domain of the multiplication operator $\mu$ to the loose-cell $E(t, 1)$ in the codomain of the extension operator $\dag$, via the loose-adjunction $E(1, t) \adj E(t, 1)$. String diagrammatically, this corresponds to bending the string associated with the tight-cell $t$ from pointing down to pointing up.

\begin{theorem}
    \label{relative-monads-are-tight-monoids}
    There is an isomorphism of categories rendering the following diagram commutative.
	\[\begin{tikzcd}
		{\RMnd(j)} && {\Mon(\X[j])} \\
		& {\X[A, E]}
		\arrow["{U_j}"', from=1-1, to=2-2]
		\arrow["{U_{\X[j]}}", from=1-3, to=2-2]
		\arrow["\iso", from=1-1, to=1-3]
	\end{tikzcd}\]
\end{theorem}

\begin{proof}
    Observe that, given a multiplication, we can define an extension operator, and conversely:
    
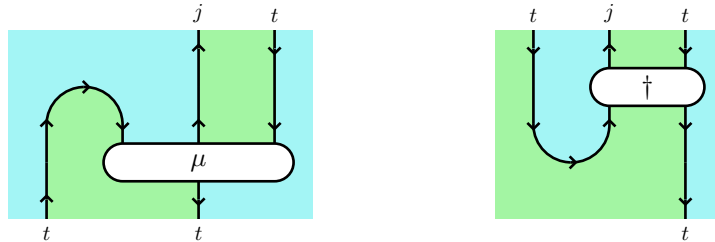
\[
	\begin{tangle}{(4,4)}[trim y]
		\tgBlank{(0,0)}{\tgColour6}
		\tgBlank{(1,0)}{\tgColour6}
		\tgBorderA{(2,0)}{\tgColour6}{\tgColour4}{\tgColour4}{\tgColour6}
		\tgBorderA{(3,0)}{\tgColour4}{\tgColour6}{\tgColour6}{\tgColour4}
		\tgBorderC{(0,1)}{3}{\tgColour6}{\tgColour4}
		\tgBorderC{(1,1)}{2}{\tgColour6}{\tgColour4}
		\tgBorderA{(2,1)}{\tgColour6}{\tgColour4}{\tgColour4}{\tgColour6}
		\tgBorderA{(3,1)}{\tgColour4}{\tgColour6}{\tgColour6}{\tgColour4}
		\tgBorderA{(0,2)}{\tgColour6}{\tgColour4}{\tgColour4}{\tgColour6}
		\tgBorderA{(1,2)}{\tgColour4}{\tgColour6}{\tgColour4}{\tgColour4}
		\tgBorderA{(2,2)}{\tgColour6}{\tgColour4}{\tgColour6}{\tgColour4}
		\tgBorderA{(3,2)}{\tgColour4}{\tgColour6}{\tgColour6}{\tgColour6}
		\tgBorderA{(0,3)}{\tgColour6}{\tgColour4}{\tgColour4}{\tgColour6}
		\tgBlank{(1,3)}{\tgColour4}
		\tgBorderA{(2,3)}{\tgColour4}{\tgColour6}{\tgColour6}{\tgColour4}
		\tgBlank{(3,3)}{\tgColour6}
		\tgCell[(2,0)]{(2,2)}{\mu}
		\tgArrow{(2,1.5)}{1}
		\tgArrow{(3,1.5)}{3}
		\tgArrow{(1,1.5)}{3}
		\tgArrow{(0.5,1)}{0}
		\tgArrow{(0,1.5)}{1}
		\tgArrow{(2,0.5)}{1}
		\tgArrow{(0,2.5)}{1}
		\tgArrow{(2,2.5)}{3}
		\tgArrow{(3,0.5)}{3}
		\tgAxisLabel{(2.5,0.75)}{south}{j}
		\tgAxisLabel{(3.5,0.75)}{south}{t}
		\tgAxisLabel{(0.5,3.25)}{north}{t}
		\tgAxisLabel{(2.5,3.25)}{north}{t}
	\end{tangle}
    \hspace{4em}
	\begin{tangle}{(3,4)}[trim y]
		\tgBorderA{(0,0)}{\tgColour4}{\tgColour6}{\tgColour6}{\tgColour4}
		\tgBorderA{(1,0)}{\tgColour6}{\tgColour4}{\tgColour4}{\tgColour6}
		\tgBorderA{(2,0)}{\tgColour4}{\tgColour6}{\tgColour6}{\tgColour4}
		\tgBorderA{(0,1)}{\tgColour4}{\tgColour6}{\tgColour6}{\tgColour4}
		\tgBorderA{(1,1)}{\tgColour6}{\tgColour4}{\tgColour4}{\tgColour6}
		\tgBorder{(1,1)}{0}{1}{0}{0}
		\tgBorderA{(2,1)}{\tgColour4}{\tgColour6}{\tgColour6}{\tgColour4}
		\tgBorder{(2,1)}{0}{0}{0}{1}
		\tgBorderC{(0,2)}{0}{\tgColour4}{\tgColour6}
		\tgBorderC{(1,2)}{1}{\tgColour4}{\tgColour6}
		\tgBorderA{(2,2)}{\tgColour4}{\tgColour6}{\tgColour6}{\tgColour4}
		\tgBlank{(0,3)}{\tgColour4}
		\tgBlank{(1,3)}{\tgColour4}
		\tgBorderA{(2,3)}{\tgColour4}{\tgColour6}{\tgColour6}{\tgColour4}
		\tgCell[(1,0)]{(1.5,1)}{\dag}
		\tgArrow{(2,1.5)}{3}
		\tgArrow{(1,1.5)}{1}
		\tgArrow{(0.5,2)}{0}
		\tgArrow{(0,1.5)}{3}
		\tgArrow{(2,2.5)}{3}
		\tgArrow{(2,0.5)}{3}
		\tgArrow{(0,0.5)}{3}
		\tgArrow{(1,0.5)}{1}
		\tgAxisLabel{(0.5,0.75)}{south}{t}
		\tgAxisLabel{(1.5,0.75)}{south}{j}
		\tgAxisLabel{(2.5,0.75)}{south}{t}
		\tgAxisLabel{(2.5,3.25)}{north}{t}
	\end{tangle}
    \]
    It is immediate that the laws for a relative monad (morphism) are precisely those for a monoid (homomorphism) under these transformations.
\end{proof}

We shall spell out the distinction between the two presentations of a relative monad in $\X = \Cat$ (for which the monoid presentation appears to be new) in \cref{enriched-relative-monads-and-relative-adjunctions}.

\begin{remark}
    \citeauthor{levy2019what}~\cite[21]{levy2019what} suggests that relative monads ought to be monoids in skew-multicategories, but does not give an explicit definition.
\end{remark}

To compare our definition of relative monad to a similar definition in the literature, we observe that, while $j$-monads are precisely monoids in $\X[j]$, they may also be viewed as particular monoids in $\X\lh j$: namely those for which the underlying loose-cell is representable. First, we make the following observation.

\begin{lemma}
    \label{monoids-in-sub-skew-multicategory-is-pullback}
    Let $\M$ be a skew-multicategory and let $\M' \ffto \M$ be a full sub-skew-multicategory. The following square forms a pullback of categories.
    \[\begin{tikzcd}
    	{\Mon(\M')} & {\Mon(\M)} \\
    	{{\M'}_1} & {\M_1}
    	\arrow[""{name=0, anchor=center, inner sep=0}, hook, from=2-1, to=2-2]
    	\arrow["{U_\M}", from=1-2, to=2-2]
    	\arrow["{U_{\M'}}"', from=1-1, to=2-1]
    	\arrow[hook, from=1-1, to=1-2]
    	\arrow["\lrcorner"{anchor=center, pos=0.125}, draw=none, from=1-1, to=0]
    \end{tikzcd}\]
\end{lemma}

\begin{proof}
    By definition, a monoid in $\M'$ is a monoid in $\M$ whose carrier is in $\M'$ and, since $\M'$ is a full sub-skew-multicategory, the homomorphisms are identical.
\end{proof}

We then have the following, identifying relative monads with representable loose relative monads, and justifying the terminology of the latter. Henceforth, when we wish to emphasise the distinction between relative monads and loose relative monads, we may call the former \emph{tight}.

\begin{proposition}
    \label{relative-monads-are-loose-relative-monads}
    The following square forms a pullback of categories.
	\[\begin{tikzcd}
		{\RMnd(j)} & {\LRMnd(j)} \\
		{\X[A, E]} & {\X\lh{A, E}}
		\arrow[""{name=0, anchor=center, inner sep=0}, hook, from=2-1, to=2-2]
		\arrow["{U_j}", from=1-2, to=2-2]
		\arrow["{U_j}"', from=1-1, to=2-1]
		\arrow[hook, from=1-1, to=1-2]
		\arrow["\lrcorner"{anchor=center, pos=0.125}, draw=none, from=1-1, to=0]
	\end{tikzcd}\]
\end{proposition}

\begin{proof}
    Direct by composing \cref{relative-monads-are-tight-monoids} with \cref{monoids-in-sub-skew-multicategory-is-pullback}, considering the inclusion ${\X[j] \ffto \X\lh j}$.
\end{proof}

Consequently, it is evident that relative monads generalise monads.

\begin{corollary}
    \label{identity-relative-monads-are-monads}
    Let $A$ be an object of $\X$. There is an isomorphism of categories rendering the following diagram commutative.
	\[\begin{tikzcd}
		{\RMnd(1_A)} && {\Mnd(A)} \\
		& {\X[A, A]}
		\arrow["{U_{1_A}}"', from=1-1, to=2-2]
		\arrow["{U_A}", from=1-3, to=2-2]
		\arrow["\iso", from=1-1, to=1-3]
	\end{tikzcd}\]
\end{corollary}

\begin{proof}
	We have the following pullbacks
	\[
	\begin{tikzcd}
		{\RMnd(1_A)} & {\LRMnd(1_A)} \\
		{\X[A, A]} & {\X\lh{A, A}}
		\arrow[""{name=0, anchor=center, inner sep=0}, hook, from=2-1, to=2-2]
		\arrow["{U_{1_A}}", from=1-2, to=2-2]
		\arrow["{U_{1_A}}"', from=1-1, to=2-1]
		\arrow[hook, from=1-1, to=1-2]
		\arrow["\lrcorner"{anchor=center, pos=0.125}, draw=none, from=1-1, to=0]
	\end{tikzcd}
	\hspace{4em}
	\begin{tikzcd}
		{\Mnd(A)} & {\LMnd(A)} \\
		{\X[A, A]} & {\X\lh{A, A}}
		\arrow[""{name=0, anchor=center, inner sep=0}, hook, from=2-1, to=2-2]
		\arrow["{U_A}", from=1-2, to=2-2]
		\arrow["{U_A}"', from=1-1, to=2-1]
		\arrow[hook, from=1-1, to=1-2]
		\arrow["\lrcorner"{anchor=center, pos=0.125}, draw=none, from=1-1, to=0]
	\end{tikzcd}
	\]
	by \cref{relative-monads-are-loose-relative-monads} and \cref{monoids-in-sub-skew-multicategory-is-pullback}. Hence both categories exhibit pullbacks of the same cospans by \cref{loose-relative-monad-to-loose-monad}.
\end{proof}

\begin{remark}
	\label{extension-systems}
	The data of a monad relative to the identity is reminiscent of the definition of \emph{extension system} of \textcite[Definition~2.3]{marmolejo2010monads}, a generalisation of the \emph{algebraic theories in extension form} of \cite[Exercise~1.3.12]{manes1976algebraic} to arbitrary 2-categories, each of which comprises a 1-cell $t \colon A \to A$, a 2-cell $\eta \colon 1 \tto t$, and a family of functions
	\[\{ \X[\cdot, A](x, t y) \to \X[\cdot, A](t x, t y) \}_{x, y \colon \cdot \to A}\]
	that is well-behaved in the sense of \cite[Definition~2.1]{marmolejo2010monads}, and subject to unitality and associativity laws. From \cite[Lemma~2.2]{marmolejo2010monads}, it follows that such families are equivalent to 2-cells $t \circ t \tto t$, and hence to 2-cells $A(1, t) \tto A(t, t)$. Thus, when $j = 1$, extension systems are essentially the same as $j$-relative monads (in extension form).

	In a similar fashion, the definition of monoid in $\X\lh{1_A}$ is reminiscent of (an analogous generalisation of) the definition of \emph{algebraic theory in clone form} of \textcite[3.2]{manes1976algebraic}, which comprises a 1-cell $t \colon A \to A$, a 2-cell $\eta \colon 1 \tto t$, and a well-behaved family of functions
	\[\{ \X[\cdot, A](x, t y) \times \X[\cdot, A](y, t z) \to \X[\cdot, A](x, t z) \}_{x, y, z \colon \cdot \to A}\]
	subject to unitality and associativity laws. As above, such families are equivalent to 2-cells $t \circ t \tto t$, and hence to 2-cells $A(1, t), A(1, t) \tto A(1, t)$. Thus, when $j = 1$, algebraic theories in clone form are essentially the same as $j$-relative monads (in monoid form).

	The approach of \textcite{marmolejo2010monads} was generalised by \textcite{lobbia2023distributive} to capture relative monads in $\Cat$. However, the failure of an analogue of \cite[Lemma~2.2]{marmolejo2010monads} in that setting means that the definition of \emph{relative monad} of \cite[Definition~2.1]{lobbia2023distributive} is not in general equivalent to our definition. We will show in \cref{relative-monads-in-VCat} that, in contrast to the definition of \citeauthor{lobbia2023distributive} (\cf~\cite[Example~2.2(iii)]{lobbia2023distributive}), our definition recovers the expected notion of enriched relative monad.
\end{remark}

Relative monads behave particularly nicely when their roots are dense. An example of this behaviour is given by the following, which permits the representation of $j$-monads as \emph{$j$-representable} loose-monads: \ie{} those loose-monads for which the underlying loose-cell is of the form ${E(j, t) \colon A \lto A}$ for some tight-cell $t \colon A \to E$.

\begin{theorem}
    \label{relative-monads-are-loose-monads}
    There is a functor $E(j, {-}) \colon \RMnd(j) \to \LMnd(A)$, \ff{} if $j$ is dense, in which case the following square forms a pullback of categories.
	\[\begin{tikzcd}
		{\RMnd(j)} & {\LMnd(A)} \\
		{\X[A, E]} & {\X\lh{A, A}}
		\arrow[""{name=0, anchor=center, inner sep=0}, "{E(j, {-})}"', hook, from=2-1, to=2-2]
		\arrow["{U_j}"', from=1-1, to=2-1]
		\arrow["{U_A}", from=1-2, to=2-2]
		\arrow["{E(j, {-})}", hook, from=1-1, to=1-2]
		\arrow["\lrcorner"{anchor=center, pos=0.125}, draw=none, from=1-1, to=0]
	\end{tikzcd}\]
\end{theorem}

\begin{proof}
	Using \cref{loose-relative-monad-to-loose-monad} and \cref{relative-monads-are-tight-monoids}, we have the following diagram of categories.
	\[\begin{tikzcd}
		{\RMnd(j)} & {\LRMnd(j)} & {\LMnd(A)} \\
		{\X[A, E]} & {\X\lh{A, E}} & {\X\lh{A, A}}
		\arrow[""{name=0, anchor=center, inner sep=0}, hook, from=2-1, to=2-2]
		\arrow["{U_j}"{description}, from=1-2, to=2-2]
		\arrow["{U_j}"', from=1-1, to=2-1]
		\arrow[hook, from=1-1, to=1-2]
		\arrow["{\ph(j, 1)}"', from=2-2, to=2-3]
		\arrow["{U_A}", from=1-3, to=2-3]
		\arrow["{\ph(j, 1)}", from=1-2, to=1-3]
		\arrow["\lrcorner"{anchor=center, pos=0.125}, draw=none, from=1-1, to=0]
	\end{tikzcd}\]
	The composite functor $\X[A, E] \to \X\lh{A, A}$ is $E(j, 1)$, which, when $j$ is dense, is \ff{} by \cref{density-implies-j*-is-ff}, and hence the outer rectangle is also a pullback by \cref{monoids-in-sub-skew-multicategory-is-pullback}. In this case, since \ff{} functors are stable under pullback, the composite functor $\RMnd(j) \to \LMnd(A)$ is also \ff{}.
\end{proof}

This characterisation will be related in \cref{examples-of-enriched-relative-monads} to several notions appearing in the literature.

\begin{remark}
	\label{formal-mw-monads}
    In the terminology of \textcite{lack2014monads}, monoids in $\X\lh{j_* \adj j^*}$ are \emph{formal mw-monads}. In their setting it is not possible to recover (tight) relative monads along the lines of \cref{relative-monads-are-loose-relative-monads}, since restricting to the monoids whose underlying loose-cell is left adjoint does not precisely recover the tight-cells (\cf{}~\cref{representables-vs-maps}). \citeauthor{lack2014monads} note the similarity of their definition to that of relative monads, but do not make this relationship precise.
\end{remark}

\begin{remark}
	Let $j^* \colon E \lto A$ be a loose-cell. The definition of loose relative monad in \cref{loose-relative-monad} admits a natural generalisation to a structure comprising a loose-cell $t_* \colon A \lto E$ admitting a composite $j^* \odot t_* \colon A \lto A$, equipped with 2-cells $\mu \colon t_*, j^*, t_* \tto t_*$ and $\eta \colon \tto j^* \odot t_*$ satisfying unit and associativity laws. This recovers as special cases various generalisations of (relative) monads that have appeared in the literature.
	\begin{itemize}
		\item When $j^*$ is corepresentable, this is precisely the definition of loose relative monad in \cref{loose-relative-monad}; when $t_*$ is furthermore representable, this is precisely the definition of relative monad in \cref{relative-monad}.
		\item When $t_*$ is representable, we recover a generalisation of relative monad proposed by \textcite[20]{levy2019what}; when $j^*$ is furthermore representable (rather than corepresentable as might be expected), this is precisely the definition of \emph{$E$-monad on $A$} of \textcite[Definition~1]{spivey2009algebras}. In particular, the correspondence between relative monads with left-adjoint roots and $E$-monads on $A$ with right-adjoint roots observed in \cite[\S6]{altenkirch2010monads} follows immediately.
	\end{itemize}
	The study of these generalised loose relative monads is deferred to future work.
\end{remark}

\subsection{Relative monads as monoids in a multicategory}

While in general we require the generality of skew-multicategories to capture relative monads, it is natural to wonder whether there are situations in which it suffices to consider simpler structures. In this section, we shall show that it often suffices to consider a (non-skew) multicategory; in the following section, we shall show that it often suffices to consider a skew-monoidal category. In proving the latter, we shall give a conceptual explanation for the skew-monoidal categories of functors studied by \textcite{altenkirch2015monads}.

First, we observe that, to capture relative monads, it suffices to consider monoids in the underlying right-normal sub-left-skew-multicategory of an associative-normal left-skew-multicategory, which is formed by restricting to those multimorphisms in which the marker $\bullet$ may appear only in the first position~(\cf{}~\cref{skew-multicategories-as-generalised-multicategories}).

\begin{lemma}
	\label{monoids-in-M-rho-are-monoids-in-M}
	Let $\M$ be an associative-normal left-skew-multicategory, and denote by $\M_\rho$ its wide (associative- and) right-normal sub-skew-multicategory. Then there is an isomorphism of categories rendering the following diagram commutative.
	\[\begin{tikzcd}
		{\Mon(\M_\rho)} && {\Mon(\M)} \\
		& {\M_1}
		\arrow["\iso", from=1-1, to=1-3]
		\arrow["{U_{\M_\rho}}"', from=1-1, to=2-2]
		\arrow["{U_\M}", from=1-3, to=2-2]
	\end{tikzcd}\]
\end{lemma}

\begin{proof}
	The data for a monoid (homomorphism) in $\M$ only involves the data of the underlying $\{ \alpha, \rho \}$-normal left-skew-multicategory.
\end{proof}

While it is not possible to directly represent relative monads as monoids in the underlying (non-skew) multicategory of $\X[j]$, since the data of a monoid involves a multimorphism with domain~$\bullet$, it is possible to characterise when $\X[j]_\rho$ is equivalent to a multicategory.

\begin{proposition}
	\label{left-normal-if-dense}
	Let $\jAE$ be a tight-cell. $\X[j]_\rho$ is left-normal if $j$ is dense.
\end{proposition}

\begin{proof}
    2-cells $E(1, f_1), \ldots, E(j, f_n) \tto E(1, g)$ are in bijection with 2-cells $E(j, f_1), \ldots, E(j, f_n) \tto E(j, g)$ when $j$ is dense by \cref{density-implies-j*-is-ff}, and hence with 2-cells $E(1, j), E(j, f_1), \ldots, E(j, f_n) \tto E(1, g)$, exhibiting the left-unitor of $\X[j]_\rho$ as invertible.
\end{proof}

Therefore, relative monads with dense roots may be represented as monoids in the multicategory $\X[j]_{\lambda\rho}$ whose multimorphisms $f_1, \ldots, f_n \to g$ are the 2-cells ${E(1, j), E(j, f_1), \ldots, E(j, f_n) \tto E(1, g)}$.

\begin{corollary}\label{j-dense-rmnd}
	Let $\jAE$ be a dense tight-cell. Then there is an isomorphism of categories rendering the following diagram commutative.
	\[\begin{tikzcd}[column sep=7em]
		{\Mon(\X[j]_{\lambda\rho})} & {\Mon(\X[j])} & {\RMnd(j)} \\
		& {\X[A, E]}
		\arrow["{U_{\X[j]_{\lambda\rho}}}"', from=1-1, to=2-2]
		\arrow["{U_{\X[j]}}"{description}, from=1-2, to=2-2]
		\arrow["\iso", from=1-1, to=1-2]
		\arrow["{U_j}", from=1-3, to=2-2]
		\arrow["\iso"', from=1-2, to=1-3]
		\arrow["{\textup{(\cref{relative-monads-are-tight-monoids})}}", draw=none, from=1-2, to=1-3]
	\end{tikzcd}\]
\end{corollary}

\begin{proof}
	By \cref{left-normal-if-dense}, a monoid (homomorphism) in $\X[j]_{\lambda\rho}$ is equivalently a monoid (homomorphism) in $\X[j]_\rho$, from which the result follows by \cref{monoids-in-M-rho-are-monoids-in-M}.
\end{proof}

\subsection{Relative monads as monoids in a skew-monoidal category}

To relate our characterisation of relative monads as monoids in a skew-multicategory to the characterisation of \textcite{altenkirch2010monads, altenkirch2015monads} of relative monads as monoids in a skew-monoidal category, we consider representability of the skew-multicategory $\X[j]_\rho$. The appropriate notion of representability turns out to be the \emph{left-representability} of \cite[Definition~4.4]{bourke2018skew}: in particular, when $\X[j]$ admits a tensor product satisfying a certain universal property with respect to right-normal multimorphisms, it is possible to equip the category $\X[j]_1$ with skew-monoidal structure $(\skt, j)$, such that monoids in $(\X[j]_1, \skt, j)$ are equivalent to monoids in $\X[j]$.

Since the definition of monoid in a skew-monoidal category has not yet appeared explicitly in the literature, we give the definition here.

\begin{definition}[{\cf{}~\cites[Theorem~5]{altenkirch2010monads}[Theorems~3.4 \& 3.5]{altenkirch2015monads}}]
    \label{monoid-in-skew-monoidal-category}
    Let $(\M, \skt, J, \alpha, \lambda, \rho)$ be a left-skew-monoidal category~\cite[Definition~2.1]{szlachanyi2012skew}. A \emph{monoid} in $\M$ comprises
    \begin{enumerate}
        \item an object $M \in \M$, the \emph{carrier};
        \item a morphism $m \colon M \skt M \to M$, the \emph{multiplication};
        \item a morphism $u \colon J \to M$, the \emph{unit},
    \end{enumerate}
    rendering the following diagrams commutative.
	\[
	\begin{tikzcd}
		{J \skt M} & {M \skt M} \\
		& M
		\arrow["{\lambda_M}"', from=1-1, to=2-2]
		\arrow["m", from=1-2, to=2-2]
		\arrow["{u \skt M}", from=1-1, to=1-2]
	\end{tikzcd}
	\hspace{2em}
	\begin{tikzcd}
		M & {M \skt J} \\
		M & {M \skt M}
		\arrow[Rightarrow, no head, from=1-1, to=2-1]
		\arrow["{M \skt u}", from=1-2, to=2-2]
		\arrow["{\rho_M}", from=1-1, to=1-2]
		\arrow["m", from=2-2, to=2-1]
	\end{tikzcd}
	\hspace{2em}
		\begin{tikzcd}[row sep=1.5em, column sep=small]
		{(M \skt M) \skt M} && {M \skt (M \skt M)} \\
		{M \skt M} && {M \skt M} \\
		& M
		\arrow["{\alpha_{M, M, M}}", from=1-1, to=1-3]
		\arrow["{M \skt m}", from=1-3, to=2-3]
		\arrow["{m \skt M}"', from=1-1, to=2-1]
		\arrow["m", from=2-3, to=3-2]
		\arrow["m"', from=2-1, to=3-2]
	\end{tikzcd}\]
    A \emph{monoid homomorphism} from $(M, u, m)$ to $(M', u', m')$ is a morphism $f \colon M \to M'$ rendering the following diagrams commutative.
	\[
	\begin{tikzcd}
		& J \\
		M && M
		\arrow["f"', from=2-1, to=2-3]
		\arrow["u"', from=1-2, to=2-1]
		\arrow["{u'}", from=1-2, to=2-3]
	\end{tikzcd}
	\hspace{4em}
	\begin{tikzcd}
		{M \skt M} & {M' \skt M'} \\
		M & {M'}
		\arrow["{m'}", from=1-2, to=2-2]
		\arrow["f"', from=2-1, to=2-2]
		\arrow["m"', from=1-1, to=2-1]
		\arrow["{f \skt f}", from=1-1, to=1-2]
	\end{tikzcd}
	\]
    Monoids in $\M$ and their homomorphisms form a category $\Mon(\M)$ functorial in $\M$. Denote by $U_{\M} \colon \Mon(\M) \to \M$ the faithful functor sending each monoid $(M, m, u)$ to its carrier $M$.
\end{definition}

\begin{theorem}
	\label{Xj1-is-skew-monoidal}
	Suppose that $\X$ admits left extensions of tight-cells $A \to E$ along a tight-cell $\jAE$. Then the category $\X[j]_1$ is equipped with left-skew-monoidal structure for which there is an isomorphism of categories rendering the following diagram commutative.
	\[\begin{tikzcd}[column sep=7em]
		{\Mon(\X[j]_1)} & {\Mon(\X[j])} & {\RMnd(j)} \\
		& {\X[A, E]}
		\arrow["{U_{\X[j]_1}}"', from=1-1, to=2-2]
		\arrow["{U_{\X[j]}}"{description}, from=1-2, to=2-2]
		\arrow["\iso", from=1-1, to=1-2]
		\arrow["{U_j}", from=1-3, to=2-2]
		\arrow["\iso"', from=1-2, to=1-3]
		\arrow["{\textup{(\cref{relative-monads-are-tight-monoids})}}", draw=none, from=1-2, to=1-3]
	\end{tikzcd}\]
	Furthermore, $\X[j]_1$ is
	\begin{enumerate}
        \item associative-normal if every such left extension $j \plx f \colon E \to E$ is $j$-absolute;
        \item left-normal if $j$ is dense;
        \item right-normal if $j$ is \ff{}.
    \end{enumerate}
\end{theorem}

\begin{proof}
	By the universal property of the left extension $j \plx f$, there is a 2-cell $E(1, f), E(j, 1) \tto E(1, j \plx f)$, and hence a 2-cell $E(1, f), E(j, 1), E(1, g) \tto E(1, (j \plx f) g)$ for all tight-cells $f, g \colon A \to E$.
    2-cells $g \d (j \plx f) \tto h$ are in bijection with 2-cells $E(1, f), E(j, 1), E(1, g) \tto E(1, h)$ by \cref{hom-cell-via-left-extension}.
	Thus, $(j \plx {-}) \c \ph$ together with the unit $E(1, j)$ exhibits $\X[j]_\rho$ as left-representable in the sense of \cite[Definition~4.4]{bourke2018skew}. Thus, by \cite[Theorem~6.1]{bourke2018skew}, $(\X[j]_\rho)_1 = \X[j]_1$ is left-skew-monoidal.

	The data of a monoid (homomorphism) in $\X[j]_1$ coincides with the data of a monoid (homomorphism) in $\X[j]$ by the above; that the laws coincide follows from the definitions of the structural transformations in the left-skew-monoidal category induced by an associative-normal left-skew-multicategory, observing that the laws in \cref{monoid-in-skew-monoidal-category} are precisely the internalisations of the laws in \cref{monoid-in-skew-multicategory}.

	Furthermore,
	\begin{enumerate}
		\item the associator $(f \skt g) \skt h \to f \skt (g \skt h)$ is given by the canonical 2-cell $h \d (j \plx (g \d (j \plx f))) \tto (h \d (j \plx g)) \d (j \plx f)$ induced by precomposition of $j \plx (g \d (j \plx f)) \tto (j \plx g) \d (j \plx f)$ by $h$, which is hence invertible if left extensions along $j$ are $j$-absolute, since $j \plx g$ preserves $j$-absolute colimits by \cref{j-plx-preserves-j-absolute-colimits};
		\item invertibility of the left-unitor follows from \cref{left-normal-if-dense}, since in this case $\X[j]_\rho$ is a multicategory, and left-representable multicategories are left-normal by \cite[Theorem~6.3]{bourke2018skew};
		\item the right-unitor $f \to f \skt j$ is given by the canonical 2-cell $f \tto j \d (j \plx f)$, which is hence invertible if $j$ is \ff{} by \cref{ff-implies-lx-unit-is-invertible}. \qedhere
	\end{enumerate}
\end{proof}

\begin{remark}
    From \cref{Xj1-is-skew-monoidal}, we recover \cites[Theorem~4]{altenkirch2010monads}[Theorems~3.1]{altenkirch2015monads} regarding skew-monoidality of $\Cat[A, E]$ given a well-behaved functor $\jAE$, \cites[Theorem~6]{altenkirch2010monads}[Theorem~4.4]{altenkirch2015monads} regarding sufficient conditions for monoidality, and \cite[Example 3.6]{uustalu2020eilenberg} regarding sufficient conditions for normality; and in conjunction with \cref{relative-monads-are-tight-monoids} recover \cites[Theorem~5]{altenkirch2010monads}[Theorems~3.4 \& 3.5]{altenkirch2015monads} regarding the equivalence between $j$-monads and monoids in $\Cat[A, E]$.

	Note that \cite[Example 3.6]{uustalu2020eilenberg} states that the sufficient conditions for normality are also necessary. However, this is not true. For example, for a counterexample for right-normality, consider the following diagram of categories.
	\[\begin{tikzcd}
		& 1 \\
		{\{ 0 \rightrightarrows 1 \}} && 1
		\arrow["\unit", from=2-1, to=1-2]
		\arrow["{1_1}", from=1-2, to=2-3]
		\arrow["\unit"', from=2-1, to=2-3]
	\end{tikzcd}\]
	The unique (identity) 2-cell $\unit \tto \unit \d 1_1$ exhibits $1_1$ as the (pointwise) left extension $\unit \plx \unit$. Therefore, the skew-monoidal category $\Cat[\unit]$ is right-normal. However $\unit$ is not \ff{}.
\end{remark}

\begin{remark}
	The conditions of \cref{Xj1-is-skew-monoidal} correspond to the \emph{well-behavedness} conditions of \cites[Definition~4]{altenkirch2010monads}[Definition~4.1]{altenkirch2010monads}, and the \emph{eleuthericity} conditions of \cite[\S7.3]{lucyshyn2016enriched}. We observe in passing that these conditions essentially characterise cocompletions under classes of weights, as observed by \textcite[\S8]{szlachanyi2017tensor} for well-behavedness, and by \textcite[Theorem~7.8]{lucyshyn2016enriched} for eleuthericity. We shall prove in future work that this holds more generally in our formal setting, and thereby deduce that relative monads in such cases are equivalent to monads preserving classes of colimits.
\end{remark}

\section{Relative adjunctions}
\label{relative-adjunctions}

The study of monads is inseparable from the study of adjunctions. An adjunction is the structure obtained by splitting the underlying endo-1-cell $A \xto t A$ of a monad into a composable pair of 1-cells $A \xto\ell C \xto r A$ in such a way that the monad structure may be recovered from corresponding structure on $(\ell, r)$. It is often convenient to study properties of monads in terms of the adjunctions that induce them: in this way, an adjunction acts as a notion of presentation for a monad. In this section, we examine the concept of relative adjunction and show that it behaves in many ways analogously to the non-relative concept, though there are subtleties in the theory not present in the non-relative setting.

\begin{definition}
    \label{relative-adjunction}
    Let $\X$ be an equipment. A \emph{relative adjunction} in $\X$ comprises
    \[\begin{tikzcd}
    	& C \\
    	A && E
    	\arrow[""{name=0, anchor=center, inner sep=0}, "\ell"{pos=0.4}, from=2-1, to=1-2]
    	\arrow[""{name=1, anchor=center, inner sep=0}, "r"{pos=0.6}, from=1-2, to=2-3]
    	\arrow["j"', from=2-1, to=2-3]
    	\arrow["\dashv"{anchor=center}, shift right=2, draw=none, from=0, to=1]
    \end{tikzcd}\]
    \begin{enumerate}
        \item a tight-cell $\jAE$, the \emph{root};
        \item a tight-cell $\ell \colon A \to C$, the \emph{left (relative) adjoint};
        \item a tight-cell $r \colon C \to E$, the \emph{right (relative) adjoint};
        \item an isomorphism $\sharp \colon C(\ell, 1) \iso E(j, r) \cocolon \flat$, the \emph{(left- and right-) transposition operators}.
    \end{enumerate}
    
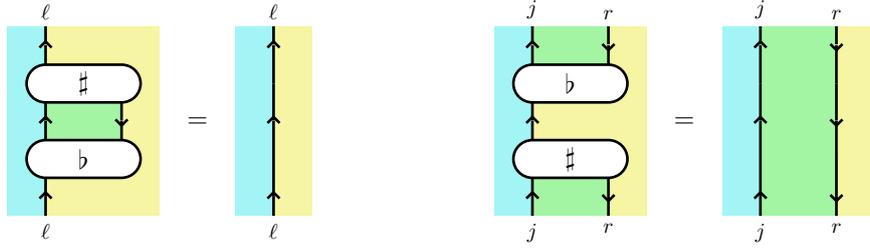
\begin{tangleeqs*}
	\begin{tangle}{(2,4)}[trim y]
		\tgBorderA{(0,0)}{\tgColour6}{\tgColour2}{\tgColour2}{\tgColour6}
		\tgBlank{(1,0)}{\tgColour2}
		\tgBorderA{(0,1)}{\tgColour6}{\tgColour2}{\tgColour4}{\tgColour6}
		\tgBorderA{(1,1)}{\tgColour2}{\tgColour2}{\tgColour2}{\tgColour4}
		\tgBorderA{(0,2)}{\tgColour6}{\tgColour4}{\tgColour2}{\tgColour6}
		\tgBorderA{(1,2)}{\tgColour4}{\tgColour2}{\tgColour2}{\tgColour2}
		\tgBorderA{(0,3)}{\tgColour6}{\tgColour2}{\tgColour2}{\tgColour6}
		\tgBlank{(1,3)}{\tgColour2}
		\tgArrow{(0,0.5)}{1}
		\tgArrow{(0,2.5)}{1}
		\tgCell[(1,0)]{(0.5,1)}{\sharp}
		\tgCell[(1,0)]{(0.5,2)}{\flat}
		\tgArrow{(0,1.5)}{1}
		\tgArrow{(1,1.5)}{3}
		\tgAxisLabel{(0.5,0.75)}{south}{\ell}
		\tgAxisLabel{(0.5,3.25)}{north}{\ell}
	\end{tangle}
    \=
	\begin{tangle}{(1,4)}[trim y]
		\tgBorderA{(0,0)}{\tgColour6}{\tgColour2}{\tgColour2}{\tgColour6}
		\tgBorderA{(0,1)}{\tgColour6}{\tgColour2}{\tgColour2}{\tgColour6}
		\tgBorderA{(0,2)}{\tgColour6}{\tgColour2}{\tgColour2}{\tgColour6}
		\tgBorderA{(0,3)}{\tgColour6}{\tgColour2}{\tgColour2}{\tgColour6}
		\tgArrow{(0,0.5)}{1}
		\tgArrow{(0,2.5)}{1}
		\tgArrow{(0,1.5)}{1}
		\tgAxisLabel{(0.5,0.75)}{south}{\ell}
		\tgAxisLabel{(0.5,3.25)}{north}{\ell}
	\end{tangle}
    \hspace{4em}
	\begin{tangle}{(2,4)}[trim y]
		\tgBorderA{(0,0)}{\tgColour6}{\tgColour4}{\tgColour4}{\tgColour6}
		\tgBorderA{(1,0)}{\tgColour4}{\tgColour2}{\tgColour2}{\tgColour4}
		\tgBorderA{(0,1)}{\tgColour6}{\tgColour4}{\tgColour2}{\tgColour6}
		\tgBorderA{(1,1)}{\tgColour4}{\tgColour2}{\tgColour2}{\tgColour2}
		\tgBorderA{(0,2)}{\tgColour6}{\tgColour2}{\tgColour4}{\tgColour6}
		\tgBorderA{(1,2)}{\tgColour2}{\tgColour2}{\tgColour2}{\tgColour4}
		\tgBorderA{(0,3)}{\tgColour6}{\tgColour4}{\tgColour4}{\tgColour6}
		\tgBorderA{(1,3)}{\tgColour4}{\tgColour2}{\tgColour2}{\tgColour4}
		\tgArrow{(0,1.5)}{1}
		\tgCell[(1,0)]{(0.5,1)}{\flat}
		\tgCell[(1,0)]{(0.5,2)}{\sharp}
		\tgArrow{(0,0.5)}{1}
		\tgArrow{(0,2.5)}{1}
		\tgArrow{(1,0.5)}{3}
		\tgArrow{(1,2.5)}{3}
		\tgAxisLabel{(0.5,0.75)}{south}{j}
		\tgAxisLabel{(1.5,0.75)}{south}{r}
		\tgAxisLabel{(0.5,3.25)}{north}{j}
		\tgAxisLabel{(1.5,3.25)}{north}{r}
	\end{tangle}
    \=
	\begin{tangle}{(2,4)}[trim y]
		\tgBorderA{(0,0)}{\tgColour6}{\tgColour4}{\tgColour4}{\tgColour6}
		\tgBorderA{(1,0)}{\tgColour4}{\tgColour2}{\tgColour2}{\tgColour4}
		\tgBorderA{(0,1)}{\tgColour6}{\tgColour4}{\tgColour4}{\tgColour6}
		\tgBorderA{(1,1)}{\tgColour4}{\tgColour2}{\tgColour2}{\tgColour4}
		\tgBorderA{(0,2)}{\tgColour6}{\tgColour4}{\tgColour4}{\tgColour6}
		\tgBorderA{(1,2)}{\tgColour4}{\tgColour2}{\tgColour2}{\tgColour4}
		\tgBorderA{(0,3)}{\tgColour6}{\tgColour4}{\tgColour4}{\tgColour6}
		\tgBorderA{(1,3)}{\tgColour4}{\tgColour2}{\tgColour2}{\tgColour4}
		\tgArrow{(0,1.5)}{1}
		\tgArrow{(0,0.5)}{1}
		\tgArrow{(0,2.5)}{1}
		\tgArrow{(1,0.5)}{3}
		\tgArrow{(1,2.5)}{3}
		\tgArrow{(1,1.5)}{3}
		\tgAxisLabel{(0.5,0.75)}{south}{j}
		\tgAxisLabel{(1.5,0.75)}{south}{r}
		\tgAxisLabel{(0.5,3.25)}{north}{j}
		\tgAxisLabel{(1.5,3.25)}{north}{r}
	\end{tangle}
	\end{tangleeqs*}
    We denote by $\ljr$ such data\footnotemark{} (by convention leaving the transposition operators implicit), and call $C$ the \emph{apex}. A \emph{$j$-relative adjunction} (alternatively \emph{adjunction relative to $j$}, or simply \emph{$j$-adjunction}) is a relative adjunction with root $j$.
    \footnotetext{Neither \textcite{ulmer1968properties} nor \textcite{altenkirch2010monads,altenkirch2015monads} introduced notation for relative adjunctions. We follow the convention of \cite[359]{street1978yoneda}, which places the root on the same side of $\adj$ as the right adjoint. However, we note that the alternative convention $\ell \mathrel{{}_j\!\!\adj} r$, which was introduced in \cite[\S0.2]{blackwell1976some} (\cf{}~\cite[\S1]{szigeti1983limits}), has also been used in recent work (\cf{}~\cite{staton2020classical,lobbia2023distributive}).}
\end{definition}

\begin{remark}
    Our definition of relative adjunction coincides with that of \cite[\S3]{wood1982abstract} in a representable equipment.
\end{remark}

\begin{remark}
    Relative adjunctions whose roots are \ff{} are sometimes called \emph{partial adjunctions}~(\eg{} in \cite[\S1.11]{kelly1982basic}), as in this case we may view the left adjoint $\ell \colon A \to C$ as being a \emph{partial morphism} from $E$ to $C$.
\end{remark}

\begin{example}
	\label{trivial-relative-adjunctions}
    Consider tight-cells $\ell, \jAE$. We have that $\ell \jadj 1_E$ if and only if $\ell \iso j$ (by \ffness{} of restriction); and that $1_A \jadj j$ if and only if $j$ if \ff{} (by
	\cref{ff-iff-invertible}).
\end{example}

There are several equivalent formulations of adjunctions~\cite[Theorem~IV.1.2]{maclane1998categories}, for which analogues exist for relative adjunctions. A subtlety is that the definition of counit is not immediately evident for relative adjunctions. We may resolve this difficulty via the techniques of \cref{skew-multicategorical-hom-categories}, using the loose-cell $E(j, 1) \colon E \lto A$ to facilitate composition of tight-cells $r \colon C \to E$ and $\ell \colon A \to C$.

\begin{lemma}
    \label{reformulations-of-relative-adjunction}
    Let $\jAE$, $\ell \colon A \to C$, and $r \colon C \to E$ be tight-cells. The following data are equivalent, exhibiting a relative adjunction $\ljr$.
    \begin{enumerate}
        \item (Hom isomorphism) An isomorphism $\sharp \colon C(\ell, 1) \iso E(j, r) \cocolon \flat$.
        \item (Universal arrow) A 2-cell $\eta \colon j \tto \ell \d r$, the \emph{unit}, and a 2-cell $\flat \colon E(j, r) \tto C(\ell, 1)$ rendering the following diagrams commutative.
		\[
    \begin{tikzcd}
      {A(1, 1)} & {E(j, j)} \\
      {C(\ell, \ell)} & {E(j, r\ell)}
      \arrow["{\flat(1, \ell)}", from=2-2, to=2-1]
      \arrow["{\pc{\ell}}"', from=1-1, to=2-1]
      \arrow["{E(j, \eta)}", from=1-2, to=2-2]
      \arrow["{\pc j}", from=1-1, to=1-2]
    \end{tikzcd}
		\hspace{4em}
    \begin{tikzcd}
      {E(j, r)} & {C(\ell, 1)} \\
      {E(j, r)} & {E(r\ell, r)}
      \arrow[Rightarrow, no head, from=1-1, to=2-1]
      \arrow["\flat", from=1-1, to=1-2]
      \arrow["{\pc r(\ell, 1)}", from=1-2, to=2-2]
      \arrow["{E(\eta, r)}", from=2-2, to=2-1]
    \end{tikzcd}
		\]
        \item (Unit--counit) A 2-cell $\eta \colon j \tto \ell \d r$, the \emph{unit}, and a 2-cell $\varepsilon \colon C(1, \ell), E(j, r) \tto C(1, 1)$, the \emph{counit}, satisfying the following equations.
        \[
        \begin{tikzcd}[column sep=huge]
        	C & A & A \\
        	C & A & A \\
        	C & A & A \\
        	C && A
        	\arrow[""{name=0, anchor=center, inner sep=0}, Rightarrow, no head, from=1-1, to=2-1]
        	\arrow["{C(1, \ell)}"', "\shortmid"{marking}, from=1-2, to=1-1]
        	\arrow[""{name=1, anchor=center, inner sep=0}, Rightarrow, no head, from=1-2, to=2-2]
        	\arrow[Rightarrow, no head, from=1-3, to=1-2]
        	\arrow[""{name=2, anchor=center, inner sep=0}, Rightarrow, no head, from=1-3, to=2-3]
        	\arrow[""{name=3, anchor=center, inner sep=0}, Rightarrow, no head, from=2-1, to=3-1]
        	\arrow["{C(1, \ell)}"{description}, from=2-2, to=2-1]
        	\arrow[""{name=4, anchor=center, inner sep=0}, Rightarrow, no head, from=2-2, to=3-2]
        	\arrow["{E(j, j)}"{description}, from=2-3, to=2-2]
        	\arrow[""{name=5, anchor=center, inner sep=0}, Rightarrow, no head, from=2-3, to=3-3]
        	\arrow[""{name=6, anchor=center, inner sep=0}, Rightarrow, no head, from=3-1, to=4-1]
        	\arrow["{C(1, \ell)}"{description}, from=3-2, to=3-1]
        	\arrow["{E(j, r \ell)}"{description}, from=3-3, to=3-2]
        	\arrow[""{name=7, anchor=center, inner sep=0}, Rightarrow, no head, from=3-3, to=4-3]
        	\arrow["{C(1, \ell)}", "\shortmid"{marking}, from=4-3, to=4-1]
        	\arrow["{=}"{description}, draw=none, from=1, to=0]
        	\arrow["{\pc j}"{description}, draw=none, from=2, to=1]
        	\arrow["{=}"{description}, draw=none, from=4, to=3]
        	\arrow["{E(j, \eta)}"{description}, draw=none, from=5, to=4]
        	\arrow["{\varepsilon(1, \ell)}"{description}, draw=none, from=7, to=6]
        \end{tikzcd}
        \quad = \quad
		\begin{tikzcd}[column sep=2.25em]
			C & A \\
			C & A
			\arrow["{C(1, \ell)}", "\shortmid"{marking}, from=2-2, to=2-1]
			\arrow[""{name=0, anchor=center, inner sep=0}, Rightarrow, no head, from=1-2, to=2-2]
			\arrow["{C(1, \ell)}"', "\shortmid"{marking}, from=1-2, to=1-1]
			\arrow[""{name=1, anchor=center, inner sep=0}, Rightarrow, no head, from=1-1, to=2-1]
			\arrow["{=}"{description}, draw=none, from=0, to=1]
		\end{tikzcd}
        \]
        \[
		\begin{tikzcd}[column sep=large]
			E & A & E & C \\
			E & A & E & C \\
			E &&& C
			\arrow["{E(1, j)}"', "\shortmid"{marking}, from=1-2, to=1-1]
			\arrow["{E(j, 1)}"', "\shortmid"{marking}, from=1-3, to=1-2]
			\arrow["{E(1, r)}"', "\shortmid"{marking}, from=1-4, to=1-3]
			\arrow["{E(1, r \ell)}"{description}, from=2-2, to=2-1]
			\arrow[""{name=0, anchor=center, inner sep=0}, Rightarrow, no head, from=1-4, to=2-4]
			\arrow[""{name=1, anchor=center, inner sep=0}, Rightarrow, no head, from=1-2, to=2-2]
			\arrow[""{name=2, anchor=center, inner sep=0}, Rightarrow, no head, from=1-1, to=2-1]
			\arrow["{E(1, r)}", "\shortmid"{marking}, from=3-4, to=3-1]
			\arrow[""{name=3, anchor=center, inner sep=0}, Rightarrow, no head, from=2-4, to=3-4]
			\arrow[""{name=4, anchor=center, inner sep=0}, Rightarrow, no head, from=2-1, to=3-1]
			\arrow["{E(1, r)}"{description}, from=2-4, to=2-3]
			\arrow["{E(j, 1)}"{description}, from=2-3, to=2-2]
			\arrow[""{name=5, anchor=center, inner sep=0}, Rightarrow, no head, from=1-3, to=2-3]
			\arrow["{E(1, \eta)}"{description}, draw=none, from=1, to=2]
			\arrow["{E(1, r), \varepsilon}"{description}, draw=none, from=3, to=4]
			\arrow["{=}"{description}, draw=none, from=0, to=5]
			\arrow["{=}"{description}, draw=none, from=5, to=1]
		\end{tikzcd}
        \quad = \quad
		\begin{tikzcd}[column sep=large]
			E & A & E & C \\
			E &&& C
			\arrow["{E(1, j)}"', "\shortmid"{marking}, from=1-2, to=1-1]
			\arrow["{E(j, 1)}"', "\shortmid"{marking}, from=1-3, to=1-2]
			\arrow["{E(1, r)}"', "\shortmid"{marking}, from=1-4, to=1-3]
			\arrow["{E(1, r)}", "\shortmid"{marking}, from=2-4, to=2-1]
			\arrow[""{name=0, anchor=center, inner sep=0}, Rightarrow, no head, from=1-4, to=2-4]
			\arrow[""{name=1, anchor=center, inner sep=0}, Rightarrow, no head, from=1-1, to=2-1]
			\arrow["{\cp j(1, r)}"{description}, draw=none, from=0, to=1]
		\end{tikzcd}
		\]
        \item (Couniversal arrow) A 2-cell $\sharp \colon C(\ell, 1) \tto E(j, r)$, and a 2-cell $\varepsilon \colon C(1, \ell), E(j, r) \tto C(1, 1)$, the \emph{counit}, satisfying the following equations.
        \[
		\begin{tikzcd}[column sep=large]
			C & A & C \\
			C & A & C \\
			C && C
			\arrow["{C(\ell, 1)}"', "\shortmid"{marking}, from=1-3, to=1-2]
			\arrow["{C(1, \ell)}"', "\shortmid"{marking}, from=1-2, to=1-1]
			\arrow["{E(j, r)}"{description}, from=2-3, to=2-2]
			\arrow["{C(1, \ell)}"{description}, from=2-2, to=2-1]
			\arrow["\shortmid"{marking}, Rightarrow, no head, from=3-3, to=3-1]
			\arrow[""{name=0, anchor=center, inner sep=0}, Rightarrow, no head, from=1-3, to=2-3]
			\arrow[""{name=1, anchor=center, inner sep=0}, Rightarrow, no head, from=2-3, to=3-3]
			\arrow[""{name=2, anchor=center, inner sep=0}, Rightarrow, no head, from=1-1, to=2-1]
			\arrow[""{name=3, anchor=center, inner sep=0}, Rightarrow, no head, from=2-1, to=3-1]
			\arrow[""{name=4, anchor=center, inner sep=0}, Rightarrow, no head, from=1-2, to=2-2]
			\arrow["{=}"{description}, draw=none, from=4, to=2]
			\arrow["\sharp"{description}, draw=none, from=0, to=4]
			\arrow["\varepsilon"{description}, draw=none, from=1, to=3]
		\end{tikzcd}
		\,=\,
		\begin{tikzcd}[column sep=1.5em]
			C & A & C \\
			C && C
			\arrow["{C(\ell, 1)}"', "\shortmid"{marking}, from=1-3, to=1-2]
			\arrow["{C(1, \ell)}"', "\shortmid"{marking}, from=1-2, to=1-1]
			\arrow["\shortmid"{marking}, Rightarrow, no head, from=2-3, to=2-1]
			\arrow[""{name=0, anchor=center, inner sep=0}, Rightarrow, no head, from=1-3, to=2-3]
			\arrow[""{name=1, anchor=center, inner sep=0}, Rightarrow, no head, from=1-1, to=2-1]
			\arrow["\cp\ell"{description}, draw=none, from=0, to=1]
		\end{tikzcd}
		\quad
		\begin{tikzcd}[column sep=large]
			A & A & C \\
			A & A & C \\
			A && C
			\arrow["{E(j, r)}"', "\shortmid"{marking}, from=1-3, to=1-2]
			\arrow[Rightarrow, no head, from=1-2, to=1-1]
			\arrow[""{name=0, anchor=center, inner sep=0}, Rightarrow, no head, from=1-3, to=2-3]
			\arrow[""{name=1, anchor=center, inner sep=0}, Rightarrow, no head, from=1-1, to=2-1]
			\arrow[""{name=2, anchor=center, inner sep=0}, Rightarrow, no head, from=1-2, to=2-2]
			\arrow["{E(j, r)}"{description}, from=2-3, to=2-2]
			\arrow["{C(\ell, \ell)}"{description}, from=2-2, to=2-1]
			\arrow["{E(j, r)}", "\shortmid"{marking}, from=3-3, to=3-1]
			\arrow[""{name=3, anchor=center, inner sep=0}, Rightarrow, no head, from=2-3, to=3-3]
			\arrow[""{name=4, anchor=center, inner sep=0}, Rightarrow, no head, from=2-1, to=3-1]
			\arrow["{=}"{description}, draw=none, from=0, to=2]
			\arrow["\pc\ell"{description}, draw=none, from=2, to=1]
			\arrow["{\sharp, \varepsilon}"{description}, draw=none, from=3, to=4]
		\end{tikzcd}
		\,=\,
		\begin{tikzcd}[column sep=1.5em]
			A & C \\
			A & C
			\arrow["{E(j, r)}"', "\shortmid"{marking}, from=1-2, to=1-1]
			\arrow[""{name=0, anchor=center, inner sep=0}, Rightarrow, no head, from=1-2, to=2-2]
			\arrow[""{name=1, anchor=center, inner sep=0}, Rightarrow, no head, from=1-1, to=2-1]
			\arrow["{E(j, r)}", "\shortmid"{marking}, from=2-2, to=2-1]
			\arrow["{=}"{description}, draw=none, from=0, to=1]
		\end{tikzcd}
		\]
        \item A $j$-representation of $C(\ell, 1)$ by $r$.
		\item A corepresentation of $E(j, r)$ by $\ell$.
		\item \label{adjunction-via-loose-adjunction} A loose-adjunction $C(1, \ell) \adj E(j, r)$.
    \end{enumerate}
\end{lemma}

\begin{proof}
    Given a 2-cell $\sharp \colon C(\ell, 1) \tto E(j, r)$, we define a unit $\eta \colon j \tto \ell \d r$ by the 2-cell on the left below; and given a 2-cell $\flat \colon E(j, r) \tto C(\ell, 1)$, we define a counit $\varepsilon \colon C(1, \ell), E(j, r) \tto C(1, 1)$ by the 2-cell on the right below.
    \begin{align*}
        \eta & \defeq
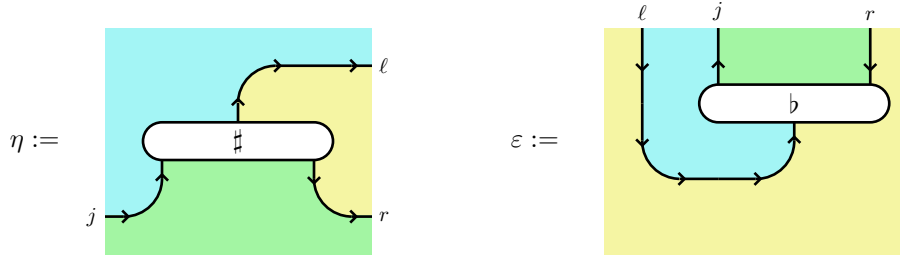
			\begin{tangle}{(5,3)}[trim x]
				\tgBlank{(0,0)}{\tgColour6}
				\tgBlank{(1,0)}{\tgColour6}
				\tgBorderC{(2,0)}{3}{\tgColour6}{\tgColour2}
				\tgBorderA{(3,0)}{\tgColour6}{\tgColour6}{\tgColour2}{\tgColour2}
				\tgBorderA{(4,0)}{\tgColour6}{\tgColour6}{\tgColour2}{\tgColour2}
				\tgBlank{(0,1)}{\tgColour6}
				\tgBorderA{(1,1)}{\tgColour6}{\tgColour6}{\tgColour4}{\tgColour6}
				\tgBorderA{(2,1)}{\tgColour6}{\tgColour2}{\tgColour4}{\tgColour4}
				\tgBorderA{(3,1)}{\tgColour2}{\tgColour2}{\tgColour2}{\tgColour4}
				\tgBlank{(4,1)}{\tgColour2}
				\tgBorderA{(0,2)}{\tgColour6}{\tgColour6}{\tgColour4}{\tgColour4}
				\tgBorderC{(1,2)}{1}{\tgColour4}{\tgColour6}
				\tgBlank{(2,2)}{\tgColour4}
				\tgBorderC{(3,2)}{0}{\tgColour4}{\tgColour2}
				\tgBorderA{(4,2)}{\tgColour2}{\tgColour2}{\tgColour4}{\tgColour4}
				\tgCell[(2,0)]{(2,1)}{\sharp}
				\tgArrow{(2,0.5)}{1}
				\tgArrow{(2.5,0)}{0}
				\tgArrow{(3.5,0)}{0}
				\tgArrow{(3,1.5)}{3}
				\tgArrow{(3.5,2)}{0}
				\tgArrow{(1,1.5)}{1}
				\tgArrow{(0.5,2)}{0}
				\tgAxisLabel{(4.25,0.5)}{west}{\ell}
				\tgAxisLabel{(0.75,2.5)}{east}{j}
				\tgAxisLabel{(4.25,2.5)}{west}{r}
			\end{tangle}
            &
        \varepsilon & \defeq
			\begin{tangle}{(4,4)}[trim y=.5]
				\tgBorderA{(0,0)}{\tgColour2}{\tgColour6}{\tgColour6}{\tgColour2}
				\tgBorderA{(1,0)}{\tgColour6}{\tgColour4}{\tgColour4}{\tgColour6}
				\tgBlank{(2,0)}{\tgColour4}
				\tgBorderA{(3,0)}{\tgColour4}{\tgColour2}{\tgColour2}{\tgColour4}
				\tgBorderA{(0,1)}{\tgColour2}{\tgColour6}{\tgColour6}{\tgColour2}
				\tgBorderA{(1,1)}{\tgColour6}{\tgColour4}{\tgColour6}{\tgColour6}
				\tgBorderA{(2,1)}{\tgColour4}{\tgColour4}{\tgColour2}{\tgColour6}
				\tgBorderA{(3,1)}{\tgColour4}{\tgColour2}{\tgColour2}{\tgColour2}
				\tgBorderC{(0,2)}{0}{\tgColour2}{\tgColour6}
				\tgBorderA{(1,2)}{\tgColour6}{\tgColour6}{\tgColour2}{\tgColour2}
				\tgBorderC{(2,2)}{1}{\tgColour2}{\tgColour6}
				\tgBlank{(3,2)}{\tgColour2}
				\tgBlank{(0,3)}{\tgColour2}
				\tgBlank{(1,3)}{\tgColour2}
				\tgBlank{(2,3)}{\tgColour2}
				\tgBlank{(3,3)}{\tgColour2}
				\tgCell[(2,0)]{(2,1)}{\flat}
				\tgArrow{(1,0.5)}{1}
				\tgArrow{(3,0.5)}{3}
				\tgArrow{(2,1.5)}{1}
				\tgArrow{(1.5,2)}{0}
				\tgArrow{(0.5,2)}{0}
				\tgArrow{(0,1.5)}{3}
				\tgArrow{(0,0.5)}{3}
				\tgAxisLabel{(0.5,.5)}{south}{\ell}
				\tgAxisLabel{(1.5,.5)}{south}{j}
				\tgAxisLabel{(3.5,.5)}{south}{r}
			\end{tangle}
    \end{align*}
    Conversely, given a 2-cell $\eta \colon j \tto \ell \d r$, we define a left-transposition operator $\sharp \colon C(\ell, 1) \tto E(j, r)$ by the 2-cell on the left below; and given a 2-cell $\varepsilon \colon C(1, \ell), E(j, r) \tto C(1, 1)$, we define a right-transposition operator $\flat \colon E(j, r) \tto C(\ell, 1)$ by the 2-cell on the right below.
    \begin{align*}
        \sharp & \defeq
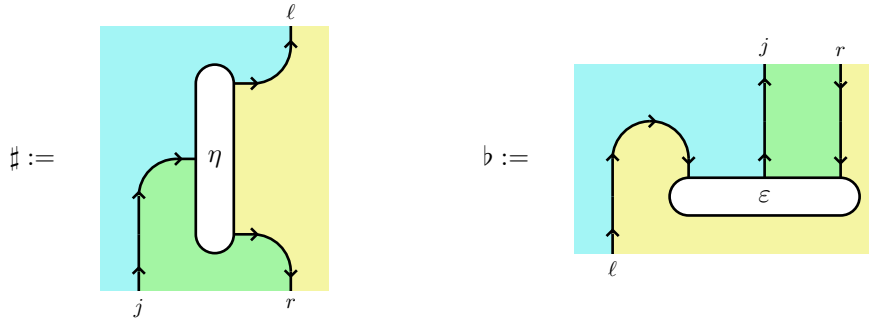
			\begin{tangle}{(3,5)}[trim y]
				\tgBlank{(0,0)}{\tgColour6}
				\tgBlank{(1,0)}{\tgColour6}
				\tgBorderA{(2,0)}{\tgColour6}{\tgColour2}{\tgColour2}{\tgColour6}
				\tgBlank{(0,1)}{\tgColour6}
				\tgBorderA{(1,1)}{\tgColour6}{\tgColour6}{\tgColour2}{\tgColour6}
				\tgBorderC{(2,1)}{1}{\tgColour2}{\tgColour6}
				\tgBorderC{(0,2)}{3}{\tgColour6}{\tgColour4}
				\tgBorderA{(1,2)}{\tgColour6}{\tgColour2}{\tgColour2}{\tgColour4}
				\tgBlank{(2,2)}{\tgColour2}
				\tgBorderA{(0,3)}{\tgColour6}{\tgColour4}{\tgColour4}{\tgColour6}
				\tgBorderA{(1,3)}{\tgColour4}{\tgColour2}{\tgColour4}{\tgColour4}
				\tgBorderC{(2,3)}{2}{\tgColour2}{\tgColour4}
				\tgBorderA{(0,4)}{\tgColour6}{\tgColour4}{\tgColour4}{\tgColour6}
				\tgBlank{(1,4)}{\tgColour4}
				\tgBorderA{(2,4)}{\tgColour4}{\tgColour2}{\tgColour2}{\tgColour4}
				\tgCell[(0,2)]{(1,2)}{\eta}
				\tgArrow{(0.5,2)}{0}
				\tgArrow{(0,2.5)}{1}
				\tgArrow{(0,3.5)}{1}
				\tgArrow{(1.5,1)}{0}
				\tgArrow{(2,0.5)}{1}
				\tgArrow{(1.5,3)}{0}
				\tgArrow{(2,3.5)}{3}
				\tgAxisLabel{(2.5,0.75)}{south}{\ell}
				\tgAxisLabel{(0.5,4.25)}{north}{j}
				\tgAxisLabel{(2.5,4.25)}{north}{r}
			\end{tangle}
            &
        \flat & \defeq
			\begin{tangle}{(4,4)}[trim y]
				\tgBlank{(0,0)}{\tgColour6}
				\tgBlank{(1,0)}{\tgColour6}
				\tgBorderA{(2,0)}{\tgColour6}{\tgColour4}{\tgColour4}{\tgColour6}
				\tgBorderA{(3,0)}{\tgColour4}{\tgColour2}{\tgColour2}{\tgColour4}
				\tgBorderC{(0,1)}{3}{\tgColour6}{\tgColour2}
				\tgBorderC{(1,1)}{2}{\tgColour6}{\tgColour2}
				\tgBorderA{(2,1)}{\tgColour6}{\tgColour4}{\tgColour4}{\tgColour6}
				\tgBorderA{(3,1)}{\tgColour4}{\tgColour2}{\tgColour2}{\tgColour4}
				\tgBorderA{(0,2)}{\tgColour6}{\tgColour2}{\tgColour2}{\tgColour6}
				\tgBorderA{(1,2)}{\tgColour2}{\tgColour6}{\tgColour2}{\tgColour2}
				\tgBorderA{(2,2)}{\tgColour6}{\tgColour4}{\tgColour2}{\tgColour2}
				\tgBorderA{(3,2)}{\tgColour4}{\tgColour2}{\tgColour2}{\tgColour2}
				\tgBorderA{(0,3)}{\tgColour6}{\tgColour2}{\tgColour2}{\tgColour6}
				\tgBlank{(1,3)}{\tgColour2}
				\tgBlank{(2,3)}{\tgColour2}
				\tgBlank{(3,3)}{\tgColour2}
				\tgCell[(2,0)]{(2,2)}{\varepsilon}
				\tgArrow{(2,1.5)}{1}
				\tgArrow{(3,1.5)}{3}
				\tgArrow{(1,1.5)}{3}
				\tgArrow{(0.5,1)}{0}
				\tgArrow{(0,1.5)}{1}
				\tgArrow{(2,0.5)}{1}
				\tgArrow{(3,0.5)}{3}
				\tgArrow{(0,2.5)}{1}
				\tgAxisLabel{(2.5,0.75)}{south}{j}
				\tgAxisLabel{(3.5,0.75)}{south}{r}
				\tgAxisLabel{(0.5,3.25)}{north}{\ell}
			\end{tangle}
    \end{align*}
	That these definitions induce a bijection between 2-cells of the form $\sharp$ and $\eta$, and $\flat$ and $\varepsilon$, follows from the zig-zag laws for restriction. That the conditions (1) -- (4) are then equivalent follows by elementary string diagrammatic reasoning.
	(5) and (6) are equivalent to (1) by definition. Finally, since $E(j, r\ell) \iso E(j, r) \odot C(1, \ell)$, by essential uniqueness of adjoints, $C(1, \ell) \adj E(j, r)$ if and only if $C(\ell, 1) \iso E(j, r)$, so that (7) is equivalent to (1).
\end{proof}

Henceforth, in the context of a relative adjunction $\ljr$, we shall use $\sharp$, $\eta$, $\flat$, and $\varepsilon$ to denote the 2-cells defined above.

When $j$ is the identity, we should anticipate that $j$-adjunctions are precisely (non-relative) adjunctions. This is indeed so.

\begin{corollary}
    Let $\ell \colon A \to C$ and $r \colon C \to A$ be tight-cells. The following are equivalent.
    \begin{enumerate}
        \item $\ell \radj{1_A} r$ (\cref{relative-adjunction}).
        \item $\ell \adj r$ (\cref{tight-adjunction}).
    \end{enumerate}
\end{corollary}

\begin{proof}
    When $j$ is the identity, condition (3) of \cref{reformulations-of-relative-adjunction} is precisely the classical $\eta$--$\varepsilon$ definition of adjunction in a 2-category.
\end{proof}

As with ordinary adjunctions, left relative adjoints are unique up to isomorphism, though, in general, it is not true that right relative adjoints are essentially unique: for instance, denoting by $0$ the empty category, every functor $r \colon C \to E$ (for arbitrary categories $C$ and $E$) is right adjoint to the unique functor $[]_C \colon 0 \to C$ relative to the unique functor $[]_E \colon 0 \to E$, but there are typically many such (non-isomorphic) functors.
\[\begin{tikzcd}
	& C \\
	0 && E
	\arrow[""{name=0, anchor=center, inner sep=0}, "{[]_C}", from=2-1, to=1-2]
	\arrow[""{name=1, anchor=center, inner sep=0}, "r", from=1-2, to=2-3]
	\arrow["{[]_E}"', from=2-1, to=2-3]
	\arrow["\dashv"{anchor=center}, shift right=2, draw=none, from=0, to=1]
\end{tikzcd}\]
However, when the root $j$ is dense, right $j$-adjoints are unique up to isomorphism. In practice, many results of interest for relative adjunctions and relative monads hold only for those with dense roots.

\begin{lemma}
    \label{uniqueness-of-relative-adjoints}
    If $\ljr$ and $\ell' \jadj r$, then $\ell \iso \ell'$. If $\ljr$ and $\ell \jadj r'$ and $j$ is dense, then $r \iso r'$.
\end{lemma}

\begin{proof}
    If $\ljr$ and $\ell' \jadj r$, then $B(\ell, 1) \iso E(j, r) \iso B(\ell', 1)$, hence $\ell \iso \ell'$. If $\ljr$ and $\ell \jadj r'$, then $E(j, r) \iso B(\ell, 1) \iso E(j, r')$, hence, if $j$ is dense, $r \iso r'$ by \cref{density-implies-j*-is-ff}.
\end{proof}

Non-relative adjoints may be computed by means of absolute lifts and extensions~\cite[Proposition~2]{street1978yoneda}. An analogous statement is true of relative adjoints, though we must replace the notion of \emph{absolute (nonpointwise) lift} in a 2-category with the notion of \emph{(pointwise) lift} in an equipment.

\begin{proposition}
    \label{left-adjoint-is-left-lift}
    Let $\jAE$ and $r \colon C \to E$ be tight-cells. A tight-cell $\ell \colon A \to C$ is left $j$-adjoint to $r$ if and only if there is a 2-cell $\eta \colon j \tto \ell \d r$ exhibiting $\ell$ as the left lift $j \plf r$ of $j$ through~$r$.
	\[\begin{tikzcd}
		& C \\
		A && E
		\arrow[""{name=0, anchor=center, inner sep=0}, "{j \plf r}", dashed, from=2-1, to=1-2]
		\arrow[""{name=1, anchor=center, inner sep=0}, "r", from=1-2, to=2-3]
		\arrow["j"', from=2-1, to=2-3]
		\arrow["\dashv"{anchor=center}, shift right=2, draw=none, from=0, to=1]
	\end{tikzcd}\]
\end{proposition}

\begin{proof}
    Immediate from \cref{pointwise-left-lift-is-hom}.
\end{proof}

\begin{remark}
	The distinction between pointwise and nonpointwise extensions is well appreciated in the categorical literature (\cf{}~\cite{dubuc1970kan}). However, the notion of \emph{pointwise lift} has not been explicitly identified in the literature. \Cref{left-adjoint-is-left-lift} provides an explanation for this seeming omission: pointwise lifts are precisely relative adjoints. (Conversely, pointwise extensions are not always relative adjoints, though often are in practice, \cf~\cref{right-adjoint-vs-absolute-left-extension}.)

	Furthermore, observe that, by \cref{pointwise-left-lift-is-absolute-left-lift}, for every relative adjunction $\ljr$, the left adjoint is the absolute left lift of $j$ through $r$ in the tight 2-category. In $\Cat$, the converse also holds: that is, $\ell$ is the left $j$-adjoint of $r$ if and only if $\ell$ is the absolute left lift of $j$ through $r$. In other words, every absolute not-necessarily-pointwise left lift is automatically pointwise in $\Cat$. However, this is not true for a general \ve{} (\cf{}~\cite[Proposition~7]{street1978yoneda}). The definition of \emph{relative adjunction} of \cite[Definition~1.4]{lobbia2023distributive}, which is equivalent to an absolute left lift~\cite[Proposition~1.6]{lobbia2023distributive}, therefore suffers from similar issues to that of relative monads \ibid{} (\cf{}~\cref{extension-systems}).
\end{remark}

\Cref{left-adjoint-is-left-lift} gives a way to compute a left relative adjoint given the right relative adjoint. We should like a converse. A subtlety is that right relative adjoints are not unique when the root is not dense.

\begin{proposition}
    \label{right-adjoint-vs-absolute-left-extension}
    Let $\jAE$ and $\ell \colon A \to C$ be tight-cells, and suppose that $j$ is \ff{}.
	\[\begin{tikzcd}
		& C \\
		A && E
		\arrow[""{name=0, anchor=center, inner sep=0}, "\ell", from=2-1, to=1-2]
		\arrow[""{name=1, anchor=center, inner sep=0}, "{\ell \plx j}", dashed, from=1-2, to=2-3]
		\arrow["j"', from=2-1, to=2-3]
		\arrow["\dashv"{anchor=center}, shift right=2, draw=none, from=0, to=1]
	\end{tikzcd}\]
    \begin{enumerate}
        \item Suppose that the left extension $\ell \plx j$ exists and is $j$-absolute. Then $\ell \jadj \ell \plx j$.
        \item Suppose that $j$ is dense and that $\ell$ has a right $j$-adjoint $r$. Then $r$ exhibits the left extension $\ell \plx j$ and this extension is $j$-absolute.
    \end{enumerate}
\end{proposition}

\begin{proof}
  First observe that, since $j$ is fully faithful, there is an isomorphism:
  \[
    C(\ell, 1) \iso E(j, j) \odot C(\ell, 1)
  \]
  For (1), $j$-absoluteness of $\ell \plx j$ implies there is an isomorphism $E(j, \ell \plx j) \iso E(j, j) \odotl C(\ell, 1)$, so we have $C(\ell, 1) \iso E(j, \ell \plx j)$ as required.
  For (2), if $\ell \jadj r$ then there is an isomorphism $E(j, j) \odot C(\ell, 1) \iso E(j, r)$, so we can conclude by applying \cref{absolute-implies-colimit}.
\end{proof}

Left relative adjoints preserve those colimits preserved by the root; while right relative adjoints preserve all limits when the root is dense.

\begin{proposition}
	\label{left-adjoints-preserve-colimits}
    Let $\jAE$ be a tight-cell. If $\ell \colon A \to C$ is a left $j$-adjoint, then $\ell$ preserves every colimit that $j$ preserves.
\end{proposition}

\begin{proof}
	Suppose that $\ell \jadj r$, let $p \colon Y \lto Z$ be a loose-cell, and let $f \colon Z \to A$ be a tight-cell admitting a $p$-colimit $p \wc f \colon Y \to A$.
	If $j$ preserves $p \wc f$, then we have the following isomorphisms.
	\begin{align*}
		C(\ell (p \wc f), 1)
		& \iso E(j (p \wc f), r) \tag{$\ljr$} \\
		& \iso E(p \wc (jf), r)  \tag{$j$ preserves $p \wc f$} \\
		& \iso E(jf, r) \rf p  \tag{\cref{colimit-and-restriction}} \\
		& \iso C(\ell f, 1) \rf p  \tag{$\ljr$}
	\end{align*}
	Hence $((p \wc f) \d \ell)$ forms the colimit $p \wc (f \d \ell)$; a simple calculation shows the universal 2-cell is the canonical one.
\end{proof}

\begin{proposition}
    \label{right-adjoints-preserve-limits}
    Let $\jAE$ be a tight-cell. If $j$ is dense and $r \colon C \to E$ is a right $j$-adjoint, then $r$ preserves limits.
\end{proposition}

\begin{proof}
  Suppose that $\ell \jadj r$, let $p \colon X \lto Y$ be a loose-cell, and let $f \colon X \to C$ be a tight-cell admitting a $p$-limit $p \wl f \colon Y \to C$.
  We have the following isomorphisms.
  \begin{align*}
    E(1, r(p \wl f))
    &\iso \tag{\cref{pointwise-extension-and-restriction}, using density of $j$}	  E(j, r(p \wl f)) \rf E(j, 1)
    \\&\iso \tag{$\ljr$}
    C(\ell, p \wl f) \rf E(j, 1)
    \\& \iso \tag{\cref{right-lift-and-restriction}}
    (p \rx C(\ell, f)) \rf E(j, 1)
    \\& \iso \tag{$\ljr$}
    (p \rx E(j, r f)) \rf E(j, 1)
    \\& \iso \tag{\cref{lifts-commute-with-extensions}}
    p \rx (E(j, r f) \rf E(j, 1))
    \\& \iso \tag{\cref{pointwise-extension-and-restriction}, using density of $j$}
    p \rx E(1, r f)
  \end{align*}
  Hence $((p \wl f) \d r)$ forms the limit $p \wl (f \d r)$; a simple calculation using the density of $j$ shows the universal 2-cell is the canonical one.
\end{proof}

\begin{remark}
	From \cref{left-adjoints-preserve-colimits,right-adjoints-preserve-limits}, we recover \cite[Theorem~2.13 \& footnote 13]{ulmer1968properties} respectively.
	As a special case, we recover \cite[Theorem~9.5.7]{riehl2022elements}, expressing that left adjoints preserve colimits (since identities trivially preserve colimits), and that right adjoints preserve limits (since identities are trivially dense).
\end{remark}

\subsection{Morphisms of relative adjunctions}

Just as relative monads are presented by relative adjunctions, so too are morphisms of relative monads presented by morphisms of relative adjunctions. In fact, there are two natural notions of morphisms of relative adjunctions, corresponding to each of the two tight-cells $\ell$ and $r$, which we term \emph{left-morphisms} and \emph{right-morphisms} respectively.

\begin{definition}
	\label{left-morphism}
    Let $\jAE$ be a tight-cell. A \emph{left-morphism} of $j$-adjunctions from $\ljr$ to $\ljrp$ comprises
    \[\begin{tikzcd}
    	& {C'} \\
    	A & C & E
    	\arrow["\ell"', from=2-1, to=2-2]
    	\arrow["r"', from=2-2, to=2-3]
    	\arrow[""{name=0, anchor=center, inner sep=0}, "{\ell'}", from=2-1, to=1-2]
    	\arrow["c"{description}, from=2-2, to=1-2]
    	\arrow["{r'}", from=1-2, to=2-3]
    	\arrow["\lambda"', shorten <=3pt, Rightarrow, from=0, to=2-2]
    \end{tikzcd}\]
    \begin{enumerate}
        \item a 1-cell $c \colon C \to C'$ such that $r = c \d r'$;
        \item a 2-cell $\lambda \colon \ell' \tto \ell \d c$,
    \end{enumerate}
    rendering the following diagram commutative.
    \[\begin{tikzcd}[column sep=large]
      {C(\ell, 1)} & {E(j, r)} \\
      {C'(c \ell, c)} & {C'(\ell', c)}
      \arrow["\sharp", from=1-1, to=1-2]
      \arrow["{\sharp'(1, c)}"', from=2-2, to=1-2]
      \arrow["{C'(\lambda, c)}"', from=2-1, to=2-2]
      \arrow["{\pc c(\ell, 1)}"', from=1-1, to=2-1]
    \end{tikzcd}\]
    It is \emph{strict} when $\lambda$ is the identity. $j$-adjunctions and their left-morphisms form a category $\RAdj_L(j)$.
\end{definition}

\begin{example}
	For any tight-cell $\jAE$ and $j$-adjunction $\ljr$, the pair $(r, \eta)$ forms a unique left-morphism from $\ljr$ to $j \jadj 1_E$ (\cref{trivial-relative-adjunctions}), exhibiting the latter as terminal in $\RAdj_L(j)$.
\end{example}

While the data of a left-morphism in \cref{left-morphism} involves both a tight-cell $c$ and a 2-cell $\lambda$, the following lemma shows that the 2-cell $\lambda$ is redundant, being uniquely determined by the tight-cell $c$. However, the analogous statement for right-morphisms is not true in general; we make the 2-cell $\lambda$ explicit for symmetry with \cref{right-morphism}.

\begin{definition}
	For each object $E$ of $\X$, denote by $\tX/E$ the category of strict slices over $E$, whose objects are tight-cells $\cdot \to E$ and whose morphisms are commutative triangles.
\end{definition}

\begin{lemma}
	\label{left-morphisms-to-slices-is-ff}
	Let $\jAE$ be a tight-cell. The functor $\RAdj_L(j) \to \tX/E$ sending each $j$-adjunction $\ljr$ to its right adjoint $r$, and sending each left-morphism $(c, \lambda)$ to its tight-cell $c$, is \ff{}.
\end{lemma}

\begin{proof}
	By pasting $\flat'$ and bending $c$, the compatibility condition for a left-morphism states that $\lambda$ is equal to the following 2-cell.
	
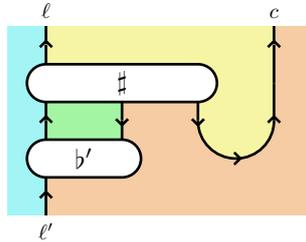
\[
	\begin{tangle}{(4,4)}[trim y]
		\tgBorderA{(0,0)}{\tgColour6}{\tgColour2}{\tgColour2}{\tgColour6}
		\tgBlank{(1,0)}{\tgColour2}
		\tgBlank{(2,0)}{\tgColour2}
		\tgBorderA{(3,0)}{\tgColour2}{\tgColour1}{\tgColour1}{\tgColour2}
		\tgBorderA{(0,1)}{\tgColour6}{\tgColour2}{\tgColour4}{\tgColour6}
		\tgBorderA{(1,1)}{\tgColour2}{\tgColour2}{\tgColour1}{\tgColour4}
		\tgBorderA{(2,1)}{\tgColour2}{\tgColour2}{\tgColour2}{\tgColour1}
		\tgBorderA{(3,1)}{\tgColour2}{\tgColour1}{\tgColour1}{\tgColour2}
		\tgBorderA{(0,2)}{\tgColour6}{\tgColour4}{\tgColour1}{\tgColour6}
		\tgBorderA{(1,2)}{\tgColour4}{\tgColour1}{\tgColour1}{\tgColour1}
		\tgBorderC{(2,2)}{0}{\tgColour1}{\tgColour2}
		\tgBorderC{(3,2)}{1}{\tgColour1}{\tgColour2}
		\tgBorderA{(0,3)}{\tgColour6}{\tgColour1}{\tgColour1}{\tgColour6}
		\tgBlank{(1,3)}{\tgColour1}
		\tgBlank{(2,3)}{\tgColour1}
		\tgBlank{(3,3)}{\tgColour1}
		\tgCell[(1,0)]{(0.5,2)}{\flat'}
		\tgArrow{(1,1.5)}{3}
		\tgArrow{(0,1.5)}{1}
		\tgArrow{(0,0.5)}{1}
		\tgArrow{(0,2.5)}{1}
		\tgCell[(2,0)]{(1,1)}{\sharp}
		\tgArrow{(2,1.5)}{3}
		\tgArrow{(2.5,2)}{0}
		\tgArrow{(3,1.5)}{1}
		\tgArrow{(3,0.5)}{1}
		\tgAxisLabel{(0.5,0.75)}{south}{\ell}
		\tgAxisLabel{(3.5,0.75)}{south}{c}
		\tgAxisLabel{(0.5,3.25)}{north}{\ell'}
	\end{tangle}
	\]
	Thus, for any tight-cell $c \colon C \to C'$ between the apices of $j$-adjunctions $\ljr$ and $\ljrp$, such that $r = c \d r'$, the 2-cell above defines a unique left-morphism $(\ljr) \to (\ljrp)$ with tight-cell $c$.
\end{proof}

\begin{definition}
    \label{right-morphism}
    Let $\jAE$ be a tight-cell in $\X$. A \emph{right-morphism} of $j$-adjunctions from $\ljr$ to $\ell' \jadj r'$ comprises
	\[\begin{tikzcd}
		& {C'} \\
		A & C & E
		\arrow["\ell"', from=2-1, to=2-2]
		\arrow["r"', from=2-2, to=2-3]
		\arrow["{\ell'}", from=2-1, to=1-2]
		\arrow["c"{description}, from=2-2, to=1-2]
		\arrow[""{name=0, anchor=center, inner sep=0}, "{r'}", from=1-2, to=2-3]
		\arrow["\rho"', shorten >=3pt, Rightarrow, from=2-2, to=0]
	\end{tikzcd}\]
    \begin{enumerate}
        \item a tight-cell $c \colon C \to C'$ such that $\ell \d c = \ell'$;
        \item a 2-cell $\rho \colon r \tto c \d r'$,
    \end{enumerate}
    rendering the following diagram commutative.
    \[\begin{tikzcd}[column sep=large]
      {C(\ell, 1)} & {E(j, r)} \\
      {C'(\ell', c)} & {E(j, r' c)}
      \arrow["\sharp", from=1-1, to=1-2]
      \arrow["{E(j, \rho)}", from=1-2, to=2-2]
      \arrow["{\sharp'(1, c)}"', from=2-1, to=2-2]
      \arrow["{\pc c(\ell, 1)}"', from=1-1, to=2-1]
    \end{tikzcd}\]
    It is \emph{strict} when $\rho$ is the identity. $j$-adjunctions and right-morphisms form a category $\RAdj_R(j)$.
\end{definition}

\begin{example}
	For any \ff{} tight-cell $\jAE$ and $j$-adjunction $\ljr$, the pair $(\ell, \eta)$ forms a unique right-morphism from $1_A \jadj j$ (\cref{trivial-relative-adjunctions}) to $\ljr$, exhibiting the former as initial in $\RAdj_R(j)$.
\end{example}

As mentioned above, the 2-cell $\rho$ in the data of a right-morphism $(c, \rho)$ is not uniquely determined by the tight-cell $c$ in general. However, in analogy with the essential uniqueness of relative adjoints (\cref{uniqueness-of-relative-adjoints}), it is uniquely determined when $j$ is dense.

\begin{definition}
	For each object $A$ of $\X$, denote by $A/\tX$ the category of strict coslices under $A$, whose objects are tight-cells $A \to \cdot$ and whose morphisms are commutative triangles.
\end{definition}

\begin{lemma}
	\label{right-morphisms-to-coslices-is-ff}
	Let $\jAE$ be a tight-cell. If $j$ is dense, then the functor $\RAdj_R(j) \to A/\tX$ sending each $j$-adjunction $\ljr$ to its left adjoint $\ell$, and sending each right-morphism $(c, \rho)$ to its tight-cell $c$, is \ff{}.
\end{lemma}

\begin{proof}
	The compatibility condition for $\rho$ equivalently states that the following square commutes.
	\[\begin{tikzcd}[column sep=huge]
		{C(\ell, 1)} & {E(j, r)} \\
		{C'(\ell', c)} & {E(j, r' c)}
		\arrow["\flat"', from=1-2, to=1-1]
		\arrow["{E(j, \rho)}", from=1-2, to=2-2]
		\arrow["{\sharp'(1, c)}"', from=2-1, to=2-2]
		\arrow["{\pc c(\ell, 1)}"', from=1-1, to=2-1]
	\end{tikzcd}\]
	Thus, for any tight-cell $c \colon C \to C'$ between the apices of $j$-adjunctions $\ljr$ and $\ljrp$ such that $\ell \d c = \ell'$, the 2-cell above uniquely determines $E(j, \rho)$.
        Hence, when $j$ is dense, it uniquely determines $\rho$ by \cref{density-implies-j*-is-ff}, thus defining a unique right-morphism $(\ljr) \to (\ljrp)$ with tight-cell $c$.
\end{proof}

The compatibility condition between $\sharp$ and $\sharp'$ in the definitions of left-morphisms and right-morphisms may be re\"expressed in terms of $\flat$, $\eta$, or $\varepsilon$ as in \cref{reformulations-of-relative-adjunction}: we leave the elementary details to the reader.

\begin{remark}
    Our definitions of left- and right-morphisms of relative adjunctions coincide with those of \cite[Definitions~5.2.20 \& 5.2.12]{arkor2022monadic} in a representable equipment, by the preceding remark.
\end{remark}

Strict morphisms of (relative) adjunctions appear more commonly in the literature (\eg{}~\cite[\S IV.7]{maclane1998categories}) than general left- and right-morphisms, and play an important role in the study of relative monads, as shall be shown in the following section.

\begin{definition}
	A \emph{strict morphism} of relative adjunctions is a strict left- (equivalently, right-) morphism of relative adjunctions. Denote by $\RAdj(j)$ the category of $j$-adjunctions and their strict morphisms.
\end{definition}

\subsection{Resolutions of relative monads}

Our motivation for introducing relative adjunctions is our interest in relative monads.
The connection between relative adjunctions and relative monads is analogous to the connection between non-relative adjunctions and non-relative monads: just as every adjunction induces a monad, every relative adjunction induces a relative monad. The converse is not necessarily true in an arbitrary equipment, though in \cref{algebras-and-opalgebras} we give sufficient conditions for a relative monad to be induced by a canonical relative adjunction, which in particular hold in $\Cat$.

\begin{theorem}
    \label{relative-adjunction-induces-relative-monad}
    Every relative adjunction $\ljr$ induces a relative monad ${\wedge_j}(\ljr)$ with underlying tight-cell $(\ell \d r)$. Furthermore, this assignment extends to functors:
	\begin{align*}
		\obslash_j & \colon \RAdj_L(j) \to \RMnd(j)\op \\
		\oslash_j & \colon \RAdj_R(j) \to \RMnd(j)
	\end{align*}
\end{theorem}

\begin{proof}
    Define $\eta$ as in \cref{reformulations-of-relative-adjunction}
    and $\dag \colon E(j, r \ell) \tto E(r \ell, r \ell)$ to be
    
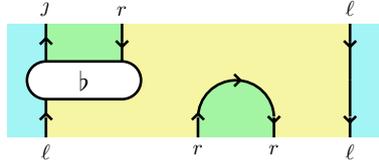
\[
	\begin{tangle}{(5,3)}[trim y]
		\tgBorderA{(0,0)}{\tgColour6}{\tgColour4}{\tgColour4}{\tgColour6}
		\tgBorderA{(1,0)}{\tgColour4}{\tgColour2}{\tgColour2}{\tgColour4}
		\tgBlank{(2,0)}{\tgColour2}
		\tgBlank{(3,0)}{\tgColour2}
		\tgBorderA{(4,0)}{\tgColour2}{\tgColour6}{\tgColour6}{\tgColour2}
		\tgBorderA{(0,1)}{\tgColour6}{\tgColour4}{\tgColour2}{\tgColour6}
		\tgBorderA{(1,1)}{\tgColour4}{\tgColour2}{\tgColour2}{\tgColour2}
		\tgBorderC{(2,1)}{3}{\tgColour2}{\tgColour4}
		\tgBorderC{(3,1)}{2}{\tgColour2}{\tgColour4}
		\tgBorderA{(4,1)}{\tgColour2}{\tgColour6}{\tgColour6}{\tgColour2}
		\tgBorderA{(0,2)}{\tgColour6}{\tgColour2}{\tgColour2}{\tgColour6}
		\tgBlank{(1,2)}{\tgColour2}
		\tgBorderA{(2,2)}{\tgColour2}{\tgColour4}{\tgColour4}{\tgColour2}
		\tgBorderA{(3,2)}{\tgColour4}{\tgColour2}{\tgColour2}{\tgColour4}
		\tgBorderA{(4,2)}{\tgColour2}{\tgColour6}{\tgColour6}{\tgColour2}
		\tgArrow{(2.5,1)}{0}
		\tgArrow{(2,1.5)}{1}
		\tgArrow{(3,1.5)}{3}
		\tgArrow{(4,1.5)}{3}
		\tgArrow{(4,0.5)}{3}
		\tgCell[(1,0)]{(0.5,1)}{\flat}
		\tgArrow{(0,1.5)}{1}
		\tgArrow{(1,0.5)}{3}
		\tgArrow{(0,0.5)}{1}
		\tgAxisLabel{(0.5,0.75)}{south}{j}
		\tgAxisLabel{(1.5,0.75)}{south}{r}
		\tgAxisLabel{(4.5,0.75)}{south}{\ell}
		\tgAxisLabel{(0.5,2.25)}{north}{\ell}
		\tgAxisLabel{(2.5,2.25)}{north}{r}
		\tgAxisLabel{(3.5,2.25)}{north}{r}
		\tgAxisLabel{(4.5,2.25)}{north}{\ell}
	\end{tangle}
    \]
    The unit laws follow from the $\sharp$--$\flat$ isomorphism. The associativity law follows by elementary string diagrammatic reasoning, in particular observing the following two identities.
	
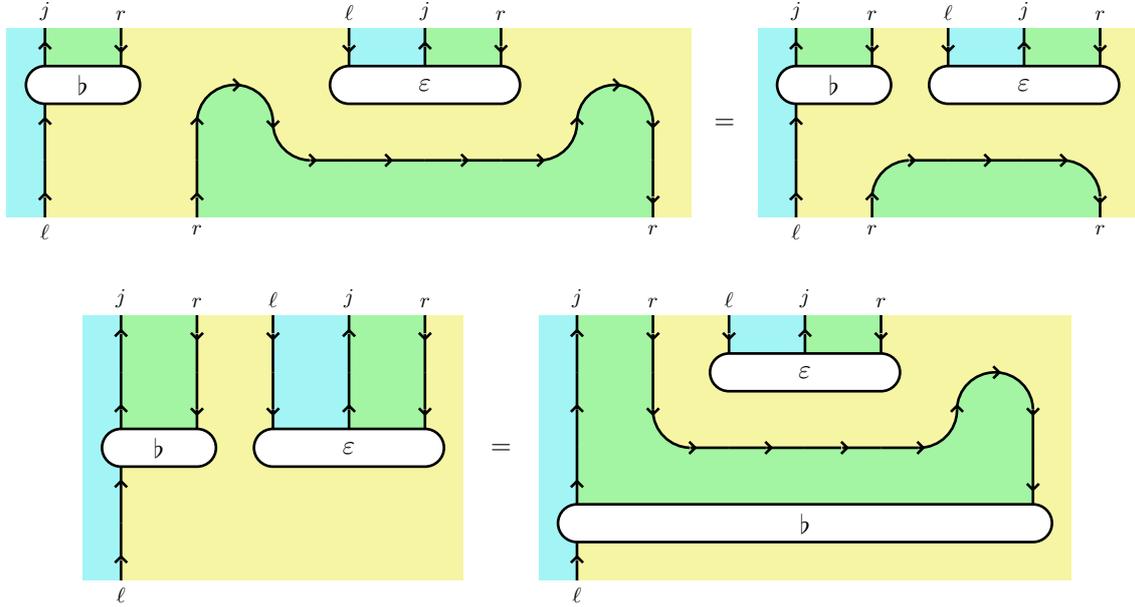
\[
	\hspace{-.5cm}
	\begin{tangle}{(9,4)}[trim y]
		\tgBorderA{(0,0)}{\tgColour6}{\tgColour4}{\tgColour4}{\tgColour6}
		\tgBorderA{(1,0)}{\tgColour4}{\tgColour2}{\tgColour2}{\tgColour4}
		\tgBlank{(2,0)}{\tgColour2}
		\tgBlank{(3,0)}{\tgColour2}
		\tgBorderA{(4,0)}{\tgColour2}{\tgColour6}{\tgColour6}{\tgColour2}
		\tgBorderA{(5,0)}{\tgColour6}{\tgColour4}{\tgColour4}{\tgColour6}
		\tgBorderA{(6,0)}{\tgColour4}{\tgColour2}{\tgColour2}{\tgColour4}
		\tgBlank{(7,0)}{\tgColour2}
		\tgBlank{(8,0)}{\tgColour2}
		\tgBorderA{(0,1)}{\tgColour6}{\tgColour4}{\tgColour2}{\tgColour6}
		\tgBorderA{(1,1)}{\tgColour4}{\tgColour2}{\tgColour2}{\tgColour2}
		\tgBorderC{(2,1)}{3}{\tgColour2}{\tgColour4}
		\tgBorderC{(3,1)}{2}{\tgColour2}{\tgColour4}
		\tgBorderA{(4,1)}{\tgColour2}{\tgColour6}{\tgColour2}{\tgColour2}
		\tgBorderA{(5,1)}{\tgColour6}{\tgColour4}{\tgColour2}{\tgColour2}
		\tgBorderA{(6,1)}{\tgColour4}{\tgColour2}{\tgColour2}{\tgColour2}
		\tgBorderC{(7,1)}{3}{\tgColour2}{\tgColour4}
		\tgBorderC{(8,1)}{2}{\tgColour2}{\tgColour4}
		\tgBorderA{(0,2)}{\tgColour6}{\tgColour2}{\tgColour2}{\tgColour6}
		\tgBlank{(1,2)}{\tgColour2}
		\tgBorderA{(2,2)}{\tgColour2}{\tgColour4}{\tgColour4}{\tgColour2}
		\tgBorderC{(3,2)}{0}{\tgColour4}{\tgColour2}
		\tgBorderA{(4,2)}{\tgColour2}{\tgColour2}{\tgColour4}{\tgColour4}
		\tgBorderA{(5,2)}{\tgColour2}{\tgColour2}{\tgColour4}{\tgColour4}
		\tgBorderA{(6,2)}{\tgColour2}{\tgColour2}{\tgColour4}{\tgColour4}
		\tgBorderC{(7,2)}{1}{\tgColour4}{\tgColour2}
		\tgBorderA{(8,2)}{\tgColour4}{\tgColour2}{\tgColour2}{\tgColour4}
		\tgBorderA{(0,3)}{\tgColour6}{\tgColour2}{\tgColour2}{\tgColour6}
		\tgBlank{(1,3)}{\tgColour2}
		\tgBorderA{(2,3)}{\tgColour2}{\tgColour4}{\tgColour4}{\tgColour2}
		\tgBlank{(3,3)}{\tgColour4}
		\tgBlank{(4,3)}{\tgColour4}
		\tgBlank{(5,3)}{\tgColour4}
		\tgBlank{(6,3)}{\tgColour4}
		\tgBlank{(7,3)}{\tgColour4}
		\tgBorderA{(8,3)}{\tgColour4}{\tgColour2}{\tgColour2}{\tgColour4}
		\tgArrow{(2.5,1)}{0}
		\tgArrow{(3,1.5)}{3}
		\tgArrow{(2,1.5)}{1}
		\tgArrow{(7,1.5)}{1}
		\tgArrow{(3.5,2)}{0}
		\tgArrow{(5.5,2)}{0}
		\tgArrow{(4.5,2)}{0}
		\tgCell[(1,0)]{(0.5,1)}{\flat}
		\tgArrow{(2,2.5)}{1}
		\tgArrow{(6.5,2)}{0}
		\tgArrow{(7.5,1)}{0}
		\tgArrow{(8,1.5)}{3}
		\tgArrow{(8,2.5)}{3}
		\tgArrow{(4,0.5)}{3}
		\tgArrow{(5,0.5)}{1}
		\tgArrow{(6,0.5)}{3}
		\tgArrow{(1,0.5)}{3}
		\tgArrow{(0,0.5)}{1}
		\tgArrow{(0,1.5)}{1}
		\tgArrow{(0,2.5)}{1}
		\tgCell[(2,0)]{(5,1)}{\varepsilon}
		\tgAxisLabel{(0.5,0.75)}{south}{j}
		\tgAxisLabel{(1.5,0.75)}{south}{r}
		\tgAxisLabel{(4.5,0.75)}{south}{\ell}
		\tgAxisLabel{(5.5,0.75)}{south}{j}
		\tgAxisLabel{(6.5,0.75)}{south}{r}
		\tgAxisLabel{(0.5,3.25)}{north}{\ell}
		\tgAxisLabel{(2.5,3.25)}{north}{r}
		\tgAxisLabel{(8.5,3.25)}{north}{r}
	\end{tangle}
	\!\tangleeq*\!
	\begin{tangle}{(5,4)}[trim y]
		\tgBorderA{(0,0)}{\tgColour6}{\tgColour4}{\tgColour4}{\tgColour6}
		\tgBorderA{(1,0)}{\tgColour4}{\tgColour2}{\tgColour2}{\tgColour4}
		\tgBorderA{(2,0)}{\tgColour2}{\tgColour6}{\tgColour6}{\tgColour2}
		\tgBorderA{(3,0)}{\tgColour6}{\tgColour4}{\tgColour4}{\tgColour6}
		\tgBorderA{(4,0)}{\tgColour4}{\tgColour2}{\tgColour2}{\tgColour4}
		\tgBorderA{(0,1)}{\tgColour6}{\tgColour4}{\tgColour2}{\tgColour6}
		\tgBorderA{(1,1)}{\tgColour4}{\tgColour2}{\tgColour2}{\tgColour2}
		\tgBorderA{(2,1)}{\tgColour2}{\tgColour6}{\tgColour2}{\tgColour2}
		\tgBorderA{(3,1)}{\tgColour6}{\tgColour4}{\tgColour2}{\tgColour2}
		\tgBorderA{(4,1)}{\tgColour4}{\tgColour2}{\tgColour2}{\tgColour2}
		\tgBorderA{(0,2)}{\tgColour6}{\tgColour2}{\tgColour2}{\tgColour6}
		\tgBorderC{(1,2)}{3}{\tgColour2}{\tgColour4}
		\tgBorderA{(2,2)}{\tgColour2}{\tgColour2}{\tgColour4}{\tgColour4}
		\tgBorderA{(3,2)}{\tgColour2}{\tgColour2}{\tgColour4}{\tgColour4}
		\tgBorderC{(4,2)}{2}{\tgColour2}{\tgColour4}
		\tgBorderA{(0,3)}{\tgColour6}{\tgColour2}{\tgColour2}{\tgColour6}
		\tgBorderA{(1,3)}{\tgColour2}{\tgColour4}{\tgColour4}{\tgColour2}
		\tgBlank{(2,3)}{\tgColour4}
		\tgBlank{(3,3)}{\tgColour4}
		\tgBorderA{(4,3)}{\tgColour4}{\tgColour2}{\tgColour2}{\tgColour4}
		\tgArrow{(0,1.5)}{1}
		\tgArrow{(1.5,2)}{0}
		\tgArrow{(2.5,2)}{0}
		\tgArrow{(0,2.5)}{1}
		\tgArrow{(2,0.5)}{3}
		\tgArrow{(3,0.5)}{1}
		\tgArrow{(4,0.5)}{3}
		\tgCell[(1,0)]{(0.5,1)}{\flat}
		\tgArrow{(0,0.5)}{1}
		\tgArrow{(1,2.5)}{1}
		\tgArrow{(1,0.5)}{3}
		\tgCell[(2,0)]{(3,1)}{\varepsilon}
		\tgArrow{(3.5,2)}{0}
		\tgArrow{(4,2.5)}{3}
		\tgAxisLabel{(0.5,0.75)}{south}{j}
		\tgAxisLabel{(1.5,0.75)}{south}{r}
		\tgAxisLabel{(2.5,0.75)}{south}{\ell}
		\tgAxisLabel{(3.5,0.75)}{south}{j}
		\tgAxisLabel{(4.5,0.75)}{south}{r}
		\tgAxisLabel{(0.5,3.25)}{north}{\ell}
		\tgAxisLabel{(1.5,3.25)}{north}{r}
		\tgAxisLabel{(4.5,3.25)}{north}{r}
	\end{tangle}
	\]
	\[
	\begin{tangle}{(5,5)}[trim y]
		\tgBorderA{(0,0)}{\tgColour6}{\tgColour4}{\tgColour4}{\tgColour6}
		\tgBorderA{(1,0)}{\tgColour4}{\tgColour2}{\tgColour2}{\tgColour4}
		\tgBorderA{(2,0)}{\tgColour2}{\tgColour6}{\tgColour6}{\tgColour2}
		\tgBorderA{(3,0)}{\tgColour6}{\tgColour4}{\tgColour4}{\tgColour6}
		\tgBorderA{(4,0)}{\tgColour4}{\tgColour2}{\tgColour2}{\tgColour4}
		\tgBorderA{(0,1)}{\tgColour6}{\tgColour4}{\tgColour4}{\tgColour6}
		\tgBorderA{(1,1)}{\tgColour4}{\tgColour2}{\tgColour2}{\tgColour4}
		\tgBorderA{(2,1)}{\tgColour2}{\tgColour6}{\tgColour6}{\tgColour2}
		\tgBorderA{(3,1)}{\tgColour6}{\tgColour4}{\tgColour4}{\tgColour6}
		\tgBorderA{(4,1)}{\tgColour4}{\tgColour2}{\tgColour2}{\tgColour4}
		\tgBorderA{(0,2)}{\tgColour6}{\tgColour4}{\tgColour2}{\tgColour6}
		\tgBorderA{(1,2)}{\tgColour4}{\tgColour2}{\tgColour2}{\tgColour2}
		\tgBorderA{(2,2)}{\tgColour2}{\tgColour6}{\tgColour2}{\tgColour2}
		\tgBorderA{(3,2)}{\tgColour6}{\tgColour4}{\tgColour2}{\tgColour2}
		\tgBorderA{(4,2)}{\tgColour4}{\tgColour2}{\tgColour2}{\tgColour2}
		\tgBorderA{(0,3)}{\tgColour6}{\tgColour2}{\tgColour2}{\tgColour6}
		\tgBlank{(1,3)}{\tgColour2}
		\tgBlank{(2,3)}{\tgColour2}
		\tgBlank{(3,3)}{\tgColour2}
		\tgBlank{(4,3)}{\tgColour2}
		\tgBorderA{(0,4)}{\tgColour6}{\tgColour2}{\tgColour2}{\tgColour6}
		\tgBlank{(1,4)}{\tgColour2}
		\tgBlank{(2,4)}{\tgColour2}
		\tgBlank{(3,4)}{\tgColour2}
		\tgBlank{(4,4)}{\tgColour2}
		\tgArrow{(0,2.5)}{1}
		\tgArrow{(2,1.5)}{3}
		\tgArrow{(3,1.5)}{1}
		\tgArrow{(4,1.5)}{3}
		\tgCell[(1,0)]{(0.5,2)}{\flat}
		\tgArrow{(0,1.5)}{1}
		\tgArrow{(1,1.5)}{3}
		\tgCell[(2,0)]{(3,2)}{\varepsilon}
		\tgArrow{(1,0.5)}{3}
		\tgArrow{(2,0.5)}{3}
		\tgArrow{(4,0.5)}{3}
		\tgArrow{(0,0.5)}{1}
		\tgArrow{(3,0.5)}{1}
		\tgArrow{(0,3.5)}{1}
		\tgAxisLabel{(0.5,0.75)}{south}{j}
		\tgAxisLabel{(1.5,0.75)}{south}{r}
		\tgAxisLabel{(2.5,0.75)}{south}{\ell}
		\tgAxisLabel{(3.5,0.75)}{south}{j}
		\tgAxisLabel{(4.5,0.75)}{south}{r}
		\tgAxisLabel{(0.5,4.25)}{north}{\ell}
	\end{tangle}
	\tangleeq*
	\begin{tangle}{(7,5)}[trim y]
		\tgBorderA{(0,0)}{\tgColour6}{\tgColour4}{\tgColour4}{\tgColour6}
		\tgBorderA{(1,0)}{\tgColour4}{\tgColour2}{\tgColour2}{\tgColour4}
		\tgBorderA{(2,0)}{\tgColour2}{\tgColour6}{\tgColour6}{\tgColour2}
		\tgBorderA{(3,0)}{\tgColour6}{\tgColour4}{\tgColour4}{\tgColour6}
		\tgBorderA{(4,0)}{\tgColour4}{\tgColour2}{\tgColour2}{\tgColour4}
		\tgBlank{(5,0)}{\tgColour2}
		\tgBlank{(6,0)}{\tgColour2}
		\tgBorderA{(0,1)}{\tgColour6}{\tgColour4}{\tgColour4}{\tgColour6}
		\tgBorderA{(1,1)}{\tgColour4}{\tgColour2}{\tgColour2}{\tgColour4}
		\tgBorderA{(2,1)}{\tgColour2}{\tgColour6}{\tgColour2}{\tgColour2}
		\tgBorderA{(3,1)}{\tgColour6}{\tgColour4}{\tgColour2}{\tgColour2}
		\tgBorderA{(4,1)}{\tgColour4}{\tgColour2}{\tgColour2}{\tgColour2}
		\tgBorderC{(5,1)}{3}{\tgColour2}{\tgColour4}
		\tgBorderC{(6,1)}{2}{\tgColour2}{\tgColour4}
		\tgBorderA{(0,2)}{\tgColour6}{\tgColour4}{\tgColour4}{\tgColour6}
		\tgBorderC{(1,2)}{0}{\tgColour4}{\tgColour2}
		\tgBorderA{(2,2)}{\tgColour2}{\tgColour2}{\tgColour4}{\tgColour4}
		\tgBorderA{(3,2)}{\tgColour2}{\tgColour2}{\tgColour4}{\tgColour4}
		\tgBorderA{(4,2)}{\tgColour2}{\tgColour2}{\tgColour4}{\tgColour4}
		\tgBorderC{(5,2)}{1}{\tgColour4}{\tgColour2}
		\tgBorderA{(6,2)}{\tgColour4}{\tgColour2}{\tgColour2}{\tgColour4}
		\tgBorderA{(0,3)}{\tgColour6}{\tgColour4}{\tgColour2}{\tgColour6}
		\tgBorderA{(1,3)}{\tgColour4}{\tgColour4}{\tgColour2}{\tgColour2}
		\tgBorderA{(2,3)}{\tgColour4}{\tgColour4}{\tgColour2}{\tgColour2}
		\tgBorderA{(3,3)}{\tgColour4}{\tgColour4}{\tgColour2}{\tgColour2}
		\tgBorderA{(4,3)}{\tgColour4}{\tgColour4}{\tgColour2}{\tgColour2}
		\tgBorderA{(5,3)}{\tgColour4}{\tgColour4}{\tgColour2}{\tgColour2}
		\tgBorderA{(6,3)}{\tgColour4}{\tgColour2}{\tgColour2}{\tgColour2}
		\tgBorderA{(0,4)}{\tgColour6}{\tgColour2}{\tgColour2}{\tgColour6}
		\tgBlank{(1,4)}{\tgColour2}
		\tgBlank{(2,4)}{\tgColour2}
		\tgBlank{(3,4)}{\tgColour2}
		\tgBlank{(4,4)}{\tgColour2}
		\tgBlank{(5,4)}{\tgColour2}
		\tgBlank{(6,4)}{\tgColour2}
		\tgArrow{(1,1.5)}{3}
		\tgArrow{(0,1.5)}{1}
		\tgArrow{(5,1.5)}{1}
		\tgArrow{(1.5,2)}{0}
		\tgArrow{(3.5,2)}{0}
		\tgArrow{(2.5,2)}{0}
		\tgArrow{(0,2.5)}{1}
		\tgArrow{(4.5,2)}{0}
		\tgArrow{(5.5,1)}{0}
		\tgArrow{(6,1.5)}{3}
		\tgArrow{(6,2.5)}{3}
		\tgArrow{(2,0.5)}{3}
		\tgArrow{(3,0.5)}{1}
		\tgArrow{(4,0.5)}{3}
		\tgCell[(2,0)]{(3,1)}{\varepsilon}
		\tgCell[(6,0)]{(3,3)}{\flat}
		\tgArrow{(0,0.5)}{1}
		\tgArrow{(1,0.5)}{3}
		\tgArrow{(0,3.5)}{1}
		\tgAxisLabel{(0.5,0.75)}{south}{j}
		\tgAxisLabel{(1.5,0.75)}{south}{r}
		\tgAxisLabel{(2.5,0.75)}{south}{\ell}
		\tgAxisLabel{(3.5,0.75)}{south}{j}
		\tgAxisLabel{(4.5,0.75)}{south}{r}
		\tgAxisLabel{(0.5,4.25)}{north}{\ell}
	\end{tangle}
	\]
	Given a left-morphism of $j$-adjunctions $(c, \lambda)$, define $\tau \colon (\ell' \d r') \tto (\ell \d r)$ by $\lambda \d r'$. Given a right-morphism of $j$-adjunctions $(c, \rho)$, define $\tau \colon (\ell \d r) \tto (\ell' \d r')$ by $\ell \d \rho$. In both cases, $\tau$ is a $j$-monad morphism, the unit preservation condition following from the reformulation of the compatibility condition for the $j$-adjunction morphism in terms of $\eta$ and $\eta'$; and the extension operator preservation condition following from the reformulation in terms of $\flat$ and $\flat'$. Preservation of identities and composites in both cases is trivial.
\end{proof}

\begin{definition}
	\label{resolution}
    Let $\jAE$ be a tight-cell, and let $T$ be a $j$-monad.
    A \emph{resolution}\footnotemark{} of $T$ is a $j$-adjunction $\ljr$ for which $T$ is
    equal to the $j$-monad ${\wedge_j}(\ljr)$ constructed in \cref{relative-adjunction-induces-relative-monad}.
	\footnotetext{We follow the terminology of \textcite[Chapter~IV]{bunge1966categories}. Resolutions of relative monads were called \emph{splittings} in \cite{altenkirch2010monads}.}
	A \emph{morphism} of resolutions of $T$ from $\ljr$ to $\ljrp$ is a tight-cell $c \colon C \to C'$ between the apices rendering the following diagram commutative.
    \[\begin{tikzcd}[row sep=small]
	& C \\
	A && E \\
	& {C'}
	\arrow["\ell", from=2-1, to=1-2]
	\arrow["r", from=1-2, to=2-3]
	\arrow["{\ell'}"', from=2-1, to=3-2]
	\arrow["{r'}"', from=3-2, to=2-3]
	\arrow["c"{description}, from=1-2, to=3-2]
    \end{tikzcd}\]
    Resolutions of $T$ and their morphisms form a category $\Res(T)$.
\end{definition}

We do not explicitly require morphisms of resolutions to be compatible with the isomorphisms $\sharp$ and $\flat$: this follows automatically from the following lemma\footnotemark{}, which in particular shows that morphisms of resolutions are strict morphisms of relative adjunctions.
\footnotetext{Compatibility with $\flat$ was imposed as an additional condition on morphisms of resolutions in \cite[Theorem~3]{altenkirch2010monads}, but is redundant by \cref{strict-morphism-implies-same-relative-monad}.}

\begin{lemma}
	\label{strict-morphism-implies-same-relative-monad}
	Let $\ljr$ and $\ljrp$ be relative adjunctions. A tight-cell $c \colon C \to C'$ between their apices satisfying $\ell \d c = \ell'$ and $c \d r = r'$ is a strict morphism from $\ljr$ to $\ljrp$ if and only if $\wedge_j(\ljr) = \wedge_j(\ljrp)$.
\end{lemma}

\begin{proof}
	By assumption, both relative monads are equal on underlying tight-cells. The compatibility condition for a strict morphism, expressed in terms of $\eta$ and $\eta'$, is equivalent to the condition that $\eta = \eta'$; and, expressed in terms of $\flat$ and $\flat'$, implies that $\dag = \dag'$ using the definition in \cref{relative-adjunction-induces-relative-monad}.
\end{proof}

Just as every relative monad induces a loose-monad (\cref{relative-monads-are-loose-monads}), every relative adjunction induces a loose-adjunction by \cref{adjunction-via-loose-adjunction}. Our primary motivation for introducing loose-adjunctions in \cref{loose-adjunction} is the following lemma, which establishes that the two methods by which we may derive a loose-monad from a relative adjunction coincide.

\begin{lemma}
	\label{relative-adjunction-to-loose-monad}
	Let $\ljr$ be a resolution of a relative monad $T$. The loose-monad induced by the loose-adjunction $C(1, \ell) \adj E(j, r)$ is isomorphic to $E(j, T)$.
\end{lemma}

\begin{proof}
	By definition, the underlying loose-cell of each is $E(j, r \ell)$. That the units and extension operators coincide follows by elementary string diagrammatic reasoning.
\end{proof}

Consequently, the following shows that the loose-monad induced by a relative monad is isomorphic to the loose-monad induced by the left adjoint of any of its resolutions.
We use this observation in \cref{opalgebra-objects} to give sufficient
conditions for the existence of opalgebra objects.

\begin{corollary}
	\label{loose-monad-associated-to-relative-monad-is-kernel-of-left-adjoint}
	Let $\ljr$ be resolution of a relative monad $T$. The loose-monad $C(\ell, \ell)$ is isomorphic to $E(j, T)$.
\end{corollary}

\begin{proof}
	By \cref{relative-adjunction-to-loose-monad}, $E(j, T)$ is induced by the loose-adjunction $C(1, \ell) \adj E(j, r)$. By definition, $E(j, r) \iso C(\ell, 1)$. Hence, $E(j, T)$ is isomorphic to the loose-monad induced by $C(1, \ell) \adj C(\ell, 1)$, which is precisely $C(\ell, \ell)$.
\end{proof}

\subsection{Composition of relative adjunctions}
\label{composition-of-relative-adjunctions}

There are two fundamental methods for constructing new relative adjunctions from existing relative adjunctions.

First, we may precompose a relative adjunction by a tight-cell (\cf{}~\cite[Lemma~2.6]{ulmer1968properties}).

\begin{proposition}
    \label{relative-adjunction-precomposition}
    Let $\ljr$ be a relative adjunction and $\ell' \colon A \to B$ be a tight-cell as below.
    \[
    \begin{tikzcd}
    	&& C \\
    	A & B && D
    	\arrow[""{name=0, anchor=center, inner sep=0}, "\ell", from=2-2, to=1-3]
    	\arrow[""{name=1, anchor=center, inner sep=0}, "r", from=1-3, to=2-4]
    	\arrow["j"', from=2-2, to=2-4]
    	\arrow["{\ell'}", from=2-1, to=2-2]
    	\arrow["\dashv"{anchor=center}, shift right=2, draw=none, from=0, to=1]
    \end{tikzcd}
    \quad\implies\quad
    \begin{tikzcd}
    	& C \\
    	A && D
    	\arrow[""{name=0, anchor=center, inner sep=0}, "{\ell' \d \ell}", from=2-1, to=1-2]
    	\arrow[""{name=1, anchor=center, inner sep=0}, "r", from=1-2, to=2-3]
    	\arrow["{\ell' \d j}"', from=2-1, to=2-3]
    	\arrow["\dashv"{anchor=center}, shift right=2, draw=none, from=0, to=1]
    \end{tikzcd}
    \]
    Then $(\ell' \d \ell) \radj{\ell' \d j} r$. Furthermore, this assignment extends to functors:
    \begin{align*}
		\ell' \d \ph & \colon \RAdj_L(j) \to \RAdj_L(\ell' \d j) \\
		\ell' \d \ph & \colon \RAdj_R(j) \to \RAdj_R(\ell' \d j)
	\end{align*}
\end{proposition}

\begin{proof}
    We have $C(\ell \ell', 1) \iso D(j \ell', r)$ using $\ljr$. Given a left-morphism $(c, \lambda)$ or right-morphism $(c, \rho)$ from $\ell_1 \jadj r_1$ to $\ell_2 \jadj r_2$, the pairs $(c, (\ell' \d \lambda))$ and $(c, \rho)$ respectively define left- and right-morphisms from $(\ell' \d \ell_1) \radj{\ell' \d j} r_1$ to $(\ell' \d \ell_2) \radj{\ell' \d j} r_2$, the compatibility conditions following immediately from those of $(c, \lambda)$ and $(c, \rho)$ respectively.
	Functoriality is trivial in both cases.
\end{proof}

Second, we have the following pasting law for relative adjunctions, analogous to the classical pasting law for pullbacks.

\begin{proposition}[Pasting law]
    \label{pasting-law}
    Consider the following diagram.
    \[\begin{tikzcd}[sep=small]
        && C \\
        &&& D \\
        A &&&& E
        \arrow["r", from=1-3, to=2-4]
        \arrow["j"', from=3-1, to=3-5]
        \arrow[""{name=0, anchor=center, inner sep=0}, "{r'}", from=2-4, to=3-5]
        \arrow[""{name=1, anchor=center, inner sep=0}, "{\ell'}"{description}, from=3-1, to=2-4]
        \arrow["\ell", from=3-1, to=1-3]
        \arrow["\dashv"{anchor=center}, shift right=1, draw=none, from=1, to=0]
    \end{tikzcd}\]
    The left triangle is a relative adjunction ($\ell \radj{\ell'} r$) if and only if the outer triangle is a relative adjunction ($\ell \jadj (r \d r'$)).
    In this case, $(r, \eta)$ exhibits a left-morphism of $j$-adjunctions $(\ell \jadj (r \d r')) \to (\ell' \jadj r')$, where $\eta$ is the unit of $\ell \radj{\ell'} r$. Furthermore, this assignment extends to functors:
    \begin{align*}
		\ph \d r' & \colon \RAdj_L(\ell') \to \RAdj_L(j) \\
		\ph \d r' & \colon \RAdj_R(\ell') \to \RAdj_R(j)
	\end{align*}
\end{proposition}

\begin{proof}
    We have $D(\ell', r) \iso E(j, r' r)$ since $\ell' \jadj r'$. The condition that the left, respectively outer, triangle be a relative adjunction asserts that the left-hand side, respectively the right-hand side, be isomorphic to $C(\ell, 1)$. That $(r, \eta)$ exhibits a left-morphism in this case follows by definition.

	Given a left-morphism $(c, \lambda)$ or right-morphism $(c, \rho)$ from $\ell_1 \radj{\ell'} r_1$ to $\ell_2 \radj{\ell'} r_2$, the pairs $(c, \lambda)$ and $(c, (\rho \d r'))$ respectively define left- and right-morphisms from $\ell_1 \radj{j'} (r_1 \d r')$ to $\ell_2 \radj{j'} (r_2 \d r')$, the compatibility conditions following immediately from those of $(c, \lambda)$ and $(c, \rho)$ respectively.
	Functoriality is trivial in both cases.
\end{proof}

\begin{example}
	Taking $j = r'$ for a \ff{} $r'$ in \cref{pasting-law} (\cf{}~\cref{trivial-relative-adjunctions}), it follows that every $\ell'$-adjunction induces an $(\ell' \d r')$-adjunction by postcomposition.
\end{example}

The pasting law may be seen as a currying result for pointwise left lifts analogous to that for pointwise left extensions (\cref{currying-pointwise-extension}), or as a pointwise analogue of \cite[Proposition~1]{street1978yoneda}. It may also be seen in one respect as a relative analogue of the classical adjoint triangle theorems~\cite{dubuc1968adjoint}. When $j = 1$, we obtain that $r$ admits a left $\ell'$-adjoint if and only if $(r \d r')$ admits a left adjoint. If $r$ has a left adjoint, then it has a left $\ell'$-adjoint by precomposing $\ell'$ (\cref{relative-adjunction-precomposition}). Adjoint triangle theorems may therefore be seen as providing a converse: giving sufficient conditions for every left $\ell'$-adjoint to extend to a left adjoint.
\[\begin{tikzcd}[sep=small]
	& C \\
	&& D \\
	E &&& E
	\arrow[""{name=0, anchor=center, inner sep=0}, "r", curve={height=-12pt}, from=1-2, to=2-3]
	\arrow[Rightarrow, no head, from=3-1, to=3-4]
	\arrow[""{name=1, anchor=center, inner sep=0}, "{r'}", from=2-3, to=3-4]
	\arrow[""{name=2, anchor=center, inner sep=0}, "{\ell'}"{description}, from=3-1, to=2-3]
	\arrow[""{name=3, anchor=center, inner sep=0}, "\ell", from=3-1, to=1-2]
	\arrow[""{name=4, anchor=center, inner sep=0}, curve={height=-12pt}, dashed, from=2-3, to=1-2]
	\arrow["\dashv"{anchor=center}, shift right=1, draw=none, from=2, to=1]
	\arrow["\dashv"{anchor=center, rotate=45}, draw=none, from=4, to=0]
	\arrow["{=}"{description}, shift left=2, draw=none, from=3, to=2]
\end{tikzcd}\]

\begin{corollary}
    \label{resolute-transposition}
    Let $(\ell_1 \d \ell_2) \jadj r$ be a relative adjunction with unit $\eta$. Then $\eta$ exhibits a relative adjunction $\ell_1 \jadj (\ell_2 \d r)$ if and only if the identity $1_{\ell_1 \d \ell_2}$ exhibits a relative adjunction $\ell_1 \radj{\ell_1 \d \ell_2} \ell_2$. In this case, the two induced $j$-monads are identical.
    \[
	\begin{tikzcd}
		B & C \\
		A && E
		\arrow["j"', from=2-1, to=2-3]
		\arrow[""{name=0, anchor=center, inner sep=0}, "r", from=1-2, to=2-3]
		\arrow[""{name=1, anchor=center, inner sep=0}, "{\ell_1 \d \ell_2}"{description}, from=2-1, to=1-2]
		\arrow["{\ell_1}", from=2-1, to=1-1]
		\arrow["{\ell_2}", from=1-1, to=1-2]
		\arrow["\dashv"{anchor=center}, shift right=2, draw=none, from=1, to=0]
	\end{tikzcd}
    \quad\implies\quad
	\begin{tikzcd}
		& B \\
		A && E
		\arrow["j"', from=2-1, to=2-3]
		\arrow[""{name=0, anchor=center, inner sep=0}, "{\ell_2 \d r}", from=1-2, to=2-3]
		\arrow[""{name=1, anchor=center, inner sep=0}, "{\ell_1}", from=2-1, to=1-2]
		\arrow["\dashv"{anchor=center}, shift right=2, draw=none, from=1, to=0]
	\end{tikzcd}
    \]
\end{corollary}

\begin{proof}
    Follows directly from \cref{pasting-law}; the two $j$-monads are identical because the unit of the relative adjunction $\ell_1 \jadj (\ell_2 \d r)$, forming a left-morphism between them, is the identity.
\end{proof}

\begin{remark}
    \Cref{resolute-transposition} recovers \cite[Proposition~6.1.6]{arkor2022monadic}, in which pairs $(\ell_1, \ell_2)$ satisfying $\ell_1 \radj{\ell_1 \d \ell_2} \ell_2$ were called \emph{resolute}~\cite[Definition~6.1.4]{arkor2022monadic}. In particular, a pair $(\ell_1, \ell_2)$ is resolute if $\ell_2$ is \ff{}, in which case we recover a well-known result for adjunctions (\eg{}~\cite[Proposition~1.1]{deleanu1975idempotent}).
\end{remark}

In contrast to non-relative adjunctions, we cannot in general compose relative adjunctions simply by composing the left adjoints and the right adjoints.
However, using \cref{relative-adjunction-precomposition} in conjunction with \cref{pasting-law}, we may compose relative adjunctions when the first left relative adjoint factors through the root of the second relative adjunction.

\begin{corollary}
	\label{relative-adjunction-composition}
	Let $\ljr$ and $(\ell' \d j) \radj{j'} r'$ be relative adjunctions as below.
	\[\begin{tikzcd}
		&& C \\
		& B && D \\
		A &&&& E
		\arrow[""{name=0, anchor=center, inner sep=0}, "\ell", from=2-2, to=1-3]
		\arrow[""{name=1, anchor=center, inner sep=0}, "r", from=1-3, to=2-4]
		\arrow["j"{description}, from=2-2, to=2-4]
		\arrow["{\ell'}", from=3-1, to=2-2]
		\arrow["{j'}"', from=3-1, to=3-5]
		\arrow[""{name=2, anchor=center, inner sep=0}, "{r'}", from=2-4, to=3-5]
		\arrow[""{name=3, anchor=center, inner sep=0}, "{\ell' \d j}"{description}, from=3-1, to=2-4]
		\arrow["\dashv"{anchor=center}, shift right=2, draw=none, from=0, to=1]
		\arrow["\dashv"{anchor=center}, shift right=2, draw=none, from=3, to=2]
	\end{tikzcd}\]
	Then $(\ell' \d \ell) \radj{j'} (r \d r')$, and $((\ell' \d \eta), r)$ is a left-morphism $((\ell' \d \ell) \radj{j'} (r \d r')) \to ((\ell' \d j) \radj{j'} r')$. Furthermore, this assignment extends to functors:
	\begin{align*}
		\ell' \d \ph \d r' & \colon \RAdj_L(j) \to \RAdj_L(j') \\
		\ell' \d \ph \d r' & \colon \RAdj_R(j) \to \RAdj_R(j')
	\end{align*}
\end{corollary}

\begin{proof}
    Follows by first precomposing $\ljr$ by $\ell'$ using \cref{relative-adjunction-precomposition}, then pasting $(\ell' \d \ell) \radj{\ell' \d j} r$ along $(\ell' \d j) \radj{j'} r'$ using \cref{pasting-law}.
\end{proof}

\begin{example}
	Taking $j = 1$ in \cref{relative-adjunction-composition}, it follows that we may compose relative adjunctions with adjunctions on their apices, recovering \cites[Lemma~2.10]{lewicki2020categories}[Lemma~18]{forster2019category}, as well as \cite[Proposition~4.6]{fiore2018relative} for relative 2-adjunctions.
\end{example}

When relative monads admit resolutions, \cref{relative-adjunction-precomposition,pasting-law} permit relative monads to be precomposed by tight-cells, and to be pasted along left relative adjoints, by first decomposing a relative monad into its resolution, applying the relevant proposition, and then taking the relative monad induced by the new relative adjunction. However, in an arbitrary equipment, it is not necessarily true that relative monads admit resolutions. In the following, we show that this assumption may be dropped.

First, we may precompose a relative monad by a tight-cell (\cf{}~\cites[Definition~1.3.1]{walters1970categorical}[Construction~2.1.15]{voevodsky2023c}, and \cites[Theorem~1]{altenkirch2010monads}[Proposition~2.3]{altenkirch2015monads} when $j = 1$).

\begin{proposition}
	\label{relative-monad-precomposition}
	Let $j \colon B \to D$ be a tight-cell, let $T = (t, \dag, \eta)$ be a $j$-monad, and let $\ell' \colon A \to B$ be a tight-cell. Then $(\ell \d t)$ may be equipped with the structure of an $(\ell' \d j)$-monad. Furthermore, this assignment extends to a functor $\RMnd(j) \to \RMnd(\ell' \d j)$ rendering the following diagram commutative.
	\[\begin{tikzcd}
		{\RAdj_L(j)\op} & {\RMnd(j)} & {\RAdj_R(j)} \\
		{\RAdj_L(\ell' \d j)\op} & {\RMnd(\ell' \d j)} & {\RAdj_R(\ell' \d j)}
		\arrow["{\obslash_j}", from=1-1, to=1-2]
		\arrow["{(\ell' \d \ph)\op}"', from=1-1, to=2-1]
		\arrow["{\obslash_{\ell' \d j}}"', from=2-1, to=2-2]
		\arrow["{\ell' \d \ph}"{description}, from=1-2, to=2-2]
		\arrow["{\oslash_j}"', from=1-3, to=1-2]
		\arrow["{\oslash_{\ell' \d j}}", from=2-3, to=2-2]
		\arrow["{\ell' \d \ph}", from=1-3, to=2-3]
	\end{tikzcd}\]
\end{proposition}

\begin{proof}
	The unit is given by $(\ell' \d \eta)$ and the extension operator is given by $\dag(\ell', \ell')$. That the relative monad laws are satisfied follows trivially from those for $T$. Given a $j$-monad morphism $\tau \colon T \to T'$, the 2-cell $(\ell' \d \tau)$ forms an $(\ell' \d j)$-monad morphism $(\ell' \d T) \to (\ell' \d T')$, the relative monad morphism laws following trivially from those for $\tau$. Functoriality of the assignment, given by precomposing $\ell'$, is trivial; commutativity of the squares follows by definition.
\end{proof}

Second, we may paste a relative monad along a left relative adjoint.

\begin{proposition}
	\label{relative-monad-relative-adjunction-pasting}
	Let $\ell' \jadj r'$ be a relative adjunction and let $T = (t, \dag, \eta)$ be an $\ell'$-monad.
	\[\begin{tikzcd}
		& D \\
		A && E
		\arrow["j"', from=2-1, to=2-3]
		\arrow[""{name=0, anchor=center, inner sep=0}, "{r'}", from=1-2, to=2-3]
		\arrow[""{name=1, anchor=center, inner sep=0}, "{\ell'}"{description}, from=2-1, to=1-2]
		\arrow[""{name=2, anchor=center, inner sep=0}, "t", curve={height=-18pt}, from=2-1, to=1-2]
		\arrow["\dashv"{anchor=center}, shift right=1, draw=none, from=1, to=0]
		\arrow["\eta"', shorten <=3pt, shorten >=3pt, Rightarrow, from=1, to=2]
	\end{tikzcd}\]
	Then $(t \d r')$ may be equipped with the structure of a $j$-monad, for which $(\eta \d r')$ is a $j$-monad morphism $\wedge_j (\ell' \jadj r') \to (T \d r')$. Furthermore, this assignment extends to a functor $\RMnd(\ell') \to \RMnd(j)$ rendering the following diagram commutative.
	\[\begin{tikzcd}
		{\RAdj_L(\ell')\op} & {\RMnd(\ell')} & {\RAdj_R(\ell')} \\
		{\RAdj_L(j)\op} & {\RMnd(j)} & {\RAdj_R(j)}
		\arrow["{\obslash_{\ell'}}", from=1-1, to=1-2]
		\arrow["{(\ph \d r')\op}"', from=1-1, to=2-1]
		\arrow["{\obslash_j}"', from=2-1, to=2-2]
		\arrow["{\ph \d r'}"{description}, from=1-2, to=2-2]
		\arrow["{\oslash_{\ell'}}"', from=1-3, to=1-2]
		\arrow["{\oslash_j}", from=2-3, to=2-2]
		\arrow["{\ph \d r'}", from=1-3, to=2-3]
	\end{tikzcd}\]
\end{proposition}

\begin{proof}
	The extension operator of the induced $j$-monad is the 2-cell on the left below; the unit is the 2-cell on the right below.
	
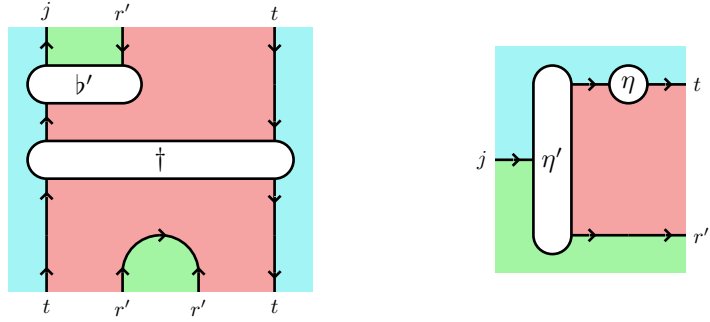
\[
	\begin{tangle}{(4,5)}[trim y]
		\tgBorderA{(0,0)}{\tgColour6}{\tgColour4}{\tgColour4}{\tgColour6}
		\tgBorderA{(1,0)}{\tgColour4}{\tgColour0}{\tgColour0}{\tgColour4}
		\tgBlank{(2,0)}{\tgColour0}
		\tgBorderA{(3,0)}{\tgColour0}{\tgColour6}{\tgColour6}{\tgColour0}
		\tgBorderA{(0,1)}{\tgColour6}{\tgColour4}{\tgColour0}{\tgColour6}
		\tgBorderA{(1,1)}{\tgColour4}{\tgColour0}{\tgColour0}{\tgColour0}
		\tgBlank{(2,1)}{\tgColour0}
		\tgBorderA{(3,1)}{\tgColour0}{\tgColour6}{\tgColour6}{\tgColour0}
		\tgBorderA{(0,2)}{\tgColour6}{\tgColour0}{\tgColour0}{\tgColour6}
		\tgBorder{(0,2)}{0}{1}{0}{0}
		\tgBorderA{(1,2)}{\tgColour0}{\tgColour0}{\tgColour0}{\tgColour0}
		\tgBorder{(1,2)}{0}{1}{0}{1}
		\tgBorderA{(2,2)}{\tgColour0}{\tgColour0}{\tgColour0}{\tgColour0}
		\tgBorder{(2,2)}{0}{1}{0}{1}
		\tgBorderA{(3,2)}{\tgColour0}{\tgColour6}{\tgColour6}{\tgColour0}
		\tgBorder{(3,2)}{0}{0}{0}{1}
		\tgBorderA{(0,3)}{\tgColour6}{\tgColour0}{\tgColour0}{\tgColour6}
		\tgBorderC{(1,3)}{3}{\tgColour0}{\tgColour4}
		\tgBorderC{(2,3)}{2}{\tgColour0}{\tgColour4}
		\tgBorderA{(3,3)}{\tgColour0}{\tgColour6}{\tgColour6}{\tgColour0}
		\tgBorderA{(0,4)}{\tgColour6}{\tgColour0}{\tgColour0}{\tgColour6}
		\tgBorderA{(1,4)}{\tgColour0}{\tgColour4}{\tgColour4}{\tgColour0}
		\tgBorderA{(2,4)}{\tgColour4}{\tgColour0}{\tgColour0}{\tgColour4}
		\tgBorderA{(3,4)}{\tgColour0}{\tgColour6}{\tgColour6}{\tgColour0}
		\tgCell[(1,0)]{(0.5,1)}{\flat'}
		\tgArrow{(2,3.5)}{1}
		\tgArrow{(3,3.5)}{3}
		\tgArrow{(1,3.5)}{1}
		\tgArrow{(0,1.5)}{1}
		\tgArrow{(0,2.5)}{1}
		\tgArrow{(0,3.5)}{1}
		\tgArrow{(0,0.5)}{1}
		\tgArrow{(1,0.5)}{3}
		\tgArrow{(3,2.5)}{3}
		\tgArrow{(3,1.5)}{3}
		\tgArrow{(3,0.5)}{3}
		\tgArrow{(1.5,3)}{0}
		\tgCell[(3,0)]{(1.5,2)}{\dag}
		\tgAxisLabel{(0.5,0.75)}{south}{j}
		\tgAxisLabel{(1.5,0.75)}{south}{r'}
		\tgAxisLabel{(3.5,0.75)}{south}{t}
		\tgAxisLabel{(0.5,4.25)}{north}{t}
		\tgAxisLabel{(1.5,4.25)}{north}{r'}
		\tgAxisLabel{(2.5,4.25)}{north}{r'}
		\tgAxisLabel{(3.5,4.25)}{north}{t}
	\end{tangle}
	\hspace{4em}
	\begin{tangle}{(4,3)}[trim x]
		\tgBlank{(0,0)}{\tgColour6}
		\tgBorderA{(1,0)}{\tgColour6}{\tgColour6}{\tgColour0}{\tgColour6}
		\tgBorderA{(2,0)}{\tgColour6}{\tgColour6}{\tgColour0}{\tgColour0}
		\tgBorderA{(3,0)}{\tgColour6}{\tgColour6}{\tgColour0}{\tgColour0}
		\tgBorderA{(0,1)}{\tgColour6}{\tgColour6}{\tgColour4}{\tgColour4}
		\tgBorderA{(1,1)}{\tgColour6}{\tgColour0}{\tgColour0}{\tgColour4}
		\tgBlank{(2,1)}{\tgColour0}
		\tgBlank{(3,1)}{\tgColour0}
		\tgBlank{(0,2)}{\tgColour4}
		\tgBorderA{(1,2)}{\tgColour4}{\tgColour0}{\tgColour4}{\tgColour4}
		\tgBorderA{(2,2)}{\tgColour0}{\tgColour0}{\tgColour4}{\tgColour4}
		\tgBorderA{(3,2)}{\tgColour0}{\tgColour0}{\tgColour4}{\tgColour4}
		\tgCell[(0,2)]{(1,1)}{\eta'}
		\tgArrow{(1.5,0)}{0}
		\tgArrow{(2.5,0)}{0}
		\tgArrow{(2.5,2)}{0}
		\tgArrow{(1.5,2)}{0}
		\tgArrow{(0.5,1)}{0}
		\tgCell{(2,0)}{\eta}
		\tgAxisLabel{(3.25,0.5)}{west}{t}
		\tgAxisLabel{(0.75,1.5)}{east}{j}
		\tgAxisLabel{(3.25,2.5)}{west}{r'}
	\end{tangle}
	\]
	The unit laws follow from the unit laws for $T$, together with the $\sharp'$--$\flat'$ isomorphism. The associativity law follows from the associativity law for $T$. That $(\eta \d r')$ is a $j'$-monad morphism follows by definition. Functoriality of the assignment, given by postcomposing $r'$, is trivial.

	The square on the left and on the right agree on objects, so to show commutativity of the object assignments, it suffices to show that the following diagram of sets commutes.
	\[\begin{tikzcd}
		{\RAdj(\ell')} & {\RAdj(j)} \\
		{\RMnd(\ell')} & {\RMnd(j)}
		\arrow["{\wedge_{\ell'}}"', from=1-1, to=2-1]
		\arrow["{\ph \d r'}", from=1-1, to=1-2]
		\arrow["{\wedge_j}", from=1-2, to=2-2]
		\arrow["{\ph \d r'}"', from=2-1, to=2-2]
	\end{tikzcd}\]
	That the assignments in the diagram above agree on the underlying tight-cell and unit is trivial; that they agree on the extension operator follows from the $\sharp'$--$\flat'$ isomorphism.

	Finally, that the two squares commute on morphisms is trivial.
\end{proof}

\begin{example}
	Taking $j = r'$ for a \ff{} $r'$ in \cref{relative-monad-relative-adjunction-pasting} (\cf{}~\cref{trivial-relative-adjunctions}), it follows that every $\ell'$-monad induces an $(\ell' \d r')$-monad by postcomposition, recovering \cite[Examples~15 \& 25]{ahrens2015terminal}.
\end{example}

\begin{example}
	Let $\ell' \adj r'$ be an adjunction on $A$ with $\ell'$ \ff{}, so that $\ell' \d r \iso 1_A$. Then, by two applications of \cref{relative-monad-relative-adjunction-pasting}, there are induced functors $\ph \d r' \colon \RMnd(\ell') \to \Mnd(A)$ and $\ph \d \ell' \colon \Mnd(A) \to \RMnd(\ell')$, for which the former is a retraction, up to isomorphism, of the latter. Furthermore, the adjunction $\ell' \adj r'$ lifts to an adjunction between categories of relative monads, recovering \cite[Remark~26]{ahrens2015terminal};
	\[\begin{tikzcd}[column sep=huge]
		{\Mnd(A)} & {\RMnd(\ell')}
		\arrow[""{name=0, anchor=center, inner sep=0}, "{\ph \d \ell'}", shift left=2, hook, from=1-1, to=1-2]
		\arrow[""{name=1, anchor=center, inner sep=0}, "{\ph \d r'}", shift left=2, from=1-2, to=1-1]
		\arrow["\dashv"{anchor=center, rotate=-90}, draw=none, from=0, to=1]
	\end{tikzcd}\]
	and, similarly, by \cref{pasting-law}, to adjunctions between categories of relative adjunctions.
	\[
	\begin{tikzcd}[column sep=huge]
		{\RAdj_L(A)} & {\RAdj_L(\ell')}
		\arrow[""{name=0, anchor=center, inner sep=0}, "{\ph \d \ell'}", shift left=2, hook, from=1-1, to=1-2]
		\arrow[""{name=1, anchor=center, inner sep=0}, "{\ph \d r'}", shift left=2, from=1-2, to=1-1]
		\arrow["\dashv"{anchor=center, rotate=-90}, draw=none, from=0, to=1]
	\end{tikzcd}
	\hspace{4em}
	\begin{tikzcd}[column sep=huge]
		{\RAdj_R(A)} & {\RAdj_R(\ell')}
		\arrow[""{name=0, anchor=center, inner sep=0}, "{\ph \d \ell'}", shift left=2, hook, from=1-1, to=1-2]
		\arrow[""{name=1, anchor=center, inner sep=0}, "{\ph \d r'}", shift left=2, from=1-2, to=1-1]
		\arrow["\dashv"{anchor=center, rotate=-90}, draw=none, from=0, to=1]
	\end{tikzcd}
	\]
\end{example}

Using \cref{relative-monad-precomposition} in conjunction with \cref{relative-monad-relative-adjunction-pasting}, we may compose relative monads with relative adjunctions whose left relative adjoint factors through the root of the relative monad.

\begin{corollary}
	\label{relative-monad-relative-adjunction-composition}
	Let $(\ell' \d j) \radj{j'} r'$ be a relative adjunction, and let $T = (t, \dag, \eta)$ be a $j$-monad.
	\[\begin{tikzcd}
		& B && D \\
		A &&&& E
		\arrow[""{name=0, anchor=center, inner sep=0}, "j"{description}, from=1-2, to=1-4]
		\arrow["{\ell'}", from=2-1, to=1-2]
		\arrow["{j'}"', from=2-1, to=2-5]
		\arrow[""{name=1, anchor=center, inner sep=0}, "{r'}", from=1-4, to=2-5]
		\arrow[""{name=2, anchor=center, inner sep=0}, "{\ell' \d j}"{description}, from=2-1, to=1-4]
		\arrow[""{name=3, anchor=center, inner sep=0}, "t", curve={height=-30pt}, from=1-2, to=1-4]
		\arrow["\dashv"{anchor=center}, shift right=2, draw=none, from=2, to=1]
		\arrow["\eta"', shorten <=8pt, shorten >=6pt, Rightarrow, from=0, to=3]
	\end{tikzcd}\]
	Then $(\ell' \d t \d r')$ may be equipped with the structure of a $j'$-monad for which $(\ell' \d \eta \d r')$ is a $j'$-monad morphism $(\ell' \d j \d r') \to (\ell' \d T \d r')$. Furthermore, this assignment extends to a functor $\RMnd(j) \to \RMnd(j')$ rendering the following diagram commutative.
	\[\begin{tikzcd}
		{\RAdj_L(j)\op} & {\RMnd(j)} & {\RAdj_R(j)} \\
		{\RAdj_L(j')\op} & {\RMnd(j')} & {\RAdj_R(j')}
		\arrow["{\obslash_j}", from=1-1, to=1-2]
		\arrow["{(\ell' \d \ph \d r')\op}"', from=1-1, to=2-1]
		\arrow["{\obslash_{j'}}"', from=2-1, to=2-2]
		\arrow["{\ell' \d \ph \d r'}"{description}, from=1-2, to=2-2]
		\arrow["{\oslash_j}"', from=1-3, to=1-2]
		\arrow["{\oslash_{j'}}", from=2-3, to=2-2]
		\arrow["{\ell' \d \ph \d r'}", from=1-3, to=2-3]
	\end{tikzcd}\]
\end{corollary}

\begin{proof}
	Follows by first precomposing $T$ by $\ell'$ using \cref{relative-monad-precomposition}, then pasting $(\ell' \d T)$ along $(\ell' \d j) \radj{j'} r'$ using \cref{relative-monad-relative-adjunction-pasting}.
\end{proof}

\begin{example}
	\Cref{relative-monad-relative-adjunction-composition} recovers several constructions in the literature.
	\begin{enumerate}
		\item Taking $j = 1$ above, it follows that every monad on the apex of a $j'$-adjunction induces a $j'$-monad. Taking furthermore $j' = 1$, we recover \cite[Theorem~4.2]{huber1961homotopy}.
		\item Taking $j = j'$, we recover \cite[Theorem~5.5]{abate2021ssprove}. \qedhere
	\end{enumerate}
\end{example}

\section{Algebras and opalgebras}
\label{algebras-and-opalgebras}

In the classical setting of relative monads in $\Cat$, there are two fundamental constructions associated to a relative monad: the \EM{} category and the Kleisli category~\cite[\S2.3]{altenkirch2015monads}. In this section, we study the corresponding notions in a \ve{}. To do so, we characterise the structures that the \EM{} category and the Kleisli category classify: namely, the \emph{algebras} and \emph{opalgebras}. While the term ``algebra'' is a familiar generalisation of the notion of algebra for a relative monad, the term ``opalgebra'' is less so. This is due in part to the fact that, in $\Cat$, the Kleisli category may be characterised as the category of free algebras, which permits the conflation of the notions of opalgebra and free algebra. However, in a general equipment this is not possible (\cf{}~\cref{comparison-tight-cell}), and it is the opalgebras that are the fundamental notion (\cf{}~\cite[\S4]{street1972formal}).

\begin{definition}
	\label{algebra}
    Let $T$ be a relative monad. An \emph{algebra for $T$} (or simply \emph{$T$-algebra}) comprises
    \begin{enumerate}
		\item an object $D$, the \emph{domain};
        \item a tight-cell $e \colon D \to E$, the \emph{carrier} or \emph{underlying tight-cell};
        \item a 2-cell $\aop \colon E(j, e) \tto E(t, e)$, the \emph{extension operator},
    \end{enumerate}
    satisfying the following equations.
    \[
	\begin{tikzcd}[column sep=large]
		A & D \\
		A & D \\
		A & D
		\arrow["{E(j, e)}"', "\shortmid"{marking}, from=1-2, to=1-1]
		\arrow["{E(t, e)}"{description}, from=2-2, to=2-1]
		\arrow[""{name=0, anchor=center, inner sep=0}, Rightarrow, no head, from=1-2, to=2-2]
		\arrow[""{name=1, anchor=center, inner sep=0}, Rightarrow, no head, from=1-1, to=2-1]
		\arrow[""{name=2, anchor=center, inner sep=0}, Rightarrow, no head, from=2-2, to=3-2]
		\arrow[""{name=3, anchor=center, inner sep=0}, Rightarrow, no head, from=2-1, to=3-1]
		\arrow["{E(j, e)}", "\shortmid"{marking}, from=3-2, to=3-1]
		\arrow["\aop"{description}, draw=none, from=0, to=1]
		\arrow["{E(\eta, e)}"{description}, draw=none, from=2, to=3]
	\end{tikzcd}
    \quad = \quad
	\begin{tikzcd}
		A & D \\
		A & D
		\arrow["{E(j, e)}"', "\shortmid"{marking}, from=1-2, to=1-1]
		\arrow["{E(j, e)}", "\shortmid"{marking}, from=2-2, to=2-1]
		\arrow[""{name=0, anchor=center, inner sep=0}, Rightarrow, no head, from=1-2, to=2-2]
		\arrow[""{name=1, anchor=center, inner sep=0}, Rightarrow, no head, from=1-1, to=2-1]
		\arrow["{=}"{description}, draw=none, from=0, to=1]
	\end{tikzcd}
	\hspace{4em}
	\begin{tangle}{(2,4)}[trim y]
		\tgBorderA{(0,0)}{\tgColour6}{\tgColour4}{\tgColour4}{\tgColour6}
		\tgBorderA{(1,0)}{\tgColour4}{\tgColour0}{\tgColour0}{\tgColour4}
		\tgBorderA{(0,1)}{\tgColour6}{\tgColour4}{\tgColour4}{\tgColour6}
		\tgBorder{(0,1)}{0}{1}{0}{0}
		\tgBorderA{(1,1)}{\tgColour4}{\tgColour0}{\tgColour0}{\tgColour4}
		\tgBorder{(1,1)}{0}{0}{0}{1}
		\tgBorderA{(0,2)}{\tgColour6}{\tgColour4}{\tgColour4}{\tgColour6}
		\tgBorderA{(1,2)}{\tgColour4}{\tgColour0}{\tgColour0}{\tgColour4}
		\tgBorderA{(0,3)}{\tgColour6}{\tgColour4}{\tgColour4}{\tgColour6}
		\tgBorderA{(1,3)}{\tgColour4}{\tgColour0}{\tgColour0}{\tgColour4}
		\tgCell[(1,0)]{(0.5,1)}{\aop}
		\tgCell{(0,2)}{\eta}
		\tgArrow{(0,1.5)}{1}
		\tgArrow{(0,2.5)}{1}
		\tgArrow{(0,0.5)}{1}
		\tgArrow{(1,0.5)}{3}
		\tgArrow{(1,1.5)}{3}
		\tgArrow{(1,2.5)}{3}
		\tgAxisLabel{(0.5,0.75)}{south}{j}
		\tgAxisLabel{(1.5,0.75)}{south}{e}
		\tgAxisLabel{(0.5,3.25)}{north}{j}
		\tgAxisLabel{(1.5,3.25)}{north}{e}
	\end{tangle}
	\tangleeq*
	\begin{tangle}{(2,4)}[trim y]
		\tgBorderA{(0,0)}{\tgColour6}{\tgColour4}{\tgColour4}{\tgColour6}
		\tgBorderA{(1,0)}{\tgColour4}{\tgColour0}{\tgColour0}{\tgColour4}
		\tgBorderA{(0,1)}{\tgColour6}{\tgColour4}{\tgColour4}{\tgColour6}
		\tgBorderA{(1,1)}{\tgColour4}{\tgColour0}{\tgColour0}{\tgColour4}
		\tgBorderA{(0,2)}{\tgColour6}{\tgColour4}{\tgColour4}{\tgColour6}
		\tgBorderA{(1,2)}{\tgColour4}{\tgColour0}{\tgColour0}{\tgColour4}
		\tgBorderA{(0,3)}{\tgColour6}{\tgColour4}{\tgColour4}{\tgColour6}
		\tgBorderA{(1,3)}{\tgColour4}{\tgColour0}{\tgColour0}{\tgColour4}
		\tgArrow{(0,1.5)}{1}
		\tgArrow{(0,2.5)}{1}
		\tgArrow{(0,0.5)}{1}
		\tgArrow{(1,0.5)}{3}
		\tgArrow{(1,1.5)}{3}
		\tgArrow{(1,2.5)}{3}
		\tgAxisLabel{(0.5,0.75)}{south}{j}
		\tgAxisLabel{(1.5,0.75)}{south}{e}
		\tgAxisLabel{(0.5,3.25)}{north}{j}
		\tgAxisLabel{(1.5,3.25)}{north}{e}
	\end{tangle}
    \]
    \[
	\begin{tikzcd}[column sep=large]
		A & A & D \\
		A & A & D \\
		A && D
		\arrow["{E(j, t)}"', "\shortmid"{marking}, from=1-2, to=1-1]
		\arrow["{E(j, e)}"', "\shortmid"{marking}, from=1-3, to=1-2]
		\arrow["{E(t, t)}"{description}, from=2-2, to=2-1]
		\arrow["{E(t, e)}"{description}, from=2-3, to=2-2]
		\arrow[""{name=0, anchor=center, inner sep=0}, Rightarrow, no head, from=2-1, to=1-1]
		\arrow[""{name=1, anchor=center, inner sep=0}, Rightarrow, no head, from=2-2, to=1-2]
		\arrow[""{name=2, anchor=center, inner sep=0}, Rightarrow, no head, from=2-3, to=1-3]
		\arrow["{E(t, e)}", "\shortmid"{marking}, from=3-3, to=3-1]
		\arrow[""{name=3, anchor=center, inner sep=0}, Rightarrow, no head, from=3-1, to=2-1]
		\arrow[""{name=4, anchor=center, inner sep=0}, Rightarrow, no head, from=3-3, to=2-3]
		\arrow["\aop"{description}, draw=none, from=1, to=2]
		\arrow["\dag"{description}, draw=none, from=0, to=1]
		\arrow["{\cp t(t, e)}"{description}, draw=none, from=3, to=4]
	\end{tikzcd}
    \quad = \quad
	\begin{tikzcd}[column sep=large]
		A & A & D \\
		A & A & D \\
		A && D \\
		A && D
		\arrow["{E(j, e)}"', "\shortmid"{marking}, from=1-3, to=1-2]
		\arrow["{E(j, t)}"{description}, from=2-2, to=2-1]
		\arrow["{E(t, e)}"{description}, from=2-3, to=2-2]
		\arrow[""{name=0, anchor=center, inner sep=0}, Rightarrow, no head, from=2-2, to=1-2]
		\arrow[""{name=1, anchor=center, inner sep=0}, Rightarrow, no head, from=2-3, to=1-3]
		\arrow["{E(j, e)}"{description}, from=3-3, to=3-1]
		\arrow[""{name=2, anchor=center, inner sep=0}, Rightarrow, no head, from=3-1, to=2-1]
		\arrow[""{name=3, anchor=center, inner sep=0}, Rightarrow, no head, from=3-3, to=2-3]
		\arrow["{E(t, e)}", "\shortmid"{marking}, from=4-3, to=4-1]
		\arrow[""{name=4, anchor=center, inner sep=0}, Rightarrow, no head, from=4-1, to=3-1]
		\arrow[""{name=5, anchor=center, inner sep=0}, Rightarrow, no head, from=4-3, to=3-3]
		\arrow["{E(j, t)}"', "\shortmid"{marking}, from=1-2, to=1-1]
		\arrow[Rightarrow, no head, from=2-1, to=1-1]
		\arrow["\aop"{description}, draw=none, from=0, to=1]
		\arrow["{\cp t(j, e)}"{description}, draw=none, from=2, to=3]
		\arrow["\aop"{description}, draw=none, from=4, to=5]
	\end{tikzcd}
    \]
    \[
	\begin{tangle}{(4,3)}[trim y=.25]
		\tgBorderA{(0,0)}{\tgColour6}{\tgColour4}{\tgColour4}{\tgColour6}
		\tgBorderA{(1,0)}{\tgColour4}{\tgColour6}{\tgColour6}{\tgColour4}
		\tgBorderA{(2,0)}{\tgColour6}{\tgColour4}{\tgColour4}{\tgColour6}
		\tgBorderA{(3,0)}{\tgColour4}{\tgColour0}{\tgColour0}{\tgColour4}
		\tgBorderA{(0,1)}{\tgColour6}{\tgColour4}{\tgColour4}{\tgColour6}
		\tgBorder{(0,1)}{0}{1}{0}{0}
		\tgBorderA{(1,1)}{\tgColour4}{\tgColour6}{\tgColour4}{\tgColour4}
		\tgBorder{(1,1)}{0}{0}{0}{1}
		\tgBorderA{(2,1)}{\tgColour6}{\tgColour4}{\tgColour4}{\tgColour4}
		\tgBorder{(2,1)}{0}{1}{0}{0}
		\tgBorderA{(3,1)}{\tgColour4}{\tgColour0}{\tgColour0}{\tgColour4}
		\tgBorder{(3,1)}{0}{0}{0}{1}
		\tgBorderA{(0,2)}{\tgColour6}{\tgColour4}{\tgColour4}{\tgColour6}
		\tgBlank{(1,2)}{\tgColour4}
		\tgBlank{(2,2)}{\tgColour4}
		\tgBorderA{(3,2)}{\tgColour4}{\tgColour0}{\tgColour0}{\tgColour4}
		\tgCell[(1,0)]{(0.5,1)}{\dag}
		\tgCell[(1,0)]{(2.5,1)}{\aop}
		\tgArrow{(1.5,1)}{0}
		\tgArrow{(0,1.5)}{1}
		\tgArrow{(3,1.5)}{3}
		\tgArrow{(0,0.5)}{1}
		\tgArrow{(2,0.5)}{1}
		\tgArrow{(1,0.5)}{3}
		\tgArrow{(3,0.5)}{3}
		\tgAxisLabel{(0.5,0.25)}{south}{j}
		\tgAxisLabel{(1.5,0.25)}{south}{t}
		\tgAxisLabel{(2.5,0.25)}{south}{j}
		\tgAxisLabel{(3.5,0.25)}{south}{e}
		\tgAxisLabel{(0.5,2.75)}{north}{t}
		\tgAxisLabel{(3.5,2.75)}{north}{e}
	\end{tangle}
	\tangleeq*
	\begin{tangle}{(4,4)}[trim y]
		\tgBorderA{(0,0)}{\tgColour6}{\tgColour4}{\tgColour4}{\tgColour6}
		\tgBorderA{(1,0)}{\tgColour4}{\tgColour6}{\tgColour6}{\tgColour4}
		\tgBorderA{(2,0)}{\tgColour6}{\tgColour4}{\tgColour4}{\tgColour6}
		\tgBorderA{(3,0)}{\tgColour4}{\tgColour0}{\tgColour0}{\tgColour4}
		\tgBorderA{(0,1)}{\tgColour6}{\tgColour4}{\tgColour4}{\tgColour6}
		\tgBorderA{(1,1)}{\tgColour4}{\tgColour6}{\tgColour4}{\tgColour4}
		\tgBorderA{(2,1)}{\tgColour6}{\tgColour4}{\tgColour0}{\tgColour4}
		\tgBorderA{(3,1)}{\tgColour4}{\tgColour0}{\tgColour0}{\tgColour0}
		\tgBorderA{(0,2)}{\tgColour6}{\tgColour4}{\tgColour4}{\tgColour6}
		\tgBorder{(0,2)}{0}{1}{0}{0}
		\tgBorderA{(1,2)}{\tgColour4}{\tgColour4}{\tgColour4}{\tgColour4}
		\tgBorder{(1,2)}{0}{1}{0}{1}
		\tgBorderA{(2,2)}{\tgColour4}{\tgColour0}{\tgColour0}{\tgColour4}
		\tgBorder{(2,2)}{0}{0}{0}{1}
		\tgBlank{(3,2)}{\tgColour0}
		\tgBorderA{(0,3)}{\tgColour6}{\tgColour4}{\tgColour4}{\tgColour6}
		\tgBlank{(1,3)}{\tgColour4}
		\tgBorderA{(2,3)}{\tgColour4}{\tgColour0}{\tgColour0}{\tgColour4}
		\tgBlank{(3,3)}{\tgColour0}
		\tgCell[(2,0)]{(1,2)}{\aop}
		\tgArrow{(0,1.5)}{1}
		\tgArrow{(2,1.5)}{3}
		\tgArrow{(0,2.5)}{1}
		\tgArrow{(0,0.5)}{1}
		\tgArrow{(2,0.5)}{1}
		\tgCell[(2,0)]{(2,1)}{\aop}
		\tgArrow{(1,0.5)}{3}
		\tgArrow{(3,0.5)}{3}
		\tgArrow{(2,2.5)}{3}
		\tgAxisLabel{(0.5,0.75)}{south}{j}
		\tgAxisLabel{(1.5,0.75)}{south}{t}
		\tgAxisLabel{(2.5,0.75)}{south}{j}
		\tgAxisLabel{(3.5,0.75)}{south}{e}
		\tgAxisLabel{(0.5,3.25)}{north}{t}
		\tgAxisLabel{(2.5,3.25)}{north}{e}
	\end{tangle}
    \]
    Let $D$ be an object of $\X$. Suppose that $(e, \aop)$ and $(e', \aop')$ are $T$-algebras with domain $D$. A \emph{$T$-algebra morphism} from $(e, \aop)$ to $(e', \aop')$ is a 2-cell $\epsilon \colon e \tto e'$ satisfying the following equation.
	\[
	\begin{tikzcd}
		{E(j, e)} & {E(t, e)} \\
		{E(j, e')} & {E(t, e')}
		\arrow["\aop", from=1-1, to=1-2]
		\arrow["{E(j, \epsilon)}"', from=1-1, to=2-1]
		\arrow["{\aop'}"', from=2-1, to=2-2]
		\arrow["{E(t, \epsilon)}", from=1-2, to=2-2]
	\end{tikzcd}
	\hspace{4em}
	\begin{tangle}{(2,5)}[trim y]
		\tgBorderA{(0,0)}{\tgColour6}{\tgColour4}{\tgColour4}{\tgColour6}
		\tgBorderA{(1,0)}{\tgColour4}{\tgColour0}{\tgColour0}{\tgColour4}
		\tgBorderA{(0,1)}{\tgColour6}{\tgColour4}{\tgColour4}{\tgColour6}
		\tgBorderA{(1,1)}{\tgColour4}{\tgColour0}{\tgColour0}{\tgColour4}
		\tgBorderA{(0,2)}{\tgColour6}{\tgColour4}{\tgColour4}{\tgColour6}
		\tgBorder{(0,2)}{0}{1}{0}{0}
		\tgBorderA{(1,2)}{\tgColour4}{\tgColour0}{\tgColour0}{\tgColour4}
		\tgBorder{(1,2)}{0}{0}{0}{1}
		\tgBorderA{(0,3)}{\tgColour6}{\tgColour4}{\tgColour4}{\tgColour6}
		\tgBorderA{(1,3)}{\tgColour4}{\tgColour0}{\tgColour0}{\tgColour4}
		\tgBorderA{(0,4)}{\tgColour6}{\tgColour4}{\tgColour4}{\tgColour6}
		\tgBorderA{(1,4)}{\tgColour4}{\tgColour0}{\tgColour0}{\tgColour4}
		\tgCell[(1,0)]{(0.5,2)}{\aop}
		\tgArrow{(0,2.5)}{1}
		\tgArrow{(0,1.5)}{1}
		\tgArrow{(1,1.5)}{3}
		\tgCell{(1,3)}{\epsilon}
		\tgArrow{(0,3.5)}{1}
		\tgArrow{(1,2.5)}{3}
		\tgArrow{(1,3.5)}{3}
		\tgArrow{(0,0.5)}{1}
		\tgArrow{(1,0.5)}{3}
		\tgAxisLabel{(0.5,0.75)}{south}{j}
		\tgAxisLabel{(1.5,0.75)}{south}{e}
		\tgAxisLabel{(0.5,4.25)}{north}{t}
		\tgAxisLabel{(1.5,4.25)}{north}{e'}
	\end{tangle}
	\tangleeq*
	\begin{tangle}{(2,5)}[trim y]
		\tgBorderA{(0,0)}{\tgColour6}{\tgColour4}{\tgColour4}{\tgColour6}
		\tgBorderA{(1,0)}{\tgColour4}{\tgColour0}{\tgColour0}{\tgColour4}
		\tgBorderA{(0,1)}{\tgColour6}{\tgColour4}{\tgColour4}{\tgColour6}
		\tgBorderA{(1,1)}{\tgColour4}{\tgColour0}{\tgColour0}{\tgColour4}
		\tgBorderA{(0,2)}{\tgColour6}{\tgColour4}{\tgColour4}{\tgColour6}
		\tgBorder{(0,2)}{0}{1}{0}{0}
		\tgBorderA{(1,2)}{\tgColour4}{\tgColour0}{\tgColour0}{\tgColour4}
		\tgBorder{(1,2)}{0}{0}{0}{1}
		\tgBorderA{(0,3)}{\tgColour6}{\tgColour4}{\tgColour4}{\tgColour6}
		\tgBorderA{(1,3)}{\tgColour4}{\tgColour0}{\tgColour0}{\tgColour4}
		\tgBorderA{(0,4)}{\tgColour6}{\tgColour4}{\tgColour4}{\tgColour6}
		\tgBorderA{(1,4)}{\tgColour4}{\tgColour0}{\tgColour0}{\tgColour4}
		\tgCell[(1,0)]{(0.5,2)}{\aop'}
		\tgArrow{(0,1.5)}{1}
		\tgArrow{(1,1.5)}{3}
		\tgArrow{(1,2.5)}{3}
		\tgArrow{(0,2.5)}{1}
		\tgArrow{(1,3.5)}{3}
		\tgArrow{(0,3.5)}{1}
		\tgCell{(1,1)}{\epsilon}
		\tgArrow{(1,0.5)}{3}
		\tgArrow{(0,0.5)}{1}
		\tgAxisLabel{(0.5,0.75)}{south}{j}
		\tgAxisLabel{(1.5,0.75)}{south}{e}
		\tgAxisLabel{(0.5,4.25)}{north}{t}
		\tgAxisLabel{(1.5,4.25)}{north}{e'}
	\end{tangle}
	\]
    $T$-algebras with domain $D$ and their morphisms form an category $T\h\Alg_D$ functorial contravariantly in $D$ and $T$.
	Denote by $U_{T, D} \colon T\h\Alg_D \to \X[D, E]$ the faithful functor sending each $T$-algebra $(e, \aop)$ to its carrier $e$.
\end{definition}

\begin{remark}
    Algebras for relative monads, and their morphisms, have been studied under that name by \textcite[Definitions~3.5.3 \& 3.5.4]{maillard2019principles}, as \emph{relative left modules} by \textcite[Definitions~4.1 \& 4.2]{lobbia2023distributive}, and as \emph{left modules} by \textcite[Definition~5.3.1]{arkor2022monadic}.
\end{remark}

\begin{example}
	\label{relative-monad-forms-algebra}
	Let $T$ be a relative monad. Then $(t, \dag)$ forms a $T$-algebra (\cf{}~\cref{monoid-forms-left-action}).
\end{example}

\begin{definition}
	\label{opalgebra}
    Let $T$ be a relative monad. An \emph{opalgebra for $T$} (or simply \emph{$T$-opalgebra}) comprises
    \begin{enumerate}
		\item an object $B$, the \emph{codomain};
        \item a tight-cell $a \colon A \to B$, the \emph{carrier} or \emph{underlying tight-cell};
        \item a 2-cell $\oop \colon E(j, t) \tto B(a, a)$, the \emph{extension operator},
    \end{enumerate}
    satisfying the following equations.
	\[
	\begin{tikzcd}[column sep=large]
		A & A \\
		A & A \\
		A & A \\
		A & A
		\arrow["{E(j, t)}"{description}, from=3-2, to=3-1]
		\arrow["{B(a, a)}", "\shortmid"{marking}, from=4-2, to=4-1]
		\arrow[""{name=0, anchor=center, inner sep=0}, Rightarrow, no head, from=3-2, to=4-2]
		\arrow[""{name=1, anchor=center, inner sep=0}, Rightarrow, no head, from=3-1, to=4-1]
		\arrow["{E(j, j)}"{description}, from=2-2, to=2-1]
		\arrow[""{name=2, anchor=center, inner sep=0}, Rightarrow, no head, from=2-2, to=3-2]
		\arrow[""{name=3, anchor=center, inner sep=0}, Rightarrow, no head, from=2-1, to=3-1]
		\arrow[""{name=4, anchor=center, inner sep=0}, Rightarrow, no head, from=1-2, to=2-2]
		\arrow[Rightarrow, no head, from=1-2, to=1-1]
		\arrow[""{name=5, anchor=center, inner sep=0}, Rightarrow, no head, from=1-1, to=2-1]
		\arrow["{\pc j}"{description}, draw=none, from=4, to=5]
		\arrow["{E(j, \eta)}"{description}, draw=none, from=2, to=3]
		\arrow["\oop"{description}, draw=none, from=0, to=1]
	\end{tikzcd}
	\quad = \quad
	\begin{tikzcd}
		A & A \\
		A & A
		\arrow["{B(a, a)}", "\shortmid"{marking}, from=2-2, to=2-1]
		\arrow[""{name=0, anchor=center, inner sep=0}, Rightarrow, no head, from=1-2, to=2-2]
		\arrow[Rightarrow, no head, from=1-2, to=1-1]
		\arrow[""{name=1, anchor=center, inner sep=0}, Rightarrow, no head, from=1-1, to=2-1]
		\arrow["{\pc a}"{description}, draw=none, from=0, to=1]
	\end{tikzcd}
	\hspace{4em}
	\begin{tangle}{(2,4)}[trim y]
		\tgBlank{(0,0)}{\tgColour6}
		\tgBlank{(1,0)}{\tgColour6}
		\tgBorderA{(0,1)}{\tgColour6}{\tgColour6}{\tgColour4}{\tgColour6}
		\tgBorderA{(1,1)}{\tgColour6}{\tgColour6}{\tgColour6}{\tgColour4}
		\tgBorderA{(0,2)}{\tgColour6}{\tgColour4}{\tgColour10}{\tgColour6}
		\tgBorderA{(1,2)}{\tgColour4}{\tgColour6}{\tgColour6}{\tgColour10}
		\tgBorderA{(0,3)}{\tgColour6}{\tgColour10}{\tgColour10}{\tgColour6}
		\tgBorderA{(1,3)}{\tgColour10}{\tgColour6}{\tgColour6}{\tgColour10}
		\tgCell[(1,0)]{(0.5,2)}{\oop}
		\tgArrow{(1,1.5)}{3}
		\tgArrow{(1,2.5)}{3}
		\tgArrow{(0,1.5)}{1}
		\tgArrow{(0,2.5)}{1}
		\tgCell[(1,0)]{(0.5,1)}{\eta}
		\tgAxisLabel{(0.5,3.25)}{north}{a}
		\tgAxisLabel{(1.5,3.25)}{north}{a}
	\end{tangle}
	\tangleeq*
	\begin{tangle}{(2,4)}[trim y]
		\tgBlank{(0,0)}{\tgColour6}
		\tgBlank{(1,0)}{\tgColour6}
		\tgBorderC{(0,1)}{3}{\tgColour6}{\tgColour10}
		\tgBorderC{(1,1)}{2}{\tgColour6}{\tgColour10}
		\tgBorderA{(0,2)}{\tgColour6}{\tgColour10}{\tgColour10}{\tgColour6}
		\tgBorderA{(1,2)}{\tgColour10}{\tgColour6}{\tgColour6}{\tgColour10}
		\tgBorderA{(0,3)}{\tgColour6}{\tgColour10}{\tgColour10}{\tgColour6}
		\tgBorderA{(1,3)}{\tgColour10}{\tgColour6}{\tgColour6}{\tgColour10}
		\tgArrow{(0.5,1)}{0}
		\tgArrow{(0,1.5)}{1}
		\tgArrow{(1,1.5)}{3}
		\tgArrow{(0,2.5)}{1}
		\tgArrow{(1,2.5)}{3}
		\tgAxisLabel{(0.5,3.25)}{north}{a}
		\tgAxisLabel{(1.5,3.25)}{north}{a}
	\end{tangle}
	\]
	\[
	\begin{tikzcd}[column sep=large]
		A & A & A \\
		A & A & A \\
		A && A
		\arrow["{E(j, t)}"', "\shortmid"{marking}, from=1-3, to=1-2]
		\arrow["{E(j, t)}"', "\shortmid"{marking}, from=1-2, to=1-1]
		\arrow["{B(a, a)}"{description}, from=2-3, to=2-2]
		\arrow["{B(a, a)}"{description}, from=2-2, to=2-1]
		\arrow[""{name=0, anchor=center, inner sep=0}, Rightarrow, no head, from=1-3, to=2-3]
		\arrow[""{name=1, anchor=center, inner sep=0}, Rightarrow, no head, from=1-2, to=2-2]
		\arrow[""{name=2, anchor=center, inner sep=0}, Rightarrow, no head, from=1-1, to=2-1]
		\arrow["{B(a, a)}", "\shortmid"{marking}, from=3-3, to=3-1]
		\arrow[""{name=3, anchor=center, inner sep=0}, Rightarrow, no head, from=2-3, to=3-3]
		\arrow[""{name=4, anchor=center, inner sep=0}, Rightarrow, no head, from=2-1, to=3-1]
		\arrow["\oop"{description}, draw=none, from=0, to=1]
		\arrow["\oop"{description}, draw=none, from=1, to=2]
		\arrow["{\cp a(a, a)}"{description}, draw=none, from=3, to=4]
	\end{tikzcd}
	\quad = \quad
	\begin{tikzcd}[column sep=large]
		A & A & A \\
		A & A & A \\
		A && A \\
		A && A
		\arrow["{E(j, t)}"', "\shortmid"{marking}, from=1-2, to=1-1]
		\arrow["{E(t, t)}"{description}, from=2-3, to=2-2]
		\arrow["{E(j, t)}"{description}, from=2-2, to=2-1]
		\arrow[""{name=0, anchor=center, inner sep=0}, Rightarrow, no head, from=1-2, to=2-2]
		\arrow[""{name=1, anchor=center, inner sep=0}, Rightarrow, no head, from=1-1, to=2-1]
		\arrow["{E(j, t)}"{description}, from=3-3, to=3-1]
		\arrow[""{name=2, anchor=center, inner sep=0}, Rightarrow, no head, from=2-3, to=3-3]
		\arrow[""{name=3, anchor=center, inner sep=0}, Rightarrow, no head, from=2-1, to=3-1]
		\arrow["{B(a, a)}", "\shortmid"{marking}, from=4-3, to=4-1]
		\arrow[""{name=4, anchor=center, inner sep=0}, Rightarrow, no head, from=3-3, to=4-3]
		\arrow[""{name=5, anchor=center, inner sep=0}, Rightarrow, no head, from=3-1, to=4-1]
		\arrow["{E(j, t)}"', "\shortmid"{marking}, from=1-3, to=1-2]
		\arrow[""{name=6, anchor=center, inner sep=0}, Rightarrow, no head, from=1-3, to=2-3]
		\arrow["\dag"{description}, draw=none, from=6, to=0]
		\arrow["{\cp t(j, t)}"{description}, draw=none, from=2, to=3]
		\arrow["\oop"{description}, draw=none, from=4, to=5]
		\arrow["{=}"{description}, draw=none, from=0, to=1]
	\end{tikzcd}
	\]
	\[
	\begin{tangle}{(4,3)}[trim y=.25]
		\tgBorderA{(0,0)}{\tgColour6}{\tgColour4}{\tgColour4}{\tgColour6}
		\tgBorderA{(1,0)}{\tgColour4}{\tgColour6}{\tgColour6}{\tgColour4}
		\tgBorderA{(2,0)}{\tgColour6}{\tgColour4}{\tgColour4}{\tgColour6}
		\tgBorderA{(3,0)}{\tgColour4}{\tgColour6}{\tgColour6}{\tgColour4}
		\tgBorderA{(0,1)}{\tgColour6}{\tgColour4}{\tgColour10}{\tgColour6}
		\tgBorderA{(1,1)}{\tgColour4}{\tgColour6}{\tgColour10}{\tgColour10}
		\tgBorderA{(2,1)}{\tgColour6}{\tgColour4}{\tgColour10}{\tgColour10}
		\tgBorderA{(3,1)}{\tgColour4}{\tgColour6}{\tgColour6}{\tgColour10}
		\tgBorderA{(0,2)}{\tgColour6}{\tgColour10}{\tgColour10}{\tgColour6}
		\tgBlank{(1,2)}{\tgColour10}
		\tgBlank{(2,2)}{\tgColour10}
		\tgBorderA{(3,2)}{\tgColour10}{\tgColour6}{\tgColour6}{\tgColour10}
		\tgCell[(1,0)]{(0.5,1)}{\oop}
		\tgCell[(1,0)]{(2.5,1)}{\oop}
		\tgArrow{(1.5,1)}{0}
		\tgArrow{(0,1.5)}{1}
		\tgArrow{(3,1.5)}{3}
		\tgArrow{(0,0.5)}{1}
		\tgArrow{(2,0.5)}{1}
		\tgArrow{(1,0.5)}{3}
		\tgArrow{(3,0.5)}{3}
		\tgAxisLabel{(0.5,0.25)}{south}{j}
		\tgAxisLabel{(1.5,0.25)}{south}{t}
		\tgAxisLabel{(2.5,0.25)}{south}{j}
		\tgAxisLabel{(3.5,0.25)}{south}{t}
		\tgAxisLabel{(0.5,2.75)}{north}{a}
		\tgAxisLabel{(3.5,2.75)}{north}{a}
	\end{tangle}
	\tangleeq*
	\begin{tangle}{(4,4)}[trim y]
		\tgBorderA{(0,0)}{\tgColour6}{\tgColour4}{\tgColour4}{\tgColour6}
		\tgBorderA{(1,0)}{\tgColour4}{\tgColour6}{\tgColour6}{\tgColour4}
		\tgBorderA{(2,0)}{\tgColour6}{\tgColour4}{\tgColour4}{\tgColour6}
		\tgBorderA{(3,0)}{\tgColour4}{\tgColour6}{\tgColour6}{\tgColour4}
		\tgBorderA{(0,1)}{\tgColour6}{\tgColour4}{\tgColour4}{\tgColour6}
		\tgBorderA{(1,1)}{\tgColour4}{\tgColour6}{\tgColour4}{\tgColour4}
		\tgBorderA{(2,1)}{\tgColour6}{\tgColour4}{\tgColour6}{\tgColour4}
		\tgBorderA{(3,1)}{\tgColour4}{\tgColour6}{\tgColour6}{\tgColour6}
		\tgBorderA{(0,2)}{\tgColour6}{\tgColour4}{\tgColour10}{\tgColour6}
		\tgBorderA{(1,2)}{\tgColour4}{\tgColour4}{\tgColour10}{\tgColour10}
		\tgBorderA{(2,2)}{\tgColour4}{\tgColour6}{\tgColour6}{\tgColour10}
		\tgBlank{(3,2)}{\tgColour6}
		\tgBorderA{(0,3)}{\tgColour6}{\tgColour10}{\tgColour10}{\tgColour6}
		\tgBlank{(1,3)}{\tgColour10}
		\tgBorderA{(2,3)}{\tgColour10}{\tgColour6}{\tgColour6}{\tgColour10}
		\tgBlank{(3,3)}{\tgColour6}
		\tgCell[(2,0)]{(1,2)}{\oop}
		\tgArrow{(0,1.5)}{1}
		\tgArrow{(2,1.5)}{3}
		\tgArrow{(0,2.5)}{1}
		\tgArrow{(0,0.5)}{1}
		\tgArrow{(2,0.5)}{1}
		\tgCell[(2,0)]{(2,1)}{\dag}
		\tgArrow{(1,0.5)}{3}
		\tgArrow{(3,0.5)}{3}
		\tgArrow{(2,2.5)}{3}
		\tgAxisLabel{(0.5,0.75)}{south}{j}
		\tgAxisLabel{(1.5,0.75)}{south}{t}
		\tgAxisLabel{(2.5,0.75)}{south}{j}
		\tgAxisLabel{(3.5,0.75)}{south}{t}
		\tgAxisLabel{(0.5,3.25)}{north}{a}
		\tgAxisLabel{(2.5,3.25)}{north}{a}
	\end{tangle}
	\]
    Let $B$ be an object of $\X$. Suppose that $(a, \oop)$ and $(a', \oop')$ are $T$-opalgebras with codomain $B$. A \emph{$T$-opalgebra morphism} from $(a, \oop)$ to $(a', \oop')$ is a 2-cell $\alpha \colon a \tto a'$ satisfying the following equation.
	\[
	\begin{tikzcd}
		{E(j, t)} & {B(a, a)} \\
		{B(a', a')} & {B(a, a')}
		\arrow["{\oop'}"', from=1-1, to=2-1]
		\arrow["\oop", from=1-1, to=1-2]
		\arrow["{B(a, \alpha)}", from=1-2, to=2-2]
		\arrow["{B(\alpha, a')}"', from=2-1, to=2-2]
	\end{tikzcd}
	\hspace{4em}
	\begin{tangle}{(2,4)}[trim y]
		\tgBorderA{(0,0)}{\tgColour6}{\tgColour4}{\tgColour4}{\tgColour6}
		\tgBorderA{(1,0)}{\tgColour4}{\tgColour6}{\tgColour6}{\tgColour4}
		\tgBorderA{(0,1)}{\tgColour6}{\tgColour4}{\tgColour10}{\tgColour6}
		\tgBorderA{(1,1)}{\tgColour4}{\tgColour6}{\tgColour6}{\tgColour10}
		\tgBorderA{(0,2)}{\tgColour6}{\tgColour10}{\tgColour10}{\tgColour6}
		\tgBorderA{(1,2)}{\tgColour10}{\tgColour6}{\tgColour6}{\tgColour10}
		\tgBorderA{(0,3)}{\tgColour6}{\tgColour10}{\tgColour10}{\tgColour6}
		\tgBorderA{(1,3)}{\tgColour10}{\tgColour6}{\tgColour6}{\tgColour10}
		\tgCell[(1,0)]{(0.5,1)}{\oop}
		\tgArrow{(0,1.5)}{1}
		\tgArrow{(0,0.5)}{1}
		\tgArrow{(1,0.5)}{3}
		\tgCell{(1,2)}{\alpha}
		\tgArrow{(0,2.5)}{1}
		\tgArrow{(1,1.5)}{3}
		\tgArrow{(1,2.5)}{3}
		\tgAxisLabel{(0.5,0.75)}{south}{j}
		\tgAxisLabel{(1.5,0.75)}{south}{t}
		\tgAxisLabel{(0.5,3.25)}{north}{a}
		\tgAxisLabel{(1.5,3.25)}{north}{a'}
	\end{tangle}
	\tangleeq*
	\begin{tangle}{(2,4)}[trim y]
		\tgBorderA{(0,0)}{\tgColour6}{\tgColour4}{\tgColour4}{\tgColour6}
		\tgBorderA{(1,0)}{\tgColour4}{\tgColour6}{\tgColour6}{\tgColour4}
		\tgBorderA{(0,1)}{\tgColour6}{\tgColour4}{\tgColour10}{\tgColour6}
		\tgBorderA{(1,1)}{\tgColour4}{\tgColour6}{\tgColour6}{\tgColour10}
		\tgBorderA{(0,2)}{\tgColour6}{\tgColour10}{\tgColour10}{\tgColour6}
		\tgBorderA{(1,2)}{\tgColour10}{\tgColour6}{\tgColour6}{\tgColour10}
		\tgBorderA{(0,3)}{\tgColour6}{\tgColour10}{\tgColour10}{\tgColour6}
		\tgBorderA{(1,3)}{\tgColour10}{\tgColour6}{\tgColour6}{\tgColour10}
		\tgCell[(1,0)]{(0.5,1)}{\oop'}
		\tgArrow{(0,0.5)}{1}
		\tgArrow{(1,0.5)}{3}
		\tgArrow{(1,1.5)}{3}
		\tgCell{(0,2)}{\alpha}
		\tgArrow{(0,1.5)}{1}
		\tgArrow{(1,2.5)}{3}
		\tgArrow{(0,2.5)}{1}
		\tgAxisLabel{(0.5,0.75)}{south}{j}
		\tgAxisLabel{(1.5,0.75)}{south}{t}
		\tgAxisLabel{(0.5,3.25)}{north}{a}
		\tgAxisLabel{(1.5,3.25)}{north}{a'}
	\end{tangle}
	\]
	$T$-opalgebras with codomain $B$ and their morphisms form a category $T\h\Opalg_B$ functorial covariantly in $B$ and contravariantly in $T$.
    Denote by $U_{T, B} \colon T\h\Opalg_B \to \X[A, B]$ the faithful functor sending each $T$-opalgebra $(a, \oop)$ to its carrier $a$.
\end{definition}

\begin{remark}
    Opalgebras for relative monads, and their morphisms, have been studied as \emph{modules over a relative monad} by \citeauthor{ahrens2012initiality}~\cites[Definitions~2.90 \& 2.94]{ahrens2012initiality}[Definitions~9 \& 14]{ahrens2016modules}, as \emph{Kleisli algebras} by \textcite[Definitions~3.5.6 \& 3.5.7]{maillard2019principles}, and as \emph{relative right modules} by \textcite[Definitions~6.1 \& 6.2]{lobbia2023distributive}.
\end{remark}

\begin{example}
	\label{relative-monad-forms-opalgebra}
	Let $T$ be a relative monad. Then $(t, \dag)$ forms a $T$-opalgebra (\cf{}~\cref{monoid-forms-right-action}).
\end{example}

It will be useful to observe that $T$-opalgebras admit an alternative description in terms of the loose-monad $E(j, T)$ associated to $T$ (\cref{relative-monads-are-loose-monads}). As a consequence of the following lemma, $j\h\Opalg_B$ forms the category of \emph{extraordinary transformations from $j$} of \cite[369]{street1978yoneda}. The connection to the \emph{bijective-on-objects} tight-cells \loccit{} will be discussed in \cref{opalgebra-objects}.

\begin{lemma}
	\label{opalgebras-are-loose-monad-morphisms}
	Let $\jAE$ be a tight-cell and let $T$ be a $j$-monad. A $T$-opalgebra comprises a tight-cell $a \colon A \to B$ and a loose-monad morphism from $E(j, T)$ to $B(a, a)$.
\end{lemma}

\begin{proof}
	The unit and compatibility laws for a $T$-opalgebra correspond respectively to the unit and multiplication laws for a loose-monad morphism.
\end{proof}

In \cref{relative-adjunctions-and-(op)algebras}, we shall discuss the relationship between algebras, opalgebras, and relative adjunctions, which motivates the study of universal algebras (called \emph{algebra objects}) in \cref{algebra-objects} and universal opalgebras (called \emph{opalgebra objects}) in \cref{opalgebra-objects}. However, before we do so, we shall consider the relationship between the definitions of algebras and opalgebras above, and the monoid presentation of a relative monad (\cref{relative-monads-as-monoids}).

\subsection{Algebras and opalgebras as actions}

In the non-relative setting, algebras and opalgebras for a monad are 1-cells equipped with actions compatible with the monad structure~\cite{street1972formal}. Recall that a monad $T$ on an object $A$ of a 2-category $\K$ is a monoid in $\K[A, A]$. For any object $D \in \K$, the hom-category $\K[D, A]$ forms a left-$\K[A, A]$-actegory via precomposition. A $T$-action in $\K[D, A]$ is precisely a $T$-algebra with domain $D$. Symmetrically, for any object $B \in \K$, the hom-category $\K[A, B]$ forms a right-$\K[A, A]$-actegory via postcomposition. A $T$-action in $\K[A, B]$ is precisely a $T$-opalgebra with codomain $B$.

We should like to characterise algebras and opalgebras for relative monads similarly, following our treatment of relative monads as monoids in skew-multicategories (\cref{skew-multicategorical-hom-categories}). However, we must generalise the notion of action accordingly, to account for skewness: in particular, we introduce the notion of \emph{skew-multiactegory}.
A skew-multiactegory may be thought of as that which acts on a skew-multicategory, in the same way that an actegory acts on a monoidal category.
Just as a skew-multicategory has multimorphisms, rather than a tensor ${\otimes} \colon \M \times \M \to \M$, a skew-multiactegory has multimorphisms, rather than an action ${\obslash} \colon \M \times \A \to \A$.

In the following, \cref{algebras,algebra-objects}, which treat algebras and algebra objects respectively, proceed analogously to \cref{opalgebras,opalgebra-objects}, which treat opalgebras and opalgebra objects respectively. The sections are structured analogously to \cref{skew-multicategorical-hom-categories}, which treated relative monads. However, since opalgebras are not formally dual to algebras (in contrast to the notions of opalgebra and algebra in a 2-category~\cite[\S4]{street1972formal}), they must be treated separately. This is reflected in their theory, which, though similar, is not identical.

\begin{remark}
	It appears likely that there exists a two-dimensional treatment of relative monads, in which algebras and opalgebras become special cases of relative monad morphisms, analogously to the formal theory of monads~\cite{street1972formal}. However, since the setting of \ve{}s and skew-multicategories already incurs a significant increase in complexity over 2-categories and monoidal categories, we defer such a two-dimensional treatment to future work.
\end{remark}

\subsection{Algebras as left-actions in a skew-multiactegory}
\label{algebras}

We start by introducing the analogue of an actegory for a skew-multicategory. There are two variants, acting on the left and on the right respectively. Since the definitions are almost identical, we define them simultaneously.

\begin{definition}
	Let $\M$ be an associative-normal left-skew-multicategory. A \emph{right- } (respectively \emph{left-}) \emph{$\M$-multiactegory} $\A$ comprises
	\begin{enumerate}
		\item a class $\ob\A$ of \emph{objects};
		\item a class $\A(A, M_1, \ldots, M_n; A')$ of \emph{multimorphisms} for each $n \geq 0$, $M_1, \ldots, M_n \in \ob\M + \{ \bullet \}$ and $A, A' \in \ob\A$;
		\item an \emph{identity} multimorphism $1_A \in \A(A; A)$ for each $A \in \ob{\A}$;
        \item for each multimorphism $g \colon A', \bullet^{m_0}, M_1, \bullet^{m_1}, \ldots, \bullet^{m_{n - 1}}, M_n, \bullet^{m_n} \to A''$ where $M_1, \ldots, M_n \in \ob\M$, $A', A'' \in \ob\A$ and $n, m_i \geq 0$, morphism $f_0 \colon A, \vec{M_0} \to A'$ in $\A$, and multimorphisms $f_1 \colon \vec{M_1} \to M_1$, \ldots, $f_n \colon \vec{M_n} \to M_n$ in $\M$, a \emph{composite} multimorphism in~$\A$: \[((f_0, f_1, \ldots, f_n) \d g) \colon A, \vec{M_0}, \bullet^{m_0}, \vec{M_1}, \bullet^{m_1}, \ldots, \bullet^{m_{n - 1}}, \vec{M_n}, \bullet^{m_n} \to A''\]
		\item a \emph{left-unitor} function \[\lambda_{(A, \vec M; A'), k} \colon \A(A, M_1, \ldots, M_k, M_{k + 1}, \ldots, M_n; A') \to \A(A, M_1, \ldots, M_k, \bullet, M_{k + 1}, \ldots, M_n; A')\] for each $0 \leq k \leq n$ (respectively for each $0 \leq k < n$);
		\item a \emph{right-unitor} function \[\rho_{(A, \vec M; A'),k} \colon \A(A, M_1, \ldots, M_k, \bullet, M_{k + 1}, \ldots, M_n; A') \to \A(A, M_1, \ldots, M_k, M_{k + 1}, \ldots, M_n; A')\] for each $0 < k \leq n$ (respectively for each $0 \leq k \leq n$),
	\end{enumerate}
	such that composition in $\A$ coheres with identities and composites in $\M$ in the following sense,
	\begin{align*}
		(1_A, 1_{M_1}, \ldots, 1_{M_n}) \d g & = g \\
		(((f_0^0, f_0^1, \ldots, f_0^{m_0}) \d g_0), &((f_1^1, \ldots, f_1^{m_1}) \d g_1), \ldots, ((f_n^1, \ldots, f_n^{m_n}) \d g_n)) \d h \\ & = (f_0^0, f_0^1, \ldots, f_0^{m_0}, f_1^1, \ldots, f_1^{m_1}, \ldots, f_n^1, \ldots, f_n^{m_n}) \d (g_0, g_1, \ldots, g_n) \d h
	\end{align*}
	and that the left- and right-unitors cohere with composition, the left- and right-unitors cohere with themselves, and the right-unitor is inverse to the left-unitor, in the sense of \cref{alpha-normal-left-skew-multicategory}.

	$\A$ is \emph{left-normal} when $\lambda$ is invertible; and is \emph{right-normal} when $\rho$ is invertible.

	A \emph{functor} between left/right-$\M$-multiactegories is a homomorphism of left/right-$\M$-multi\-actegories.
\end{definition}

By notational convention, we shall write the hom-sets of a left-multiactegory $\A$ in the form $\A(M_1, \ldots, M_n, A; A')$.

\begin{definition}
    Given a left/right-skew-multiactegory $\A$, denote by $\A_1$ the category of unary multimorphisms, \ie{} the category whose objects are those of $\A$ and whose hom-set $\A_1(A, A') \defeq \A(A; A')$.
\end{definition}

\begin{definition}
	\label{left-action}
    Let $\M$ be an associative-normal left-skew-multicategory and let $\A$ be a left-$\M$-multiactegory. An \emph{action} in $\A$ for a monoid $(M, m, u)$ in $\M$ (or simply \emph{$(M, m, u)$-action}) comprises
    \begin{enumerate}
        \item an object $A \in \A$, the \emph{carrier};
        \item a multimorphism $a \colon M, A \to A$, the \emph{action},
    \end{enumerate}
    satisfying the following equations.
	\begin{align*}
		(u, 1_A) \d a & = \lambda_{(A; A),0}(1_A) &
		(m, 1_A) \d a & = (1_M, a) \d a
	\end{align*}
    An \emph{action homomorphism} from $(A, a)$ to $(A', a')$ is a unary multimorphism $f \colon A \to A'$ satisfying the following equation.
	\begin{align*}
		a \d f = (1_M, f) \d a'
	\end{align*}
	$(M, m, u)$-actions and their homomorphisms form a category $\Act(\A, (M, m, u))$ functorial covariantly in $\A$ and contravariantly in $(M, m, u)$. Denote by $U_{\A, (M, m, u)} \colon \Act(\A, (M, m, u)) \to \A_1$ the faithful functor sending each action $(A, a)$ to its carrier $a$.
\end{definition}

Every skew-multicategory acts on itself trivially.

\begin{proposition}
    \label{monoid-forms-left-action}
    Let $\M$ be an associative-normal left-skew-multicategory. Then $\M$ forms a left-$\M$-multiactegory. Furthermore any monoid $(M, m, u)$ in $\M$ forms an $(M, m, u)$-action therein.
\end{proposition}

\begin{proof}
    The left-$\M$-multiactegory structure is defined to have the same objects, multimorphisms, and composition as $\M$, from which the laws hold trivially. Given a monoid $(M, m, u)$, we define an action $(M, m)$: the unit and multiplication laws follow from those of the monoid.
\end{proof}

\begin{proposition}
	Let $\X$ be a \vdc{} with a loose-adjunction $j_* \adj j^* \colon E \lto A$ and an object $D$. The loose-cells $e \colon D \lto E$ in $\X$ together with 2-cells of the form $p_1, j^*, p_2, j^*, \ldots, j^*, p_n, j^*, e \tto e'$ form a left-$\X\lh{j_* \adj j^*}$-multiactegory.
\end{proposition}

\begin{proof}
    We define a left-$\X\lh{j_* \adj j^*}$-multiactegory $\X\lh{D, j_* \adj j^*}$ as follows. The class of objects is given by those of $\X\lh{D, E}$. The left- and right-normal multimorphisms $p_1, \ldots, p_n, e \to e'$ ($n \geq 0$) are 2-cells  $p_1, j^*, p_2, j^*, \ldots, j^*, p_n, j^*, e \tto e'$.
	\[\begin{tikzcd}
		E & A & \cdots & E & A & E & D \\
		E &&&&&& D
		\arrow["{j^*}"', "\shortmid"{marking}, from=1-3, to=1-2]
		\arrow["{p_1}"', "\shortmid"{marking}, from=1-2, to=1-1]
		\arrow["{e'}", "\shortmid"{marking}, from=2-7, to=2-1]
		\arrow[""{name=0, anchor=center, inner sep=0}, Rightarrow, no head, from=1-7, to=2-7]
		\arrow["{j^*}"', "\shortmid"{marking}, from=1-4, to=1-3]
		\arrow["{p_n}"', "\shortmid"{marking}, from=1-5, to=1-4]
		\arrow["{j^*}"', "\shortmid"{marking}, from=1-6, to=1-5]
		\arrow["e"', "\shortmid"{marking}, from=1-7, to=1-6]
		\arrow[""{name=1, anchor=center, inner sep=0}, Rightarrow, no head, from=1-1, to=2-1]
		\arrow["\phi"{description}, draw=none, from=0, to=1]
	\end{tikzcd}\]
	The general multimorphisms, composition structure, and left- and right-unitors are defined as in \cref{skew-multicategorical-hom}, and satisfy the laws for the same reasons.
	Functoriality in $D$ follows from pasting on the right.
\end{proof}

\begin{definition}
	\label{XDj}
	Let $\X$ be an equipment with a tight-cell $\jAE$. Denote by $\X\lh{D, j}$ the $\X\lh j$-multiactegory $\X\lh{D, E(1, j) \adj E(j, 1)}$. Define $\X[D, j]$ to be the full sub-multiactegory of $\X\lh{D, j}$ spanned by the representable loose-cells.
\end{definition}

In particular $\X\lh{D, j}_1 = \X\lh{D, E}$ and $\X[D, j]_1 = \X[D, E]$. We shall unwrap the definition of an action in $\X[D, j]$ to compare it with the classical definition of an algebra for a relative monad. Explicitly, an action in $\X[D, j]$ comprises
\begin{enumerate}
	\item a tight-cell $e \colon D \to E$;
	\item a 2-cell ${\lambda \colon E(1, t), E(j, 1), E(1, e) \tto E(1, e)}$,
\end{enumerate}
satisfying the following equations.
\[
\begin{tikzcd}[column sep=large]
	E & A & E & D \\
	E & A & E & D \\
	E &&& D
	\arrow["{E(1, e)}"{description}, from=2-4, to=2-3]
	\arrow["{E(j, 1)}"{description}, from=2-3, to=2-2]
	\arrow["{E(1, t)}"{description}, from=2-2, to=2-1]
	\arrow["{E(1, e)}", "\shortmid"{marking}, from=3-4, to=3-1]
	\arrow[""{name=0, anchor=center, inner sep=0}, Rightarrow, no head, from=2-4, to=3-4]
	\arrow[""{name=1, anchor=center, inner sep=0}, Rightarrow, no head, from=2-1, to=3-1]
	\arrow[""{name=2, anchor=center, inner sep=0}, Rightarrow, no head, from=1-2, to=2-2]
	\arrow[""{name=3, anchor=center, inner sep=0}, Rightarrow, no head, from=1-1, to=2-1]
	\arrow["{E(1, j)}"', "\shortmid"{marking}, from=1-2, to=1-1]
	\arrow["{E(1, e)}"', "\shortmid"{marking}, from=1-4, to=1-3]
	\arrow["{E(j, 1)}"', "\shortmid"{marking}, from=1-3, to=1-2]
	\arrow[""{name=4, anchor=center, inner sep=0}, Rightarrow, no head, from=1-4, to=2-4]
	\arrow[""{name=5, anchor=center, inner sep=0}, Rightarrow, no head, from=1-3, to=2-3]
	\arrow["\lambda"{description}, draw=none, from=0, to=1]
	\arrow["{E(1, \eta)}"{description}, draw=none, from=2, to=3]
	\arrow["{=}"{description}, draw=none, from=4, to=5]
	\arrow["{=}"{description}, draw=none, from=5, to=2]
\end{tikzcd}
\quad = \quad
\begin{tikzcd}[column sep=large]
	E & A & E & D \\
	E && E & D \\
	E &&& D
	\arrow["{E(1, e)}"{description}, from=2-4, to=2-3]
	\arrow[""{name=0, anchor=center, inner sep=0}, Rightarrow, no head, from=1-1, to=2-1]
	\arrow["{E(1, j)}"', "\shortmid"{marking}, from=1-2, to=1-1]
	\arrow["{E(1, e)}"', "\shortmid"{marking}, from=1-4, to=1-3]
	\arrow["{E(j, 1)}"', "\shortmid"{marking}, from=1-3, to=1-2]
	\arrow[""{name=1, anchor=center, inner sep=0}, Rightarrow, no head, from=1-4, to=2-4]
	\arrow[""{name=2, anchor=center, inner sep=0}, Rightarrow, no head, from=1-3, to=2-3]
	\arrow["\shortmid"{marking}, Rightarrow, no head, from=2-3, to=2-1]
	\arrow["{E(1, e)}", "\shortmid"{marking}, from=3-4, to=3-1]
	\arrow[""{name=3, anchor=center, inner sep=0}, Rightarrow, no head, from=2-4, to=3-4]
	\arrow[""{name=4, anchor=center, inner sep=0}, Rightarrow, no head, from=2-1, to=3-1]
	\arrow["{=}"{description}, draw=none, from=1, to=2]
	\arrow["{\cp j}"{description}, draw=none, from=2, to=0]
	\arrow["\opcart"{description}, draw=none, from=3, to=4]
\end{tikzcd}
\]\[
\begin{tangle}{(3,4)}[trim y]
	\tgBorderA{(0,0)}{\tgColour4}{\tgColour6}{\tgColour6}{\tgColour4}
	\tgBorderA{(1,0)}{\tgColour6}{\tgColour4}{\tgColour4}{\tgColour6}
	\tgBorderA{(2,0)}{\tgColour4}{\tgColour0}{\tgColour0}{\tgColour4}
	\tgBorderA{(0,1)}{\tgColour4}{\tgColour6}{\tgColour6}{\tgColour4}
	\tgBorderA{(1,1)}{\tgColour6}{\tgColour4}{\tgColour4}{\tgColour6}
	\tgBorderA{(2,1)}{\tgColour4}{\tgColour0}{\tgColour0}{\tgColour4}
	\tgBorderA{(0,2)}{\tgColour4}{\tgColour6}{\tgColour4}{\tgColour4}
	\tgBorderA{(1,2)}{\tgColour6}{\tgColour4}{\tgColour0}{\tgColour4}
	\tgBorderA{(2,2)}{\tgColour4}{\tgColour0}{\tgColour0}{\tgColour0}
	\tgBlank{(0,3)}{\tgColour4}
	\tgBorderA{(1,3)}{\tgColour4}{\tgColour0}{\tgColour0}{\tgColour4}
	\tgBlank{(2,3)}{\tgColour0}
	\tgCell[(2,0)]{(1,2)}{\lambda}
	\tgArrow{(1,1.5)}{1}
	\tgArrow{(0,1.5)}{3}
	\tgArrow{(2,1.5)}{3}
	\tgCell{(0,1)}{\eta}
	\tgArrow{(0,0.5)}{3}
	\tgArrow{(2,0.5)}{3}
	\tgArrow{(1,2.5)}{3}
	\tgArrow{(1,0.5)}{1}
	\tgAxisLabel{(0.5,0.75)}{south}{j}
	\tgAxisLabel{(1.5,0.75)}{south}{j}
	\tgAxisLabel{(2.5,0.75)}{south}{e}
	\tgAxisLabel{(1.5,3.25)}{north}{e}
\end{tangle}
\tangleeq*
\begin{tangle}{(3,4)}[trim y]
	\tgBorderA{(0,0)}{\tgColour4}{\tgColour6}{\tgColour6}{\tgColour4}
	\tgBorderA{(1,0)}{\tgColour6}{\tgColour4}{\tgColour4}{\tgColour6}
	\tgBorderA{(2,0)}{\tgColour4}{\tgColour0}{\tgColour0}{\tgColour4}
	\tgBorderA{(0,1)}{\tgColour4}{\tgColour6}{\tgColour6}{\tgColour4}
	\tgBorderA{(1,1)}{\tgColour6}{\tgColour4}{\tgColour4}{\tgColour6}
	\tgBorderA{(2,1)}{\tgColour4}{\tgColour0}{\tgColour0}{\tgColour4}
	\tgBorderC{(0,2)}{0}{\tgColour4}{\tgColour6}
	\tgBorderC{(1,2)}{1}{\tgColour4}{\tgColour6}
	\tgBorderA{(2,2)}{\tgColour4}{\tgColour0}{\tgColour0}{\tgColour4}
	\tgBlank{(0,3)}{\tgColour4}
	\tgBlank{(1,3)}{\tgColour4}
	\tgBorderA{(2,3)}{\tgColour4}{\tgColour0}{\tgColour0}{\tgColour4}
	\tgArrow{(0.5,2)}{0}
	\tgArrow{(0,1.5)}{3}
	\tgArrow{(2,1.5)}{3}
	\tgArrow{(2,2.5)}{3}
	\tgArrow{(1,1.5)}{1}
	\tgArrow{(2,0.5)}{3}
	\tgArrow{(0,0.5)}{3}
	\tgArrow{(1,0.5)}{1}
	\tgAxisLabel{(0.5,0.75)}{south}{j}
	\tgAxisLabel{(1.5,0.75)}{south}{j}
	\tgAxisLabel{(2.5,0.75)}{south}{e}
	\tgAxisLabel{(2.5,3.25)}{north}{e}
\end{tangle}
\]
\[
\begin{tikzcd}[column sep=1.54em]
	E & A & E & A && E && D \\
	E &&& A && E && D \\
	E &&&&&&& D
	\arrow["{E(1, e)}"{description}, from=2-8, to=2-6]
	\arrow["{E(j, 1)}"{description}, from=2-6, to=2-4]
	\arrow["{E(1, t)}"{description}, from=2-4, to=2-1]
	\arrow["{E(1, e)}", "\shortmid"{marking}, from=3-8, to=3-1]
	\arrow[""{name=0, anchor=center, inner sep=0}, Rightarrow, no head, from=2-8, to=3-8]
	\arrow[""{name=1, anchor=center, inner sep=0}, Rightarrow, no head, from=2-1, to=3-1]
	\arrow["{E(1, e)}"', "\shortmid"{marking}, from=1-8, to=1-6]
	\arrow["{E(j, 1)}"', "\shortmid"{marking}, from=1-6, to=1-4]
	\arrow["{E(1, t)}"', "\shortmid"{marking}, from=1-4, to=1-3]
	\arrow["{E(j, 1)}"', "\shortmid"{marking}, from=1-3, to=1-2]
	\arrow["{E(1, t)}"', "\shortmid"{marking}, from=1-2, to=1-1]
	\arrow[""{name=2, anchor=center, inner sep=0}, Rightarrow, no head, from=1-1, to=2-1]
	\arrow[""{name=3, anchor=center, inner sep=0}, Rightarrow, no head, from=1-8, to=2-8]
	\arrow[""{name=4, anchor=center, inner sep=0}, Rightarrow, no head, from=1-4, to=2-4]
	\arrow[""{name=5, anchor=center, inner sep=0}, Rightarrow, no head, from=1-6, to=2-6]
	\arrow["\lambda"{description}, draw=none, from=0, to=1]
	\arrow["\mu"{description}, draw=none, from=4, to=2]
	\arrow["{=}"{description}, draw=none, from=3, to=5]
	\arrow["{=}"{description}, draw=none, from=5, to=4]
\end{tikzcd}
=
\begin{tikzcd}[column sep=1.54em]
	E && A && E & A & E & D \\
	E && A && E &&& D \\
	E &&&&&&& D
	\arrow["{E(1, e)}"{description}, from=2-8, to=2-5]
	\arrow["{E(j, 1)}"{description}, from=2-5, to=2-3]
	\arrow["{E(1, t)}"{description}, from=2-3, to=2-1]
	\arrow["{E(1, e)}", "\shortmid"{marking}, from=3-8, to=3-1]
	\arrow[""{name=0, anchor=center, inner sep=0}, Rightarrow, no head, from=2-8, to=3-8]
	\arrow[""{name=1, anchor=center, inner sep=0}, Rightarrow, no head, from=2-1, to=3-1]
	\arrow["{E(1, e)}"', "\shortmid"{marking}, from=1-8, to=1-7]
	\arrow["{E(j, 1)}"', "\shortmid"{marking}, from=1-7, to=1-6]
	\arrow["{E(1, t)}"', "\shortmid"{marking}, from=1-6, to=1-5]
	\arrow["{E(j, 1)}"', "\shortmid"{marking}, from=1-5, to=1-3]
	\arrow["{E(1, t)}"', "\shortmid"{marking}, from=1-3, to=1-1]
	\arrow[""{name=2, anchor=center, inner sep=0}, Rightarrow, no head, from=1-1, to=2-1]
	\arrow[""{name=3, anchor=center, inner sep=0}, Rightarrow, no head, from=1-5, to=2-5]
	\arrow[""{name=4, anchor=center, inner sep=0}, Rightarrow, no head, from=1-8, to=2-8]
	\arrow[""{name=5, anchor=center, inner sep=0}, Rightarrow, no head, from=1-3, to=2-3]
	\arrow["\lambda"{description}, draw=none, from=0, to=1]
	\arrow["\lambda"{description}, draw=none, from=4, to=3]
	\arrow["{=}"{description}, draw=none, from=3, to=5]
	\arrow["{=}"{description}, draw=none, from=5, to=2]
\end{tikzcd}
\]\[
\begin{tangle}{(6,4)}[trim y]
	\tgBorderA{(0,0)}{\tgColour4}{\tgColour6}{\tgColour6}{\tgColour4}
	\tgBorderA{(1,0)}{\tgColour6}{\tgColour4}{\tgColour4}{\tgColour6}
	\tgBorderA{(2,0)}{\tgColour4}{\tgColour6}{\tgColour6}{\tgColour4}
	\tgBorderA{(3,0)}{\tgColour6}{\tgColour4}{\tgColour4}{\tgColour6}
	\tgBlank{(4,0)}{\tgColour4}
	\tgBorderA{(5,0)}{\tgColour4}{\tgColour0}{\tgColour0}{\tgColour4}
	\tgBorderA{(0,1)}{\tgColour4}{\tgColour6}{\tgColour4}{\tgColour4}
	\tgBorderA{(1,1)}{\tgColour6}{\tgColour4}{\tgColour6}{\tgColour4}
	\tgBorderA{(2,1)}{\tgColour4}{\tgColour6}{\tgColour6}{\tgColour6}
	\tgBorderA{(3,1)}{\tgColour6}{\tgColour4}{\tgColour4}{\tgColour6}
	\tgBlank{(4,1)}{\tgColour4}
	\tgBorderA{(5,1)}{\tgColour4}{\tgColour0}{\tgColour0}{\tgColour4}
	\tgBlank{(0,2)}{\tgColour4}
	\tgBorderA{(1,2)}{\tgColour4}{\tgColour6}{\tgColour4}{\tgColour4}
	\tgBorderA{(2,2)}{\tgColour6}{\tgColour6}{\tgColour4}{\tgColour4}
	\tgBorderA{(3,2)}{\tgColour6}{\tgColour4}{\tgColour0}{\tgColour4}
	\tgBorderA{(4,2)}{\tgColour4}{\tgColour4}{\tgColour0}{\tgColour0}
	\tgBorderA{(5,2)}{\tgColour4}{\tgColour0}{\tgColour0}{\tgColour0}
	\tgBlank{(0,3)}{\tgColour4}
	\tgBlank{(1,3)}{\tgColour4}
	\tgBlank{(2,3)}{\tgColour4}
	\tgBorderA{(3,3)}{\tgColour4}{\tgColour0}{\tgColour0}{\tgColour4}
	\tgBlank{(4,3)}{\tgColour0}
	\tgBlank{(5,3)}{\tgColour0}
	\tgCell[(2,0)]{(1,1)}{\mu}
	\tgCell[(4,0)]{(3,2)}{\lambda}
	\tgArrow{(3,1.5)}{1}
	\tgArrow{(5,1.5)}{3}
	\tgArrow{(1,1.5)}{3}
	\tgArrow{(0,0.5)}{3}
	\tgArrow{(2,0.5)}{3}
	\tgArrow{(5,0.5)}{3}
	\tgArrow{(3,2.5)}{3}
	\tgArrow{(1,0.5)}{1}
	\tgArrow{(3,0.5)}{1}
	\tgAxisLabel{(0.5,0.75)}{south}{t}
	\tgAxisLabel{(1.5,0.75)}{south}{j}
	\tgAxisLabel{(2.5,0.75)}{south}{t}
	\tgAxisLabel{(3.5,0.75)}{south}{j}
	\tgAxisLabel{(5.5,0.75)}{south}{e}
	\tgAxisLabel{(3.5,3.25)}{north}{e}
\end{tangle}
\tangleeq*
\begin{tangle}{(6,4)}[trim y]
	\tgBorderA{(0,0)}{\tgColour4}{\tgColour6}{\tgColour6}{\tgColour4}
	\tgBlank{(1,0)}{\tgColour6}
	\tgBorderA{(2,0)}{\tgColour6}{\tgColour4}{\tgColour4}{\tgColour6}
	\tgBorderA{(3,0)}{\tgColour4}{\tgColour6}{\tgColour6}{\tgColour4}
	\tgBorderA{(4,0)}{\tgColour6}{\tgColour4}{\tgColour4}{\tgColour6}
	\tgBorderA{(5,0)}{\tgColour4}{\tgColour0}{\tgColour0}{\tgColour4}
	\tgBorderA{(0,1)}{\tgColour4}{\tgColour6}{\tgColour6}{\tgColour4}
	\tgBlank{(1,1)}{\tgColour6}
	\tgBorderA{(2,1)}{\tgColour6}{\tgColour4}{\tgColour4}{\tgColour6}
	\tgBorderA{(3,1)}{\tgColour4}{\tgColour6}{\tgColour4}{\tgColour4}
	\tgBorderA{(4,1)}{\tgColour6}{\tgColour4}{\tgColour0}{\tgColour4}
	\tgBorderA{(5,1)}{\tgColour4}{\tgColour0}{\tgColour0}{\tgColour0}
	\tgBorderA{(0,2)}{\tgColour4}{\tgColour6}{\tgColour4}{\tgColour4}
	\tgBorderA{(1,2)}{\tgColour6}{\tgColour6}{\tgColour4}{\tgColour4}
	\tgBorderA{(2,2)}{\tgColour6}{\tgColour4}{\tgColour0}{\tgColour4}
	\tgBorderA{(3,2)}{\tgColour4}{\tgColour4}{\tgColour0}{\tgColour0}
	\tgBorderA{(4,2)}{\tgColour4}{\tgColour0}{\tgColour0}{\tgColour0}
	\tgBlank{(5,2)}{\tgColour0}
	\tgBlank{(0,3)}{\tgColour4}
	\tgBlank{(1,3)}{\tgColour4}
	\tgBorderA{(2,3)}{\tgColour4}{\tgColour0}{\tgColour0}{\tgColour4}
	\tgBlank{(3,3)}{\tgColour0}
	\tgBlank{(4,3)}{\tgColour0}
	\tgBlank{(5,3)}{\tgColour0}
	\tgCell[(4,0)]{(2,2)}{\lambda}
	\tgArrow{(2,1.5)}{1}
	\tgArrow{(4,1.5)}{3}
	\tgArrow{(0,1.5)}{3}
	\tgCell[(2,0)]{(4,1)}{\lambda}
	\tgArrow{(3,0.5)}{3}
	\tgArrow{(5,0.5)}{3}
	\tgArrow{(2,2.5)}{3}
	\tgArrow{(4,0.5)}{1}
	\tgArrow{(0,0.5)}{3}
	\tgArrow{(2,0.5)}{1}
	\tgAxisLabel{(0.5,0.75)}{south}{t}
	\tgAxisLabel{(2.5,0.75)}{south}{j}
	\tgAxisLabel{(3.5,0.75)}{south}{t}
	\tgAxisLabel{(4.5,0.75)}{south}{j}
	\tgAxisLabel{(5.5,0.75)}{south}{e}
	\tgAxisLabel{(2.5,3.25)}{north}{e}
\end{tangle}
\]
An action homomorphism is a 2-cell $E(1, \epsilon) \colon E(1, e) \tto E(1, e')$ satisfying the following equation.
\[
\begin{tikzcd}[column sep=large]
	E & A & E & D \\
	E &&& D \\
	E &&& D
	\arrow["{E(1, e)}"', "\shortmid"{marking}, from=1-4, to=1-3]
	\arrow["{E(j, 1)}"', "\shortmid"{marking}, from=1-3, to=1-2]
	\arrow["{E(1, t)}"', "\shortmid"{marking}, from=1-2, to=1-1]
	\arrow["{E(1, e)}"{description}, from=2-4, to=2-1]
	\arrow[""{name=0, anchor=center, inner sep=0}, Rightarrow, no head, from=1-4, to=2-4]
	\arrow[""{name=1, anchor=center, inner sep=0}, Rightarrow, no head, from=1-1, to=2-1]
	\arrow["{E(1, e')}", "\shortmid"{marking}, from=3-4, to=3-1]
	\arrow[""{name=2, anchor=center, inner sep=0}, Rightarrow, no head, from=2-4, to=3-4]
	\arrow[""{name=3, anchor=center, inner sep=0}, Rightarrow, no head, from=2-1, to=3-1]
	\arrow["\lambda"{description}, draw=none, from=0, to=1]
	\arrow["{E(1, \epsilon)}"{description}, draw=none, from=2, to=3]
\end{tikzcd}
\quad = \quad
\begin{tikzcd}[column sep=large]
	E & A & E & D \\
	E & A & E & D \\
	E &&& D
	\arrow["{E(1, e')}"{description}, from=2-4, to=2-3]
	\arrow["{E(j, 1)}"{description}, from=2-3, to=2-2]
	\arrow["{E(1, t)}"{description}, from=2-2, to=2-1]
	\arrow["{E(1, e')}", "\shortmid"{marking}, from=3-4, to=3-1]
	\arrow[""{name=0, anchor=center, inner sep=0}, Rightarrow, no head, from=2-4, to=3-4]
	\arrow[""{name=1, anchor=center, inner sep=0}, Rightarrow, no head, from=2-1, to=3-1]
	\arrow["{E(1, e)}"', "\shortmid"{marking}, from=1-4, to=1-3]
	\arrow["{E(j, 1)}"', "\shortmid"{marking}, from=1-3, to=1-2]
	\arrow["{E(1, t)}"', "\shortmid"{marking}, from=1-2, to=1-1]
	\arrow[""{name=2, anchor=center, inner sep=0}, Rightarrow, no head, from=1-2, to=2-2]
	\arrow[""{name=3, anchor=center, inner sep=0}, Rightarrow, no head, from=1-1, to=2-1]
	\arrow[""{name=4, anchor=center, inner sep=0}, Rightarrow, no head, from=1-4, to=2-4]
	\arrow[""{name=5, anchor=center, inner sep=0}, Rightarrow, no head, from=1-3, to=2-3]
	\arrow["{\lambda'}"{description}, draw=none, from=0, to=1]
	\arrow["{E(1, \epsilon)}"{description}, draw=none, from=4, to=5]
	\arrow["{=}"{description}, draw=none, from=5, to=2]
	\arrow["{=}"{description}, draw=none, from=2, to=3]
\end{tikzcd}
\]\[
\begin{tangle}{(3,4)}[trim y]
	\tgBorderA{(0,0)}{\tgColour4}{\tgColour6}{\tgColour6}{\tgColour4}
	\tgBorderA{(1,0)}{\tgColour6}{\tgColour4}{\tgColour4}{\tgColour6}
	\tgBorderA{(2,0)}{\tgColour4}{\tgColour0}{\tgColour0}{\tgColour4}
	\tgBorderA{(0,1)}{\tgColour4}{\tgColour6}{\tgColour4}{\tgColour4}
	\tgBorderA{(1,1)}{\tgColour6}{\tgColour4}{\tgColour0}{\tgColour4}
	\tgBorderA{(2,1)}{\tgColour4}{\tgColour0}{\tgColour0}{\tgColour0}
	\tgBlank{(0,2)}{\tgColour4}
	\tgBorderA{(1,2)}{\tgColour4}{\tgColour0}{\tgColour0}{\tgColour4}
	\tgBlank{(2,2)}{\tgColour0}
	\tgBlank{(0,3)}{\tgColour4}
	\tgBorderA{(1,3)}{\tgColour4}{\tgColour0}{\tgColour0}{\tgColour4}
	\tgBlank{(2,3)}{\tgColour0}
	\tgCell[(2,0)]{(1,1)}{\lambda}
	\tgCell{(1,2)}{\epsilon}
	\tgArrow{(1,1.5)}{3}
	\tgArrow{(1,2.5)}{3}
	\tgArrow{(0,0.5)}{3}
	\tgArrow{(2,0.5)}{3}
	\tgArrow{(1,0.5)}{1}
	\tgAxisLabel{(0.5,0.75)}{south}{t}
	\tgAxisLabel{(1.5,0.75)}{south}{j}
	\tgAxisLabel{(2.5,0.75)}{south}{e}
	\tgAxisLabel{(1.5,3.25)}{north}{e'}
\end{tangle}
\tangleeq*
\begin{tangle}{(3,4)}[trim y]
	\tgBorderA{(0,0)}{\tgColour4}{\tgColour6}{\tgColour6}{\tgColour4}
	\tgBorderA{(1,0)}{\tgColour6}{\tgColour4}{\tgColour4}{\tgColour6}
	\tgBorderA{(2,0)}{\tgColour4}{\tgColour0}{\tgColour0}{\tgColour4}
	\tgBorderA{(0,1)}{\tgColour4}{\tgColour6}{\tgColour6}{\tgColour4}
	\tgBorderA{(1,1)}{\tgColour6}{\tgColour4}{\tgColour4}{\tgColour6}
	\tgBorderA{(2,1)}{\tgColour4}{\tgColour0}{\tgColour0}{\tgColour4}
	\tgBorderA{(0,2)}{\tgColour4}{\tgColour6}{\tgColour4}{\tgColour4}
	\tgBorderA{(1,2)}{\tgColour6}{\tgColour4}{\tgColour0}{\tgColour4}
	\tgBorderA{(2,2)}{\tgColour4}{\tgColour0}{\tgColour0}{\tgColour0}
	\tgBlank{(0,3)}{\tgColour4}
	\tgBorderA{(1,3)}{\tgColour4}{\tgColour0}{\tgColour0}{\tgColour4}
	\tgBlank{(2,3)}{\tgColour0}
	\tgCell[(2,0)]{(1,2)}{\lambda'}
	\tgCell{(2,1)}{\epsilon}
	\tgArrow{(1,1.5)}{1}
	\tgArrow{(0,1.5)}{3}
	\tgArrow{(2,1.5)}{3}
	\tgArrow{(0,0.5)}{3}
	\tgArrow{(2,0.5)}{3}
	\tgArrow{(1,2.5)}{3}
	\tgArrow{(1,0.5)}{1}
	\tgAxisLabel{(0.5,0.75)}{south}{t}
	\tgAxisLabel{(1.5,0.75)}{south}{j}
	\tgAxisLabel{(2.5,0.75)}{south}{e}
	\tgAxisLabel{(1.5,3.25)}{north}{e'}
\end{tangle}
\]

We may now exhibit the $T$-algebras of \cref{algebra} as actions in $\X[D, j]$.

\begin{theorem}
    \label{algebras-are-left-actions}
    There is an isomorphism of categories rendering the following diagram commutative, natural in $D$ and $T$.
	\[\begin{tikzcd}
		{T\h\Alg_D} && {\Act(\X[D, j], T)} \\
		& {\X[D, E]}
		\arrow["{U_{T, D}}"', from=1-1, to=2-2]
		\arrow["{U_{\X[D, j], T}}", from=1-3, to=2-2]
		\arrow["\iso", from=1-1, to=1-3]
	\end{tikzcd}\]
\end{theorem}

\begin{proof}
    Observe that, given an action, we can define an algebra structure, and conversely:
    
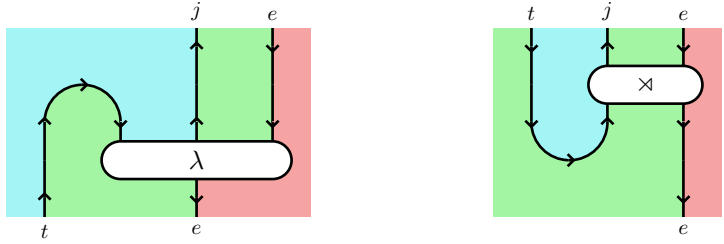
\[
	\begin{tangle}{(4,4)}[trim y]
		\tgBlank{(0,0)}{\tgColour6}
		\tgBlank{(1,0)}{\tgColour6}
		\tgBorderA{(2,0)}{\tgColour6}{\tgColour4}{\tgColour4}{\tgColour6}
		\tgBorderA{(3,0)}{\tgColour4}{\tgColour0}{\tgColour0}{\tgColour4}
		\tgBorderC{(0,1)}{3}{\tgColour6}{\tgColour4}
		\tgBorderC{(1,1)}{2}{\tgColour6}{\tgColour4}
		\tgBorderA{(2,1)}{\tgColour6}{\tgColour4}{\tgColour4}{\tgColour6}
		\tgBorderA{(3,1)}{\tgColour4}{\tgColour0}{\tgColour0}{\tgColour4}
		\tgBorderA{(0,2)}{\tgColour6}{\tgColour4}{\tgColour4}{\tgColour6}
		\tgBorderA{(1,2)}{\tgColour4}{\tgColour6}{\tgColour4}{\tgColour4}
		\tgBorderA{(2,2)}{\tgColour6}{\tgColour4}{\tgColour0}{\tgColour4}
		\tgBorderA{(3,2)}{\tgColour4}{\tgColour0}{\tgColour0}{\tgColour0}
		\tgBorderA{(0,3)}{\tgColour6}{\tgColour4}{\tgColour4}{\tgColour6}
		\tgBlank{(1,3)}{\tgColour4}
		\tgBorderA{(2,3)}{\tgColour4}{\tgColour0}{\tgColour0}{\tgColour4}
		\tgBlank{(3,3)}{\tgColour0}
		\tgCell[(2,0)]{(2,2)}{\lambda}
		\tgArrow{(2,1.5)}{1}
		\tgArrow{(3,1.5)}{3}
		\tgArrow{(1,1.5)}{3}
		\tgArrow{(0.5,1)}{0}
		\tgArrow{(0,1.5)}{1}
		\tgArrow{(3,0.5)}{3}
		\tgArrow{(2,2.5)}{3}
		\tgArrow{(0,2.5)}{1}
		\tgArrow{(2,0.5)}{1}
		\tgAxisLabel{(2.5,0.75)}{south}{j}
		\tgAxisLabel{(3.5,0.75)}{south}{e}
		\tgAxisLabel{(0.5,3.25)}{north}{t}
		\tgAxisLabel{(2.5,3.25)}{north}{e}
	\end{tangle}
    \hspace{4em}
	\begin{tangle}{(3,4)}[trim y]
		\tgBorderA{(0,0)}{\tgColour4}{\tgColour6}{\tgColour6}{\tgColour4}
		\tgBorderA{(1,0)}{\tgColour6}{\tgColour4}{\tgColour4}{\tgColour6}
		\tgBorderA{(2,0)}{\tgColour4}{\tgColour0}{\tgColour0}{\tgColour4}
		\tgBorderA{(0,1)}{\tgColour4}{\tgColour6}{\tgColour6}{\tgColour4}
		\tgBorderA{(1,1)}{\tgColour6}{\tgColour4}{\tgColour4}{\tgColour6}
		\tgBorder{(1,1)}{0}{1}{0}{0}
		\tgBorderA{(2,1)}{\tgColour4}{\tgColour0}{\tgColour0}{\tgColour4}
		\tgBorder{(2,1)}{0}{0}{0}{1}
		\tgBorderC{(0,2)}{0}{\tgColour4}{\tgColour6}
		\tgBorderC{(1,2)}{1}{\tgColour4}{\tgColour6}
		\tgBorderA{(2,2)}{\tgColour4}{\tgColour0}{\tgColour0}{\tgColour4}
		\tgBlank{(0,3)}{\tgColour4}
		\tgBlank{(1,3)}{\tgColour4}
		\tgBorderA{(2,3)}{\tgColour4}{\tgColour0}{\tgColour0}{\tgColour4}
		\tgCell[(1,0)]{(1.5,1)}{\aop}
		\tgArrow{(2,1.5)}{3}
		\tgArrow{(1,1.5)}{1}
		\tgArrow{(0.5,2)}{0}
		\tgArrow{(0,1.5)}{3}
		\tgArrow{(2,2.5)}{3}
		\tgArrow{(0,0.5)}{3}
		\tgArrow{(2,0.5)}{3}
		\tgArrow{(1,0.5)}{1}
		\tgAxisLabel{(0.5,0.75)}{south}{t}
		\tgAxisLabel{(1.5,0.75)}{south}{j}
		\tgAxisLabel{(2.5,0.75)}{south}{e}
		\tgAxisLabel{(2.5,3.25)}{north}{e}
	\end{tangle}
    \]
    It is immediate that the laws for an algebra (morphism) are precisely those for an action (homomorphism) under these transformations.
\end{proof}

\begin{remark}
	\label{representability-for-skew-multiactegories}
	While we shall not formally introduce representability for (skew) multiactegories, the evident generalisation of \cref{Xj1-is-skew-monoidal} suggests that the skew-multiactegory $\X[D, j]$ of \cref{XDj} will be representable in an appropriate sense when $\X$ admits left extensions of tight-cells $A \to E$ along $j$, in which case we recover \cite[Theorem~3.7]{altenkirch2015monads}.
\end{remark}

\begin{corollary}
	\label{algebras-for-identity-relative-monads-are-algebras-for-monads}
    An algebra (morphism) for a $1_E$-monad is precisely an algebra (morphism) for the corresponding monad on $E$ (\cref{identity-relative-monads-are-monads}).
\end{corollary}

\begin{proof}
    By \cref{algebras-are-left-actions}, it suffices to consider action (morphisms) in place of algebra (morphisms). An action in $\X[D, 1_E]$ for a monad $T$ comprises a tight-cell $e \colon D \to E$ and a 2-cell $\lambda \colon (e \d t) \tto e$ rendering the following diagrams commutative.
	\[
	\begin{tikzcd}
		e & {e \d t} \\
		& e
		\arrow["{e \d \eta}", from=1-1, to=1-2]
		\arrow["\lambda", from=1-2, to=2-2]
		\arrow[Rightarrow, no head, from=1-1, to=2-2]
	\end{tikzcd}
	\hspace{4em}
	\begin{tikzcd}
		{e \d t \d t} & {e \d t} \\
		{e \d t} & e
		\arrow["{e \d \mu}", from=1-1, to=1-2]
		\arrow["\lambda", from=1-2, to=2-2]
		\arrow["{\lambda \d t}"', from=1-1, to=2-1]
		\arrow["\lambda"', from=2-1, to=2-2]
	\end{tikzcd}
	\]
	An action morphism in $\X[D, 1_E]$ is a 2-cell $\epsilon \colon e \tto e'$ rendering the following diagram commutative.
	\[\begin{tikzcd}
		{e \d t} & {e' \d t} \\
		e & {e'}
		\arrow["\lambda"', from=1-1, to=2-1]
		\arrow["\epsilon"', from=2-1, to=2-2]
		\arrow["{\epsilon \d t}", from=1-1, to=1-2]
		\arrow["{\lambda'}", from=1-2, to=2-2]
	\end{tikzcd}\]
	This is precisely the definition of an algebra (morphism) for a monad (\cf{}~\cite[\S3.1]{kelly1974review}).
\end{proof}

\subsection{Opalgebras as right-actions in a skew-multiactegory}
\label{opalgebras}

\begin{definition}
	\label{right-action}
    Let $\M$ be an associative-normal left-skew-multicategory and let $\A$ be a right-$\M$-multiactegory. An \emph{action} in $\A$ for a monoid $(M, m, u)$ in $\M$ (or simply \emph{$(M, \mu, u)$-action}) comprises
    \begin{enumerate}
        \item an object $A \in \A$, the \emph{carrier};
        \item a multimorphism $a \colon A, M \to A$, the \emph{action},
    \end{enumerate}
    satisfying the following equations.
	\begin{align*}
		\rho_{(A; A), 1}((1_A, u) \d a) & = 1_A &
		(a, 1_M) \d a = (1_A, m) \d a
	\end{align*}
    An \emph{action homomorphism} from $(A, a)$ to $(A', a')$ is a unary multimorphism $f \colon A \to A'$ satisfying the following equation.
	\begin{align*}
		a \d f & = (f, 1_M) \d a'
	\end{align*}
	$(M, m, u)$-actions and their homomorphisms form a category $\Act(\A, (M, m, u))$ functorial covariantly in $\A$ and contravariantly in $(M, m, u)$. Denote by $U_{\A, (M, m, u)} \colon \Act(\A, (M, m, u)) \to \A_1$ the faithful functor sending each action $(A, a)$ to its carrier $a$.
\end{definition}

\begin{proposition}
	\label{monoid-forms-right-action}
    Let $\M$ be an associative-normal left-skew-multicategory. Then $\M$ forms a right-$\M$-multiactegory. Furthermore any monoid $(M, m, u)$ in $\M$ forms an $(M, m, u)$-action therein.
\end{proposition}

\begin{proof}
    The right-$\M$-multiactegory structure is defined to have the same objects, multimorphisms, and composition as $\M$, from which the laws hold trivially. Given a monoid $(M, m, u)$, we define an action $(M, m)$: the unit and multiplication laws follow from those of the monoid.
\end{proof}

\begin{proposition}
	Let $\X$ be a \vdc{} with a loose-adjunction $j_* \adj j^* \colon E \lto A$ and an object $B$. The loose-cells $a \colon A \lto B$ in $\X$ together with 2-cells of the form $a, j^*, p_1, j^*, p_2, j^*, \ldots, j^*, p_n \tto a'$ form a right-$\X\lh{j_* \adj j^*}$-multiactegory.
\end{proposition}

\begin{proof}
    We define a right-$\X\lh{j_* \adj j^*}$-multiactegory $\X\lh{j_* \adj j^*, B}$ as follows. The class of objects is given by those of $\X\lh{A, B}$. The left- and right-normal multimorphisms $a, p_1, \ldots, p_n \to a'$ ($n \geq 0$) are 2-cells $a, j^*, p_1, j^*, \ldots, j^*, p_n \tto a'$.
	\[\begin{tikzcd}
		B & A & E & A & \cdots & E & A \\
		B &&&&&& A
		\arrow["{p_n}"', "\shortmid"{marking}, from=1-7, to=1-6]
		\arrow["{j^*}"', "\shortmid"{marking}, from=1-5, to=1-4]
		\arrow["{p_1}"', "\shortmid"{marking}, from=1-4, to=1-3]
		\arrow["{a'}", "\shortmid"{marking}, from=2-7, to=2-1]
		\arrow[""{name=0, anchor=center, inner sep=0}, Rightarrow, no head, from=1-7, to=2-7]
		\arrow["{j^*}"', "\shortmid"{marking}, from=1-6, to=1-5]
		\arrow["{j^*}"', "\shortmid"{marking}, from=1-3, to=1-2]
		\arrow["a"', "\shortmid"{marking}, from=1-2, to=1-1]
		\arrow[""{name=1, anchor=center, inner sep=0}, Rightarrow, no head, from=1-1, to=2-1]
		\arrow["\phi"{description}, draw=none, from=0, to=1]
	\end{tikzcd}\]
	The general multimorphisms, composition structure, and left- and right-unitors are defined as in \cref{skew-multicategorical-hom}, and satisfy the laws for the same reasons.
	Functoriality in $B$ follows from pasting on the left.
\end{proof}

\begin{definition}
	\label{XjB}
	Let $\X$ be an equipment with a tight-cell $\jAE$. Denote by $\X\lh{j, B}$ the $\X\lh j$-multiactegory $\X\lh{E(1, j) \adj E(j, 1), B}$. Define $\X[j, B]$ to be the full sub-multiactegory of $\X\lh{j, B}$ spanned by the representable loose-cells.
\end{definition}

In particular $\X\lh{j, B}_1 = \X\lh{A, B}$ and $\X[j, B]_1 = \X[A, B]$. We shall unwrap the definition of an action in $\X[j, B]$ to compare it with the definition of an opalgebra for a relative monad. Explicitly, an action in $\X[j, B]$ comprises
\begin{enumerate}
	\item a tight-cell $a \colon A \to B$,
	\item a 2-cell ${\rho \colon B(1, a), E(j, 1), E(1, t) \tto B(1, a)}$,
\end{enumerate}
satisfying the following equations.
\[
\begin{tikzcd}[column sep=large]
	B & A & A \\
	B & A & A \\
	B & A & A \\
	B && A
	\arrow["{B(1, a)}"{description}, from=3-2, to=3-1]
	\arrow["{B(1, a)}", "\shortmid"{marking}, from=4-3, to=4-1]
	\arrow[""{name=0, anchor=center, inner sep=0}, Rightarrow, no head, from=3-3, to=4-3]
	\arrow[""{name=1, anchor=center, inner sep=0}, Rightarrow, no head, from=3-1, to=4-1]
	\arrow[""{name=2, anchor=center, inner sep=0}, Rightarrow, no head, from=2-2, to=3-2]
	\arrow[""{name=3, anchor=center, inner sep=0}, Rightarrow, no head, from=2-1, to=3-1]
	\arrow["{B(1, a)}"{description}, from=2-2, to=2-1]
	\arrow["{E(j, t)}"{description}, from=3-3, to=3-2]
	\arrow[Rightarrow, no head, from=1-3, to=1-2]
	\arrow["{E(j, j)}"{description}, from=2-3, to=2-2]
	\arrow[""{name=4, anchor=center, inner sep=0}, Rightarrow, no head, from=2-3, to=3-3]
	\arrow[""{name=5, anchor=center, inner sep=0}, Rightarrow, no head, from=1-3, to=2-3]
	\arrow[""{name=6, anchor=center, inner sep=0}, Rightarrow, no head, from=1-2, to=2-2]
	\arrow[""{name=7, anchor=center, inner sep=0}, Rightarrow, no head, from=1-1, to=2-1]
	\arrow["{B(1, a)}"', "\shortmid"{marking}, from=1-2, to=1-1]
	\arrow["\rho"{description}, draw=none, from=0, to=1]
	\arrow["{=}"{description}, draw=none, from=2, to=3]
	\arrow["{E(j, \eta)}"{description}, draw=none, from=4, to=2]
	\arrow["{=}"{description}, draw=none, from=6, to=7]
	\arrow["{\pc j}"{description}, draw=none, from=5, to=6]
\end{tikzcd}
\quad = \quad
\begin{tikzcd}
	B & A \\
	B & A
	\arrow["{B(1, a)}", "\shortmid"{marking}, no head, from=2-2, to=2-1]
	\arrow[""{name=0, anchor=center, inner sep=0}, Rightarrow, no head, from=1-2, to=2-2]
	\arrow[""{name=1, anchor=center, inner sep=0}, Rightarrow, no head, from=1-1, to=2-1]
	\arrow["{B(1, a)}"', "\shortmid"{marking}, from=1-2, to=1-1]
	\arrow["{=}"{description}, draw=none, from=0, to=1]
\end{tikzcd}
\hspace{4em}
\begin{tangle}{(3,4)}[trim y]
	\tgBorderA{(0,0)}{\tgColour10}{\tgColour6}{\tgColour6}{\tgColour10}
	\tgBlank{(1,0)}{\tgColour6}
	\tgBlank{(2,0)}{\tgColour6}
	\tgBorderA{(0,1)}{\tgColour10}{\tgColour6}{\tgColour6}{\tgColour10}
	\tgBorderA{(1,1)}{\tgColour6}{\tgColour6}{\tgColour4}{\tgColour6}
	\tgBorderA{(2,1)}{\tgColour6}{\tgColour6}{\tgColour6}{\tgColour4}
	\tgBorderA{(0,2)}{\tgColour10}{\tgColour6}{\tgColour10}{\tgColour10}
	\tgBorderA{(1,2)}{\tgColour6}{\tgColour4}{\tgColour6}{\tgColour10}
	\tgBorderA{(2,2)}{\tgColour4}{\tgColour6}{\tgColour6}{\tgColour6}
	\tgBlank{(0,3)}{\tgColour10}
	\tgBorderA{(1,3)}{\tgColour10}{\tgColour6}{\tgColour6}{\tgColour10}
	\tgBlank{(2,3)}{\tgColour6}
	\tgCell[(2,0)]{(1,2)}{\rho}
	\tgArrow{(1,1.5)}{1}
	\tgArrow{(2,1.5)}{3}
	\tgArrow{(0,1.5)}{3}
	\tgCell[(1,0)]{(1.5,1)}{\eta}
	\tgArrow{(0,0.5)}{3}
	\tgArrow{(1,2.5)}{3}
	\tgAxisLabel{(0.5,0.75)}{south}{a}
	\tgAxisLabel{(1.5,3.25)}{north}{a}
\end{tangle}
\tangleeq*
\begin{tangle}{(1,4)}[trim y]
	\tgBorderA{(0,0)}{\tgColour10}{\tgColour6}{\tgColour6}{\tgColour10}
	\tgBorderA{(0,1)}{\tgColour10}{\tgColour6}{\tgColour6}{\tgColour10}
	\tgBorderA{(0,2)}{\tgColour10}{\tgColour6}{\tgColour6}{\tgColour10}
	\tgBorderA{(0,3)}{\tgColour10}{\tgColour6}{\tgColour6}{\tgColour10}
	\tgArrow{(0,1.5)}{3}
	\tgArrow{(0,2.5)}{3}
	\tgArrow{(0,0.5)}{3}
	\tgAxisLabel{(0.5,0.75)}{south}{a}
	\tgAxisLabel{(0.5,3.25)}{north}{a}
\end{tangle}
\]
\[
\begin{tikzcd}[column sep=1.55em]
	B & A & E & A && E && A \\
	B &&& A && E && A \\
	B &&&&&&& A
	\arrow["{E(1, t)}"{description}, from=2-8, to=2-6]
	\arrow["{E(j, 1)}"{description}, from=2-6, to=2-4]
	\arrow["{B(1, a)}"{description}, from=2-4, to=2-1]
	\arrow["{B(1, a)}", "\shortmid"{marking}, from=3-8, to=3-1]
	\arrow[""{name=0, anchor=center, inner sep=0}, Rightarrow, no head, from=2-8, to=3-8]
	\arrow[""{name=1, anchor=center, inner sep=0}, Rightarrow, no head, from=2-1, to=3-1]
	\arrow["{E(1, t)}"', "\shortmid"{marking}, from=1-8, to=1-6]
	\arrow["{E(j, 1)}"', "\shortmid"{marking}, from=1-6, to=1-4]
	\arrow["{E(1, t)}"', "\shortmid"{marking}, from=1-4, to=1-3]
	\arrow["{E(j, 1)}"', "\shortmid"{marking}, from=1-3, to=1-2]
	\arrow["{B(1, a)}"', "\shortmid"{marking}, from=1-2, to=1-1]
	\arrow[""{name=2, anchor=center, inner sep=0}, Rightarrow, no head, from=1-1, to=2-1]
	\arrow[""{name=3, anchor=center, inner sep=0}, Rightarrow, no head, from=1-8, to=2-8]
	\arrow[""{name=4, anchor=center, inner sep=0}, Rightarrow, no head, from=1-4, to=2-4]
	\arrow[""{name=5, anchor=center, inner sep=0}, Rightarrow, no head, from=1-6, to=2-6]
	\arrow["\rho"{description}, draw=none, from=0, to=1]
	\arrow["\rho"{description}, draw=none, from=4, to=2]
	\arrow["{=}"{description}, draw=none, from=3, to=5]
	\arrow["{=}"{description}, draw=none, from=5, to=4]
\end{tikzcd}
\!=\!
\begin{tikzcd}[column sep=1.54em]
	B && A && E & A & E & A \\
	B && A && E &&& A \\
	B &&&&&&& A
	\arrow["{E(1, t)}"{description}, from=2-8, to=2-5]
	\arrow["{E(j, 1)}"{description}, from=2-5, to=2-3]
	\arrow["{B(1, a)}"{description}, from=2-3, to=2-1]
	\arrow["{B(1, a)}", "\shortmid"{marking}, from=3-8, to=3-1]
	\arrow[""{name=0, anchor=center, inner sep=0}, Rightarrow, no head, from=2-8, to=3-8]
	\arrow[""{name=1, anchor=center, inner sep=0}, Rightarrow, no head, from=2-1, to=3-1]
	\arrow["{E(1, t)}"', "\shortmid"{marking}, from=1-8, to=1-7]
	\arrow["{E(j, 1)}"', "\shortmid"{marking}, from=1-7, to=1-6]
	\arrow["{E(1, t)}"', "\shortmid"{marking}, from=1-6, to=1-5]
	\arrow["{E(j, 1)}"', "\shortmid"{marking}, from=1-5, to=1-3]
	\arrow["{B(1, a)}"', "\shortmid"{marking}, from=1-3, to=1-1]
	\arrow[""{name=2, anchor=center, inner sep=0}, Rightarrow, no head, from=1-1, to=2-1]
	\arrow[""{name=3, anchor=center, inner sep=0}, Rightarrow, no head, from=1-5, to=2-5]
	\arrow[""{name=4, anchor=center, inner sep=0}, Rightarrow, no head, from=1-8, to=2-8]
	\arrow[""{name=5, anchor=center, inner sep=0}, Rightarrow, no head, from=1-3, to=2-3]
	\arrow["\rho"{description}, draw=none, from=0, to=1]
	\arrow["\mu"{description}, draw=none, from=4, to=3]
	\arrow["{=}"{description}, draw=none, from=3, to=5]
	\arrow["{=}"{description}, draw=none, from=5, to=2]
\end{tikzcd}
\]\[
\begin{tangle}{(6,4)}[trim y]
	\tgBorderA{(0,0)}{\tgColour10}{\tgColour6}{\tgColour6}{\tgColour10}
	\tgBorderA{(1,0)}{\tgColour6}{\tgColour4}{\tgColour4}{\tgColour6}
	\tgBorderA{(2,0)}{\tgColour4}{\tgColour6}{\tgColour6}{\tgColour4}
	\tgBorderA{(3,0)}{\tgColour6}{\tgColour4}{\tgColour4}{\tgColour6}
	\tgBlank{(4,0)}{\tgColour4}
	\tgBorderA{(5,0)}{\tgColour4}{\tgColour6}{\tgColour6}{\tgColour4}
	\tgBorderA{(0,1)}{\tgColour10}{\tgColour6}{\tgColour10}{\tgColour10}
	\tgBorderA{(1,1)}{\tgColour6}{\tgColour4}{\tgColour6}{\tgColour10}
	\tgBorderA{(2,1)}{\tgColour4}{\tgColour6}{\tgColour6}{\tgColour6}
	\tgBorderA{(3,1)}{\tgColour6}{\tgColour4}{\tgColour4}{\tgColour6}
	\tgBlank{(4,1)}{\tgColour4}
	\tgBorderA{(5,1)}{\tgColour4}{\tgColour6}{\tgColour6}{\tgColour4}
	\tgBlank{(0,2)}{\tgColour10}
	\tgBorderA{(1,2)}{\tgColour10}{\tgColour6}{\tgColour10}{\tgColour10}
	\tgBorderA{(2,2)}{\tgColour6}{\tgColour6}{\tgColour10}{\tgColour10}
	\tgBorderA{(3,2)}{\tgColour6}{\tgColour4}{\tgColour6}{\tgColour10}
	\tgBorderA{(4,2)}{\tgColour4}{\tgColour4}{\tgColour6}{\tgColour6}
	\tgBorderA{(5,2)}{\tgColour4}{\tgColour6}{\tgColour6}{\tgColour6}
	\tgBlank{(0,3)}{\tgColour10}
	\tgBlank{(1,3)}{\tgColour10}
	\tgBlank{(2,3)}{\tgColour10}
	\tgBorderA{(3,3)}{\tgColour10}{\tgColour6}{\tgColour6}{\tgColour10}
	\tgBlank{(4,3)}{\tgColour6}
	\tgBlank{(5,3)}{\tgColour6}
	\tgCell[(2,0)]{(1,1)}{\rho}
	\tgCell[(4,0)]{(3,2)}{\rho}
	\tgArrow{(3,1.5)}{1}
	\tgArrow{(5,1.5)}{3}
	\tgArrow{(1,1.5)}{3}
	\tgArrow{(0,0.5)}{3}
	\tgArrow{(2,0.5)}{3}
	\tgArrow{(5,0.5)}{3}
	\tgArrow{(3,2.5)}{3}
	\tgArrow{(3,0.5)}{1}
	\tgArrow{(1,0.5)}{1}
	\tgAxisLabel{(0.5,0.75)}{south}{a}
	\tgAxisLabel{(1.5,0.75)}{south}{j}
	\tgAxisLabel{(2.5,0.75)}{south}{t}
	\tgAxisLabel{(3.5,0.75)}{south}{j}
	\tgAxisLabel{(5.5,0.75)}{south}{t}
	\tgAxisLabel{(3.5,3.25)}{north}{a}
\end{tangle}
\tangleeq*
\begin{tangle}{(6,4)}[trim y]
	\tgBorderA{(0,0)}{\tgColour10}{\tgColour6}{\tgColour6}{\tgColour10}
	\tgBlank{(1,0)}{\tgColour6}
	\tgBorderA{(2,0)}{\tgColour6}{\tgColour4}{\tgColour4}{\tgColour6}
	\tgBorderA{(3,0)}{\tgColour4}{\tgColour6}{\tgColour6}{\tgColour4}
	\tgBorderA{(4,0)}{\tgColour6}{\tgColour4}{\tgColour4}{\tgColour6}
	\tgBorderA{(5,0)}{\tgColour4}{\tgColour6}{\tgColour6}{\tgColour4}
	\tgBorderA{(0,1)}{\tgColour10}{\tgColour6}{\tgColour6}{\tgColour10}
	\tgBlank{(1,1)}{\tgColour6}
	\tgBorderA{(2,1)}{\tgColour6}{\tgColour4}{\tgColour4}{\tgColour6}
	\tgBorderA{(3,1)}{\tgColour4}{\tgColour6}{\tgColour4}{\tgColour4}
	\tgBorderA{(4,1)}{\tgColour6}{\tgColour4}{\tgColour6}{\tgColour4}
	\tgBorderA{(5,1)}{\tgColour4}{\tgColour6}{\tgColour6}{\tgColour6}
	\tgBorderA{(0,2)}{\tgColour10}{\tgColour6}{\tgColour10}{\tgColour10}
	\tgBorderA{(1,2)}{\tgColour6}{\tgColour6}{\tgColour10}{\tgColour10}
	\tgBorderA{(2,2)}{\tgColour6}{\tgColour4}{\tgColour6}{\tgColour10}
	\tgBorderA{(3,2)}{\tgColour4}{\tgColour4}{\tgColour6}{\tgColour6}
	\tgBorderA{(4,2)}{\tgColour4}{\tgColour6}{\tgColour6}{\tgColour6}
	\tgBlank{(5,2)}{\tgColour6}
	\tgBlank{(0,3)}{\tgColour10}
	\tgBlank{(1,3)}{\tgColour10}
	\tgBorderA{(2,3)}{\tgColour10}{\tgColour6}{\tgColour6}{\tgColour10}
	\tgBlank{(3,3)}{\tgColour6}
	\tgBlank{(4,3)}{\tgColour6}
	\tgBlank{(5,3)}{\tgColour6}
	\tgCell[(4,0)]{(2,2)}{\rho}
	\tgArrow{(2,1.5)}{1}
	\tgArrow{(4,1.5)}{3}
	\tgArrow{(0,1.5)}{3}
	\tgCell[(2,0)]{(4,1)}{\mu}
	\tgArrow{(0,0.5)}{3}
	\tgArrow{(3,0.5)}{3}
	\tgArrow{(5,0.5)}{3}
	\tgArrow{(2,2.5)}{3}
	\tgArrow{(2,0.5)}{1}
	\tgArrow{(4,0.5)}{1}
	\tgAxisLabel{(0.5,0.75)}{south}{a}
	\tgAxisLabel{(2.5,0.75)}{south}{j}
	\tgAxisLabel{(3.5,0.75)}{south}{t}
	\tgAxisLabel{(4.5,0.75)}{south}{j}
	\tgAxisLabel{(5.5,0.75)}{south}{t}
	\tgAxisLabel{(2.5,3.25)}{north}{a}
\end{tangle}
\]
An action homomorphism is a 2-cell $B(1, \alpha) \colon B(1, a) \tto B(1, a')$ satisfying the following equation.
\[
\begin{tikzcd}[column sep=large]
	B & A & E & A \\
	B &&& A \\
	B &&& A
	\arrow["{E(1, t)}"', "\shortmid"{marking}, from=1-4, to=1-3]
	\arrow["{E(j, 1)}"', "\shortmid"{marking}, from=1-3, to=1-2]
	\arrow["{B(1, a)}"', "\shortmid"{marking}, from=1-2, to=1-1]
	\arrow["{B(1, a)}"{description}, from=2-4, to=2-1]
	\arrow[""{name=0, anchor=center, inner sep=0}, Rightarrow, no head, from=1-4, to=2-4]
	\arrow[""{name=1, anchor=center, inner sep=0}, Rightarrow, no head, from=1-1, to=2-1]
	\arrow["{B(1, a')}", "\shortmid"{marking}, from=3-4, to=3-1]
	\arrow[""{name=2, anchor=center, inner sep=0}, Rightarrow, no head, from=2-4, to=3-4]
	\arrow[""{name=3, anchor=center, inner sep=0}, Rightarrow, no head, from=2-1, to=3-1]
	\arrow["\rho"{description}, draw=none, from=0, to=1]
	\arrow["{B(1, \alpha)}"{description}, draw=none, from=2, to=3]
\end{tikzcd}
\quad = \quad
\begin{tikzcd}[column sep=large]
	B & A & E & A \\
	B & A & E & A \\
	B &&& A
	\arrow["{E(1, t)}"{description}, from=2-4, to=2-3]
	\arrow["{E(j, 1)}"{description}, from=2-3, to=2-2]
	\arrow["{B(1, a')}"{description}, from=2-2, to=2-1]
	\arrow["{B(1, a')}", "\shortmid"{marking}, from=3-4, to=3-1]
	\arrow[""{name=0, anchor=center, inner sep=0}, Rightarrow, no head, from=2-4, to=3-4]
	\arrow[""{name=1, anchor=center, inner sep=0}, Rightarrow, no head, from=2-1, to=3-1]
	\arrow["{E(1, t)}"', "\shortmid"{marking}, from=1-4, to=1-3]
	\arrow["{E(j, 1)}"', "\shortmid"{marking}, from=1-3, to=1-2]
	\arrow["{B(1, a)}"', "\shortmid"{marking}, from=1-2, to=1-1]
	\arrow[""{name=2, anchor=center, inner sep=0}, Rightarrow, no head, from=1-2, to=2-2]
	\arrow[""{name=3, anchor=center, inner sep=0}, Rightarrow, no head, from=1-1, to=2-1]
	\arrow[""{name=4, anchor=center, inner sep=0}, Rightarrow, no head, from=1-4, to=2-4]
	\arrow[""{name=5, anchor=center, inner sep=0}, Rightarrow, no head, from=1-3, to=2-3]
	\arrow["{\rho'}"{description}, draw=none, from=0, to=1]
	\arrow["{B(1, \alpha)}"{description}, draw=none, from=2, to=3]
	\arrow["{=}"{description}, draw=none, from=4, to=5]
	\arrow["{=}"{description}, draw=none, from=5, to=2]
\end{tikzcd}
\]\[
\begin{tangle}{(3,4)}[trim y]
	\tgBorderA{(0,0)}{\tgColour10}{\tgColour6}{\tgColour6}{\tgColour10}
	\tgBorderA{(1,0)}{\tgColour6}{\tgColour4}{\tgColour4}{\tgColour6}
	\tgBorderA{(2,0)}{\tgColour4}{\tgColour6}{\tgColour6}{\tgColour4}
	\tgBorderA{(0,1)}{\tgColour10}{\tgColour6}{\tgColour10}{\tgColour10}
	\tgBorderA{(1,1)}{\tgColour6}{\tgColour4}{\tgColour6}{\tgColour10}
	\tgBorderA{(2,1)}{\tgColour4}{\tgColour6}{\tgColour6}{\tgColour6}
	\tgBlank{(0,2)}{\tgColour10}
	\tgBorderA{(1,2)}{\tgColour10}{\tgColour6}{\tgColour6}{\tgColour10}
	\tgBlank{(2,2)}{\tgColour6}
	\tgBlank{(0,3)}{\tgColour10}
	\tgBorderA{(1,3)}{\tgColour10}{\tgColour6}{\tgColour6}{\tgColour10}
	\tgBlank{(2,3)}{\tgColour6}
	\tgCell[(2,0)]{(1,1)}{\rho}
	\tgCell{(1,2)}{\alpha}
	\tgArrow{(1,1.5)}{3}
	\tgArrow{(0,0.5)}{3}
	\tgArrow{(2,0.5)}{3}
	\tgArrow{(1,2.5)}{3}
	\tgArrow{(1,0.5)}{1}
	\tgAxisLabel{(0.5,0.75)}{south}{a}
	\tgAxisLabel{(1.5,0.75)}{south}{j}
	\tgAxisLabel{(2.5,0.75)}{south}{t}
	\tgAxisLabel{(1.5,3.25)}{north}{a'}
\end{tangle}
\tangleeq*
\begin{tangle}{(3,4)}[trim y]
	\tgBorderA{(0,0)}{\tgColour10}{\tgColour6}{\tgColour6}{\tgColour10}
	\tgBorderA{(1,0)}{\tgColour6}{\tgColour4}{\tgColour4}{\tgColour6}
	\tgBorderA{(2,0)}{\tgColour4}{\tgColour6}{\tgColour6}{\tgColour4}
	\tgBorderA{(0,1)}{\tgColour10}{\tgColour6}{\tgColour6}{\tgColour10}
	\tgBorderA{(1,1)}{\tgColour6}{\tgColour4}{\tgColour4}{\tgColour6}
	\tgBorderA{(2,1)}{\tgColour4}{\tgColour6}{\tgColour6}{\tgColour4}
	\tgBorderA{(0,2)}{\tgColour10}{\tgColour6}{\tgColour10}{\tgColour10}
	\tgBorderA{(1,2)}{\tgColour6}{\tgColour4}{\tgColour6}{\tgColour10}
	\tgBorderA{(2,2)}{\tgColour4}{\tgColour6}{\tgColour6}{\tgColour6}
	\tgBlank{(0,3)}{\tgColour10}
	\tgBorderA{(1,3)}{\tgColour10}{\tgColour6}{\tgColour6}{\tgColour10}
	\tgBlank{(2,3)}{\tgColour6}
	\tgCell[(2,0)]{(1,2)}{\rho'}
	\tgCell{(2,1)}{\alpha}
	\tgArrow{(1,1.5)}{1}
	\tgArrow{(0,1.5)}{3}
	\tgArrow{(2,1.5)}{3}
	\tgArrow{(2,0.5)}{3}
	\tgArrow{(0,0.5)}{3}
	\tgArrow{(1,0.5)}{1}
	\tgArrow{(1,2.5)}{3}
	\tgAxisLabel{(0.5,0.75)}{south}{a}
	\tgAxisLabel{(1.5,0.75)}{south}{j}
	\tgAxisLabel{(2.5,0.75)}{south}{t}
	\tgAxisLabel{(1.5,3.25)}{north}{a'}
\end{tangle}
\]

We may now exhibit the $T$-opalgebras of \cref{opalgebra} as actions in $\X[j, B]$.

\begin{theorem}
    \label{opalgebras-are-right-actions}
    There is an isomorphism of categories rendering the following diagram commutative, natural in $B$ and $T$.
	\[\begin{tikzcd}
		{T\h\Opalg_B} && {\Act(\X[j, B], T)} \\
		& {\X[A, B]}
		\arrow["{U_{T, B}}"', from=1-1, to=2-2]
		\arrow["{U_{\X[j, B], T}}", from=1-3, to=2-2]
		\arrow["\iso", from=1-1, to=1-3]
	\end{tikzcd}\]
\end{theorem}

\begin{proof}
    Observe that, given an action, we can define an opalgebra structure, and conversely:
    
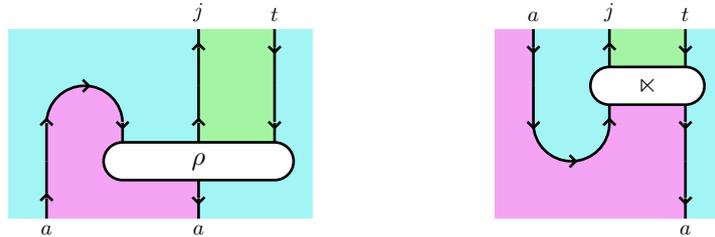
\[
	\begin{tangle}{(4,4)}[trim y]
		\tgBlank{(0,0)}{\tgColour6}
		\tgBlank{(1,0)}{\tgColour6}
		\tgBorderA{(2,0)}{\tgColour6}{\tgColour4}{\tgColour4}{\tgColour6}
		\tgBorderA{(3,0)}{\tgColour4}{\tgColour6}{\tgColour6}{\tgColour4}
		\tgBorderC{(0,1)}{3}{\tgColour6}{\tgColour10}
		\tgBorderC{(1,1)}{2}{\tgColour6}{\tgColour10}
		\tgBorderA{(2,1)}{\tgColour6}{\tgColour4}{\tgColour4}{\tgColour6}
		\tgBorderA{(3,1)}{\tgColour4}{\tgColour6}{\tgColour6}{\tgColour4}
		\tgBorderA{(0,2)}{\tgColour6}{\tgColour10}{\tgColour10}{\tgColour6}
		\tgBorderA{(1,2)}{\tgColour10}{\tgColour6}{\tgColour10}{\tgColour10}
		\tgBorderA{(2,2)}{\tgColour6}{\tgColour4}{\tgColour6}{\tgColour10}
		\tgBorderA{(3,2)}{\tgColour4}{\tgColour6}{\tgColour6}{\tgColour6}
		\tgBorderA{(0,3)}{\tgColour6}{\tgColour10}{\tgColour10}{\tgColour6}
		\tgBlank{(1,3)}{\tgColour10}
		\tgBorderA{(2,3)}{\tgColour10}{\tgColour6}{\tgColour6}{\tgColour10}
		\tgBlank{(3,3)}{\tgColour6}
		\tgCell[(2,0)]{(2,2)}{\rho}
		\tgArrow{(2,1.5)}{1}
		\tgArrow{(3,1.5)}{3}
		\tgArrow{(1,1.5)}{3}
		\tgArrow{(0.5,1)}{0}
		\tgArrow{(0,1.5)}{1}
		\tgArrow{(2,0.5)}{1}
		\tgArrow{(3,0.5)}{3}
		\tgArrow{(2,2.5)}{3}
		\tgArrow{(0,2.5)}{1}
		\tgAxisLabel{(2.5,0.75)}{south}{j}
		\tgAxisLabel{(3.5,0.75)}{south}{t}
		\tgAxisLabel{(0.5,3.25)}{north}{a}
		\tgAxisLabel{(2.5,3.25)}{north}{a}
	\end{tangle}
    \hspace{4em}
	\begin{tangle}{(3,4)}[trim y]
		\tgBorderA{(0,0)}{\tgColour10}{\tgColour6}{\tgColour6}{\tgColour10}
		\tgBorderA{(1,0)}{\tgColour6}{\tgColour4}{\tgColour4}{\tgColour6}
		\tgBorderA{(2,0)}{\tgColour4}{\tgColour6}{\tgColour6}{\tgColour4}
		\tgBorderA{(0,1)}{\tgColour10}{\tgColour6}{\tgColour6}{\tgColour10}
		\tgBorderA{(1,1)}{\tgColour6}{\tgColour4}{\tgColour10}{\tgColour6}
		\tgBorderA{(2,1)}{\tgColour4}{\tgColour6}{\tgColour6}{\tgColour10}
		\tgBorderC{(0,2)}{0}{\tgColour10}{\tgColour6}
		\tgBorderC{(1,2)}{1}{\tgColour10}{\tgColour6}
		\tgBorderA{(2,2)}{\tgColour10}{\tgColour6}{\tgColour6}{\tgColour10}
		\tgBlank{(0,3)}{\tgColour10}
		\tgBlank{(1,3)}{\tgColour10}
		\tgBorderA{(2,3)}{\tgColour10}{\tgColour6}{\tgColour6}{\tgColour10}
		\tgCell[(1,0)]{(1.5,1)}{\oop}
		\tgArrow{(2,1.5)}{3}
		\tgArrow{(1,1.5)}{1}
		\tgArrow{(0.5,2)}{0}
		\tgArrow{(0,1.5)}{3}
		\tgArrow{(2,2.5)}{3}
		\tgArrow{(2,0.5)}{3}
		\tgArrow{(1,0.5)}{1}
		\tgArrow{(0,0.5)}{3}
		\tgAxisLabel{(0.5,0.75)}{south}{a}
		\tgAxisLabel{(1.5,0.75)}{south}{j}
		\tgAxisLabel{(2.5,0.75)}{south}{t}
		\tgAxisLabel{(2.5,3.25)}{north}{a}
	\end{tangle}
    \]
    It is immediate that the laws for an opalgebra (morphism) are precisely those for an action (homomorphism) under these transformations.
\end{proof}

\begin{remark}
	Similarly to \cref{representability-for-skew-multiactegories}, the evident generalisation of \cref{Xj1-is-skew-monoidal} suggests that the skew-multiactegory $\X[j, B]$ of \cref{XjB} will be representable in an appropriate sense when $\X$ admits left extensions of tight-cells $A \to B$ along $j$. However, \textcite{altenkirch2015monads} do not give a characterisation of opalgebras for relative monads as right-actions for monoids in the skew-monoidal $\X[j]$.
\end{remark}

\begin{corollary}
	\label{opalgebras-for-identity-relative-monads-are-opalgebras-for-monads}
    An opalgebra (morphism) for a $1_A$-monad is precisely an opalgebra (morphism) for the corresponding monad on $A$ (\cref{identity-relative-monads-are-monads}).
\end{corollary}

\begin{proof}
	By \cref{opalgebras-are-right-actions}, it suffices to consider action (morphisms) in place of opalgebra (morphisms). An action in $\X[1_A, B]$ for a monad $T$ comprises a tight-cell $a \colon A \to B$ and a 2-cell $\rho \colon (t \d a) \tto a$ rendering the following diagrams commutative.
	\[
	\begin{tikzcd}
		a & {t \d a} \\
		& a
		\arrow["{\eta \d a}", from=1-1, to=1-2]
		\arrow["\rho", from=1-2, to=2-2]
		\arrow[Rightarrow, no head, from=1-1, to=2-2]
	\end{tikzcd}
	\hspace{4em}
	\begin{tikzcd}
		{t \d t \d a} & {t \d a} \\
		{t \d a} & a
		\arrow["{t \d \rho}", from=1-1, to=1-2]
		\arrow["\rho", from=1-2, to=2-2]
		\arrow["{\mu \d a}"', from=1-1, to=2-1]
		\arrow["\rho"', from=2-1, to=2-2]
	\end{tikzcd}
	\]
	An action morphism in $\X[1_A, B]$ is a 2-cell $\alpha \colon a \tto a'$ rendering the following diagram commutative.
	\[\begin{tikzcd}
		{t \d a} & {t \d a'} \\
		a & {a'}
		\arrow["\rho"', from=1-1, to=2-1]
		\arrow["\alpha"', from=2-1, to=2-2]
		\arrow["{t \d \alpha}", from=1-1, to=1-2]
		\arrow["{\rho'}", from=1-2, to=2-2]
	\end{tikzcd}\]
	This is precisely the definition of an opalgebra (morphism) for a monad (\cf{}~\cite[\S3.1]{kelly1974review}).
\end{proof}

\subsection{Relative adjunctions and (op)algebras}
\label{relative-adjunctions-and-(op)algebras}

Just as a relative adjunction induces a relative monad (\cref{relative-adjunction-induces-relative-monad}), so too does it induce an algebra and opalgebra for the induced relative monad.

\begin{proposition}
	\label{relative-adjunction-induces-algebra-and-opalgebra}
    Let $\ljr$ be a relative adjunction and denote by $T$ the induced $j$-monad. The left $j$-adjoint forms a $T$-opalgebra; and the right $j$-adjoint forms a $T$-algebra.\footnotemark{}
	\footnotetext{Note that a left relative adjoint induces a \emph{right}-action, whereas a right relative adjoint induces a \emph{left}-action. This apparent discrepancy arises from the convention to consider diagrammatic composition of tight-cells (for the adjoints), and nondiagrammatic composition of loose-cells (for the actions).}
\end{proposition}

\begin{proof}
	Define $\oop \colon E(j, r \ell) \tto C(\ell, \ell)$ to be the 2-cell on the left below, and $\aop \colon E(j, r) \tto E(r \ell, r)$ to be the 2-cell on the right below.
    
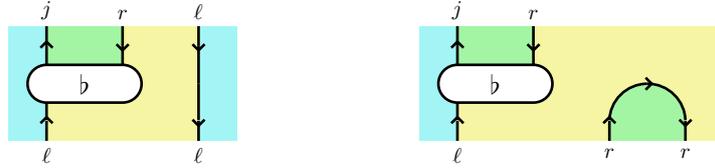
\[
	\begin{tangle}{(3,3)}[trim y]
		\tgBorderA{(0,0)}{\tgColour6}{\tgColour4}{\tgColour4}{\tgColour6}
		\tgBorderA{(1,0)}{\tgColour4}{\tgColour2}{\tgColour2}{\tgColour4}
		\tgBorderA{(2,0)}{\tgColour2}{\tgColour6}{\tgColour6}{\tgColour2}
		\tgBorderA{(0,1)}{\tgColour6}{\tgColour4}{\tgColour2}{\tgColour6}
		\tgBorderA{(1,1)}{\tgColour4}{\tgColour2}{\tgColour2}{\tgColour2}
		\tgBorderA{(2,1)}{\tgColour2}{\tgColour6}{\tgColour6}{\tgColour2}
		\tgBorderA{(0,2)}{\tgColour6}{\tgColour2}{\tgColour2}{\tgColour6}
		\tgBlank{(1,2)}{\tgColour2}
		\tgBorderA{(2,2)}{\tgColour2}{\tgColour6}{\tgColour6}{\tgColour2}
		\tgArrow{(2,0.5)}{3}
		\tgArrow{(2,1.5)}{3}
		\tgArrow{(0,1.5)}{1}
		\tgArrow{(0,0.5)}{1}
		\tgArrow{(1,0.5)}{3}
		\tgCell[(1,0)]{(0.5,1)}{\flat}
		\tgAxisLabel{(0.5,0.75)}{south}{j}
		\tgAxisLabel{(1.5,0.75)}{south}{r}
		\tgAxisLabel{(2.5,0.75)}{south}{\ell}
		\tgAxisLabel{(0.5,2.25)}{north}{\ell}
		\tgAxisLabel{(2.5,2.25)}{north}{\ell}
	\end{tangle}
	\hspace{4em}
	\begin{tangle}{(4,3)}[trim y]
		\tgBorderA{(0,0)}{\tgColour6}{\tgColour4}{\tgColour4}{\tgColour6}
		\tgBorderA{(1,0)}{\tgColour4}{\tgColour2}{\tgColour2}{\tgColour4}
		\tgBlank{(2,0)}{\tgColour2}
		\tgBlank{(3,0)}{\tgColour2}
		\tgBorderA{(0,1)}{\tgColour6}{\tgColour4}{\tgColour2}{\tgColour6}
		\tgBorderA{(1,1)}{\tgColour4}{\tgColour2}{\tgColour2}{\tgColour2}
		\tgBorderC{(2,1)}{3}{\tgColour2}{\tgColour4}
		\tgBorderC{(3,1)}{2}{\tgColour2}{\tgColour4}
		\tgBorderA{(0,2)}{\tgColour6}{\tgColour2}{\tgColour2}{\tgColour6}
		\tgBlank{(1,2)}{\tgColour2}
		\tgBorderA{(2,2)}{\tgColour2}{\tgColour4}{\tgColour4}{\tgColour2}
		\tgBorderA{(3,2)}{\tgColour4}{\tgColour2}{\tgColour2}{\tgColour4}
		\tgArrow{(0,1.5)}{1}
		\tgArrow{(0,0.5)}{1}
		\tgArrow{(1,0.5)}{3}
		\tgCell[(1,0)]{(0.5,1)}{\flat}
		\tgArrow{(2.5,1)}{0}
		\tgArrow{(3,1.5)}{3}
		\tgArrow{(2,1.5)}{1}
		\tgAxisLabel{(0.5,0.75)}{south}{j}
		\tgAxisLabel{(1.5,0.75)}{south}{r}
		\tgAxisLabel{(0.5,2.25)}{north}{\ell}
		\tgAxisLabel{(2.5,2.25)}{north}{r}
		\tgAxisLabel{(3.5,2.25)}{north}{r}
	\end{tangle}
    \]
    The proof of each $T$-(op)algebra law is analogous to the proof of the corresponding relative monad law for $T$ (\cref{relative-adjunction-induces-relative-monad}).
\end{proof}

Furthermore, just as relative adjunctions and relative monads may be constructed from existing relative adjunctions and relative monads (\cref{composition-of-relative-adjunctions}), so too may (op)algebras be constructed from existing (op)algebras, and this coheres with the process of forming (op)algebras from relative adjunctions (\cf{}~\cite[Construction~2.2.10]{voevodsky2023c}).

\begin{proposition}
	\label{algebra-and-opalgebra-precomposition}
	Let $j \colon B \to D$ be a tight-cell, let $T$ be a $j$-monad, and let $\ell' \colon A \to B$ be a tight-cell, as in \cref{relative-monad-precomposition}. Every $T$-opalgebra induces an $(\ell' \d T)$-opalgebra; and every $T$-algebra (morphism) induces an $(\ell' \d T)$-algebra (morphism), forming functors
	\begin{align*}
		\ell' \d \ph & \colon T\h\Opalg_C \to (\ell' \d T)\h\Opalg_C \\
		\ell' \d \ph & \colon T\h\Alg_C \to (\ell' \d T)\h\Alg_C
	\end{align*}
	natural in $T$ and $C$. Furthermore, for every object $C$, the following diagram commutes,
	\[\begin{tikzcd}
		{\ob{T\h\Opalg_C}} & {\ob{\Res(T)_C}} & {\ob{T\h\Alg_C}} \\
		{\ob{(\ell' \d T)\h\Opalg_C}} & {\ob{\Res(\ell' \d T)_C}} & {\ob{(\ell' \d T)\h\Alg_C}}
		\arrow["{\ob{\ell' \d \ph}}"{description}, from=1-2, to=2-2]
		\arrow["{\ob{\ell' \d \ph}}"', from=1-1, to=2-1]
		\arrow[from=1-2, to=1-1]
		\arrow[from=2-2, to=2-1]
		\arrow[from=1-2, to=1-3]
		\arrow["{\ob{\ell' \d \ph}}", from=1-3, to=2-3]
		\arrow[from=2-2, to=2-3]
	\end{tikzcd}\]
	where we denote by $\Res(T)_C$ the full subcategory of $\Res(T)$ spanned by relative adjunctions with apex $C$, and where the unlabelled functions are given by \cref{relative-adjunction-induces-algebra-and-opalgebra}.
\end{proposition}

\begin{proof}
	Given a $T$-opalgebra $(b, \oop)$, we have a 2-cell $E(j \ell', t \ell') \tto C(b \ell', b \ell')$ by precomposing $\ell'$ in both arguments; given a $T$-algebra $(d, \aop)$, we have a 2-cell $E(j \ell', d) \tto E(t \ell', d)$ by precomposing $\ell'$ in the first argument.
	The proof that these 2-cells define an $(\ell' \d T)$-opalgebra and an $(\ell' \d T)$-algebra respectively is analogous to the proof that $(\ell' \d T)$ forms an $(\ell' \d j)$-monad (\cref{relative-monad-precomposition}). Functoriality of the assignments, given by precomposing $\ell'$, together with naturality, is trivial. That the specified diagram commutes follows directly from the definitions of the respective opalgebra and algebra structures.
\end{proof}

\begin{proposition}
	\label{algebra-and-opalgebra-pasting}
	Let $\ell' \jadj r'$ be a relative adjunction and let $T = (t, \dag, \eta)$ be an $\ell'$-monad, as in \cref{relative-monad-relative-adjunction-pasting}. Every $T$-opalgebra (morphism) induces a $(T \d r')$-opalgebra; and every $T$-algebra (morphism) induces an $(T \d r')$-algebra, forming functors
	\begin{align*}
		\ph \d r' & \colon T\h\Opalg_C \to (T \d r')\h\Opalg_C \\
		\ph \d r' & \colon T\h\Alg_C \to (T \d r')\h\Alg_C
	\end{align*}
	natural in $T$ and $C$. Furthermore, for every object $C$, the following diagram commutes,
	\[\begin{tikzcd}
		{\ob{T\h\Opalg_C}} & {\ob{\Res(T)_C}} & {\ob{T\h\Alg_C}} \\
		{\ob{(T \d r')\h\Opalg_C}} & {\ob{\Res(T \d r')_C}} & {\ob{(T \d r')\h\Alg_C}}
		\arrow["{\ob{\ph \d r'}}"{description}, from=1-2, to=2-2]
		\arrow["{\ob{\ph \d r'}}"', from=1-1, to=2-1]
		\arrow[from=1-2, to=1-1]
		\arrow[from=2-2, to=2-1]
		\arrow[from=1-2, to=1-3]
		\arrow["{\ob{\ph \d r'}}", from=1-3, to=2-3]
		\arrow[from=2-2, to=2-3]
	\end{tikzcd}\]
	where the unlabelled functions are given by \cref{relative-adjunction-induces-algebra-and-opalgebra}.
\end{proposition}

\begin{proof}
	Given a $T$-opalgebra $(a, \oop)$, define a 2-cell $E(j, r' t) \iso D(\ell', t) \tto C(a, a)$ by applying the relative adjunction; given a $T$-algebra $(e, \aop)$, define a 2-cell $E(j, r' d) \tto E(r' t, r' d)$ as below.
	
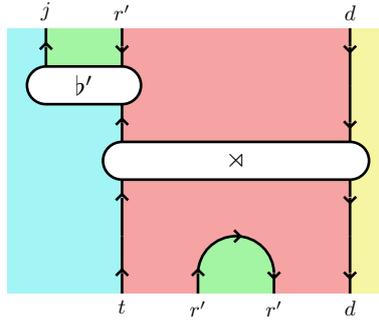
\[
	\begin{tangle}{(5,5)}[trim y]
		\tgBorderA{(0,0)}{\tgColour6}{\tgColour4}{\tgColour4}{\tgColour6}
		\tgBorderA{(1,0)}{\tgColour4}{\tgColour0}{\tgColour0}{\tgColour4}
		\tgBlank{(2,0)}{\tgColour0}
		\tgBlank{(3,0)}{\tgColour0}
		\tgBorderA{(4,0)}{\tgColour0}{\tgColour2}{\tgColour2}{\tgColour0}
		\tgBorderA{(0,1)}{\tgColour6}{\tgColour4}{\tgColour6}{\tgColour6}
		\tgBorderA{(1,1)}{\tgColour4}{\tgColour0}{\tgColour0}{\tgColour6}
		\tgBlank{(2,1)}{\tgColour0}
		\tgBlank{(3,1)}{\tgColour0}
		\tgBorderA{(4,1)}{\tgColour0}{\tgColour2}{\tgColour2}{\tgColour0}
		\tgBlank{(0,2)}{\tgColour6}
		\tgBorderA{(1,2)}{\tgColour6}{\tgColour0}{\tgColour0}{\tgColour6}
		\tgBorder{(1,2)}{0}{1}{0}{0}
		\tgBorderA{(2,2)}{\tgColour0}{\tgColour0}{\tgColour0}{\tgColour0}
		\tgBorder{(2,2)}{0}{1}{0}{1}
		\tgBorderA{(3,2)}{\tgColour0}{\tgColour0}{\tgColour0}{\tgColour0}
		\tgBorder{(3,2)}{0}{1}{0}{1}
		\tgBorderA{(4,2)}{\tgColour0}{\tgColour2}{\tgColour2}{\tgColour0}
		\tgBorder{(4,2)}{0}{0}{0}{1}
		\tgBlank{(0,3)}{\tgColour6}
		\tgBorderA{(1,3)}{\tgColour6}{\tgColour0}{\tgColour0}{\tgColour6}
		\tgBorderC{(2,3)}{3}{\tgColour0}{\tgColour4}
		\tgBorderC{(3,3)}{2}{\tgColour0}{\tgColour4}
		\tgBorderA{(4,3)}{\tgColour0}{\tgColour2}{\tgColour2}{\tgColour0}
		\tgBlank{(0,4)}{\tgColour6}
		\tgBorderA{(1,4)}{\tgColour6}{\tgColour0}{\tgColour0}{\tgColour6}
		\tgBorderA{(2,4)}{\tgColour0}{\tgColour4}{\tgColour4}{\tgColour0}
		\tgBorderA{(3,4)}{\tgColour4}{\tgColour0}{\tgColour0}{\tgColour4}
		\tgBorderA{(4,4)}{\tgColour0}{\tgColour2}{\tgColour2}{\tgColour0}
		\tgCell[(3,0)]{(2.5,2)}{\aop}
		\tgCell[(1,0)]{(0.5,1)}{\flat'}
		\tgArrow{(2,3.5)}{1}
		\tgArrow{(3,3.5)}{3}
		\tgArrow{(2.5,3)}{0}
		\tgArrow{(1,2.5)}{1}
		\tgArrow{(1,3.5)}{1}
		\tgArrow{(1,1.5)}{1}
		\tgArrow{(1,0.5)}{3}
		\tgArrow{(4,0.5)}{3}
		\tgArrow{(4,1.5)}{3}
		\tgArrow{(4,2.5)}{3}
		\tgArrow{(4,3.5)}{3}
		\tgArrow{(0,0.5)}{1}
		\tgAxisLabel{(0.5,0.75)}{south}{j}
		\tgAxisLabel{(1.5,0.75)}{south}{r'}
		\tgAxisLabel{(4.5,0.75)}{south}{d}
		\tgAxisLabel{(1.5,4.25)}{north}{t}
		\tgAxisLabel{(2.5,4.25)}{north}{r'}
		\tgAxisLabel{(3.5,4.25)}{north}{r'}
		\tgAxisLabel{(4.5,4.25)}{north}{d}
	\end{tangle}
	\]
	The proof that these 2-cells define a $(T \d r')$-opalgebra and a $(T \d r')$-algebra respectively is analogous to the proof that $(T \d r')$ forms a $j$-monad (\cref{relative-monad-relative-adjunction-pasting}). Functoriality of the assignments, given by postcomposing $r'$, together with naturality, is trivial. That the specified diagram commutes follows directly from the definitions of the respective opalgebra and algebra structures.
\end{proof}

\begin{remark}
	The proof of \cref{algebra-and-opalgebra-pasting} actually establishes a stronger property: that the functor $(\ph \d r') \colon T\h\Opalg_C \to (T \d r')\h\Opalg_C$ is invertible.
\end{remark}

The converse to \cref{relative-adjunction-induces-algebra-and-opalgebra} is not generally true: that is, not every (op)algebra arises from a relative adjunction. However, we might be led to wonder whether there are any natural conditions on an algebra or opalgebra that ensure they arise from a relative adjunction. In the case of non-relative monads, the answer is affirmative: namely, an algebra object for a monad $T$ is always induced by a resolution of $T$~\cite[Theorem~2 \& Theorem~3]{street1972formal}; and the same is true for an opalgebra object by duality.

We should like to deduce something similar for relative monads. However, the na\"ive definitions of (op)algebra objects, defined to be (op)algebras universal with respect to (op)algebra morphisms turn out to be insufficient. Instead, it is necessary to identify stronger universal properties making use of the equipment structure.

\subsection{Algebra objects}
\label{algebra-objects}

The definition of algebra morphism in \cref{algebra} is given only between algebras with the same domain. We now give a more general definition of (graded) morphism between any algebras for a relative monad, which is necessary to express the universal property of algebra objects for relative monads.

\begin{definition}
	\label{graded-algebra-morphism}
    Let $(e \colon D \to E, \aop)$ and $(e' \colon D' \to E, \aop')$ be $T$-algebras. A \emph{$(p_1, \ldots, p_n)$-graded $T$-algebra morphism} from $(e, \aop)$ to $(e', \aop')$ is a 2-cell \[\epsilon \colon E(1, e), p_1, \ldots, p_n \tto E(1, e')\] satisfying the following equation (defining $\lambda$ and $\lambda'$ as in \cref{algebras-are-left-actions}).
    \[
	\hspace{-.2em}
	\begin{tikzcd}
		E & A & D & \cdots & {D'} \\
		E && D & \cdots & {D'} \\
		E &&&& {D'}
		\arrow[""{name=0, anchor=center, inner sep=0}, Rightarrow, no head, from=1-5, to=2-5]
		\arrow["{p_n}"', "\shortmid"{marking}, from=1-5, to=1-4]
		\arrow["{E(1, e')}", "\shortmid"{marking}, from=3-5, to=3-1]
		\arrow["{p_1}"', "\shortmid"{marking}, from=1-4, to=1-3]
		\arrow[draw=none, from=3-1, to=2-1]
		\arrow[""{name=1, anchor=center, inner sep=0}, Rightarrow, no head, from=2-1, to=3-1]
		\arrow["{E(j, e)}"', "\shortmid"{marking}, from=1-3, to=1-2]
		\arrow[""{name=2, anchor=center, inner sep=0}, Rightarrow, no head, from=2-5, to=3-5]
		\arrow[""{name=3, anchor=center, inner sep=0}, Rightarrow, no head, from=1-3, to=2-3]
		\arrow["{E(1, e)}"{description}, from=2-3, to=2-1]
		\arrow["{p_n}"{description}, from=2-5, to=2-4]
		\arrow["{p_1}"{description}, from=2-4, to=2-3]
		\arrow["{E(1, t)}"', "\shortmid"{marking}, from=1-2, to=1-1]
		\arrow[""{name=4, anchor=center, inner sep=0}, Rightarrow, no head, from=1-1, to=2-1]
		\arrow["{=}"{description}, draw=none, from=0, to=3]
		\arrow["\epsilon"{description}, draw=none, from=2, to=1]
		\arrow["\lambda"{description}, draw=none, from=3, to=4]
	\end{tikzcd}
	\hspace{.2em} = \hspace{.2em}
	\begin{tikzcd}
		E & A & E & D & \cdots & {D'} \\
		E & A & E &&& {D'} \\
		E &&&&& {D'}
		\arrow["{E(1, e')}", "\shortmid"{marking}, from=3-6, to=3-1]
		\arrow[""{name=0, anchor=center, inner sep=0}, Rightarrow, no head, from=1-6, to=2-6]
		\arrow["{p_n}"', "\shortmid"{marking}, from=1-6, to=1-5]
		\arrow["{p_1}"', "\shortmid"{marking}, from=1-5, to=1-4]
		\arrow[""{name=1, anchor=center, inner sep=0}, Rightarrow, no head, from=2-6, to=3-6]
		\arrow["{E(1, t)}"', "\shortmid"{marking}, from=1-2, to=1-1]
		\arrow[""{name=2, anchor=center, inner sep=0}, Rightarrow, no head, from=1-1, to=2-1]
		\arrow[""{name=3, anchor=center, inner sep=0}, Rightarrow, no head, from=2-1, to=3-1]
		\arrow[""{name=4, anchor=center, inner sep=0}, Rightarrow, no head, from=1-2, to=2-2]
		\arrow["{E(1, t)}"{description}, from=2-2, to=2-1]
		\arrow["{E(1, e)}"', "\shortmid"{marking}, from=1-4, to=1-3]
		\arrow["{E(j, 1)}"', "\shortmid"{marking}, from=1-3, to=1-2]
		\arrow["{E(j, 1)}"{description}, from=2-3, to=2-2]
		\arrow[""{name=5, anchor=center, inner sep=0}, Rightarrow, no head, from=1-3, to=2-3]
		\arrow["{E(1, e')}"{description}, from=2-6, to=2-3]
		\arrow["{\lambda'}"{description}, draw=none, from=1, to=3]
		\arrow["{=}"{description}, draw=none, from=4, to=2]
		\arrow["{=}"{description}, draw=none, from=5, to=4]
		\arrow["\epsilon"{description}, draw=none, from=0, to=5]
	\end{tikzcd}
	\]\[
	\begin{tangle}{(6,4)}[trim y]
		\tgBorderA{(0,0)}{\tgColour4}{\tgColour6}{\tgColour6}{\tgColour4}
		\tgBorderA{(1,0)}{\tgColour6}{\tgColour4}{\tgColour4}{\tgColour6}
		\tgBorderA{(2,0)}{\tgColour4}{\tgColour0}{\tgColour0}{\tgColour4}
		\tgBorderA{(3,0)}{\tgColour0}{white}{white}{\tgColour0}
		\tgBlank{(4,0)}{white}
		\tgBorderA{(5,0)}{white}{\tgColour11}{\tgColour11}{white}
		\tgBorderA{(0,1)}{\tgColour4}{\tgColour6}{\tgColour4}{\tgColour4}
		\tgBorderA{(1,1)}{\tgColour6}{\tgColour4}{\tgColour0}{\tgColour4}
		\tgBorderA{(2,1)}{\tgColour4}{\tgColour0}{\tgColour0}{\tgColour0}
		\tgBorderA{(3,1)}{\tgColour0}{white}{white}{\tgColour0}
		\tgBlank{(4,1)}{white}
		\tgBorderA{(5,1)}{white}{\tgColour11}{\tgColour11}{white}
		\tgBlank{(0,2)}{\tgColour4}
		\tgBorderA{(1,2)}{\tgColour4}{\tgColour0}{\tgColour11}{\tgColour4}
		\tgBorderA{(2,2)}{\tgColour0}{\tgColour0}{\tgColour11}{\tgColour11}
		\tgBorderA{(3,2)}{\tgColour0}{white}{\tgColour11}{\tgColour11}
		\tgBorderA{(4,2)}{white}{white}{\tgColour11}{\tgColour11}
		\tgBorderA{(5,2)}{white}{\tgColour11}{\tgColour11}{\tgColour11}
		\tgBlank{(0,3)}{\tgColour4}
		\tgBorderA{(1,3)}{\tgColour4}{\tgColour11}{\tgColour11}{\tgColour4}
		\tgBlank{(2,3)}{\tgColour11}
		\tgBlank{(3,3)}{\tgColour11}
		\tgBlank{(4,3)}{\tgColour11}
		\tgBlank{(5,3)}{\tgColour11}
		\tgArrow{(2,0.5)}{3}
		\tgArrow{(1,1.5)}{3}
		\tgArrow{(1,2.5)}{3}
		\tgArrow{(1,0.5)}{1}
		\tgCell[(4,0)]{(3,2)}{\epsilon}
		\tgCell[(2,0)]{(1,1)}{\lambda}
		\tgArrow{(0,0.5)}{3}
		\tgAxisLabel{(0.5,0.75)}{south}{t}
		\tgAxisLabel{(1.5,0.75)}{south}{j}
		\tgAxisLabel{(2.5,0.75)}{south}{e}
		\tgAxisLabel{(3.5,0.75)}{south}{p_1}
		\tgAxisLabel{(5.5,0.75)}{south}{p_n}
		\tgAxisLabel{(1.5,3.25)}{north}{e'}
		\node at (4.5,1.9) {$\cdots$};
	\end{tangle}
	\tangleeq*
	\begin{tangle}{(6,4)}[trim y]
		\tgBorderA{(0,0)}{\tgColour4}{\tgColour6}{\tgColour6}{\tgColour4}
		\tgBorderA{(1,0)}{\tgColour6}{\tgColour4}{\tgColour4}{\tgColour6}
		\tgBorderA{(2,0)}{\tgColour4}{\tgColour0}{\tgColour0}{\tgColour4}
		\tgBorderA{(3,0)}{\tgColour0}{white}{white}{\tgColour0}
		\tgBlank{(4,0)}{white}
		\tgBorderA{(5,0)}{white}{\tgColour11}{\tgColour11}{white}
		\tgBorderA{(0,1)}{\tgColour4}{\tgColour6}{\tgColour6}{\tgColour4}
		\tgBorderA{(1,1)}{\tgColour6}{\tgColour4}{\tgColour4}{\tgColour6}
		\tgBorderA{(2,1)}{\tgColour4}{\tgColour0}{\tgColour11}{\tgColour4}
		\tgBorderA{(3,1)}{\tgColour0}{white}{\tgColour11}{\tgColour11}
		\tgBorderA{(4,1)}{white}{white}{\tgColour11}{\tgColour11}
		\tgBorderA{(5,1)}{white}{\tgColour11}{\tgColour11}{\tgColour11}
		\tgBorderA{(0,2)}{\tgColour4}{\tgColour6}{\tgColour4}{\tgColour4}
		\tgBorderA{(1,2)}{\tgColour6}{\tgColour4}{\tgColour11}{\tgColour4}
		\tgBorderA{(2,2)}{\tgColour4}{\tgColour11}{\tgColour11}{\tgColour11}
		\tgBlank{(3,2)}{\tgColour11}
		\tgBlank{(4,2)}{\tgColour11}
		\tgBlank{(5,2)}{\tgColour11}
		\tgBlank{(0,3)}{\tgColour4}
		\tgBorderA{(1,3)}{\tgColour4}{\tgColour11}{\tgColour11}{\tgColour4}
		\tgBlank{(2,3)}{\tgColour11}
		\tgBlank{(3,3)}{\tgColour11}
		\tgBlank{(4,3)}{\tgColour11}
		\tgBlank{(5,3)}{\tgColour11}
		\tgCell[(3,0)]{(3.5,1)}{\epsilon}
		\tgArrow{(2,0.5)}{3}
		\tgArrow{(2,1.5)}{3}
		\tgArrow{(1,0.5)}{1}
		\tgArrow{(1,1.5)}{1}
		\tgArrow{(1,2.5)}{3}
		\tgCell[(2,0)]{(1,2)}{\lambda'}
		\tgArrow{(0,0.5)}{3}
		\tgArrow{(0,1.5)}{3}
		\tgAxisLabel{(0.5,0.75)}{south}{t}
		\tgAxisLabel{(1.5,0.75)}{south}{j}
		\tgAxisLabel{(2.5,0.75)}{south}{e}
		\tgAxisLabel{(3.5,0.75)}{south}{p_1}
		\tgAxisLabel{(5.5,0.75)}{south}{p_n}
		\tgAxisLabel{(1.5,3.25)}{north}{e'}
		\node at (4.5,.9) {$\cdots$};
	\end{tangle}
	\]
	When $n = 0$, we call such a morphism \emph{ungraded}.
\end{definition}

In particular, ungraded algebra morphisms are precisely those given in \cref{algebra}.

\begin{remark}
    \label{algebra-morphisms-equivalent}
	It will be convenient to have the following alternative description of \cref{graded-algebra-morphism}.
	A $(p_1, \dots, p_n)$-graded $T$-algebra morphism is equivalently a 2-cell
	\[\epsilon \colon p_1, \ldots, p_n \tto E(e, e')\]
  	satisfying the following equation.
    \[
	\begin{tikzcd}[column sep=large]
		A & D & \cdots & {D'} \\
		A & D && {D'} \\
		A &&& {D'}
		\arrow["{E(e, e')}"{description}, from=2-4, to=2-2]
		\arrow[""{name=0, anchor=center, inner sep=0}, Rightarrow, no head, from=1-2, to=2-2]
		\arrow[""{name=1, anchor=center, inner sep=0}, Rightarrow, no head, from=1-4, to=2-4]
		\arrow["{p_n}"', "\shortmid"{marking}, from=1-4, to=1-3]
		\arrow["{E(t, e')}", "\shortmid"{marking}, from=3-4, to=3-1]
		\arrow["{p_1}"', "\shortmid"{marking}, from=1-3, to=1-2]
		\arrow[draw=none, from=3-1, to=2-1]
		\arrow[""{name=2, anchor=center, inner sep=0}, Rightarrow, no head, from=2-1, to=3-1]
		\arrow["{E(t, e)}"{description}, from=2-2, to=2-1]
		\arrow["{E(j, e)}"', "\shortmid"{marking}, from=1-2, to=1-1]
		\arrow[""{name=3, anchor=center, inner sep=0}, Rightarrow, no head, from=1-1, to=2-1]
		\arrow[""{name=4, anchor=center, inner sep=0}, Rightarrow, no head, from=2-4, to=3-4]
		\arrow["\epsilon"{description}, draw=none, from=1, to=0]
		\arrow["\aop"{description}, draw=none, from=0, to=3]
		\arrow["{\cp e(t, e')}"{description}, draw=none, from=4, to=2]
	\end{tikzcd}
    \quad = \quad
	\begin{tikzcd}[column sep=large]
		A & D & \cdots & {D'} \\
		A & D && {D'} \\
		A &&& {D'} \\
		A &&& {D'}
		\arrow[""{name=0, anchor=center, inner sep=0}, Rightarrow, no head, from=1-2, to=2-2]
		\arrow["{E(t, e')}", "\shortmid"{marking}, from=4-4, to=4-1]
		\arrow[""{name=1, anchor=center, inner sep=0}, Rightarrow, no head, from=1-4, to=2-4]
		\arrow["{E(e, e')}"{description}, from=2-4, to=2-2]
		\arrow["{p_n}"', "\shortmid"{marking}, from=1-4, to=1-3]
		\arrow["{p_1}"', "\shortmid"{marking}, from=1-3, to=1-2]
		\arrow["{E(j, e)}"', "\shortmid"{marking}, from=1-2, to=1-1]
		\arrow["{E(j, e)}"{description}, from=2-2, to=2-1]
		\arrow[""{name=2, anchor=center, inner sep=0}, Rightarrow, no head, from=1-1, to=2-1]
		\arrow[""{name=3, anchor=center, inner sep=0}, Rightarrow, no head, from=2-1, to=3-1]
		\arrow[""{name=4, anchor=center, inner sep=0}, Rightarrow, no head, from=2-4, to=3-4]
		\arrow["{E(j, e')}"{description}, from=3-4, to=3-1]
		\arrow[""{name=5, anchor=center, inner sep=0}, Rightarrow, no head, from=3-4, to=4-4]
		\arrow[""{name=6, anchor=center, inner sep=0}, Rightarrow, no head, from=3-1, to=4-1]
		\arrow["\epsilon"{description}, draw=none, from=1, to=0]
		\arrow["{\cp e(j, e')}"{description}, draw=none, from=4, to=3]
		\arrow["{\aop'}"{description}, draw=none, from=5, to=6]
		\arrow["{=}"{description}, draw=none, from=0, to=2]
	\end{tikzcd}
    \]\[
	\begin{tangle}{(5,4)}[trim y]
		\tgBorderA{(0,0)}{\tgColour6}{\tgColour4}{\tgColour4}{\tgColour6}
		\tgBorderA{(1,0)}{\tgColour4}{\tgColour0}{\tgColour0}{\tgColour4}
		\tgBorderA{(2,0)}{\tgColour0}{white}{white}{\tgColour0}
		\tgBlank{(3,0)}{white}
		\tgBorderA{(4,0)}{white}{\tgColour11}{\tgColour11}{white}
		\tgBorderA{(0,1)}{\tgColour6}{\tgColour4}{\tgColour4}{\tgColour6}
		\tgBorder{(0,1)}{0}{1}{0}{0}
		\tgBorderA{(1,1)}{\tgColour4}{\tgColour0}{\tgColour0}{\tgColour4}
		\tgBorder{(1,1)}{0}{0}{0}{1}
		\tgBorderA{(2,1)}{\tgColour0}{white}{white}{\tgColour0}
		\tgBlank{(3,1)}{white}
		\tgBorderA{(4,1)}{white}{\tgColour11}{\tgColour11}{white}
		\tgBorderA{(0,2)}{\tgColour6}{\tgColour4}{\tgColour4}{\tgColour6}
		\tgBorderA{(1,2)}{\tgColour4}{\tgColour0}{\tgColour11}{\tgColour4}
		\tgBorderA{(2,2)}{\tgColour0}{white}{\tgColour11}{\tgColour11}
		\tgBorderA{(3,2)}{white}{white}{\tgColour11}{\tgColour11}
		\tgBorderA{(4,2)}{white}{\tgColour11}{\tgColour11}{\tgColour11}
		\tgBorderA{(0,3)}{\tgColour6}{\tgColour4}{\tgColour4}{\tgColour6}
		\tgBorderA{(1,3)}{\tgColour4}{\tgColour11}{\tgColour11}{\tgColour4}
		\tgBlank{(2,3)}{\tgColour11}
		\tgBlank{(3,3)}{\tgColour11}
		\tgBlank{(4,3)}{\tgColour11}
		\tgCell[(3,0)]{(2.5,2)}{\epsilon}
		\tgArrow{(1,1.5)}{3}
		\tgArrow{(1,2.5)}{3}
		\tgCell[(1,0)]{(0.5,1)}{\aop}
		\tgArrow{(1,0.5)}{3}
		\tgArrow{(0,1.5)}{1}
		\tgArrow{(0,2.5)}{1}
		\tgArrow{(0,0.5)}{1}
		\tgAxisLabel{(0.5,0.75)}{south}{j}
		\tgAxisLabel{(1.5,0.75)}{south}{e}
		\tgAxisLabel{(2.5,0.75)}{south}{p_1}
		\tgAxisLabel{(4.5,0.75)}{south}{p_n}
		\tgAxisLabel{(0.5,3.25)}{north}{t}
		\tgAxisLabel{(1.5,3.25)}{north}{e'}
		\node at (3.5,1.9) {$\cdots$};
	\end{tangle}
	\tangleeq*
	\begin{tangle}{(5,4)}[trim y]
		\tgBorderA{(0,0)}{\tgColour6}{\tgColour4}{\tgColour4}{\tgColour6}
		\tgBorderA{(1,0)}{\tgColour4}{\tgColour0}{\tgColour0}{\tgColour4}
		\tgBorderA{(2,0)}{\tgColour0}{white}{white}{\tgColour0}
		\tgBlank{(3,0)}{white}
		\tgBorderA{(4,0)}{white}{\tgColour11}{\tgColour11}{white}
		\tgBorderA{(0,1)}{\tgColour6}{\tgColour4}{\tgColour4}{\tgColour6}
		\tgBorderA{(1,1)}{\tgColour4}{\tgColour0}{\tgColour11}{\tgColour4}
		\tgBorderA{(2,1)}{\tgColour0}{white}{\tgColour11}{\tgColour11}
		\tgBorderA{(3,1)}{white}{white}{\tgColour11}{\tgColour11}
		\tgBorderA{(4,1)}{white}{\tgColour11}{\tgColour11}{\tgColour11}
		\tgBorderA{(0,2)}{\tgColour6}{\tgColour4}{\tgColour4}{\tgColour6}
		\tgBorder{(0,2)}{0}{1}{0}{0}
		\tgBorderA{(1,2)}{\tgColour4}{\tgColour11}{\tgColour11}{\tgColour4}
		\tgBorder{(1,2)}{0}{0}{0}{1}
		\tgBlank{(2,2)}{\tgColour11}
		\tgBlank{(3,2)}{\tgColour11}
		\tgBlank{(4,2)}{\tgColour11}
		\tgBorderA{(0,3)}{\tgColour6}{\tgColour4}{\tgColour4}{\tgColour6}
		\tgBorderA{(1,3)}{\tgColour4}{\tgColour11}{\tgColour11}{\tgColour4}
		\tgBlank{(2,3)}{\tgColour11}
		\tgBlank{(3,3)}{\tgColour11}
		\tgBlank{(4,3)}{\tgColour11}
		\tgCell[(3,0)]{(2.5,1)}{\epsilon}
		\tgArrow{(1,0.5)}{3}
		\tgArrow{(1,1.5)}{3}
		\tgCell[(1,0)]{(0.5,2)}{\aop'}
		\tgArrow{(0,0.5)}{1}
		\tgArrow{(0,1.5)}{1}
		\tgArrow{(0,2.5)}{1}
		\tgArrow{(1,2.5)}{3}
		\tgAxisLabel{(0.5,0.75)}{south}{j}
		\tgAxisLabel{(1.5,0.75)}{south}{e}
		\tgAxisLabel{(2.5,0.75)}{south}{p_1}
		\tgAxisLabel{(4.5,0.75)}{south}{p_n}
		\tgAxisLabel{(0.5,3.25)}{north}{t}
		\tgAxisLabel{(1.5,3.25)}{north}{e'}
		\node at (3.5,.9) {$\cdots$};
	\end{tangle}
	\]
\end{remark}

\begin{example}
	\label{algebra-extension-operator-is-graded-morphism}
    For a given relative monad $T$, the extension operator $\aop \colon E(j, e) \tto E(t, e)$ of a $T$-algebra $(e, \aop)$ is an $(E(j, e))$-graded $T$-algebra morphism from $(t, \dag)$ to $(e, \aop)$ in the sense of \cref{algebra-morphisms-equivalent}, the law for the graded morphism being precisely the extension law for the algebra.
\end{example}

\begin{remark}
	For each relative monad $T$, the $\ob\X$-indexed family of categories $T\h\Alg_{\ph}$ (\cref{algebra}) assembles into a category \emph{locally graded} by $\X$. Its morphisms, which are graded by chains of loose-cells in $\X$, are the graded algebra morphisms of \cref{graded-algebra-morphism}; the same is true of the graded opalgebra morphisms of \cref{graded-opalgebra-morphism}. The construction of an $(\ell' \d T)$-(op)algebra from a $T$-(op)algebra described in \cref{algebra-and-opalgebra-precomposition}, and the construction of a $(T \d r')$-(op)algebra from a $T$-(op)algebra described in \cref{algebra-and-opalgebra-pasting} then extend to locally graded functors between locally graded categories of (op)algebras. Since explicating this idea would require the introduction of categories locally graded by \vdcs{}, rather than simply monoidal categories as have been considered in the literature~\cite{wood1976indicial,levy2019locally}, we do not do so here.
\end{remark}

\begin{definition}
	\label{algebra-object}
	Let $T$ be a relative monad. A $T$-algebra $(u_T \colon \Alg(T) \to E, \aop_T)$ is called an
	\emph{algebra object} for $T$ when
	\begin{enumerate}
		\item \label{algebra-object-tight-cell-UP} for every $T$-algebra $(e \colon D \to E, \aop)$, there is a
		unique tight-cell $\unit_{(e, \aop)} \colon D \to \Alg(T)$ such that $\unit_{(e, \aop)} \d u_T = e$ and $\aop_T(1, \unit_{(e, \aop)}) = \aop$;
		\item \label{algebra-object-2-cell-UP} for every graded $T$-algebra morphism $\epsilon \colon E(1, e), p_1, \dots, p_n \tto E(1, e')$ there is a unique 2-cell $\unit_\epsilon \colon \Alg(T)(1, \unit_{(e, \aop)}), p_1, \dots, p_n \tto \Alg(T)(1, \unit_{(e', \aop')})$ such that:
		\[
		\epsilon~=~
		\begin{tikzcd}
			E & D & \cdots & {D'} \\
			E &&& {D'}
			\arrow[""{name=0, anchor=center, inner sep=0}, Rightarrow, no head, from=1-4, to=2-4]
			\arrow["{p_n}"', "\shortmid"{marking}, from=1-4, to=1-3]
			\arrow["{p_1}"', "\shortmid"{marking}, from=1-3, to=1-2]
			\arrow["{E(1, e)}"', "\shortmid"{marking}, from=1-2, to=1-1]
			\arrow["{E(1, e')}", "\shortmid"{marking}, from=2-4, to=2-1]
			\arrow[""{name=1, anchor=center, inner sep=0}, Rightarrow, no head, from=1-1, to=2-1]
			\arrow["{E(1, u_T), \unit_\epsilon}"{description}, draw=none, from=0, to=1]
		\end{tikzcd}
		\]
	\end{enumerate}
\end{definition}

In writing $\epsilon$ as a 2-cell $(E(1, u_T), \unit_\epsilon)$ in \ref{algebra-object-2-cell-UP} above, we are using the fact (from \ref{algebra-object-tight-cell-UP}) that $\unit_{(e, \aop)} \d u_T = e$.
The property $\aop_T(1, \unit_{(e, \aop)}) = \aop$ implies that the equation in \ref{algebra-object-2-cell-UP} expresses a bijection between graded $T$-algebra morphisms $E(1, e), p_1, \dots, p_n \tto E(1, e')$ and 2-cells $\Alg(T)(1, \unit_{(e, \aop)}), p_1, \dots, p_n \tto \Alg(T){(1, \unit_{(e', \aop')})}$.

Our definition of algebra object deserves some elaboration. The motivation for \cref{algebra-object} is to capture the universal property possessed by the category of algebras for a relative monad. To reconstruct the 2-categorical universal property of a given category $C$, it is necessary not only to consider the objects and morphisms of the category (viewed as \emph{global objects}, \ie{} functors $1 \to C$, and \emph{global morphisms}, \ie{} natural transformations between such functors), but, more generally, arbitrary functors into $C$ (viewed as \emph{generalised objects}) and natural transformations between them (viewed as \emph{generalised morphisms}). Doing so in the case of the category of algebras for an (enriched) monad leads to the universal property considered by \textcite[Proposition~II.1.1]{dubuc1970kan} and formalised in the definition of an algebra object by \textcite[\S1]{street1972formal} (\cf{}~\cite[\S3.3]{kelly1974review}). However, in a \ve{} we may identify a stronger universal property than the 2-categorical one: rather than considering generalised morphisms, we may consider \emph{graded} generalised morphisms, which permits the consideration of 2-cells $p_1, \ldots, p_n \tto C(c, c')$ between tight-cells $c \colon X \to C$ and $c' \colon X' \to C$ with different domains. \Cref{algebra-object} thus arises by considering the virtual double categorical universal property of the category of algebras for a relative monad in terms of generalised objects and graded generalised morphisms.

Consequently, the universal property of \cref{algebra-object} is stronger than that of an algebra object in the sense of \cite{street1972formal}, even for non-relative monads: while, as observed in \cref{algebras-for-identity-relative-monads-are-algebras-for-monads}, algebras and ungraded morphisms in the sense of \cref{algebra} for $1_E$-monads are algebras and morphisms for monads on $E$ in the classical sense, the definition of graded algebra morphism cannot be stated in an arbitrary 2-category. In \cref{relative-monads-in-VCat}, we will show that algebra objects for relative monads exist in $\VCat$, and thus that algebra objects for enriched (relative) monads satisfy a stronger universal property with respect to distributors than has traditionally been recognised. Similar considerations apply to the definition of opalgebra object given in the next section.

In light of \cref{relative-adjunction-induces-algebra-and-opalgebra}, given a relative adjunction $\ljr$ inducing a monad $T$ that admits an algebra object, we will denote by $\unit_{\ljr} \colon C \to \Alg(T)$ the mediating morphism induced by the $T$-algebra structure on the right $j$-adjoint.

\begin{definition}
	Let $T$ be a relative monad admitting an algebra object. Denote by ${f_T \colon A \to \Alg(T)}$ the mediating tight-cell $\unit_{(t, \dag)}$ induced by the $T$-algebra $(t, \dag)$ (\cref{relative-monad-forms-algebra}).
\end{definition}

\begin{lemma}
	\label{algebra-object-induces-resolution}
	Let $T$ be a relative monad. If $T$ admits an algebra object, then $T$ admits a resolution.
	\[\begin{tikzcd}
		& {\Alg(T)} \\
		A && E
		\arrow[""{name=0, anchor=center, inner sep=0}, "{f_T}", from=2-1, to=1-2]
		\arrow[""{name=1, anchor=center, inner sep=0}, "{u_T}", from=1-2, to=2-3]
		\arrow["j"', from=2-1, to=2-3]
		\arrow["\dashv"{anchor=center}, shift right=2, draw=none, from=0, to=1]
	\end{tikzcd}\]
\end{lemma}

\begin{proof}
	The unit of $T$ is a 2-cell $\eta \colon j \tto (t = f_T \d u_T)$. By \cref{algebra-extension-operator-is-graded-morphism}, the extension operator $\aop_T$ is an $(E(j, u_T))$-graded algebra morphism from $(t, \dag)$ to $(u_T, \aop_T)$, so that the universal property of the algebra object (with respect to algebra morphisms in the form of \cref{algebra-morphisms-equivalent}) induces a 2-cell $\unit_{\aop_T} \colon E(j, u_T) \tto \Alg(T)(f_T, 1)$. We prove that $\eta$ together with $\unit_{\aop_T}$ forms a $j$-adjunction (in universal arrow form).
	First, we have
	\begin{tangleeqs}
	\begin{tangle}{(4,3)}[trim y=.25]
		\tgBlank{(0,0)}{\tgColour6}
		\tgBlank{(1,0)}{\tgColour6}
		\tgBorderA{(2,0)}{\tgColour6}{\tgColour4}{\tgColour4}{\tgColour6}
		\tgBorderA{(3,0)}{\tgColour4}{\tgColour0}{\tgColour0}{\tgColour4}
		\tgBorderA{(0,1)}{\tgColour6}{\tgColour6}{\tgColour4}{\tgColour6}
		\tgBorderA{(1,1)}{\tgColour6}{\tgColour6}{\tgColour0}{\tgColour4}
		\tgBorderA{(2,1)}{\tgColour6}{\tgColour4}{\tgColour0}{\tgColour0}
		\tgBorderA{(3,1)}{\tgColour4}{\tgColour0}{\tgColour0}{\tgColour0}
		\tgBorderA{(0,2)}{\tgColour6}{\tgColour4}{\tgColour4}{\tgColour6}
		\tgBorderA{(1,2)}{\tgColour4}{\tgColour0}{\tgColour0}{\tgColour4}
		\tgBlank{(2,2)}{\tgColour0}
		\tgBlank{(3,2)}{\tgColour0}
		\tgCell[(1,0)]{(0.5,1)}{\eta}
		\tgCell[(1,0)]{(2.5,1)}{\unit_{\aop_T}}
		\tgArrow{(2,0.5)}{1}
		\tgArrow{(3,0.5)}{3}
		\tgArrow{(1,1.5)}{3}
		\tgArrow{(0,1.5)}{1}
		\tgArrow{(1.5,1)}{0}
		\tgAxisLabel{(2.5,0.25)}{south}{j}
		\tgAxisLabel{(3.5,0.25)}{south}{u_T}
		\tgAxisLabel{(0.5,2.75)}{north}{j}
		\tgAxisLabel{(1.5,2.75)}{north}{u_T}
	\end{tangle}
    \=
	\begin{tangle}{(5,4)}[trim y]
		\tgBlank{(0,0)}{\tgColour6}
		\tgBlank{(1,0)}{\tgColour6}
		\tgBlank{(2,0)}{\tgColour6}
		\tgBorderA{(3,0)}{\tgColour6}{\tgColour4}{\tgColour4}{\tgColour6}
		\tgBorderA{(4,0)}{\tgColour4}{\tgColour0}{\tgColour0}{\tgColour4}
		\tgBorderA{(0,1)}{\tgColour6}{\tgColour6}{\tgColour4}{\tgColour6}
		\tgBorderA{(1,1)}{\tgColour6}{\tgColour6}{\tgColour0}{\tgColour4}
		\tgBorderA{(2,1)}{\tgColour6}{\tgColour6}{\tgColour6}{\tgColour0}
		\tgBorderA{(3,1)}{\tgColour6}{\tgColour4}{\tgColour4}{\tgColour6}
		\tgBorderA{(4,1)}{\tgColour4}{\tgColour0}{\tgColour0}{\tgColour4}
		\tgBorderA{(0,2)}{\tgColour6}{\tgColour4}{\tgColour4}{\tgColour6}
		\tgBorderA{(1,2)}{\tgColour4}{\tgColour0}{\tgColour0}{\tgColour4}
		\tgBorderA{(2,2)}{\tgColour0}{\tgColour6}{\tgColour0}{\tgColour0}
		\tgBorderA{(3,2)}{\tgColour6}{\tgColour4}{\tgColour0}{\tgColour0}
		\tgBorderA{(4,2)}{\tgColour4}{\tgColour0}{\tgColour0}{\tgColour0}
		\tgBorderA{(0,3)}{\tgColour6}{\tgColour4}{\tgColour4}{\tgColour6}
		\tgBorderA{(1,3)}{\tgColour4}{\tgColour0}{\tgColour0}{\tgColour4}
		\tgBlank{(2,3)}{\tgColour0}
		\tgBlank{(3,3)}{\tgColour0}
		\tgBlank{(4,3)}{\tgColour0}
		\tgCell[(2,0)]{(1,1)}{\eta}
		\tgCell[(2,0)]{(3,2)}{\unit_{\aop_T}}
		\tgArrow{(0,1.5)}{1}
		\tgArrow{(1,1.5)}{3}
		\tgArrow{(2,1.5)}{3}
		\tgArrow{(3,1.5)}{1}
		\tgArrow{(4,1.5)}{3}
		\tgArrow{(4,0.5)}{3}
		\tgArrow{(1,2.5)}{3}
		\tgArrow{(0,2.5)}{1}
		\tgArrow{(3,0.5)}{1}
		\tgAxisLabel{(3.5,0.75)}{south}{j}
		\tgAxisLabel{(4.5,0.75)}{south}{u_T}
		\tgAxisLabel{(0.5,3.25)}{north}{j}
		\tgAxisLabel{(1.5,3.25)}{north}{u_T}
	\end{tangle}
    \\
	\begin{tangle}{(5,4)}[trim y]
		\tgBlank{(0,0)}{\tgColour6}
		\tgBlank{(1,0)}{\tgColour6}
		\tgBlank{(2,0)}{\tgColour6}
		\tgBorderA{(3,0)}{\tgColour6}{\tgColour4}{\tgColour4}{\tgColour6}
		\tgBorderA{(4,0)}{\tgColour4}{\tgColour0}{\tgColour0}{\tgColour4}
		\tgBorderA{(0,1)}{\tgColour6}{\tgColour6}{\tgColour4}{\tgColour6}
		\tgBorderA{(1,1)}{\tgColour6}{\tgColour6}{\tgColour0}{\tgColour4}
		\tgBorderA{(2,1)}{\tgColour6}{\tgColour6}{\tgColour6}{\tgColour0}
		\tgBorderA{(3,1)}{\tgColour6}{\tgColour4}{\tgColour4}{\tgColour6}
		\tgBorderA{(4,1)}{\tgColour4}{\tgColour0}{\tgColour0}{\tgColour4}
		\tgBorderA{(0,2)}{\tgColour6}{\tgColour4}{\tgColour4}{\tgColour6}
		\tgBorderA{(1,2)}{\tgColour4}{\tgColour0}{\tgColour0}{\tgColour4}
		\tgBorder{(1,2)}{0}{1}{0}{0}
		\tgBorderA{(2,2)}{\tgColour0}{\tgColour6}{\tgColour0}{\tgColour0}
		\tgBorder{(2,2)}{0}{0}{0}{1}
		\tgBorderA{(3,2)}{\tgColour6}{\tgColour4}{\tgColour0}{\tgColour0}
		\tgBorderA{(4,2)}{\tgColour4}{\tgColour0}{\tgColour0}{\tgColour0}
		\tgBorderA{(0,3)}{\tgColour6}{\tgColour4}{\tgColour4}{\tgColour6}
		\tgBorderA{(1,3)}{\tgColour4}{\tgColour0}{\tgColour0}{\tgColour4}
		\tgBlank{(2,3)}{\tgColour0}
		\tgBlank{(3,3)}{\tgColour0}
		\tgBlank{(4,3)}{\tgColour0}
		\tgCell[(2,0)]{(1,1)}{\eta}
		\tgArrow{(0,1.5)}{1}
		\tgArrow{(1,1.5)}{3}
		\tgArrow{(2,1.5)}{3}
		\tgArrow{(3,1.5)}{1}
		\tgArrow{(4,1.5)}{3}
		\tgCell[(3,0)]{(2.5,2)}{\aop_T}
		\tgArrow{(4,0.5)}{3}
		\tgArrow{(1,2.5)}{3}
		\tgArrow{(0,2.5)}{1}
		\tgArrow{(3,0.5)}{1}
		\tgAxisLabel{(3.5,0.75)}{south}{j}
		\tgAxisLabel{(4.5,0.75)}{south}{u_T}
		\tgAxisLabel{(0.5,3.25)}{north}{j}
		\tgAxisLabel{(1.5,3.25)}{north}{u_T}
	\end{tangle}
    \=
	\begin{tangle}{(2,4)}[trim y]
		\tgBorderA{(0,0)}{\tgColour6}{\tgColour4}{\tgColour4}{\tgColour6}
		\tgBorderA{(1,0)}{\tgColour4}{\tgColour0}{\tgColour0}{\tgColour4}
		\tgBorderA{(0,1)}{\tgColour6}{\tgColour4}{\tgColour4}{\tgColour6}
		\tgBorderA{(1,1)}{\tgColour4}{\tgColour0}{\tgColour0}{\tgColour4}
		\tgBorderA{(0,2)}{\tgColour6}{\tgColour4}{\tgColour4}{\tgColour6}
		\tgBorderA{(1,2)}{\tgColour4}{\tgColour0}{\tgColour0}{\tgColour4}
		\tgBorderA{(0,3)}{\tgColour6}{\tgColour4}{\tgColour4}{\tgColour6}
		\tgBorderA{(1,3)}{\tgColour4}{\tgColour0}{\tgColour0}{\tgColour4}
		\tgArrow{(1,1.5)}{3}
		\tgArrow{(0,1.5)}{1}
		\tgArrow{(1,0.5)}{3}
		\tgArrow{(1,2.5)}{3}
		\tgArrow{(0,0.5)}{1}
		\tgArrow{(0,2.5)}{1}
		\tgAxisLabel{(0.5,0.75)}{south}{j}
		\tgAxisLabel{(1.5,0.75)}{south}{u_T}
		\tgAxisLabel{(0.5,3.25)}{north}{j}
		\tgAxisLabel{(1.5,3.25)}{north}{u_T}
	\end{tangle}
	\end{tangleeqs}
	by: \eqstepref{1.1} bending $f_T$; \eqstepref{1.2} using the definition of $\unit_{\aop_T}$; and \eqstepref{1.3} the unit law for the algebra $(u_T, \aop_T)$. Second, we have
    \begin{tangleeqs}
	\begin{tangle}{(5,4)}[trim y]
		\tgBorderA{(0,0)}{\tgColour4}{\tgColour0}{\tgColour0}{\tgColour4}
		\tgBorderA{(1,0)}{\tgColour0}{\tgColour6}{\tgColour6}{\tgColour0}
		\tgBlank{(2,0)}{\tgColour6}
		\tgBlank{(3,0)}{\tgColour6}
		\tgBlank{(4,0)}{\tgColour6}
		\tgBorderA{(0,1)}{\tgColour4}{\tgColour0}{\tgColour0}{\tgColour4}
		\tgBorderA{(1,1)}{\tgColour0}{\tgColour6}{\tgColour6}{\tgColour0}
		\tgBorderA{(2,1)}{\tgColour6}{\tgColour6}{\tgColour4}{\tgColour6}
		\tgBorderA{(3,1)}{\tgColour6}{\tgColour6}{\tgColour0}{\tgColour4}
		\tgBorderA{(4,1)}{\tgColour6}{\tgColour6}{\tgColour6}{\tgColour0}
		\tgBorderA{(0,2)}{\tgColour4}{\tgColour0}{\tgColour0}{\tgColour4}
		\tgBorderA{(1,2)}{\tgColour0}{\tgColour6}{\tgColour0}{\tgColour0}
		\tgBorderA{(2,2)}{\tgColour6}{\tgColour4}{\tgColour0}{\tgColour0}
		\tgBorderA{(3,2)}{\tgColour4}{\tgColour0}{\tgColour0}{\tgColour0}
		\tgBorderA{(4,2)}{\tgColour0}{\tgColour6}{\tgColour6}{\tgColour0}
		\tgBorderA{(0,3)}{\tgColour4}{\tgColour0}{\tgColour0}{\tgColour4}
		\tgBlank{(1,3)}{\tgColour0}
		\tgBlank{(2,3)}{\tgColour0}
		\tgBlank{(3,3)}{\tgColour0}
		\tgBorderA{(4,3)}{\tgColour0}{\tgColour6}{\tgColour6}{\tgColour0}
		\tgCell[(2,0)]{(2,2)}{\unit_{\aop_T}}
		\tgCell[(2,0)]{(3,1)}{\eta}
		\tgArrow{(4,1.5)}{3}
		\tgArrow{(2,1.5)}{1}
		\tgArrow{(3,1.5)}{3}
		\tgArrow{(0,1.5)}{3}
		\tgArrow{(1,1.5)}{3}
		\tgArrow{(4,2.5)}{3}
		\tgArrow{(1,0.5)}{3}
		\tgArrow{(0,0.5)}{3}
		\tgArrow{(0,2.5)}{3}
		\tgAxisLabel{(0.5,0.75)}{south}{u_T}
		\tgAxisLabel{(1.5,0.75)}{south}{f_T}
		\tgAxisLabel{(0.5,3.25)}{north}{u_T}
		\tgAxisLabel{(4.5,3.25)}{north}{f_T}
	\end{tangle}
    \=
	\begin{tangle}{(5,4)}[trim y]
		\tgBorderA{(0,0)}{\tgColour4}{\tgColour0}{\tgColour0}{\tgColour4}
		\tgBorderA{(1,0)}{\tgColour0}{\tgColour6}{\tgColour6}{\tgColour0}
		\tgBlank{(2,0)}{\tgColour6}
		\tgBlank{(3,0)}{\tgColour6}
		\tgBlank{(4,0)}{\tgColour6}
		\tgBorderA{(0,1)}{\tgColour4}{\tgColour0}{\tgColour0}{\tgColour4}
		\tgBorderA{(1,1)}{\tgColour0}{\tgColour6}{\tgColour6}{\tgColour0}
		\tgBorderA{(2,1)}{\tgColour6}{\tgColour6}{\tgColour4}{\tgColour6}
		\tgBorderA{(3,1)}{\tgColour6}{\tgColour6}{\tgColour0}{\tgColour4}
		\tgBorderA{(4,1)}{\tgColour6}{\tgColour6}{\tgColour6}{\tgColour0}
		\tgBorderA{(0,2)}{\tgColour4}{\tgColour0}{\tgColour0}{\tgColour4}
		\tgBorder{(0,2)}{0}{1}{0}{0}
		\tgBorderA{(1,2)}{\tgColour0}{\tgColour6}{\tgColour0}{\tgColour0}
		\tgBorder{(1,2)}{0}{0}{0}{1}
		\tgBorderA{(2,2)}{\tgColour6}{\tgColour4}{\tgColour0}{\tgColour0}
		\tgBorderA{(3,2)}{\tgColour4}{\tgColour0}{\tgColour0}{\tgColour0}
		\tgBorderA{(4,2)}{\tgColour0}{\tgColour6}{\tgColour6}{\tgColour0}
		\tgBorderA{(0,3)}{\tgColour4}{\tgColour0}{\tgColour0}{\tgColour4}
		\tgBlank{(1,3)}{\tgColour0}
		\tgBlank{(2,3)}{\tgColour0}
		\tgBlank{(3,3)}{\tgColour0}
		\tgBorderA{(4,3)}{\tgColour0}{\tgColour6}{\tgColour6}{\tgColour0}
		\tgCell[(2,0)]{(3,1)}{\eta}
		\tgArrow{(4,1.5)}{3}
		\tgArrow{(2,1.5)}{1}
		\tgArrow{(3,1.5)}{3}
		\tgArrow{(0,1.5)}{3}
		\tgArrow{(1,1.5)}{3}
		\tgCell[(3,0)]{(1.5,2)}{\aop_T}
		\tgArrow{(4,2.5)}{3}
		\tgArrow{(0,2.5)}{3}
		\tgArrow{(1,0.5)}{3}
		\tgArrow{(0,0.5)}{3}
		\tgAxisLabel{(0.5,0.75)}{south}{u_T}
		\tgAxisLabel{(1.5,0.75)}{south}{f_T}
		\tgAxisLabel{(0.5,3.25)}{north}{u_T}
		\tgAxisLabel{(4.5,3.25)}{north}{f_T}
	\end{tangle}
    \\
	\begin{tangle}{(4,4)}[trim y]
		\tgBlank{(0,0)}{\tgColour6}
		\tgBlank{(1,0)}{\tgColour6}
		\tgBlank{(2,0)}{\tgColour6}
		\tgBlank{(3,0)}{\tgColour6}
		\tgBlank{(0,1)}{\tgColour6}
		\tgBorderA{(1,1)}{\tgColour6}{\tgColour6}{\tgColour4}{\tgColour6}
		\tgBorderA{(2,1)}{\tgColour6}{\tgColour6}{\tgColour0}{\tgColour4}
		\tgBorderA{(3,1)}{\tgColour6}{\tgColour6}{\tgColour6}{\tgColour0}
		\tgBorderA{(0,2)}{\tgColour6}{\tgColour6}{\tgColour4}{\tgColour6}
		\tgBorderA{(1,2)}{\tgColour6}{\tgColour4}{\tgColour4}{\tgColour4}
		\tgBorder{(1,2)}{0}{1}{0}{0}
		\tgBorderA{(2,2)}{\tgColour4}{\tgColour0}{\tgColour0}{\tgColour4}
		\tgBorder{(2,2)}{0}{0}{0}{1}
		\tgBorderA{(3,2)}{\tgColour0}{\tgColour6}{\tgColour6}{\tgColour0}
		\tgBorderA{(0,3)}{\tgColour6}{\tgColour4}{\tgColour4}{\tgColour6}
		\tgBlank{(1,3)}{\tgColour4}
		\tgBorderA{(2,3)}{\tgColour4}{\tgColour0}{\tgColour0}{\tgColour4}
		\tgBorderA{(3,3)}{\tgColour0}{\tgColour6}{\tgColour6}{\tgColour0}
		\tgCell[(2,0)]{(2,1)}{\eta}
		\tgArrow{(3,1.5)}{3}
		\tgArrow{(1,1.5)}{1}
		\tgArrow{(2,1.5)}{3}
		\tgCell[(2,0)]{(1,2)}{\aop_T}
		\tgArrow{(3,2.5)}{3}
		\tgArrow{(2,2.5)}{3}
		\tgArrow{(0,2.5)}{3}
		\tgAxisLabel{(0.5,3.25)}{north}{t}
		\tgAxisLabel{(2.5,3.25)}{north}{u_T}
		\tgAxisLabel{(3.5,3.25)}{north}{f_T}
	\end{tangle}
    \=
	\begin{tangle}{(2,4)}[trim y]
		\tgBlank{(0,0)}{\tgColour6}
		\tgBlank{(1,0)}{\tgColour6}
		\tgBorderA{(0,1)}{\tgColour6}{\tgColour6}{\tgColour4}{\tgColour6}
		\tgBorderA{(1,1)}{\tgColour6}{\tgColour6}{\tgColour6}{\tgColour4}
		\tgBorderA{(0,2)}{\tgColour6}{\tgColour4}{\tgColour4}{\tgColour6}
		\tgBorder{(0,2)}{0}{1}{0}{0}
		\tgBorderA{(1,2)}{\tgColour4}{\tgColour6}{\tgColour6}{\tgColour4}
		\tgBorder{(1,2)}{0}{0}{0}{1}
		\tgBorderA{(0,3)}{\tgColour6}{\tgColour4}{\tgColour4}{\tgColour6}
		\tgBorderA{(1,3)}{\tgColour4}{\tgColour6}{\tgColour6}{\tgColour4}
		\tgCell[(1,0)]{(0.5,1)}{\eta}
		\tgCell[(1,0)]{(0.5,2)}{\dag}
		\tgArrow{(0,1.5)}{1}
		\tgArrow{(0,2.5)}{1}
		\tgArrow{(1,1.5)}{3}
		\tgArrow{(1,2.5)}{3}
		\tgAxisLabel{(0.5,3.25)}{north}{t}
		\tgAxisLabel{(1.5,3.25)}{north}{t}
	\end{tangle}
    \=
	\begin{tangle}{(2,4)}[trim y]
		\tgBorderA{(0,0)}{\tgColour4}{\tgColour0}{\tgColour0}{\tgColour4}
		\tgBorderA{(1,0)}{\tgColour0}{\tgColour6}{\tgColour6}{\tgColour0}
		\tgBorderA{(0,1)}{\tgColour4}{\tgColour0}{\tgColour0}{\tgColour4}
		\tgBorderA{(1,1)}{\tgColour0}{\tgColour6}{\tgColour6}{\tgColour0}
		\tgBorderA{(0,2)}{\tgColour4}{\tgColour0}{\tgColour0}{\tgColour4}
		\tgBorderA{(1,2)}{\tgColour0}{\tgColour6}{\tgColour6}{\tgColour0}
		\tgBorderA{(0,3)}{\tgColour4}{\tgColour0}{\tgColour0}{\tgColour4}
		\tgBorderA{(1,3)}{\tgColour0}{\tgColour6}{\tgColour6}{\tgColour0}
		\tgArrow{(0,1.5)}{3}
		\tgArrow{(1,1.5)}{3}
		\tgArrow{(1,0.5)}{3}
		\tgArrow{(1,2.5)}{3}
		\tgArrow{(0,0.5)}{3}
		\tgArrow{(0,2.5)}{3}
		\tgAxisLabel{(0.5,0.75)}{south}{u_T}
		\tgAxisLabel{(1.5,0.75)}{south}{f_T}
		\tgAxisLabel{(0.5,3.25)}{north}{u_T}
		\tgAxisLabel{(1.5,3.25)}{north}{f_T}
	\end{tangle}
	\end{tangleeqs}
	using: \eqstepref{2.1} the definition of $\unit_{\aop_T}$; \eqstepref{2.2} bending $(f_T \d u_T) = t$; \eqstepref{2.3} using $\aop_T(1, f_T) = \dag$; and \eqstepref{2.4} the right unit law for $T$. Hence, since $1_t$ is an (ungraded) algebra endomorphism on $(t, \dag)$, by uniqueness the 2-cell above is induced from $1_{\Alg(T)(1, f_T)}$ by postcomposing $1_{E(1, u_T)}$. Hence the zig-zag laws are satisfied, and so $f_T \jadj u_T$. The induced operator is given by
	
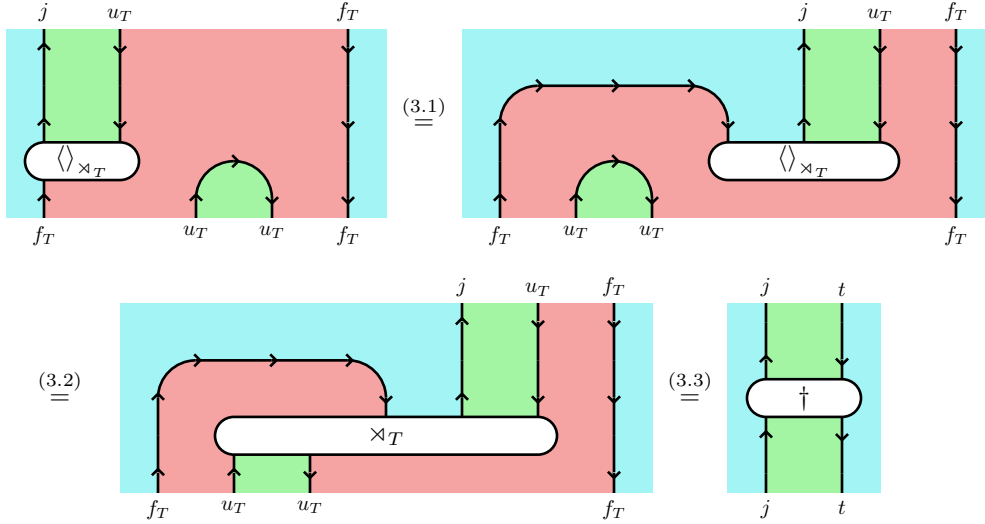
\begin{tangleeqs}
	\begin{tangle}{(5,4)}[trim y]
		\tgBorderA{(0,0)}{\tgColour6}{\tgColour4}{\tgColour4}{\tgColour6}
		\tgBorderA{(1,0)}{\tgColour4}{\tgColour0}{\tgColour0}{\tgColour4}
		\tgBlank{(2,0)}{\tgColour0}
		\tgBlank{(3,0)}{\tgColour0}
		\tgBorderA{(4,0)}{\tgColour0}{\tgColour6}{\tgColour6}{\tgColour0}
		\tgBorderA{(0,1)}{\tgColour6}{\tgColour4}{\tgColour4}{\tgColour6}
		\tgBorderA{(1,1)}{\tgColour4}{\tgColour0}{\tgColour0}{\tgColour4}
		\tgBlank{(2,1)}{\tgColour0}
		\tgBlank{(3,1)}{\tgColour0}
		\tgBorderA{(4,1)}{\tgColour0}{\tgColour6}{\tgColour6}{\tgColour0}
		\tgBorderA{(0,2)}{\tgColour6}{\tgColour4}{\tgColour0}{\tgColour6}
		\tgBorderA{(1,2)}{\tgColour4}{\tgColour0}{\tgColour0}{\tgColour0}
		\tgBorderC{(2,2)}{3}{\tgColour0}{\tgColour4}
		\tgBorderC{(3,2)}{2}{\tgColour0}{\tgColour4}
		\tgBorderA{(4,2)}{\tgColour0}{\tgColour6}{\tgColour6}{\tgColour0}
		\tgBorderA{(0,3)}{\tgColour6}{\tgColour0}{\tgColour0}{\tgColour6}
		\tgBlank{(1,3)}{\tgColour0}
		\tgBorderA{(2,3)}{\tgColour0}{\tgColour4}{\tgColour4}{\tgColour0}
		\tgBorderA{(3,3)}{\tgColour4}{\tgColour0}{\tgColour0}{\tgColour4}
		\tgBorderA{(4,3)}{\tgColour0}{\tgColour6}{\tgColour6}{\tgColour0}
		\tgArrow{(2.5,2)}{0}
		\tgArrow{(4,2.5)}{3}
		\tgArrow{(2,2.5)}{1}
		\tgArrow{(0,2.5)}{1}
		\tgCell[(1,0)]{(0.5,2)}{\unit_{\aop_T}}
		\tgArrow{(0,1.5)}{1}
		\tgArrow{(1,1.5)}{3}
		\tgArrow{(3,2.5)}{3}
		\tgArrow{(4,1.5)}{3}
		\tgArrow{(4,0.5)}{3}
		\tgArrow{(1,0.5)}{3}
		\tgArrow{(0,0.5)}{1}
		\tgAxisLabel{(0.5,0.75)}{south}{j}
		\tgAxisLabel{(1.5,0.75)}{south}{u_T}
		\tgAxisLabel{(4.5,0.75)}{south}{f_T}
		\tgAxisLabel{(0.5,3.25)}{north}{f_T}
		\tgAxisLabel{(2.5,3.25)}{north}{u_T}
		\tgAxisLabel{(3.5,3.25)}{north}{u_T}
		\tgAxisLabel{(4.5,3.25)}{north}{f_T}
	\end{tangle}
	\=
	\begin{tangle}{(7,4)}[trim y]
		\tgBlank{(0,0)}{\tgColour6}
		\tgBlank{(1,0)}{\tgColour6}
		\tgBlank{(2,0)}{\tgColour6}
		\tgBlank{(3,0)}{\tgColour6}
		\tgBorderA{(4,0)}{\tgColour6}{\tgColour4}{\tgColour4}{\tgColour6}
		\tgBorderA{(5,0)}{\tgColour4}{\tgColour0}{\tgColour0}{\tgColour4}
		\tgBorderA{(6,0)}{\tgColour0}{\tgColour6}{\tgColour6}{\tgColour0}
		\tgBorderC{(0,1)}{3}{\tgColour6}{\tgColour0}
		\tgBorderA{(1,1)}{\tgColour6}{\tgColour6}{\tgColour0}{\tgColour0}
		\tgBorderA{(2,1)}{\tgColour6}{\tgColour6}{\tgColour0}{\tgColour0}
		\tgBorderC{(3,1)}{2}{\tgColour6}{\tgColour0}
		\tgBorderA{(4,1)}{\tgColour6}{\tgColour4}{\tgColour4}{\tgColour6}
		\tgBorderA{(5,1)}{\tgColour4}{\tgColour0}{\tgColour0}{\tgColour4}
		\tgBorderA{(6,1)}{\tgColour0}{\tgColour6}{\tgColour6}{\tgColour0}
		\tgBorderA{(0,2)}{\tgColour6}{\tgColour0}{\tgColour0}{\tgColour6}
		\tgBorderC{(1,2)}{3}{\tgColour0}{\tgColour4}
		\tgBorderC{(2,2)}{2}{\tgColour0}{\tgColour4}
		\tgBorderA{(3,2)}{\tgColour0}{\tgColour6}{\tgColour0}{\tgColour0}
		\tgBorderA{(4,2)}{\tgColour6}{\tgColour4}{\tgColour0}{\tgColour0}
		\tgBorderA{(5,2)}{\tgColour4}{\tgColour0}{\tgColour0}{\tgColour0}
		\tgBorderA{(6,2)}{\tgColour0}{\tgColour6}{\tgColour6}{\tgColour0}
		\tgBorderA{(0,3)}{\tgColour6}{\tgColour0}{\tgColour0}{\tgColour6}
		\tgBorderA{(1,3)}{\tgColour0}{\tgColour4}{\tgColour4}{\tgColour0}
		\tgBorderA{(2,3)}{\tgColour4}{\tgColour0}{\tgColour0}{\tgColour4}
		\tgBlank{(3,3)}{\tgColour0}
		\tgBlank{(4,3)}{\tgColour0}
		\tgBlank{(5,3)}{\tgColour0}
		\tgBorderA{(6,3)}{\tgColour0}{\tgColour6}{\tgColour6}{\tgColour0}
		\tgCell[(2,0)]{(4,2)}{\unit_{\aop_T}}
		\tgArrow{(1,2.5)}{1}
		\tgArrow{(2,2.5)}{3}
		\tgArrow{(1.5,2)}{0}
		\tgArrow{(0,2.5)}{1}
		\tgArrow{(0,1.5)}{1}
		\tgArrow{(0.5,1)}{0}
		\tgArrow{(1.5,1)}{0}
		\tgArrow{(2.5,1)}{0}
		\tgArrow{(3,1.5)}{3}
		\tgArrow{(5,0.5)}{3}
		\tgArrow{(5,1.5)}{3}
		\tgArrow{(6,0.5)}{3}
		\tgArrow{(6,1.5)}{3}
		\tgArrow{(6,2.5)}{3}
		\tgArrow{(4,1.5)}{1}
		\tgArrow{(4,0.5)}{1}
		\tgAxisLabel{(4.5,0.75)}{south}{j}
		\tgAxisLabel{(5.5,0.75)}{south}{u_T}
		\tgAxisLabel{(6.5,0.75)}{south}{f_T}
		\tgAxisLabel{(0.5,3.25)}{north}{f_T}
		\tgAxisLabel{(1.5,3.25)}{north}{u_T}
		\tgAxisLabel{(2.5,3.25)}{north}{u_T}
		\tgAxisLabel{(6.5,3.25)}{north}{f_T}
	\end{tangle}
	\\
	\begin{tangle}{(7,4)}[trim y]
		\tgBlank{(0,0)}{\tgColour6}
		\tgBlank{(1,0)}{\tgColour6}
		\tgBlank{(2,0)}{\tgColour6}
		\tgBlank{(3,0)}{\tgColour6}
		\tgBorderA{(4,0)}{\tgColour6}{\tgColour4}{\tgColour4}{\tgColour6}
		\tgBorderA{(5,0)}{\tgColour4}{\tgColour0}{\tgColour0}{\tgColour4}
		\tgBorderA{(6,0)}{\tgColour0}{\tgColour6}{\tgColour6}{\tgColour0}
		\tgBorderC{(0,1)}{3}{\tgColour6}{\tgColour0}
		\tgBorderA{(1,1)}{\tgColour6}{\tgColour6}{\tgColour0}{\tgColour0}
		\tgBorderA{(2,1)}{\tgColour6}{\tgColour6}{\tgColour0}{\tgColour0}
		\tgBorderC{(3,1)}{2}{\tgColour6}{\tgColour0}
		\tgBorderA{(4,1)}{\tgColour6}{\tgColour4}{\tgColour4}{\tgColour6}
		\tgBorderA{(5,1)}{\tgColour4}{\tgColour0}{\tgColour0}{\tgColour4}
		\tgBorderA{(6,1)}{\tgColour0}{\tgColour6}{\tgColour6}{\tgColour0}
		\tgBorderA{(0,2)}{\tgColour6}{\tgColour0}{\tgColour0}{\tgColour6}
		\tgBorderA{(1,2)}{\tgColour0}{\tgColour0}{\tgColour4}{\tgColour0}
		\tgBorderA{(2,2)}{\tgColour0}{\tgColour0}{\tgColour0}{\tgColour4}
		\tgBorder{(2,2)}{0}{1}{0}{0}
		\tgBorderA{(3,2)}{\tgColour0}{\tgColour6}{\tgColour0}{\tgColour0}
		\tgBorder{(3,2)}{0}{0}{0}{1}
		\tgBorderA{(4,2)}{\tgColour6}{\tgColour4}{\tgColour0}{\tgColour0}
		\tgBorderA{(5,2)}{\tgColour4}{\tgColour0}{\tgColour0}{\tgColour0}
		\tgBorderA{(6,2)}{\tgColour0}{\tgColour6}{\tgColour6}{\tgColour0}
		\tgBorderA{(0,3)}{\tgColour6}{\tgColour0}{\tgColour0}{\tgColour6}
		\tgBorderA{(1,3)}{\tgColour0}{\tgColour4}{\tgColour4}{\tgColour0}
		\tgBorderA{(2,3)}{\tgColour4}{\tgColour0}{\tgColour0}{\tgColour4}
		\tgBlank{(3,3)}{\tgColour0}
		\tgBlank{(4,3)}{\tgColour0}
		\tgBlank{(5,3)}{\tgColour0}
		\tgBorderA{(6,3)}{\tgColour0}{\tgColour6}{\tgColour6}{\tgColour0}
		\tgArrow{(1,2.5)}{1}
		\tgArrow{(2,2.5)}{3}
		\tgArrow{(1.5,2)}{0}
		\tgArrow{(0,2.5)}{1}
		\tgArrow{(0,1.5)}{1}
		\tgArrow{(0.5,1)}{0}
		\tgArrow{(1.5,1)}{0}
		\tgArrow{(2.5,1)}{0}
		\tgArrow{(3,1.5)}{3}
		\tgArrow{(5,0.5)}{3}
		\tgArrow{(5,1.5)}{3}
		\tgArrow{(6,0.5)}{3}
		\tgArrow{(6,1.5)}{3}
		\tgArrow{(6,2.5)}{3}
		\tgArrow{(4,1.5)}{1}
		\tgArrow{(4,0.5)}{1}
		\tgCell[(4,0)]{(3,2)}{\aop_T}
		\tgAxisLabel{(4.5,0.75)}{south}{j}
		\tgAxisLabel{(5.5,0.75)}{south}{u_T}
		\tgAxisLabel{(6.5,0.75)}{south}{f_T}
		\tgAxisLabel{(0.5,3.25)}{north}{f_T}
		\tgAxisLabel{(1.5,3.25)}{north}{u_T}
		\tgAxisLabel{(2.5,3.25)}{north}{u_T}
		\tgAxisLabel{(6.5,3.25)}{north}{f_T}
	\end{tangle}
	\=
	\begin{tangle}{(2,3)}[trim y=.25]
		\tgBorderA{(0,0)}{\tgColour6}{\tgColour4}{\tgColour4}{\tgColour6}
		\tgBorderA{(1,0)}{\tgColour4}{\tgColour6}{\tgColour6}{\tgColour4}
		\tgBorderA{(0,1)}{\tgColour6}{\tgColour4}{\tgColour4}{\tgColour6}
		\tgBorder{(0,1)}{0}{1}{0}{0}
		\tgBorderA{(1,1)}{\tgColour4}{\tgColour6}{\tgColour6}{\tgColour4}
		\tgBorder{(1,1)}{0}{0}{0}{1}
		\tgBorderA{(0,2)}{\tgColour6}{\tgColour4}{\tgColour4}{\tgColour6}
		\tgBorderA{(1,2)}{\tgColour4}{\tgColour6}{\tgColour6}{\tgColour4}
		\tgCell[(1,0)]{(0.5,1)}{\dag}
		\tgArrow{(0,1.5)}{1}
		\tgArrow{(0,0.5)}{1}
		\tgArrow{(1,0.5)}{3}
		\tgArrow{(1,1.5)}{3}
		\tgAxisLabel{(0.5,0.25)}{south}{j}
		\tgAxisLabel{(1.5,0.25)}{south}{t}
		\tgAxisLabel{(0.5,2.75)}{north}{j}
		\tgAxisLabel{(1.5,2.75)}{north}{t}
	\end{tangle}
	\end{tangleeqs}
	by: \eqstepref{3.1} bending $f_T$; \eqstepref{3.2} using the definition of $\unit_{\aop_T}$; and \eqstepref{3.3} that $\aop_T(1, f_T) = \dag$. Therefore $f_T \jadj u_T$ is a resolution of $T$.
\end{proof}

\begin{remark}
	The proof of \cref{algebra-object-induces-resolution} (and later that of \cref{opalgebra-object-induces-resolution} for opalgebra objects) makes crucial use of the universal property of $\Alg(T)$ with respect to \emph{graded} algebra morphisms. This may appear surprising, since the analogous result for algebra objects for non-relative monads requires no such grading~\cite[Theorem~3]{street1978yoneda}. The reason lies in the distinction between an adjunction and a relative adjunction: an adjunction may be defined entirely 2-diagrammatically, so that a universal property for algebra objects involving solely tight-cells and 2-cells is sufficient to construct an adjunction; whereas the definition of a relative adjunction requires equipment structure. Therefore, we should expect the universal property for an algebra object to also involve the equipment structure. In particular, \cref{algebra-object} rectifies inadequacies in the definition of the \emph{\EM{} objects} of \cite[Definition~3.5.5]{maillard2019principles} and the \emph{algebra-objects} of \cite[Definition~5.3.4]{arkor2022monadic}, which satisfy a universal property only with respect to ungraded morphisms.
\end{remark}

\begin{remark}
  Let $\jAE$ be a tight-cell and let $T$ be a relative monad admitting an algebra object. For any $T$-algebra $(e \colon D \to E, \aop)$, the extension operator $\aop$ factors through the algebra object in two equivalent ways. First, the universal property of $\Alg(T)$ on algebras induces a tight-cell $\unit_{(e, \aop)} \colon D \to \Alg(T)$ such that $\aop_T(1, \unit_{(e, \aop)}) = \aop$. Second, since $\aop$ may be viewed as an $E(j, e)$-graded algebra morphism by \cref{algebra-extension-operator-is-graded-morphism}, the universal property of $\Alg(T)$ on graded algebra morphisms induces a 2-cell $\unit_\aop \colon E(j, e) \tto E(f_T, \unit_{(e, \aop)})$ such that $\unit_\aop \d u_T = \aop$. Applying the latter factorisation to the algebra object $(u_T, \aop_T)$ itself, we obtain $\flat \d u_T = \aop_T$, which is precisely the construction of the extension operator from the relative adjunction $f_T \jadj u_T$ provided by \cref{relative-adjunction-induces-algebra-and-opalgebra}. Similar considerations apply to opalgebras (\cf{}~\cref{opalgebra-extension-operator-is-graded-morphism}).
\end{remark}

Before stating the main theorem of this section, we recall a useful lemma for establishing the existence of strict reflections.

\begin{lemma}
	\label{adjoint-section}
	Let $\ell \colon A \to C$ be a functor. Then $\ell$ admits a right-adjoint section if and only if the function $\ob \ell$ admits a section $\ob r \colon \ob C \to \ob A$ such that, for all $a \in A$ and $c \in C$, the function
	\[\ell_{a, \ob r c} \colon A(a, \ob r c) \to C(\ob \ell a, \ob \ell \ob r c) = C(\ob \ell a, c)\]
	admits an inverse.
	In this case, $\ob r$ extends uniquely to a functor $r \colon C \to A$, which is a section for $\ell$, such that $\ell \adj r$.
\end{lemma}

\begin{proof}
	Follows directly from \cite[Theorem~IV.1.2]{maclane1998categories}, taking the counit to be the identity.
\end{proof}

\begin{theorem}
	\label{algebra-objects-induce-j-monadic-resolutions}
	Let $\jAE$ be a tight-cell. The functor $\obslash_j \colon \RAdj_L(j) \to \RMnd(j)\op$ admits a partial right-adjoint section, defined on those $j$-monads admitting algebra objects.
	\[\begin{tikzcd}[column sep=huge]
		{\RAdj_L(j)} & {\RMnd(j)\op}
		\arrow[""{name=0, anchor=center, inner sep=0}, "{\obslash_j}", shift left=2, from=1-1, to=1-2]
		\arrow[""{name=1, anchor=center, inner sep=0}, "{{f_{\ph}} \jadj\;{u_{\ph}}}", shift left=2, dashed, hook, from=1-2, to=1-1]
		\arrow["\dashv"{anchor=center, rotate=-90}, draw=none, from=0, to=1]
	\end{tikzcd}\]
   Moreover, a left-morphism is strict if and only if its transpose is the identity $j$-monad morphism.
\end{theorem}

\begin{proof}
	We shall use \cref{adjoint-section}, which permits us to elide details of functoriality and naturality.

	\Cref{algebra-object-induces-resolution} gives a partial assignment $\RMnd(j)\op \to \RAdj_L(j)$ on objects. Let $T$ and $T'$ be $j$-monads, and denote by $\ljr$ a resolution of $T$. Assume that $T'$ admits an algebra object. We aim to define an inverse to the function
	\[(\obslash_j)_{(\ljr), (f_{T'} \jadj u_{T'})} \colon \RAdj_L(j)(\ljr, f_{T'} \jadj u_{T'}) \to \RMnd(j)(T', T)\]
	Recall that $r$ and $t$ form $T$-algebras by \cref{relative-adjunction-induces-algebra-and-opalgebra} and \cref{relative-monad-forms-algebra} respectively, so that a $j$-monad morphism $\tau \colon T' \to T$ induces $T'$-algebra structures on each by functoriality of $\ph\h\Alg$.
	The universal property of $\Alg(T')$ thus induces a unique tight-cell $\unit_{\ljr} \colon C \to \Alg(T')$ such that $r = \unit_{\ljr} \d u_{T'}$ and $\flat \d r \d E(\tau, r) = \aop_{T'}(1, \unit_{\ljr})$. Furthermore, the 2-cell $\tau$ forms an ungraded $T'$-algebra morphism from $(t', \dag')$ to the induced $T'$-algebra structure on $t$, the compatibility law following from the extension operator law for $\tau$, and hence induces a 2-cell $\unit_\tau \colon f_{T'} \tto \ell \d \unit_{\ljr}$ by the universal property of $\Alg(T')$. The pair $(\unit_{\ljr}, \unit_\tau)$ forms a left-morphism, the compatibility law following from the unit law for $\tau$. This assignment defines a function:
	\[\unit_{\ph} \colon \RMnd(j)(T', T) \to \RAdj_L(j)(\ljr, f_{T'} \jadj u_{T'})\]

	To establish that these functions are inverse, let $(c, \lambda)$ be a left-morphism from $\ljr$ to $f_{T'} \jadj u_{T'}$, inducing the $j$-monad morphism $(\lambda \d u_{T'})$. We have that $r = c \d u_{T'}$ and $\flat \d r \d E(\lambda \d u_{T'}, r) = \aop_{T'}(1, c)$ by definition of a left-morphism, so that $c = \unit_{\ljr}$ by uniqueness of the universal property; that $\unit_{\lambda \d u_{T'}} = \lambda$ is trivial. Conversely, let $\tau$ be a $j$-monad morphism from $T'$ to $T$, inducing a left-morphism $(\unit_{\ljr}, \unit_\tau)$. We have that $\unit_\tau \d u_{T'} = \tau$ by definition. Thus $\obslash_j$ admits a partial right-adjoint section.

	Finally, let $(c, \lambda)$ be a left-morphism from $\ljr$ to $f_{T'} \jadj u_{T'}$. If $\lambda$ is the identity, then the induced $j$-monad morphism is trivially also the identity. Conversely, suppose that the induced $j$-monad morphism is the identity. Then we have $(\ell \d c) \d u_{T'} = \ell \d (c \d u_{T'}) = \ell \d r = t'$ and $\aop_{T'}(1, c \ell) = \aop_{T'}(1, c)(1, \ell) = (\flat, \pc r)(1, \ell) = \dag'$, so that $\ell \d c = f_{T'}$ by uniqueness of the mediating tight-cell for $\Alg(T')$. The universal property of $\Alg(T')$ on algebra morphisms thus implies that $\lambda$ is the identity, so that the left-morphism is necessarily strict.
\end{proof}

\begin{corollary}
	\label{relative-monads-form-full-subcategory-of-slices}
	Let $\jAE$ be a tight-cell. The partial functor $u_{\ph} \colon \RMnd(j)\op \to \tX/E$, defined on those $j$-monads $T$ admitting algebra objects, is \ff{}.
\end{corollary}

\begin{proof}
	Direct by composing the \ff{} functors of \cref{algebra-objects-induce-j-monadic-resolutions,left-morphisms-to-slices-is-ff}.
\end{proof}

\begin{corollary}
	\label{algebra-object-is-j-monadic}
	Let $T$ be a relative monad admitting an algebra object. The resolution $f_T \jadj u_T$ is terminal in $\Res(T)$.
\end{corollary}

\begin{proof}
	Suppose $\ljr$ is a resolution of $T$. From \cref{algebra-objects-induce-j-monadic-resolutions}, we have that strict morphisms from $\ljr$ to $f_T \jadj u_T$ necessarily correspond via transposition to the identity morphism on $T$, hence are unique.
\end{proof}

\Cref{algebra-objects-induce-j-monadic-resolutions} justifies our study of left-morphisms of relative adjunctions: in particular, the well-known universal property of algebra objects as terminal resolutions (\cref{algebra-object-is-j-monadic}) is a consequence of a more general universal property that is functorial in the relative monad.

As is to be expected from the non-relative setting, algebra objects for trivial relative monads are trivial.

\begin{proposition}
	\label{algebra-object-for-trivial-relative-monad}
	Let $\jAE$ be a tight-cell. Then $(1_E, 1_{E(j, 1)})$ exhibits an algebra object for the trivial $j$-monad.
\end{proposition}

\begin{proof}
	Let $(e \colon D \to E, \aop)$ be a $j$-algebra. $e$ trivially exhibits a unique mediating tight-cell $\unit_{(e, \aop)} \colon D \to E$. Trivially, every graded algebra morphism factors uniquely through the identity on~$E$.
\end{proof}

\subsection{Opalgebra objects}
\label{opalgebra-objects}

The definition of opalgebra morphism in \cref{opalgebra} is given only between opalgebras with the same codomain. We now give a more general definition of (graded) morphism between any opalgebras for a relative monad, which is necessary to express the universal property of opalgebra objects for relative monads.

\begin{definition}
	\label{graded-opalgebra-morphism}
  Let $(a \colon A \to B, \oop)$ and $(a' \colon A \to B', \oop')$ be $T$-opalgebras. A \emph{$(p_1, \dots, p_n)$-graded $T$-opalgebra morphism} from $(a, \oop)$ to $(a', \oop')$ is a 2-cell
  \[\alpha \colon p_1, \ldots, p_n, B(1, a) \tto B'(1, a')\]
  satisfying the following equation (defining $\rho$ and $\rho'$ as in \cref{opalgebras-are-right-actions}).
  \[
	\begin{tikzcd}
		{B'} & \cdots & B & A & A \\
		{B'} & \cdots & B && A \\
		{B'} &&&& A
		\arrow["{B(1, a)}"{description}, from=2-5, to=2-3]
		\arrow[""{name=0, anchor=center, inner sep=0}, Rightarrow, no head, from=2-5, to=3-5]
		\arrow["{B'(1, a')}", "\shortmid"{marking}, from=3-5, to=3-1]
		\arrow[""{name=1, anchor=center, inner sep=0}, Rightarrow, no head, from=2-1, to=3-1]
		\arrow["{E(j, t)}"', "\shortmid"{marking}, from=1-5, to=1-4]
		\arrow["{B(1, a)}"', "\shortmid"{marking}, from=1-4, to=1-3]
		\arrow[""{name=2, anchor=center, inner sep=0}, Rightarrow, no head, from=1-5, to=2-5]
		\arrow[""{name=3, anchor=center, inner sep=0}, Rightarrow, no head, from=1-3, to=2-3]
		\arrow[""{name=4, anchor=center, inner sep=0}, Rightarrow, no head, from=1-1, to=2-1]
		\arrow["{p_n}"', "\shortmid"{marking}, from=1-3, to=1-2]
		\arrow["{p_n}"{description}, from=2-3, to=2-2]
		\arrow["{p_1}"{description}, from=2-2, to=2-1]
		\arrow["{p_1}"', "\shortmid"{marking}, from=1-2, to=1-1]
		\arrow["\rho"{description}, draw=none, from=2, to=3]
		\arrow["\alpha"{description}, draw=none, from=1, to=0]
		\arrow["{=}"{description}, draw=none, from=3, to=4]
	\end{tikzcd}
  	\quad = \quad
	\begin{tikzcd}
		{B'} & \cdots & B & A && A \\
		{B'} &&& A && A \\
		{B'} &&&&& A
		\arrow["{B(1, a)}"', "\shortmid"{marking}, from=1-4, to=1-3]
		\arrow[""{name=0, anchor=center, inner sep=0}, Rightarrow, no head, from=1-4, to=2-4]
		\arrow["{B'(1, a')}"{description}, from=2-4, to=2-1]
		\arrow[""{name=1, anchor=center, inner sep=0}, Rightarrow, no head, from=1-1, to=2-1]
		\arrow["{E(j, t)}"{description}, from=2-6, to=2-4]
		\arrow["{E(j, t)}"', "\shortmid"{marking}, from=1-6, to=1-4]
		\arrow[""{name=2, anchor=center, inner sep=0}, Rightarrow, no head, from=1-6, to=2-6]
		\arrow["{B'(1, a')}", "\shortmid"{marking}, from=3-6, to=3-1]
		\arrow[""{name=3, anchor=center, inner sep=0}, Rightarrow, no head, from=2-1, to=3-1]
		\arrow[""{name=4, anchor=center, inner sep=0}, Rightarrow, no head, from=2-6, to=3-6]
		\arrow["{p_n}"', "\shortmid"{marking}, from=1-3, to=1-2]
		\arrow["{p_1}"', "\shortmid"{marking}, from=1-2, to=1-1]
		\arrow["\alpha"{description}, draw=none, from=1, to=0]
		\arrow["{\rho'}"{description}, draw=none, from=3, to=4]
		\arrow["{=}"{description}, draw=none, from=2, to=0]
	\end{tikzcd}
	\]\[
	\begin{tangle}{(6,4)}[trim y]
		\tgBorderA{(0,0)}{\tgColour9}{white}{white}{\tgColour9}
		\tgBlank{(1,0)}{white}
		\tgBorderA{(2,0)}{white}{\tgColour10}{\tgColour10}{white}
		\tgBorderA{(3,0)}{\tgColour10}{\tgColour6}{\tgColour6}{\tgColour10}
		\tgBorderA{(4,0)}{\tgColour6}{\tgColour4}{\tgColour4}{\tgColour6}
		\tgBorderA{(5,0)}{\tgColour4}{\tgColour6}{\tgColour6}{\tgColour4}
		\tgBorderA{(0,1)}{\tgColour9}{white}{white}{\tgColour9}
		\tgBlank{(1,1)}{white}
		\tgBorderA{(2,1)}{white}{\tgColour10}{\tgColour10}{white}
		\tgBorderA{(3,1)}{\tgColour10}{\tgColour6}{\tgColour10}{\tgColour10}
		\tgBorderA{(4,1)}{\tgColour6}{\tgColour4}{\tgColour6}{\tgColour10}
		\tgBorderA{(5,1)}{\tgColour4}{\tgColour6}{\tgColour6}{\tgColour6}
		\tgBorderA{(0,2)}{\tgColour9}{white}{\tgColour9}{\tgColour9}
		\tgBorderA{(1,2)}{white}{white}{\tgColour9}{\tgColour9}
		\tgBorderA{(2,2)}{white}{\tgColour10}{\tgColour9}{\tgColour9}
		\tgBorderA{(3,2)}{\tgColour10}{\tgColour10}{\tgColour9}{\tgColour9}
		\tgBorderA{(4,2)}{\tgColour10}{\tgColour6}{\tgColour6}{\tgColour9}
		\tgBlank{(5,2)}{\tgColour6}
		\tgBlank{(0,3)}{\tgColour9}
		\tgBlank{(1,3)}{\tgColour9}
		\tgBlank{(2,3)}{\tgColour9}
		\tgBlank{(3,3)}{\tgColour9}
		\tgBorderA{(4,3)}{\tgColour9}{\tgColour6}{\tgColour6}{\tgColour9}
		\tgBlank{(5,3)}{\tgColour6}
		\tgCell[(2,0)]{(4,1)}{\rho}
		\tgCell[(4,0)]{(2,2)}{\alpha}
		\tgArrow{(4,2.5)}{3}
		\tgArrow{(4,1.5)}{3}
		\tgArrow{(5,0.5)}{3}
		\tgArrow{(3,0.5)}{3}
		\tgArrow{(4,0.5)}{1}
		\tgAxisLabel{(0.5,0.75)}{south}{p_1}
		\tgAxisLabel{(2.5,0.75)}{south}{p_n}
		\tgAxisLabel{(3.5,0.75)}{south}{a}
		\tgAxisLabel{(4.5,0.75)}{south}{j}
		\tgAxisLabel{(5.5,0.75)}{south}{t}
		\tgAxisLabel{(4.5,3.25)}{north}{a'}
		\node at (1.5,1.9) {$\cdots$};
	\end{tangle}
	\tangleeq*
	\begin{tangle}{(6,4)}[trim y]
		\tgBorderA{(0,0)}{\tgColour9}{white}{white}{\tgColour9}
		\tgBlank{(1,0)}{white}
		\tgBorderA{(2,0)}{white}{\tgColour10}{\tgColour10}{white}
		\tgBorderA{(3,0)}{\tgColour10}{\tgColour6}{\tgColour6}{\tgColour10}
		\tgBorderA{(4,0)}{\tgColour6}{\tgColour4}{\tgColour4}{\tgColour6}
		\tgBorderA{(5,0)}{\tgColour4}{\tgColour6}{\tgColour6}{\tgColour4}
		\tgBorderA{(0,1)}{\tgColour9}{white}{\tgColour9}{\tgColour9}
		\tgBorderA{(1,1)}{white}{white}{\tgColour9}{\tgColour9}
		\tgBorderA{(2,1)}{white}{\tgColour10}{\tgColour9}{\tgColour9}
		\tgBorderA{(3,1)}{\tgColour10}{\tgColour6}{\tgColour6}{\tgColour9}
		\tgBorderA{(4,1)}{\tgColour6}{\tgColour4}{\tgColour4}{\tgColour6}
		\tgBorderA{(5,1)}{\tgColour4}{\tgColour6}{\tgColour6}{\tgColour4}
		\tgBlank{(0,2)}{\tgColour9}
		\tgBlank{(1,2)}{\tgColour9}
		\tgBlank{(2,2)}{\tgColour9}
		\tgBorderA{(3,2)}{\tgColour9}{\tgColour6}{\tgColour9}{\tgColour9}
		\tgBorderA{(4,2)}{\tgColour6}{\tgColour4}{\tgColour6}{\tgColour9}
		\tgBorderA{(5,2)}{\tgColour4}{\tgColour6}{\tgColour6}{\tgColour6}
		\tgBlank{(0,3)}{\tgColour9}
		\tgBlank{(1,3)}{\tgColour9}
		\tgBlank{(2,3)}{\tgColour9}
		\tgBlank{(3,3)}{\tgColour9}
		\tgBorderA{(4,3)}{\tgColour9}{\tgColour6}{\tgColour6}{\tgColour9}
		\tgBlank{(5,3)}{\tgColour6}
		\tgCell[(3,0)]{(1.5,1)}{\alpha}
		\tgArrow{(3,0.5)}{3}
		\tgArrow{(3,1.5)}{3}
		\tgCell[(2,0)]{(4,2)}{\rho'}
		\tgArrow{(5,0.5)}{3}
		\tgArrow{(5,1.5)}{3}
		\tgArrow{(4,2.5)}{3}
		\tgArrow{(4,1.5)}{1}
		\tgArrow{(4,0.5)}{1}
		\tgAxisLabel{(0.5,0.75)}{south}{p_1}
		\tgAxisLabel{(2.5,0.75)}{south}{p_n}
		\tgAxisLabel{(3.5,0.75)}{south}{a}
		\tgAxisLabel{(4.5,0.75)}{south}{j}
		\tgAxisLabel{(5.5,0.75)}{south}{t}
		\tgAxisLabel{(4.5,3.25)}{north}{a'}
		\node at (1.5,.9) {$\cdots$};
	\end{tangle}
	\]
	When $n = 0$, we call such a morphism \emph{ungraded}.
\end{definition}

In particular, ungraded opalgebra morphisms are precisely those given in \cref{opalgebra}.

\begin{example}
	\label{opalgebra-extension-operator-is-graded-morphism}
    Analogously to \cref{algebra-extension-operator-is-graded-morphism}, for a given relative monad $T$, the extension operator $\oop \colon E(j, t) \tto B(a, a)$ of a $T$-opalgebra $(a, \oop)$, viewed as a right-action $\rho \colon B(1, a), E(j, t) \tto B(1, a)$ by \cref{opalgebras-are-right-actions}, is a $(B(1, a), E(j, 1))$-graded $T$-opalgebra morphism from $(t, \dag)$ to $(a, \oop)$, the law for the graded morphism being precisely the extension law for the opalgebra.
\end{example}

\begin{definition}
	\label{opalgebra-object}
	Let $T$ be a relative monad. A $T$-opalgebra $(k_T \colon A \to \Opalg(T), \oop_T)$ is called an
	\emph{opalgebra object} for $T$ when
	\begin{enumerate}
		\item for every $T$-opalgebra $(a \colon A \to B, \oop)$, there is a unique tight-cell $[]_{(a, \oop)} \colon \Opalg(T) \to B$ such that $k_T \d []_{(a, \oop)} = a$ and $\oop_T \d []_{(a, \oop)} = \oop$;
		\item for every graded $T$-opalgebra morphism $\alpha \colon p_1, \dots, p_n, B(1, a) \tto B'(1, a')$ there is a unique 2-cell $[]_\alpha \colon p_1, \dots, p_n, B(1, []_{(a, \oop)}) \tto B'(1, []_{(a', \oop')})$ such that:
		\[
		\alpha~=~
		\begin{tikzcd}
			{B'} & \cdots & B & A \\
			{B'} &&& A
			\arrow["{B(1, a)}"', "\shortmid"{marking}, from=1-4, to=1-3]
			\arrow[""{name=0, anchor=center, inner sep=0}, Rightarrow, no head, from=1-4, to=2-4]
			\arrow["{B'(1, a')}", "\shortmid"{marking}, from=2-4, to=2-1]
			\arrow[""{name=1, anchor=center, inner sep=0}, Rightarrow, no head, from=1-1, to=2-1]
			\arrow["{p_n}"', "\shortmid"{marking}, from=1-3, to=1-2]
			\arrow["{p_1}"', "\shortmid"{marking}, from=1-2, to=1-1]
			\arrow["{[]_\alpha(1, k_T)}"{description}, draw=none, from=0, to=1]
		\end{tikzcd}
		\]
  \end{enumerate}
\end{definition}

As with algebras, being an opalgebra object for a (non-relative) monad in the sense of \cref{opalgebra-object} is a stronger condition than being an opalgebra object in a 2-category in the sense of \cite[\S4]{street1972formal}.

In light of \cref{relative-adjunction-induces-algebra-and-opalgebra}, given a relative adjunction $\ljr$ inducing a monad $T$ that admits an opalgebra object, we will denote by $[]_{\ljr} \colon \Opalg(T) \to C$ the mediating morphism induced by the $T$-opalgebra structure on the left $j$-adjoint.

\begin{definition}
	Let $T$ be a relative monad admitting an opalgebra object. Denote by $v_T \colon \Opalg(T) \to E$ the mediating tight-cell $[]_{(t, \dag)}$ induced by the $T$-opalgebra $(t, \dag)$ (\cref{relative-monad-forms-opalgebra}).
\end{definition}

\begin{lemma}
	\label{opalgebra-object-induces-resolution}
	Let $T$ be a relative monad. If $T$ admits an opalgebra object, then $T$ admits a resolution.
	\[\begin{tikzcd}
		& {\Opalg(T)} \\
		A && E
		\arrow[""{name=0, anchor=center, inner sep=0}, "{k_T}", from=2-1, to=1-2]
		\arrow[""{name=1, anchor=center, inner sep=0}, "{v_T}", from=1-2, to=2-3]
		\arrow["j"', from=2-1, to=2-3]
		\arrow["\dashv"{anchor=center}, shift right=2, draw=none, from=0, to=1]
	\end{tikzcd}\]
\end{lemma}

\begin{proof}
	The unit of $T$ provides a 2-cell $\eta \colon j \tto (t = k_T \d v_T)$. By \cref{opalgebra-extension-operator-is-graded-morphism}, the extension operator $\oop_T$ is an $(\Opalg(T)(1, k_T), E(j, 1))$-graded opalgebra morphism from $(t, \dag)$ to $(k_T, \oop_T)$, so that the universal property of the opalgebra object induces a 2-cell $[]_{\oop_T} \colon \Opalg(T)(1, k_T), E(j, v_T) \tto \Opalg(T)(1, 1)$. We prove that $\eta$ together with $[]_{\oop_T}$ forms a $j$-adjunction (in unit--counit form). First, we have
	\begin{tangleeqs}
	\begin{tangle}{(5,3)}[trim y]
		\tgBlank{(0,0)}{\tgColour6}
		\tgBlank{(1,0)}{\tgColour6}
		\tgBorderA{(2,0)}{\tgColour6}{\tgColour4}{\tgColour4}{\tgColour6}
		\tgBorderA{(3,0)}{\tgColour4}{\tgColour10}{\tgColour10}{\tgColour4}
		\tgBorderA{(4,0)}{\tgColour10}{\tgColour6}{\tgColour6}{\tgColour10}
		\tgBorderA{(0,1)}{\tgColour6}{\tgColour6}{\tgColour4}{\tgColour6}
		\tgBorderA{(1,1)}{\tgColour6}{\tgColour6}{\tgColour10}{\tgColour4}
		\tgBorderA{(2,1)}{\tgColour6}{\tgColour4}{\tgColour10}{\tgColour10}
		\tgBorderA{(3,1)}{\tgColour4}{\tgColour10}{\tgColour10}{\tgColour10}
		\tgBorderA{(4,1)}{\tgColour10}{\tgColour6}{\tgColour6}{\tgColour10}
		\tgBorderA{(0,2)}{\tgColour6}{\tgColour4}{\tgColour4}{\tgColour6}
		\tgBorderA{(1,2)}{\tgColour4}{\tgColour10}{\tgColour10}{\tgColour4}
		\tgBlank{(2,2)}{\tgColour10}
		\tgBlank{(3,2)}{\tgColour10}
		\tgBorderA{(4,2)}{\tgColour10}{\tgColour6}{\tgColour6}{\tgColour10}
		\tgCell[(1,0)]{(0.5,1)}{\eta}
		\tgCell[(1,0)]{(2.5,1)}{[]_{\oop_T}}
		\tgArrow{(4,0.5)}{3}
		\tgArrow{(1,1.5)}{3}
		\tgArrow{(3,0.5)}{3}
		\tgArrow{(2,0.5)}{1}
		\tgArrow{(0,1.5)}{1}
		\tgArrow{(4,1.5)}{3}
		\tgArrow{(1.5,1)}{0}
		\tgAxisLabel{(2.5,0.75)}{south}{j}
		\tgAxisLabel{(3.5,0.75)}{south}{v_T}
		\tgAxisLabel{(4.5,0.75)}{south}{k_T}
		\tgAxisLabel{(0.5,2.25)}{north}{j}
		\tgAxisLabel{(1.5,2.25)}{north}{v_T}
		\tgAxisLabel{(4.5,2.25)}{north}{k_T}
	\end{tangle}
    \=
	\begin{tangle}{(5,3)}[trim y]
		\tgBlank{(0,0)}{\tgColour6}
		\tgBlank{(1,0)}{\tgColour6}
		\tgBorderA{(2,0)}{\tgColour6}{\tgColour4}{\tgColour4}{\tgColour6}
		\tgBorderA{(3,0)}{\tgColour4}{\tgColour10}{\tgColour10}{\tgColour4}
		\tgBorderA{(4,0)}{\tgColour10}{\tgColour6}{\tgColour6}{\tgColour10}
		\tgBorderA{(0,1)}{\tgColour6}{\tgColour6}{\tgColour4}{\tgColour6}
		\tgBorderA{(1,1)}{\tgColour6}{\tgColour6}{\tgColour10}{\tgColour4}
		\tgBorderA{(2,1)}{\tgColour6}{\tgColour4}{\tgColour10}{\tgColour10}
		\tgBorderA{(3,1)}{\tgColour4}{\tgColour10}{\tgColour10}{\tgColour10}
		\tgBorder{(3,1)}{0}{1}{0}{0}
		\tgBorderA{(4,1)}{\tgColour10}{\tgColour6}{\tgColour6}{\tgColour10}
		\tgBorder{(4,1)}{0}{0}{0}{1}
		\tgBorderA{(0,2)}{\tgColour6}{\tgColour4}{\tgColour4}{\tgColour6}
		\tgBorderA{(1,2)}{\tgColour4}{\tgColour10}{\tgColour10}{\tgColour4}
		\tgBlank{(2,2)}{\tgColour10}
		\tgBlank{(3,2)}{\tgColour10}
		\tgBorderA{(4,2)}{\tgColour10}{\tgColour6}{\tgColour6}{\tgColour10}
		\tgCell[(1,0)]{(0.5,1)}{\eta}
		\tgArrow{(4,0.5)}{3}
		\tgArrow{(1,1.5)}{3}
		\tgArrow{(3,0.5)}{3}
		\tgArrow{(2,0.5)}{1}
		\tgArrow{(0,1.5)}{1}
		\tgArrow{(4,1.5)}{3}
		\tgArrow{(1.5,1)}{0}
		\tgCell[(2,0)]{(3,1)}{\oop_T}
		\tgAxisLabel{(2.5,0.75)}{south}{j}
		\tgAxisLabel{(3.5,0.75)}{south}{v_T}
		\tgAxisLabel{(4.5,0.75)}{south}{k_T}
		\tgAxisLabel{(0.5,2.25)}{north}{j}
		\tgAxisLabel{(1.5,2.25)}{north}{v_T}
		\tgAxisLabel{(4.5,2.25)}{north}{k_T}
	\end{tangle}
	\\
	\begin{tangle}{(5,5)}[trim y]
		\tgBlank{(0,0)}{\tgColour6}
		\tgBorderA{(1,0)}{\tgColour6}{\tgColour4}{\tgColour4}{\tgColour6}
		\tgBlank{(2,0)}{\tgColour4}
		\tgBlank{(3,0)}{\tgColour4}
		\tgBorderA{(4,0)}{\tgColour4}{\tgColour6}{\tgColour6}{\tgColour4}
		\tgBlank{(0,1)}{\tgColour6}
		\tgBorderA{(1,1)}{\tgColour6}{\tgColour4}{\tgColour10}{\tgColour6}
		\tgBorderA{(2,1)}{\tgColour4}{\tgColour4}{\tgColour10}{\tgColour10}
		\tgBorderA{(3,1)}{\tgColour4}{\tgColour4}{\tgColour10}{\tgColour10}
		\tgBorderA{(4,1)}{\tgColour4}{\tgColour6}{\tgColour6}{\tgColour10}
		\tgBlank{(0,2)}{\tgColour6}
		\tgBorderA{(1,2)}{\tgColour6}{\tgColour10}{\tgColour10}{\tgColour6}
		\tgBorderC{(2,2)}{3}{\tgColour10}{\tgColour4}
		\tgBorderC{(3,2)}{2}{\tgColour10}{\tgColour4}
		\tgBorderA{(4,2)}{\tgColour10}{\tgColour6}{\tgColour6}{\tgColour10}
		\tgBorderA{(0,3)}{\tgColour6}{\tgColour6}{\tgColour4}{\tgColour6}
		\tgBorderA{(1,3)}{\tgColour6}{\tgColour10}{\tgColour4}{\tgColour4}
		\tgBorderA{(2,3)}{\tgColour10}{\tgColour4}{\tgColour4}{\tgColour4}
		\tgBorderA{(3,3)}{\tgColour4}{\tgColour10}{\tgColour10}{\tgColour4}
		\tgBorderA{(4,3)}{\tgColour10}{\tgColour6}{\tgColour6}{\tgColour10}
		\tgBorderA{(0,4)}{\tgColour6}{\tgColour4}{\tgColour4}{\tgColour6}
		\tgBlank{(1,4)}{\tgColour4}
		\tgBlank{(2,4)}{\tgColour4}
		\tgBorderA{(3,4)}{\tgColour4}{\tgColour10}{\tgColour10}{\tgColour4}
		\tgBorderA{(4,4)}{\tgColour10}{\tgColour6}{\tgColour6}{\tgColour10}
		\tgCell[(3,0)]{(2.5,1)}{\oop_T}
		\tgArrow{(2,2.5)}{1}
		\tgArrow{(2.5,2)}{0}
		\tgArrow{(3,2.5)}{3}
		\tgArrow{(3,3.5)}{3}
		\tgArrow{(1,2.5)}{1}
		\tgArrow{(1,1.5)}{1}
		\tgArrow{(4,1.5)}{3}
		\tgArrow{(4,2.5)}{3}
		\tgArrow{(4,3.5)}{3}
		\tgArrow{(4,0.5)}{3}
		\tgArrow{(1,0.5)}{1}
		\tgArrow{(0,3.5)}{1}
		\tgCell[(2,0)]{(1,3)}{\eta}
		\tgAxisLabel{(1.5,0.75)}{south}{j}
		\tgAxisLabel{(4.5,0.75)}{south}{t}
		\tgAxisLabel{(0.5,4.25)}{north}{j}
		\tgAxisLabel{(3.5,4.25)}{north}{v_T}
		\tgAxisLabel{(4.5,4.25)}{north}{k_T}
	\end{tangle}
    \=
	\begin{tangle}{(2,5)}[trim y]
		\tgBorderA{(0,0)}{\tgColour6}{\tgColour4}{\tgColour4}{\tgColour6}
		\tgBorderA{(1,0)}{\tgColour4}{\tgColour6}{\tgColour6}{\tgColour4}
		\tgBorderA{(0,1)}{\tgColour6}{\tgColour4}{\tgColour4}{\tgColour6}
		\tgBorderA{(1,1)}{\tgColour4}{\tgColour6}{\tgColour6}{\tgColour4}
		\tgBorderA{(0,2)}{\tgColour6}{\tgColour4}{\tgColour4}{\tgColour6}
		\tgBorder{(0,2)}{0}{1}{0}{0}
		\tgBorderA{(1,2)}{\tgColour4}{\tgColour6}{\tgColour6}{\tgColour4}
		\tgBorder{(1,2)}{0}{0}{0}{1}
		\tgBorderA{(0,3)}{\tgColour6}{\tgColour4}{\tgColour4}{\tgColour6}
		\tgBorderA{(1,3)}{\tgColour4}{\tgColour6}{\tgColour6}{\tgColour4}
		\tgBorderA{(0,4)}{\tgColour6}{\tgColour4}{\tgColour4}{\tgColour6}
		\tgBorderA{(1,4)}{\tgColour4}{\tgColour6}{\tgColour6}{\tgColour4}
		\tgCell[(1,0)]{(0.5,2)}{\dag}
		\tgCell{(0,3)}{\eta}
		\tgArrow{(0,2.5)}{1}
		\tgArrow{(0,3.5)}{1}
		\tgArrow{(0,1.5)}{1}
		\tgArrow{(1,1.5)}{3}
		\tgArrow{(1,2.5)}{3}
		\tgArrow{(1,3.5)}{3}
		\tgArrow{(1,0.5)}{3}
		\tgArrow{(0,0.5)}{1}
		\tgAxisLabel{(0.5,0.75)}{south}{j}
		\tgAxisLabel{(1.5,0.75)}{south}{t}
		\tgAxisLabel{(0.5,4.25)}{north}{j}
		\tgAxisLabel{(1.5,4.25)}{north}{t}
	\end{tangle}
	\=
	\begin{tangle}{(2,5)}[trim y]
		\tgBorderA{(0,0)}{\tgColour6}{\tgColour4}{\tgColour4}{\tgColour6}
		\tgBorderA{(1,0)}{\tgColour4}{\tgColour6}{\tgColour6}{\tgColour4}
		\tgBorderA{(0,1)}{\tgColour6}{\tgColour4}{\tgColour4}{\tgColour6}
		\tgBorderA{(1,1)}{\tgColour4}{\tgColour6}{\tgColour6}{\tgColour4}
		\tgBorderA{(0,2)}{\tgColour6}{\tgColour4}{\tgColour4}{\tgColour6}
		\tgBorderA{(1,2)}{\tgColour4}{\tgColour6}{\tgColour6}{\tgColour4}
		\tgBorderA{(0,3)}{\tgColour6}{\tgColour4}{\tgColour4}{\tgColour6}
		\tgBorderA{(1,3)}{\tgColour4}{\tgColour6}{\tgColour6}{\tgColour4}
		\tgBorderA{(0,4)}{\tgColour6}{\tgColour4}{\tgColour4}{\tgColour6}
		\tgBorderA{(1,4)}{\tgColour4}{\tgColour6}{\tgColour6}{\tgColour4}
		\tgArrow{(0,2.5)}{1}
		\tgArrow{(1,2.5)}{3}
		\tgArrow{(1,1.5)}{3}
		\tgArrow{(1,3.5)}{3}
		\tgArrow{(0,1.5)}{1}
		\tgArrow{(0,3.5)}{1}
		\tgArrow{(0,0.5)}{1}
		\tgArrow{(1,0.5)}{3}
		\tgAxisLabel{(0.5,0.75)}{south}{j}
		\tgAxisLabel{(1.5,0.75)}{south}{t}
		\tgAxisLabel{(0.5,4.25)}{north}{j}
		\tgAxisLabel{(1.5,4.25)}{north}{t}
	\end{tangle}
	\end{tangleeqs}
	by: \eqstepref{4.1} using the definition of $[]_{\oop_T}$; \eqstepref{4.2} bending $v_T$; \eqstepref{4.3} that $\oop_T \d v_T = \dag$; and \eqstepref{4.4} the left unit law for $T$. Second, we have
	\begin{tangleeqs}
	\begin{tangle}{(2,6)}[trim y]
		\tgBorderA{(0,0)}{\tgColour6}{\tgColour10}{\tgColour10}{\tgColour6}
		\tgBlank{(1,0)}{\tgColour10}
		\tgBorderA{(0,1)}{\tgColour6}{\tgColour10}{\tgColour10}{\tgColour6}
		\tgBlank{(1,1)}{\tgColour10}
		\tgBorderA{(0,2)}{\tgColour6}{\tgColour10}{\tgColour4}{\tgColour6}
		\tgBorderA{(1,2)}{\tgColour10}{\tgColour10}{\tgColour10}{\tgColour4}
		\tgBorderA{(0,3)}{\tgColour6}{\tgColour4}{\tgColour10}{\tgColour6}
		\tgBorderA{(1,3)}{\tgColour4}{\tgColour10}{\tgColour10}{\tgColour10}
		\tgBorderA{(0,4)}{\tgColour6}{\tgColour10}{\tgColour10}{\tgColour6}
		\tgBlank{(1,4)}{\tgColour10}
		\tgBorderA{(0,5)}{\tgColour6}{\tgColour10}{\tgColour10}{\tgColour6}
		\tgBlank{(1,5)}{\tgColour10}
		\tgCell[(1,0)]{(0.5,2)}{\eta}
		\tgCell[(1,0)]{(0.5,3)}{[]_{\oop_T}}
		\tgArrow{(1,2.5)}{3}
		\tgArrow{(0,2.5)}{1}
		\tgArrow{(0,3.5)}{1}
		\tgArrow{(0,1.5)}{1}
		\tgArrow{(0,0.5)}{1}
		\tgArrow{(0,4.5)}{1}
		\tgAxisLabel{(0.5,0.75)}{south}{k_T}
		\tgAxisLabel{(0.5,5.25)}{north}{k_T}
	\end{tangle}
	\=
	\begin{tangle}{(4,6)}[trim y]
		\tgBlank{(0,0)}{\tgColour6}
		\tgBlank{(1,0)}{\tgColour6}
		\tgBlank{(2,0)}{\tgColour6}
		\tgBorderA{(3,0)}{\tgColour6}{\tgColour10}{\tgColour10}{\tgColour6}
		\tgBorderC{(0,1)}{3}{\tgColour6}{\tgColour10}
		\tgBorderA{(1,1)}{\tgColour6}{\tgColour6}{\tgColour10}{\tgColour10}
		\tgBorderC{(2,1)}{2}{\tgColour6}{\tgColour10}
		\tgBorderA{(3,1)}{\tgColour6}{\tgColour10}{\tgColour10}{\tgColour6}
		\tgBorderA{(0,2)}{\tgColour6}{\tgColour10}{\tgColour4}{\tgColour6}
		\tgBorderA{(1,2)}{\tgColour10}{\tgColour10}{\tgColour10}{\tgColour4}
		\tgBorderA{(2,2)}{\tgColour10}{\tgColour6}{\tgColour6}{\tgColour10}
		\tgBorderA{(3,2)}{\tgColour6}{\tgColour10}{\tgColour10}{\tgColour6}
		\tgBorderA{(0,3)}{\tgColour6}{\tgColour4}{\tgColour10}{\tgColour6}
		\tgBorderA{(1,3)}{\tgColour4}{\tgColour10}{\tgColour10}{\tgColour10}
		\tgBorderA{(2,3)}{\tgColour10}{\tgColour6}{\tgColour6}{\tgColour10}
		\tgBorderA{(3,3)}{\tgColour6}{\tgColour10}{\tgColour10}{\tgColour6}
		\tgBorderA{(0,4)}{\tgColour6}{\tgColour10}{\tgColour10}{\tgColour6}
		\tgBlank{(1,4)}{\tgColour10}
		\tgBorderC{(2,4)}{0}{\tgColour10}{\tgColour6}
		\tgBorderC{(3,4)}{1}{\tgColour10}{\tgColour6}
		\tgBorderA{(0,5)}{\tgColour6}{\tgColour10}{\tgColour10}{\tgColour6}
		\tgBlank{(1,5)}{\tgColour10}
		\tgBlank{(2,5)}{\tgColour10}
		\tgBlank{(3,5)}{\tgColour10}
		\tgCell[(1,0)]{(0.5,2)}{\eta}
		\tgCell[(1,0)]{(0.5,3)}{[]_{\oop_T}}
		\tgArrow{(1,2.5)}{3}
		\tgArrow{(0,2.5)}{1}
		\tgArrow{(0,3.5)}{1}
		\tgArrow{(0,1.5)}{1}
		\tgArrow{(0.5,1)}{0}
		\tgArrow{(1.5,1)}{0}
		\tgArrow{(2,1.5)}{3}
		\tgArrow{(2,2.5)}{3}
		\tgArrow{(2,3.5)}{3}
		\tgArrow{(2.5,4)}{0}
		\tgArrow{(3,3.5)}{1}
		\tgArrow{(3,2.5)}{1}
		\tgArrow{(3,1.5)}{1}
		\tgArrow{(3,0.5)}{1}
		\tgArrow{(0,4.5)}{1}
		\tgAxisLabel{(3.5,0.75)}{south}{k_T}
		\tgAxisLabel{(0.5,5.25)}{north}{k_T}
	\end{tangle}
    \=
	\begin{tangle}{(3,6)}[trim y]
		\tgBlank{(0,0)}{\tgColour6}
		\tgBlank{(1,0)}{\tgColour6}
		\tgBorderA{(2,0)}{\tgColour6}{\tgColour10}{\tgColour10}{\tgColour6}
		\tgBlank{(0,1)}{\tgColour6}
		\tgBlank{(1,1)}{\tgColour6}
		\tgBorderA{(2,1)}{\tgColour6}{\tgColour10}{\tgColour10}{\tgColour6}
		\tgBorderA{(0,2)}{\tgColour6}{\tgColour6}{\tgColour4}{\tgColour6}
		\tgBorderA{(1,2)}{\tgColour6}{\tgColour6}{\tgColour6}{\tgColour4}
		\tgBorderA{(2,2)}{\tgColour6}{\tgColour10}{\tgColour10}{\tgColour6}
		\tgBorderA{(0,3)}{\tgColour6}{\tgColour4}{\tgColour10}{\tgColour6}
		\tgBorderA{(1,3)}{\tgColour4}{\tgColour6}{\tgColour6}{\tgColour10}
		\tgBorderA{(2,3)}{\tgColour6}{\tgColour10}{\tgColour10}{\tgColour6}
		\tgBorderA{(0,4)}{\tgColour6}{\tgColour10}{\tgColour10}{\tgColour6}
		\tgBorderC{(1,4)}{0}{\tgColour10}{\tgColour6}
		\tgBorderC{(2,4)}{1}{\tgColour10}{\tgColour6}
		\tgBorderA{(0,5)}{\tgColour6}{\tgColour10}{\tgColour10}{\tgColour6}
		\tgBlank{(1,5)}{\tgColour10}
		\tgBlank{(2,5)}{\tgColour10}
		\tgCell[(1,0)]{(0.5,2)}{\eta}
		\tgCell[(1,0)]{(0.5,3)}{\oop_T}
		\tgArrow{(1,2.5)}{3}
		\tgArrow{(0,2.5)}{1}
		\tgArrow{(0,3.5)}{1}
		\tgArrow{(2,2.5)}{1}
		\tgArrow{(2,1.5)}{1}
		\tgArrow{(2,0.5)}{1}
		\tgArrow{(0,4.5)}{1}
		\tgArrow{(1,3.5)}{3}
		\tgArrow{(1.5,4)}{0}
		\tgArrow{(2,3.5)}{1}
		\tgAxisLabel{(2.5,0.75)}{south}{k_T}
		\tgAxisLabel{(0.5,5.25)}{north}{k_T}
	\end{tangle}
    \=
	\begin{tangle}{(1,6)}[trim y]
		\tgBorderA{(0,0)}{\tgColour6}{\tgColour10}{\tgColour10}{\tgColour6}
		\tgBorderA{(0,1)}{\tgColour6}{\tgColour10}{\tgColour10}{\tgColour6}
		\tgBorderA{(0,2)}{\tgColour6}{\tgColour10}{\tgColour10}{\tgColour6}
		\tgBorderA{(0,3)}{\tgColour6}{\tgColour10}{\tgColour10}{\tgColour6}
		\tgBorderA{(0,4)}{\tgColour6}{\tgColour10}{\tgColour10}{\tgColour6}
		\tgBorderA{(0,5)}{\tgColour6}{\tgColour10}{\tgColour10}{\tgColour6}
		\tgArrow{(0,2.5)}{1}
		\tgArrow{(0,1.5)}{1}
		\tgArrow{(0,0.5)}{1}
		\tgArrow{(0,3.5)}{1}
		\tgArrow{(0,4.5)}{1}
		\tgAxisLabel{(0.5,0.75)}{south}{k_T}
		\tgAxisLabel{(0.5,5.25)}{north}{k_T}
	\end{tangle}
	\end{tangleeqs}
	by: \eqstepref{5.1} bending $k_T$; \eqstepref{5.2} the definition of $[]_{\oop_T}$; and \eqstepref{5.3} the unit law for the opalgebra $(k_T, \oop_T)$.
	Hence the zig-zag laws are satisfied, and so $k_T \jadj v_T$. The induced operator is given by
	\begin{tangleeqs}
	\begin{tangle}{(5,4)}[trim y]
		\tgBorderA{(0,0)}{\tgColour6}{\tgColour4}{\tgColour4}{\tgColour6}
		\tgBorderA{(1,0)}{\tgColour4}{\tgColour10}{\tgColour10}{\tgColour4}
		\tgBlank{(2,0)}{\tgColour10}
		\tgBlank{(3,0)}{\tgColour10}
		\tgBorderA{(4,0)}{\tgColour10}{\tgColour6}{\tgColour6}{\tgColour10}
		\tgBorderA{(0,1)}{\tgColour6}{\tgColour4}{\tgColour4}{\tgColour6}
		\tgBorderA{(1,1)}{\tgColour4}{\tgColour10}{\tgColour10}{\tgColour4}
		\tgBlank{(2,1)}{\tgColour10}
		\tgBlank{(3,1)}{\tgColour10}
		\tgBorderA{(4,1)}{\tgColour10}{\tgColour6}{\tgColour6}{\tgColour10}
		\tgBorderA{(0,2)}{\tgColour6}{\tgColour4}{\tgColour10}{\tgColour6}
		\tgBorderA{(1,2)}{\tgColour4}{\tgColour10}{\tgColour10}{\tgColour10}
		\tgBorderC{(2,2)}{3}{\tgColour10}{\tgColour4}
		\tgBorderC{(3,2)}{2}{\tgColour10}{\tgColour4}
		\tgBorderA{(4,2)}{\tgColour10}{\tgColour6}{\tgColour6}{\tgColour10}
		\tgBorderA{(0,3)}{\tgColour6}{\tgColour10}{\tgColour10}{\tgColour6}
		\tgBlank{(1,3)}{\tgColour10}
		\tgBorderA{(2,3)}{\tgColour10}{\tgColour4}{\tgColour4}{\tgColour10}
		\tgBorderA{(3,3)}{\tgColour4}{\tgColour10}{\tgColour10}{\tgColour4}
		\tgBorderA{(4,3)}{\tgColour10}{\tgColour6}{\tgColour6}{\tgColour10}
		\tgCell[(1,0)]{(0.5,2)}{[]_{\oop_T}}
		\tgArrow{(1,1.5)}{3}
		\tgArrow{(3,2.5)}{3}
		\tgArrow{(4,2.5)}{3}
		\tgArrow{(4,1.5)}{3}
		\tgArrow{(2.5,2)}{0}
		\tgArrow{(2,2.5)}{1}
		\tgArrow{(0,1.5)}{1}
		\tgArrow{(0,2.5)}{1}
		\tgArrow{(1,0.5)}{3}
		\tgArrow{(4,0.5)}{3}
		\tgArrow{(0,0.5)}{1}
		\tgAxisLabel{(0.5,0.75)}{south}{j}
		\tgAxisLabel{(1.5,0.75)}{south}{v_T}
		\tgAxisLabel{(4.5,0.75)}{south}{k_T}
		\tgAxisLabel{(0.5,3.25)}{north}{k_T}
		\tgAxisLabel{(2.5,3.25)}{north}{v_T}
		\tgAxisLabel{(3.5,3.25)}{north}{v_T}
		\tgAxisLabel{(4.5,3.25)}{north}{k_T}
	\end{tangle}
	\=
	\begin{tangle}{(4,4)}[trim y]
		\tgBorderA{(0,0)}{\tgColour6}{\tgColour4}{\tgColour4}{\tgColour6}
		\tgBlank{(1,0)}{\tgColour4}
		\tgBlank{(2,0)}{\tgColour4}
		\tgBorderA{(3,0)}{\tgColour4}{\tgColour6}{\tgColour6}{\tgColour4}
		\tgBorderA{(0,1)}{\tgColour6}{\tgColour4}{\tgColour10}{\tgColour6}
		\tgBorderA{(1,1)}{\tgColour4}{\tgColour4}{\tgColour10}{\tgColour10}
		\tgBorderA{(2,1)}{\tgColour4}{\tgColour4}{\tgColour10}{\tgColour10}
		\tgBorderA{(3,1)}{\tgColour4}{\tgColour6}{\tgColour6}{\tgColour10}
		\tgBorderA{(0,2)}{\tgColour6}{\tgColour10}{\tgColour10}{\tgColour6}
		\tgBorderC{(1,2)}{3}{\tgColour10}{\tgColour4}
		\tgBorderC{(2,2)}{2}{\tgColour10}{\tgColour4}
		\tgBorderA{(3,2)}{\tgColour10}{\tgColour6}{\tgColour6}{\tgColour10}
		\tgBorderA{(0,3)}{\tgColour6}{\tgColour10}{\tgColour10}{\tgColour6}
		\tgBorderA{(1,3)}{\tgColour10}{\tgColour4}{\tgColour4}{\tgColour10}
		\tgBorderA{(2,3)}{\tgColour4}{\tgColour10}{\tgColour10}{\tgColour4}
		\tgBorderA{(3,3)}{\tgColour10}{\tgColour6}{\tgColour6}{\tgColour10}
		\tgCell[(3,0)]{(1.5,1)}{\oop_T}
		\tgArrow{(1.5,2)}{0}
		\tgArrow{(2,2.5)}{3}
		\tgArrow{(1,2.5)}{1}
		\tgArrow{(0,1.5)}{1}
		\tgArrow{(0,2.5)}{1}
		\tgArrow{(0,0.5)}{1}
		\tgArrow{(3,0.5)}{3}
		\tgArrow{(3,1.5)}{3}
		\tgArrow{(3,2.5)}{3}
		\tgAxisLabel{(0.5,0.75)}{south}{j}
		\tgAxisLabel{(3.5,0.75)}{south}{t}
		\tgAxisLabel{(0.5,3.25)}{north}{k_T}
		\tgAxisLabel{(1.5,3.25)}{north}{v_T}
		\tgAxisLabel{(2.5,3.25)}{north}{v_T}
		\tgAxisLabel{(3.5,3.25)}{north}{k_T}
	\end{tangle}
	\=
	\begin{tangle}{(2,3)}[trim y=.25]
		\tgBorderA{(0,0)}{\tgColour6}{\tgColour4}{\tgColour4}{\tgColour6}
		\tgBorderA{(1,0)}{\tgColour4}{\tgColour6}{\tgColour6}{\tgColour4}
		\tgBorderA{(0,1)}{\tgColour6}{\tgColour4}{\tgColour4}{\tgColour6}
		\tgBorder{(0,1)}{0}{1}{0}{0}
		\tgBorderA{(1,1)}{\tgColour4}{\tgColour6}{\tgColour6}{\tgColour4}
		\tgBorder{(1,1)}{0}{0}{0}{1}
		\tgBorderA{(0,2)}{\tgColour6}{\tgColour4}{\tgColour4}{\tgColour6}
		\tgBorderA{(1,2)}{\tgColour4}{\tgColour6}{\tgColour6}{\tgColour4}
		\tgCell[(1,0)]{(0.5,1)}{\dag}
		\tgArrow{(0,1.5)}{1}
		\tgArrow{(0,0.5)}{1}
		\tgArrow{(1,0.5)}{3}
		\tgArrow{(1,1.5)}{3}
		\tgAxisLabel{(0.5,0.25)}{south}{j}
		\tgAxisLabel{(1.5,0.25)}{south}{t}
		\tgAxisLabel{(0.5,2.75)}{north}{j}
		\tgAxisLabel{(1.5,2.75)}{north}{t}
	\end{tangle}
	\end{tangleeqs}
	using: \eqstepref{6.1} the definition of $[]_{\oop_T}$; and \eqstepref{6.2} that $\oop_T \d v_T = \dag$. Therefore $k_T \jadj v_T$ is a resolution of~$T$.
\end{proof}

\begin{corollary}
	\label{opalgebra-object-action-is-invertible}
	Let $T$ be a relative monad admitting an opalgebra object $(k_T, \oop_T)$. Then $\oop_T$ is necessarily invertible.
\end{corollary}

\begin{proof}
	From \cref{opalgebra-object-induces-resolution}, we have that $k_j \jadj v_T$, and hence that $\Opalg(T)(k_T, k_T) \iso E(j, v_T k_T) = E(j, t)$. By construction of the relative adjunction, this invertible 2-cell is precisely $\oop_T$.
\end{proof}

\begin{theorem}
	\label{opalgebra-objects-induce-j-opmonadic-resolutions}
	$\oslash_j \colon \RAdj_R(j) \to \RMnd(j)$ admits a partial
        left-adjoint section, defined on those $j$-monads admitting
        opalgebra objects.
	\[\begin{tikzcd}[column sep=huge]
		{\RAdj_R(j)} & {\RMnd(j)}
		\arrow[""{name=0, anchor=center, inner sep=0}, "{\oslash_j}"', shift right=2, from=1-1, to=1-2]
		\arrow[""{name=1, anchor=center, inner sep=0}, "{{k_{\ph}} \jadj\;{v_{\ph}}}"', shift right=2, dashed, hook', from=1-2, to=1-1]
		\arrow["\dashv"{anchor=center, rotate=-90}, draw=none, from=1, to=0]
	\end{tikzcd}\]
    Moreover, a right-morphism is strict if and only if its transpose is the identity $j$-monad morphism.
\end{theorem}

\begin{proof}
	We shall use (the dual of) \cref{adjoint-section}, which permits us to elide details of functoriality and naturality.

	\Cref{algebra-object-induces-resolution} gives a partial assignment $\RMnd(j) \to \RAdj_R(j)$ on objects. Let $T$ and $T'$ be $j$-monads, and denote by $\ljrp$ a resolution of $T'$. Assume that $T$ admits an opalgebra object. We aim to define an inverse to the function
	\[(\oslash_j)_{k_T \jadj v_T, \ljrp} \colon \RAdj_R(j)(k_T \jadj v_T, \ljrp) \to \RMnd(j)(T, T')\]

	Recall that $\ell'$ and $t$ form $T$-opalgebras by \cref{relative-adjunction-induces-algebra-and-opalgebra} and \cref{relative-monad-forms-opalgebra} respectively, so that a $j$-monad morphism $\tau \colon T \to T'$ induces $T$-opalgebra structures on each by functoriality of $\ph\h\Opalg$.
	The universal property of $\Opalg(T)$ thus induces a unique tight-cell $[]_{\ljrp} \colon \Opalg(T) \to C$ such that $\ell = k_T \d []_{\ljrp}$ and $E(j, \tau) \d \flat'(1, \ell') = \oop_T \d []_{\ljrp}$. Furthermore, the 2-cell $\tau$ forms an ungraded $T$-opalgebra morphism from $(t, \dag)$ to the induced $T$-opalgebra structure on $t'$, the compatibility law following from the extension operator law for $\tau$, and hence induces a 2-cell $[]_\tau \colon v_T \tto []_{\ljrp} \d r'$ by the universal property of $\Opalg(T)$. The pair $([]_{\ljrp}, []_\tau)$ forms a right-morphism, the compatibility law following from the unit law for $\tau$. This assignment defines a function:
	\[[]_{\ph} \colon \RMnd(j)(T, T') \to \RAdj_R(j)(k_T \jadj v_T, \ljrp)\]

	To establish that these functions are inverse, let $(c, \rho)$ be a right-morphism from $k_T \jadj v_T$ to $\ljrp$, inducing the $j$-monad morphism $(k_T \d \rho)$. We have that $\ell' = k_T \d c$ and $E(j, k_T \d \rho) \d \flat'(1, \ell') = \oop_T \d c$ by definition of a right-morphism, so that $c = []_{\ljrp}$ by uniqueness of the universal property; that $[]_{k_T \d \rho} = \rho$ is trivial. Conversely, let $\tau$ be a $j$-monad morphism from $T$ to $T'$, inducing a right-morphism $([]_{\ljrp}, []_\tau)$. We have that $k_T \d []_{\ljrp} = \tau$ by definition. Thus $\oslash_j$ admits a partial left-adjoint section.

	Finally, let $(c, \rho)$ be a right-morphism from $k_T \jadj v_T$ to $\ljrp$. If $\rho$ is the identity, then the induced $j$-monad morphism is trivially also the identity. Conversely, suppose that the induced $j$-monad morphism is the identity. Then we have $k_T \d (c \d r') = (k_T \d c) \d r' = \ell' \d r' = t$ and $\oop_T \d (c \d r') = (\oop_T \d c) \d r' = \flat'(1, \ell') \d r' = \dag$, so that $c \d r' = v_T$ by uniqueness of the mediating tight-cell for $\Opalg(T)$. The universal property of $\Opalg(T)$ on opalgebra morphisms thus implies that $\rho$ is the identity, so that the right-morphism is necessarily strict.
\end{proof}

\begin{corollary}
	\label{relative-monads-form-full-subcategory-of-coslices}
	Let $\jAE$ be a dense tight-cell. The partial functor $k_{\ph} \colon \RMnd(j) \to A/\tX$, defined on those $j$-monads $T$ admitting opalgebra objects, is \ff{}.
\end{corollary}

\begin{proof}
	Direct by composing the \ff{} functors of \cref{opalgebra-objects-induce-j-opmonadic-resolutions,right-morphisms-to-coslices-is-ff}.
\end{proof}

\begin{corollary}
	\label{opalgebra-object-is-j-opmonadic}
	Let $T$ be a relative monad admitting an opalgebra object. The resolution $k_T \jadj v_T$ is initial in $\Res(T)$.
\end{corollary}

\begin{proof}
	Suppose $\ljr$ is a resolution of $T$. From \cref{opalgebra-objects-induce-j-opmonadic-resolutions}, we have that strict morphisms from $k_T \jadj v_T$ to $\ljr$ necessarily correspond via transposition to the identity morphism on $T$, hence are unique.
\end{proof}

\begin{remark}
	From \cref{algebra-object-is-j-monadic,opalgebra-object-is-j-opmonadic}, we recover \cites[Theorem~1.2.6]{walters1970categorical}[Theorem~3]{altenkirch2010monads}[Theorem~2.12]{altenkirch2015monads} for relative monads in $\Cat$, in particular subsuming \cites[Theorem~2.2]{eilenberg1965adjoint}[Reviewer's remark]{huber1966review}[Propositions~1 \& 2]{coppey1970factorisations} when $j = 1$. From \cref{algebra-objects-induce-j-monadic-resolutions,opalgebra-objects-induce-j-opmonadic-resolutions}, we recover \cite[Theorem~2 \& Theorem~1]{maranda1966fundamental} when $j = 1$.

	From \cref{relative-monads-form-full-subcategory-of-slices}, we recover \cite[Lemma~4.5.2]{maillard2019principles}; \cites[Theorem~3]{frei1969some}[Theorem~2.3]{wiesler1970remarks} for non-relative monads in $\Cat$ and $\VCat$, and \cite[Theorem~1.5.4]{walters1970categorical} for relative monads in $\Cat$ with \ff{} roots (\cf{}~\cref{device}).
\end{remark}

\begin{remark}
    Our definition of opalgebra object rectifies an inadequacy in the definition of the \emph{relative Kleisli objects} of \cite[Definition~6.4]{lobbia2023distributive}, which do not appear to form initial resolutions, in contrast to the \emph{relative EM objects} \ibid{}, which do form terminal resolutions (\cf{}~\cite[Remark~6.7]{lobbia2023distributive}).
\end{remark}

\begin{remark}
    \label{comparison-tight-cell}
	As a consequence of \cref{opalgebra-object-is-j-opmonadic} together with \cref{algebra-object-is-j-monadic}, for any relative monad admitting both an opalgebra object and an algebra object, there is a unique \emph{comparison} tight-cell $i_T \colon \Opalg(T) \to \Alg(T)$ (given equivalently by $[]_{f_T \jadj u_T}$ and by $\unit_{k_T \jadj v_T}$) rendering the following triangles commutative.
	\[
	\begin{tikzcd}
		{\Opalg(T)} & {\Alg(T)} \\
		A
		\arrow["{i_T}", from=1-1, to=1-2]
		\arrow["{k_T}", from=2-1, to=1-1]
		\arrow["{f_T}"', from=2-1, to=1-2]
	\end{tikzcd}
	\hspace{4em}
	\begin{tikzcd}
		{\Opalg(T)} & {\Alg(T)} \\
		& E
		\arrow["{i_T}", from=1-1, to=1-2]
		\arrow["{v_T}"', from=1-1, to=2-2]
		\arrow["{u_T}", from=1-2, to=2-2]
	\end{tikzcd}
	\]
	Note that, in a general equipment, in contrast to the situation in $\Cat$, there is no reason to expect that $i_T$ should be \ff{}.
\end{remark}

As is to be expected from the non-relative setting, and from the analogous result for algebra objects (\cref{algebra-object-for-trivial-relative-monad}), the existence of opalgebra objects for trivial relative monads is trivial, at least for \ff{} tight-cells.

\begin{proposition}
	\label{opalgebra-object-for-trivial-relative-monad-with-ff-root}
	Let $\jAE$ be a tight-cell. There exists a $j$-opalgebra $(1_A, \oop_j)$ exhibiting an opalgebra object for the trivial $j$-monad if and only if $j$ is \ff{}.
\end{proposition}

\begin{proof}
	Suppose that $j$ is \ff{}. Then, by \cref{opalgebras-are-loose-monad-morphisms}, $(1_A, (\pc j)\inv)$ forms a $j$-opalgebra since the inverse of a loose-monad morphism is a loose-monad morphism. We shall prove that it furthermore forms an opalgebra object. For each $j$-opalgebra $(a, \oop)$, define $[]_{(a, \oop)} \defeq a$. Trivially, $1_A \d a = a$, and, by the unit law for the opalgebra, $\oop = (\pc j)\inv \d \pc a$. For each $(p_1, \ldots, p_n)$-graded $j$-opalgebra morphism $\alpha$ from $(a, \oop)$ to $(a', \oop')$, define $[]_\alpha \defeq \alpha$, which trivially satisfies the required property. Thus $(1_A, (\pc j)\inv)$ forms an opalgebra object.

	Conversely, suppose that there exists an opalgebra object $(1_A, \oop_j)$ for $j$. Then, by \cref{opalgebra-object-action-is-invertible}, we have that $\oop_j \colon E(j, j) \tto A(1, 1)$ is invertible. Thus the canonical 2-cell $\pc j \colon A(1, 1) \tto E(j, j)$ is invertible by \cref{ff-iff-invertible}.
\end{proof}

For tight-cells $\jAE$ that are not \ff{}, the situation is more interesting. Suppose that an opalgebra object $(k_j, \oop_j)$ for the trivial $j$-monad exists. By the above, $k_j$ may not be invertible, hence the opalgebra object is nontrivial. However, the universal property implies that $k_j$ is \emph{surjective} in a certain sense.
\begin{enumerate}
	\item The functor $k_j \d \ph \colon \X[\Opalg(j), B] \to j\h\Opalg_B$ is an isomorphism of categories: thus, following the discussion preceding \cref{opalgebras-are-loose-monad-morphisms}, $k_j$ is \emph{bijective-on-objects} in the terminology of \textcite[369]{street1978yoneda}.
	\item By \cref{opalgebra-object-action-is-invertible}, $\Opalg(j)(k_j, k_j) \iso E(j, v_j k_j)$, and so $k_j$ is a \emph{surjection} in the terminology of \textcite[141]{wood1985proarrows}.
	\item The 1-categorical universal property of the opalgebra object for $j$ is the condition that $k_j$ \emph{factors hom-actions} in the terminology of \textcite[Definition~6.1.2]{arkor2022monadic}.
\end{enumerate}
In $\Cat$, for instance, an opalgebra object for a functor $\jAE$ is precisely the \emph{full image factorisation} of $j$ into a bijective-on-objects functor $k_j$ followed by a \ff{} functor $v_j$. The formal relationship between opalgebra objects for trivial relative monads and factorisation systems will be expounded elsewhere.

The nontriviality of opalgebra objects for trivial relative monads is also illustrated by the following observation, which in particular significantly generalises \cite[Example~2.16]{altenkirch2015monads}.

\begin{proposition}
	\label{coincidence-of-opalgebra-objects}
	Let $T$ be a $j$-monad, and suppose that the loose-monad $E(j, T)$ is induced by a tight-cell $\ell \colon A \to C$ (\cref{co-representable-resolution}). Composition with the isomorphism $E(j, T) \iso C(\ell, \ell)$ induces a bijection between $T$-opalgebras and $\ell$-opalgebras, and their graded morphisms. Thus $T$ admits an opalgebra object if and only if the trivial $\ell$-monad admits an opalgebra object, in which case there is an isomorphism rendering the following diagram commutative. In particular, this holds when $T$ admits a resolution $\ljr$.
	\[\begin{tikzcd}
		{\Opalg(\ell)} && {\Opalg(T)} \\
		& A
		\arrow["{k_\ell}", from=2-2, to=1-1]
		\arrow["{k_T}"', from=2-2, to=1-3]
		\arrow["\iso", from=1-1, to=1-3]
	\end{tikzcd}\]
\end{proposition}

\begin{proof}
	By \cref{opalgebras-are-loose-monad-morphisms}, $T$-opalgebras are precisely tight-cells $a \colon A \to B$ equipped with loose-monad morphisms $E(j, T) \tto B(a, a)$, while $\ell$-opalgebras are tight-cells $a \colon A \to B$ equipped with loose-monad morphisms $C(\ell, \ell) \tto B(a, a)$. Hence the isomorphism $E(j, T) \iso C(\ell, \ell)$ of loose-monads induces a bijection between them and their graded morphisms. Consequently, the two opalgebra objects satisfy the same universal property, exhibiting them as isomorphic. In particular, when $T$ admits a resolution $\ljr$, \cref{loose-monad-associated-to-relative-monad-is-kernel-of-left-adjoint} implies that $\ell$ is such a tight-cell.
\end{proof}

In practice, this means that demonstrating the existence of opalgebra objects in general may often be reduced to demonstrating the existence of opalgebra objects for trivial relative monads.

\begin{corollary}
	If every loose-monad in $\X$ is induced by a tight-cell, every relative monad admits an opalgebra object if and only if every trivial relative monad admits an opalgebra object.
\end{corollary}

\begin{proof}
	Assume that every loose-monad in $\X$ is induced by a tight-cell. Then, in particular, for any $j$-monad $T$, the assumptions of \cref{coincidence-of-opalgebra-objects} are satisfied, so that if trivial relative monads admit opalgebra objects, then $T$ admits an opalgebra object. The converse is trivial.
\end{proof}

The assumption that every loose-monad is induced by a tight-cell is verified, for instance, in $\Cat$~\cites[6.22]{justesen1968bikategorien}[Proposition~39]{wood1985proarrows}, and more generally in any equipment that is \emph{exact} in the sense of \textcite[Definition~5.1]{schultz2015regular}.

\subsection{(Op)algebra objects and composition of relative adjunctions}
\label{algebra-and-opalgebra-objects-and-composition}

Suppose we have the following situation, as in \cref{relative-monad-relative-adjunction-composition}.
\[\begin{tikzcd}
	& B & D \\
	A &&& E
	\arrow["j", from=1-2, to=1-3]
	\arrow["{\ell'}", from=2-1, to=1-2]
	\arrow["{j'}"', from=2-1, to=2-4]
	\arrow[""{name=0, anchor=center, inner sep=0}, "{r'}", from=1-3, to=2-4]
	\arrow[""{name=1, anchor=center, inner sep=0}, "{\ell' \d j}"{description}, from=2-1, to=1-3]
	\arrow["\dashv"{anchor=center}, shift right=2, draw=none, from=1, to=0]
\end{tikzcd}\]
Let $T$ be a $j$-monad. Suppose that $T$ and $(\ell' \d T \d r')$ admit opalgebra objects. Then $(k_T, \oop_T)$ induces an $(\ell' \d T \d r')$-opalgebra structure on $(\ell' \d k_T)$ by \cref{algebra-and-opalgebra-precomposition,algebra-and-opalgebra-pasting}, and consequently the universal property of $\Opalg(\ell' \d T \d r')$ induces a tight-cell $[]_T \colon \Opalg(\ell' \d T \d r') \to \Opalg(T)$ under $A$. Similarly, suppose that $T$ and $(\ell' \d T \d r')$ admit algebra objects. Then $(u_T, \aop_T)$ induces an $(\ell' \d T \d r')$-algebra structure on $(u_T \d r')$ by \cref{algebra-and-opalgebra-precomposition,algebra-and-opalgebra-pasting}, and consequently the universal property of $\Alg(\ell' \d T \d r')$ induces a tight-cell $\unit_T \colon \Alg(T) \to \Alg(\ell' \d T \d r')$ over $E$. When both opalgebra and algebra objects exist, we have a commutative diagram as follows.
\[\begin{tikzcd}
	{\Opalg(\ell' \d T \d r')} & {\Opalg(T)} & {\Alg(T)} & {\Alg(\ell' \d T \d r')} \\
	A & B & D & E
	\arrow["{i_T}"{description}, from=1-2, to=1-3]
	\arrow["{k_T}"{description}, from=2-2, to=1-2]
	\arrow["{u_T}"{description}, from=1-3, to=2-3]
	\arrow["t"', from=2-2, to=2-3]
	\arrow["{\ell'}"', from=2-1, to=2-2]
	\arrow["{r'}"', from=2-3, to=2-4]
	\arrow["{k_{\ell' \d T \d r'}}", from=2-1, to=1-1]
	\arrow["{u_{\ell' \d T \d r'}}", from=1-4, to=2-4]
	\arrow["{[]_T}"{description}, from=1-1, to=1-2]
	\arrow["{\unit_T}"{description}, from=1-3, to=1-4]
	\arrow["{i_{\ell' \d T \d r'}}", curve={height=-18pt}, from=1-1, to=1-4]
\end{tikzcd}\]
Furthermore, in this situation, the opalgebra object and algebra object for $(\ell' \d T \d r')$ satisfy a universal property with respect to the opalgebra object and algebra object for $T$, as follows.

\begin{proposition}
	\label{universal-property-of-lp-t-rp}
	Let $(\ell' \d j) \radj{j'} r'$ be a relative adjunction, and let $T$ be a $j$-monad, as in \cref{relative-monad-relative-adjunction-composition}.
	\begin{enumerate}
		\item Suppose $j'$ is dense. If $T$ and $(\ell' \d T \d r')$ admit opalgebra objects, then, for every $j'$-monad $T'$ admitting an opalgebra object and for every tight-cell $\Opalg(T') \to \Opalg(T)$ under $A$, there is a unique tight-cell $\Opalg(T') \to \Opalg(\ell' \d T \d r')$ rendering the following diagram commutative.
		\[\begin{tikzcd}
			{\Opalg(T')} & {\Opalg(\ell' \d T \d r')} & {\Opalg(T)} \\
			& A
			\arrow[dashed, from=1-1, to=1-2]
			\arrow["{[]_T}"{description}, from=1-2, to=1-3]
			\arrow[curve={height=-18pt}, from=1-1, to=1-3]
			\arrow["{k_{T'}}", from=2-2, to=1-1]
			\arrow["{\ell' \d k_T}"', from=2-2, to=1-3]
			\arrow["{k_{\ell' \d T \d r'}}"{description}, from=2-2, to=1-2]
		\end{tikzcd}\]
		\item If $T$ and $(\ell' \d T \d r')$ admit algebra objects, then, for every $j'$-monad $T'$ admitting an algebra object and for every tight-cell $\Alg(T) \to \Alg(T')$ over $E$, there is a unique tight-cell $\Alg(\ell' \d T \d r') \to \Alg(T')$ rendering the following diagram commutative.
		\[\begin{tikzcd}
			{\Alg(T)} & {\Alg(\ell' \d T \d r')} & {\Alg(T')} \\
			& E
			\arrow["{\unit_T}"{description}, from=1-1, to=1-2]
			\arrow[dashed, from=1-2, to=1-3]
			\arrow[curve={height=-18pt}, from=1-1, to=1-3]
			\arrow["{u_T \d r'}"', from=1-1, to=2-2]
			\arrow["{u_{T'}}", from=1-3, to=2-2]
			\arrow["{u_{\ell' \d T \d r'}}"{description}, from=1-2, to=2-2]
		\end{tikzcd}\]
	\end{enumerate}
\end{proposition}

\begin{proof}
	The proofs for (1) and (2) proceed similarly.
	\begin{enumerate}
		\item
		\begin{align*}
			A/\X(k_{T'}, (\ell' \d k_T)) & \iso \RAdj_R(j')((k_{T'} \radj{j'} v_{T'}), (\ell' \d k_T \radj{j'} v_T \d r')) \tag{\cref{right-morphisms-to-coslices-is-ff}} \\
				& \iso \RMnd(j')(T', (\ell' \d T \d r')) \tag{\cref{opalgebra-objects-induce-j-opmonadic-resolutions}} \\
				& \iso \RAdj_R(j')((k_{T'} \radj{j'} v_{T'}), (k_{\ell' \d T \d r'} \radj{j'} v_{\ell' \d T \d r'})) \tag{\cref{opalgebra-objects-induce-j-opmonadic-resolutions}} \\
				& \iso A/\X(k_{T'}, k_{\ell' \d T \d r'}) \tag{\cref{right-morphisms-to-coslices-is-ff}}
		\end{align*}
		\item
		\begin{align*}
			\X/E((u_T \d r'), u_{T'}) & \iso \RAdj_L(j')((\ell' \d f_T \radj{j'} u_T \d r'), (f_{T'} \radj{j'} u_{T'})) \tag{\cref{left-morphisms-to-slices-is-ff}} \\
				& \iso \RMnd(j')(T', (\ell' \d T \d r')) \tag{\cref{algebra-objects-induce-j-monadic-resolutions}} \\
				& \iso \RAdj_L(j')((f_{\ell' \d T \d r'} \radj{j'} u_{\ell' \d T \d r'}), (f_{T'} \radj{j'} u_{T'})) \tag{\cref{algebra-objects-induce-j-monadic-resolutions}} \\
				& \iso \X/E(u_{\ell' \d T \d r'}, u_{T'}) \tag{\cref{left-morphisms-to-slices-is-ff}}
		\end{align*}
	\end{enumerate}
\end{proof}

Conceptually, \cref{universal-property-of-lp-t-rp} may be viewed as expressing that $(\ell' \d T \d r')$ is the universal $j'$-monad associated to the $j$-monad $T$. However, we shall not make this intuition precise here, as doing so requires a notion of morphism between relative monads with different roots.

\section{Relative comonads and relative coadjunctions}
\label{duality}

In the formal theory of monads in a 2-category, duality is used to great effect~\cite[\S4]{street1972formal}. Every 2-category $\K$ has three duals: a dual on 1-cells $\K\op$; a dual on 2-cells $\K\co$; and a dual on 1-cells and 2-cells $\K\coop$. We may consider monads and their (op)algebras in each of the four 2-categories $\K$, $\K\op$, $\K\co$, and $\K\coop$. A monad in $\K\op$ is simply a monad in $\K$; while a monad in either of $\K\co$ or $\K\coop$ is precisely a comonad in $\K$. Furthermore, for a monad $T$ in $\K\op$ (equivalently in $\K$), a $T$-algebra in $\K\op$ is precisely a $T$-opalgebra in $\K$. Therefore the theory of opalgebras and opalgebra objects follows formally from the theory of algebras and algebra objects (and conversely).
Orthogonally, for a monad $D$ in $\K\co$ (equivalently a comonad in $\K$), a $D$-algebra in $\K\co$ is a $D$-coalgebra in $\K$. This relationship is summarised in the table below.

\begin{center}
	\begin{tblr}{c|c|c}
		An algebra for a monad in \underline{\hspace{0.5cm}} & is a/an \underline{\hspace{0.5cm}} & for a \underline{\hspace{0.5cm}} in $\K$. \\
		\hline
		$\K$ & algebra & monad \\
		$\K\op$ & opalgebra & monad \\
		$\K\co$ & coalgebra & comonad \\
		$\K\coop$ & co\"opalgebra\footnotemark{} & comonad
	\end{tblr}
    \footnotetext{We prefer the use of a diaeresis to denote a vowel separation, but \emph{co-opalgebra} and \emph{coopalgebra} are also possible.}
\end{center}

For the formal theory of relative monads in a \ve{}, this relationship breaks down. A \vdc{} $\X$ only admits one notion of dual: namely, the dual on loose-cells $\X\co$, which corresponds both to the $\co$ of the tight 2-category, and to the $\op$ of the underlying (virtual) bicategory of loose-cells. Consequently, only one axis of the usual duality theory of monads in a 2-category generalises to relative monads in a \ve{}: while relative monads in $\X\co$ are precisely relative comonads in $\X$, the concept of (co)algebra is formally distinct to that of (co)opalgebra.

While this bifurcation in the relative setting may appear surprising, the reason is clear from the perspective of skew-multiactegories. For a relative monad qua monoid in the skew-multicategory $\X[j]$, we have the notions of action both in a left-$\X[j]$-multiactegory (\cref{left-action}) and in a right-$\X[j]$--multiactegory (\cref{right-action}). When $j = 1$, these notions are formally dual, so that actions in a left-multiactegory may be defined in terms of actions in a right-multiactegory, and conversely. However, in general the two notions are not dual: for instance, the definition of action in a skew-left-multiactegory involves the left-unitor $\lambda$, while the definition of action in a skew-right-multiactegory involves the right-unitor $\rho$.

We shall briefly review the theory of relative comonads and relative coadjunctions. However, since the theory is entirely dual to the theory of relative monads and relative adjunctions, we shall give only definitions, and leave the reader to dualise the theorems as desired. We omit the string diagram presentations of the laws, which are obtained simply by horizontally reflecting those for relative monads and relative adjunctions.

\begin{remark}
	While the study of relative comonads and relative coadjunctions may be reduced to the study of relative monads and relative adjunctions via duality, it appears likely that it is worthwhile to study the interaction between relative adjunctions and relative coadjunctions, and between relative monads and relative comonads, which cannot be thus reduced (\cf{}~\cite[\S2.2 \& \S2.4]{lewicki2020categories}).
\end{remark}

\begin{definition}
    Let $\X$ be a \ve{}. A \emph{relative comonad} in $\X$ is a relative monad in $\X\co$. Explicitly, this comprises
	\begin{enumerate}
        \item a tight-cell $i \colon Z \to V$, the \emph{coroot};
        \item a tight-cell $d \colon Z \to V$, the \emph{underlying tight-cell};
        \item a 2-cell $\flip\dag \colon V(d, i) \tto V(d, d)$, the \emph{coextension operator};
        \item a 2-cell $\varepsilon \colon d \tto i$, the \emph{counit},
    \end{enumerate}
	satisfying the following three equations.
	\[
	\begin{tikzcd}[column sep=large]
		Z & Z \\
		Z & Z \\
		Z & Z
		\arrow["{V(d, i)}"', "\shortmid"{marking}, from=1-2, to=1-1]
		\arrow["{V(d, d)}"{description}, from=2-2, to=2-1]
		\arrow[""{name=0, anchor=center, inner sep=0}, Rightarrow, no head, from=1-2, to=2-2]
		\arrow[""{name=1, anchor=center, inner sep=0}, Rightarrow, no head, from=1-1, to=2-1]
		\arrow[""{name=2, anchor=center, inner sep=0}, Rightarrow, no head, from=2-2, to=3-2]
		\arrow[""{name=3, anchor=center, inner sep=0}, Rightarrow, no head, from=2-1, to=3-1]
		\arrow["{V(d, i)}", "\shortmid"{marking}, from=3-2, to=3-1]
		\arrow["\flip\dag"{description}, draw=none, from=0, to=1]
		\arrow["{V(d, \varepsilon)}"{description}, draw=none, from=2, to=3]
	\end{tikzcd}
    ~=~
	\begin{tikzcd}
		Z & Z \\
		Z & Z
		\arrow["{V(d, i)}"', "\shortmid"{marking}, from=1-2, to=1-1]
		\arrow["{V(d, i)}", "\shortmid"{marking}, from=2-2, to=2-1]
		\arrow[""{name=0, anchor=center, inner sep=0}, Rightarrow, no head, from=1-2, to=2-2]
		\arrow[""{name=1, anchor=center, inner sep=0}, Rightarrow, no head, from=1-1, to=2-1]
		\arrow["{=}"{description}, draw=none, from=0, to=1]
	\end{tikzcd}
    \hspace{8em}
	\begin{tikzcd}[column sep=large]
		Z & Z \\
		Z & Z \\
		Z & Z \\
		Z & Z
		\arrow["{V(d, i)}"{description}, from=3-2, to=3-1]
		\arrow["{V(d, d)}", "\shortmid"{marking}, from=4-2, to=4-1]
		\arrow[""{name=0, anchor=center, inner sep=0}, Rightarrow, no head, from=3-2, to=4-2]
		\arrow[""{name=1, anchor=center, inner sep=0}, Rightarrow, no head, from=3-1, to=4-1]
		\arrow["{V(i, i)}"{description}, from=2-2, to=2-1]
		\arrow[""{name=2, anchor=center, inner sep=0}, Rightarrow, no head, from=2-2, to=3-2]
		\arrow[""{name=3, anchor=center, inner sep=0}, Rightarrow, no head, from=2-1, to=3-1]
		\arrow[""{name=4, anchor=center, inner sep=0}, Rightarrow, no head, from=1-2, to=2-2]
		\arrow[Rightarrow, no head, from=1-2, to=1-1]
		\arrow[""{name=5, anchor=center, inner sep=0}, Rightarrow, no head, from=1-1, to=2-1]
		\arrow["{\pc i}"{description}, draw=none, from=4, to=5]
		\arrow["{V(\varepsilon, i)}"{description}, draw=none, from=2, to=3]
		\arrow["\flip\dag"{description}, draw=none, from=0, to=1]
	\end{tikzcd}
    ~=~
	\begin{tikzcd}
		Z & Z \\
		Z & Z
		\arrow["{V(d, d)}", "\shortmid"{marking}, from=2-2, to=2-1]
		\arrow[""{name=0, anchor=center, inner sep=0}, Rightarrow, no head, from=1-2, to=2-2]
		\arrow[Rightarrow, no head, from=1-2, to=1-1]
		\arrow[""{name=1, anchor=center, inner sep=0}, Rightarrow, no head, from=1-1, to=2-1]
		\arrow["{\pc d}"{description}, draw=none, from=0, to=1]
	\end{tikzcd}
    \]
    \[
	\begin{tikzcd}[column sep=large]
		Z & Z & Z \\
		Z & Z & Z \\
		Z && Z
		\arrow["{V(d, i)}"', "\shortmid"{marking}, from=1-3, to=1-2]
		\arrow["{V(d, i)}"', "\shortmid"{marking}, from=1-2, to=1-1]
		\arrow["{V(d, d)}"{description}, from=2-3, to=2-2]
		\arrow["{V(d, d)}"{description}, from=2-2, to=2-1]
		\arrow[""{name=0, anchor=center, inner sep=0}, Rightarrow, no head, from=1-3, to=2-3]
		\arrow[""{name=1, anchor=center, inner sep=0}, Rightarrow, no head, from=1-2, to=2-2]
		\arrow[""{name=2, anchor=center, inner sep=0}, Rightarrow, no head, from=1-1, to=2-1]
		\arrow["{V(d, d)}", "\shortmid"{marking}, from=3-3, to=3-1]
		\arrow[""{name=3, anchor=center, inner sep=0}, Rightarrow, no head, from=2-3, to=3-3]
		\arrow[""{name=4, anchor=center, inner sep=0}, Rightarrow, no head, from=2-1, to=3-1]
		\arrow["{\cp d(d, d)}"{description}, draw=none, from=3, to=4]
		\arrow["\flip\dag"{description}, draw=none, from=0, to=1]
		\arrow["\flip\dag"{description}, draw=none, from=1, to=2]
	\end{tikzcd}
    \quad = \quad
	\begin{tikzcd}[column sep=large]
		Z & Z & Z \\
		Z & Z & Z \\
		Z && Z \\
		Z && Z
		\arrow["{V(d, i)}"', "\shortmid"{marking}, from=1-2, to=1-1]
		\arrow["{V(d, i)}"{description}, from=2-3, to=2-2]
		\arrow["{V(d, d)}"{description}, from=2-2, to=2-1]
		\arrow[""{name=0, anchor=center, inner sep=0}, Rightarrow, no head, from=1-2, to=2-2]
		\arrow[""{name=1, anchor=center, inner sep=0}, Rightarrow, no head, from=1-1, to=2-1]
		\arrow["{V(d, i)}"{description}, from=3-3, to=3-1]
		\arrow[""{name=2, anchor=center, inner sep=0}, Rightarrow, no head, from=2-3, to=3-3]
		\arrow[""{name=3, anchor=center, inner sep=0}, Rightarrow, no head, from=2-1, to=3-1]
		\arrow["{V(d, d)}", "\shortmid"{marking}, from=4-3, to=4-1]
		\arrow[""{name=4, anchor=center, inner sep=0}, Rightarrow, no head, from=3-3, to=4-3]
		\arrow[""{name=5, anchor=center, inner sep=0}, Rightarrow, no head, from=3-1, to=4-1]
		\arrow["{V(d, i)}"', "\shortmid"{marking}, from=1-3, to=1-2]
		\arrow[""{name=6, anchor=center, inner sep=0}, Rightarrow, no head, from=1-3, to=2-3]
		\arrow["{\cp d(d, i)}"{description}, draw=none, from=2, to=3]
		\arrow["\flip\dag"{description}, draw=none, from=4, to=5]
		\arrow["\flip\dag"{description}, draw=none, from=0, to=1]
		\arrow["{=}"{description}, draw=none, from=6, to=0]
	\end{tikzcd}
    \]
    An \emph{$i$-relative comonad} (alternatively \emph{comonad on $i$}, \emph{comonad relative to $i$}, or simply \emph{$i$-comonad}) is a relative comonad with coroot $i$. A \emph{morphism} of $i$-comonads from $(d, \flip\dag, \varepsilon)$ to $(d', \flip\dag', \varepsilon')$ is a morphism of the corresponding relative monads in $\X\co$. Explicitly, this is a 2-cell $\delta \colon d' \tto d$ rendering the following diagrams commutative.
	\[
	\begin{tikzcd}
		d && {d'} \\
		& i
		\arrow["\varepsilon"', from=1-1, to=2-2]
		\arrow["{\varepsilon'}", from=1-3, to=2-2]
		\arrow["\delta"', from=1-3, to=1-1]
	\end{tikzcd}
	\hspace{4em}
        \begin{tikzcd}[sep=small]
		{V(d, i)} && {V(d', i)} \\
		{V(d, d)} && {V(d', d')} \\
		& {V(d', d)}
		\arrow["\flip\dag"', from=1-1, to=2-1]
		\arrow["{V(\delta, d)}"', from=2-1, to=3-2]
		\arrow["{V(\delta, i)}", from=1-1, to=1-3]
		\arrow["{\flip\dag'}", from=1-3, to=2-3]
		\arrow["{V(d', \delta)}", from=2-3, to=3-2]
	\end{tikzcd}
	\]
    $i$-comonads and their morphisms form a category $\RCmnd(i)$.
\end{definition}

One may view relative comonads as monoids in a \emph{right-}skew-multicategory of tight-cells $Z \to V$ in $\X$, dually to the theory of \cref{skew-multicategorical-hom-categories}. In this case, the multimorphisms are given by 2-cells
\[\begin{tikzcd}
	Z & V & Z & V & \cdots & Z & V & Z & V \\
	Z &&&&&&&& V
	\arrow["{p_n}"', "\shortmid"{marking}, from=1-9, to=1-8]
	\arrow["{V(1, i)}"', "\shortmid"{marking}, from=1-8, to=1-7]
	\arrow["{p_{n - 1}}"', "\shortmid"{marking}, from=1-7, to=1-6]
	\arrow["{V(1, i)}"', "\shortmid"{marking}, from=1-6, to=1-5]
	\arrow["{V(1, i)}"', "\shortmid"{marking}, from=1-5, to=1-4]
	\arrow["q", "\shortmid"{marking}, from=2-9, to=2-1]
	\arrow[""{name=0, anchor=center, inner sep=0}, Rightarrow, no head, from=1-9, to=2-9]
	\arrow[""{name=1, anchor=center, inner sep=0}, Rightarrow, no head, from=1-1, to=2-1]
	\arrow["{p_1}"', "\shortmid"{marking}, from=1-2, to=1-1]
	\arrow["{p_2}"', "\shortmid"{marking}, from=1-4, to=1-3]
	\arrow["{V(1, i)}"', "\shortmid"{marking}, from=1-3, to=1-2]
	\arrow["\psi"{description}, draw=none, from=0, to=1]
\end{tikzcd}\]
after which we restrict to the corepresentable loose-cells. Note that relative comonads, like relative monads, are \emph{monoids} in skew-multicategories, rather than comonoids. Consequently, an $i$-comonad $D$ induces a loose-monad $V(D, i)$. If $\X$ admits right extensions of tight-cells $Z \to V$ along the tight-cell $i \colon Z \to V$, then the category $\X[Z, V]$ is equipped with right-skew-monoidal structure in which a comonoid is precisely an $i$-comonad, by the dual of \cref{Xj1-is-skew-monoidal}.

\begin{definition}
    A \emph{relative coadjunction}\footnotemark{} is a relative adjunction in $\X\co$.
	\footnotetext{In older works on category theory, the terms \emph{adjoint} and \emph{coadjoint} are occasionally encountered, typically meaning \emph{left adjoint} and \emph{right adjoint} respectively (\eg{}~\cite{lawvere1963functorial, thiebaud1971relative}). The terms \emph{adjunction} and \emph{coadjunction} are also occasionally found to refer to the unit and counit of an adjunction (\eg{}~\cite{lambek1973localization}). Since this terminology has fallen out of usage, and is particularly convenient in the context of relative adjoints, we feel there is no danger of confusion in repurposing the terminology.}
	Explicitly, this comprises
	\[\begin{tikzcd}
		Z && V \\
		& X
		\arrow[""{name=0, anchor=center, inner sep=0}, "r"'{pos=0.4}, from=1-1, to=2-2]
		\arrow[""{name=1, anchor=center, inner sep=0}, "\ell"'{pos=0.6}, from=2-2, to=1-3]
		\arrow["i", from=1-1, to=1-3]
		\arrow["\dashv"{anchor=center, rotate=-180}, shift right=2, draw=none, from=1, to=0]
	\end{tikzcd}\]
    \begin{enumerate}
        \item a tight-cell $i \colon Z \to V$, the \emph{coroot};
        \item a tight-cell $\ell \colon X \to V$, the \emph{left (relative) coadjoint};
        \item a tight-cell $r \colon Z \to X$, the \emph{right (relative) coadjoint};
        \item an isomorphism $\sharp \colon V(\ell, i) \iso X(1, r) \cocolon \flat$, the \emph{(left- and right-) transposition operators}.
    \end{enumerate}%
    We denote by $r \rcadj i \ell$ such data\footnotemark{} (by convention leaving the transposition operators implicit), and call $X$ the \emph{nadir}. An \emph{$i$-relative coadjunction} (or simply \emph{$i$-coadjunction}) is a relative coadjunction with coroot $i$.
	\footnotetext{We prefer to place the right coadjoint on the left, reflecting the composition order $(r \d \ell)$. However, the mirrored convention $\ell \mathrel{{}_i\!\!\adj} r$, which is used, for instance, in \cite[Def.~2]{reggiani1983doppie}, is also possible and is compatible with our convention.}
\end{definition}

We leave the reader to dualise the definitions of left- and right-morphisms of relative adjunctions (\cref{left-morphism,right-morphism}).

The distinction between relative adjunctions and relative coadjunctions disappears when the (co)root is the identity: that is, $\ell \radj 1 r$ if and only if $r \rcadj 1 \ell$. For this reason, relative adjunctions and relative coadjunctions have not been adequately distinguished in the literature, and authors have often used the term \emph{relative adjunction} to refer to either concept, disambiguating only via context\footnotemark{}. However, it is helpful to distinguish between the two concepts: for instance, while a relative adjunction induces a relative monad, a relative coadjunction induces a relative comonad.
\footnotetext{The naming convention of \textcite{ulmer1968properties} suggests \emph{$j$-left adjunction} for our $j$-adjunction, and \emph{$i$-right adjunction} for our $i$-coadjunction. In the terminology of \citeauthor{ulmer1968properties}, $\ell$ is \emph{$j$-left adjoint} to $r$ when $\ljr$; and $r$ is \emph{$i$-right adjoint} to $\ell$ when $r \rcadj i \ell$. This has the significant shortcoming that it leaves no convenient terminology for the right $j$-adjoint or left $i$-coadjoint.}

\begin{remark}
	Let $\ell$ and $r$ be antiparallel tight-cells in an equipment $\X$. Then the following are equivalent.
	\begin{enumerate}
		\item $\ell$ is left-adjoint to $r$ in $\X$.
		\item $\ell$ is left-coadjoint to $r$ in $\X$.
		\item $\ell$ is right-adjoint to $r$ in $\X\co$.
		\item $\ell$ is right-coadjoint to $r$ in $\X\co$. \qedhere
	\end{enumerate}
\end{remark}

\section{Enriched relative monads}
\label{relative-monads-in-VCat}

A motivating setting for this paper is that of enriched category theory: while the theory of relative monads in ordinary category theory has been developed to some extent~\cite{altenkirch2010monads,altenkirch2015monads}, the theory of relative monads in enriched category theory remains largely undeveloped. The formal theory we have developed herein allows us to deduce the theorems of interest for enriched categories by specialising to equipments of enriched categories. For simplicity, we work with enrichment in monoidal categories~\cite{benabou1965categories,maranda1965formal,kelly1982basic}, though we shall not need to impose symmetry, closure, or (co)completeness assumptions in general. In future work, we shall show that the results of interest hold for much more general bases of enrichment.

Throughout this section, we assume a fixed monoidal category $(\V, \tensor, \tensorI)$.
To simplify the notation, we work as though $\V$ is strict, but occasionally make explicit the unitors $\lambda_v \colon \tensorI \tensor v \to v$ and $\rho_v \colon v \to v \tensor \tensorI$ for clarity.

\begin{definition}
	\label{VCat}
    The \vdc{} $\VCat$ of \emph{categories enriched in $\V$} (or simply \emph{$\V$-categories}) is defined as follows.
    \begin{enumerate}
      \item An object is a \emph{$\V$-category}~\cite[\S1.2]{kelly1982basic}, comprising a class $\ob C$ of objects, an object $C(x, y)$ of $\V$ for each $x, y \in \ob C$, a morphism $\I_x \colon \tensorI \to C(x, x)$ in $\V$ for each $x \in \ob C$, and a morphism $\circ_{x, y, z} \colon C(x, y) \otimes C(y, z) \to C(x, z)$ in $\V$ for each $x, y, z \in \ob C$, subject to unitality and associativity.
      \item A tight-cell $f \colon C \to D$ is a \emph{$\V$-functor}~\cite[\S1.2]{kelly1982basic}, comprising a function $\ob f \colon \ob C \to \ob D$, together with a morphism $f_{x, y} \colon C(x, y) \to D(\ob f x, \ob f y)$ in $\V$ for each $x, y \in \ob C$, preserving identities and composites.
      \item A loose-cell $p \colon D \lto C$ is a \emph{$\V$-distributor}\footnotemark{}~\cite[\S3.1.c]{benabou1973distributeurs}, comprising an object $p(x, y)$ of $\V$ for each $x \in \ob C$ and $y \in \ob D$, and morphisms $\circ_{x', x, y} \colon C(x', x) \otimes p(x, y) \to p(x', y)$ and $\circ_{x, y, y'} \colon p(x, y) \otimes D(y, y') \to p(x, y')$ in $\V$ compatible with each other, and with composition and identities in $C$ and $D$.
    	\footnotetext{$\V$-distributors are alternatively called \emph{$\V$-profunctors} or \emph{$\V$-(bi)modules}.}
        \item A 2-cell
		\[\begin{tikzcd}
			{C_0} & \cdots & {C_n} \\
			D && {D'}
			\arrow["q", "\shortmid"{marking}, from=2-3, to=2-1]
			\arrow[""{name=0, anchor=center, inner sep=0}, "g", from=1-3, to=2-3]
			\arrow[""{name=1, anchor=center, inner sep=0}, "f"', from=1-1, to=2-1]
			\arrow["{p_n}"', "\shortmid"{marking}, from=1-3, to=1-2]
			\arrow["{p_1}"', "\shortmid"{marking}, from=1-2, to=1-1]
			\arrow["\phi"{description}, draw=none, from=0, to=1]
		\end{tikzcd}\]
        is a \emph{$\V$-natural transformation}\footnotemark{}, comprising a morphism
        \[\phi_{x_0, \ldots, x_n} \colon p_1(x_0, x_1) \otimes \cdots \otimes p_n(x_{n - 1}, x_n) \to q(\ob f x_0, \ob g x_n)\]
		in $\V$ for each $x_0 \in \ob{C_0}, \dots, x_n \in \ob{C_n}$, rendering the following \emph{$\V$-naturality} diagrams commutative.
		\footnotetext{The notion of $\V$-natural transformation defined here is more general than the usual notion of $\V$-natural transformation between $\V$-functors (\cf{}~\cite[(1.7)]{kelly1982basic}), which is recovered when $n = 0$ and $q = D(1, 1)$ (\cf{}~\cite[Example~6.4]{cruttwell2010unified}). When $f$ and $g$ are identities, this definition recovers the notion of \emph{$\V$-form} introduced by \textcite[134]{day1997monoidal}, though note that the definition \loccit{} is incomplete, as it omits the coherence condition for nullary $\V$-forms.}
\[
\begin{tikzcd}[sep=small]
	{\tensorI \otimes C_0(x, x')} && {C_0(x, x') \otimes \tensorI} \\
	\\
	{q(\ob f x, \ob g x) \otimes D'(\ob g x, \ob g x')} && {D(\ob f x, \ob f x') \otimes q(\ob f x', \ob g x')} \\
	& {q(\ob f x, \ob g x')}
	\arrow["{\phi_x \otimes g_{x, x'}}"', from=1-1, to=3-1]
	\arrow["{f_{x, x'} \otimes \phi_{x'}}", from=1-3, to=3-3]
	\arrow["{\circ^q_{\ob f x, \ob g x, \ob g x'}}"', from=3-1, to=4-2]
	\arrow["{\circ^q_{\ob f x, \ob f x', \ob g x'}}", from=3-3, to=4-2]
	\arrow["{\lambda_{C_0(x, x')} \d \rho_{C_0(x, x')}}", from=1-1, to=1-3]
\end{tikzcd}
\tag{$n = 0$}
\]
\[
\begin{tikzcd}
	{C_0(c, x_0) \otimes p_1(x_0, x_1) \otimes \cdots \otimes p_n(x_{n - 1}, x_n)} & {p_1(c, x_1) \otimes \cdots \otimes p_n(x_{n - 1}, x_n)} \\
	{D(\ob f c, \ob f x_0) \otimes p_1(x_0, x_1) \otimes \cdots \otimes p_n(x_{n - 1}, x_n)} \\
	{D(\ob f c, \ob f x_0) \otimes q(\ob f x_0, \ob g x_n)} & {q(\ob f c, \ob g x_n)}
	\arrow["{f_{c, x_0} \otimes p_1(x_0, x_1) \otimes \cdots \otimes p_n(x_{n - 1}, x_n)}"{description}, from=1-1, to=2-1]
	\arrow["{D(\ob f c, \ob f x_0) \otimes \phi_{x_0, \ldots, x_n}}"{description}, from=2-1, to=3-1]
	\arrow["{\circ^q_{\ob f c, \ob f x_0, \ob g x_n}}"', from=3-1, to=3-2]
	\arrow["{\circ^{p_1}_{c, x_0, x_1} \otimes \cdots}", from=1-1, to=1-2]
	\arrow["{\phi_{c, x_1, \ldots, x_n}}", from=1-2, to=3-2]
\end{tikzcd}
\tag{$n \geq 1$}
\]
\[
\begin{tikzcd}
	{p_1(x_0, x_1) \otimes \cdots \otimes p_n(x_{n - 1}, x_n) \otimes C_n(x_n, c)} & {p_1(x_0, x_1) \otimes \cdots \otimes p_n(x_{n - 1}, c)} \\
	{p_1(x_0, x_1) \otimes \cdots \otimes p_n(x_{n - 1}, x_n) \otimes D'(\ob g x_n, \ob g c)} \\
	{q(\ob f x_0, \ob g x_n) \otimes D'(\ob g x_n, \ob g c)} & {q(\ob f x_0, \ob g c)}
	\arrow["{p_1(x_0, x_1) \otimes \cdots \otimes p_n(x_{n - 1}, x_n) \otimes g_{x_n, c}}"{description}, from=1-1, to=2-1]
	\arrow["{\phi_{x_0, \ldots, x_n} \otimes D'(\ob g x_n, \ob g c)}"{description}, from=2-1, to=3-1]
	\arrow["{\circ^q_{\ob f x_0, \ob g x_n, \ob g c}}"', from=3-1, to=3-2]
	\arrow["{\cdots \otimes \circ^{p_n}_{x_{n - 1}, x_n, c}}", from=1-1, to=1-2]
	\arrow["{\phi_{x_0, \ldots, x_{n - 1}, c}}", from=1-2, to=3-2]
\end{tikzcd}
\tag{$n \geq 1$}
\]
\[
\begin{tikzcd}[column sep=-14em]
	& {p_1(x_0, x_1) \otimes \cdots \otimes p_i(x_{i-1}, x_i) \otimes C_i(x_i, x'_i) \otimes p_{i+1}(x'_i, x_{i+1}) \otimes \cdots \otimes p_n(x_{n-1}, x_n)} \\
	{\cdots \otimes p_i(x_{i-1}, x'_i) \otimes p_{i+1}(x'_i, x_{i+1}) \otimes \cdots} && {\cdots \otimes p_i(x_{i-1}, x_i) \otimes p_{i+1}(x_i, x_{i+1}) \otimes \cdots} \\
	& {q(\ob f x_0, \ob g x_n)}
	\arrow["{\cdots \otimes \circ^{p_i}_{x_{i-1}, x_i, x'_i} \otimes \cdots}"'{pos=0.8}, from=1-2, to=2-1]
	\arrow["{\cdots \otimes \circ^{p_{i + 1}}_{x_i, x'_i, x_{i+1}} \otimes \cdots}"{pos=0.8}, from=1-2, to=2-3]
	\arrow["{\phi_{x_0, \ldots, x_{i - 1}, x_i', x_{i + 1}, \ldots, x_n}}"', from=2-1, to=3-2]
	\arrow["{\phi_{x_0, \ldots, x_{i - 1}, x_i, x_{i + 1}, \ldots, x_n}}", from=2-3, to=3-2]
\end{tikzcd}
\tag{$1 \leq i < n$}
\]
    \end{enumerate}
\end{definition}

\begin{remark}
	\label{vdc-naming-convention}
    We use the name $\VCat$ following our decision to name \vdcs{} after their objects
    (\cf{}~\cref{name-after-objects}). In papers following the convention to name \vdcs{} after their loose-cells, the names
    $\V\h\dc{Dist}$, $\V\h\dc{Prof}$, and $\V\h\dc{Mod}$ are common alternatives.
\end{remark}

\begin{remark}
	Under modest assumptions on the monoidal category $\V$, \cref{VCat} may be simplified (\cf{}~\cite[Example~2.9]{cruttwell2010unified}). For instance, when $\V$ is closed (permitting $\V$ to be viewed as a $\V$-category itself) and symmetric (permitting both the construction of opposite $\V$-categories, and of the tensor product $\boxtimes$ of $\V$-categories), a $\V$-distributor $p \colon D \lto C$ is equivalently a $\V$-functor $C\op \boxtimes D \to \V$. When $\V$ is furthermore cocomplete, then the composite $q \odot p \colon D \lto B$ of $\V$-distributors $p \colon D \lto C$ and $q \colon C \lto B$ exists for small $C$, being exhibited by the following coend. In such cases, it suffices to consider only unary $\V$-natural transformations.
	\[(q \odot p)(b, d) \iso \int^{c \in \ob C} q(b, c) \otimes p(c, d) \qedhere\]
\end{remark}

Each $\V$-category $C$ admits a loose-identity, given by the $\V$-distributor $(x, y) \mapsto C(x, y)$ equipped with the composition structure of $C$. Furthermore, each $\V$-distributor $p \colon B \lto C$ admits a restriction along tight-cells $f \colon D \to C$ and $g \colon A \to B$, given by pre- and postcomposition:
\[p(f, g)(x, y) \defeq p(\ob f x, \ob g y)\]
The associated cartesian 2-cell is the identity on each component, so that a $\V$-natural transformation with unary domain is cartesian precisely when all of its components are isomorphisms in $\V$. In this way, the \vdc{} $\VCat$ forms an equipment~\cite[Examples~7.3]{cruttwell2010unified}.
The tight 2-category $\u\VCat$ is the usual 2-category of $\V$-categories, $\V$-functors, and $\V$-natural transformations.

\subsection{Formal category theory}
\label{formal-category-theory-in-VCat}

We describe concretely the various definitions of \cref{formal-category-theory} in the setting of $\VCat$, to show that we recover the usual notions.

\begin{lemma}
	A $\V$-functor $f \colon C \to D$ is \ff{} in the sense of \cref{full-faithfulness} if and only if for each pair of objects $x, y \in \ob C$, the morphism $f_{x, y} \colon C(x, y) \to D(\ob f x, \ob f y)$ is invertible.
\end{lemma}

\begin{proof}
    Immediate from the fact that a $\V$-natural transformation with unary domain is cartesian precisely when every component is invertible.
\end{proof}

The concrete definitions of weighted colimits involve enriched presheaves.
We will also use $\V$-presheaves as part of a sufficient condition for the existence of algebra objects in $\VCat$~(\cref{VCat-admits-algebra-objects-ii}).

\begin{definition}\label{presheaf}
  A \emph{$\V$-presheaf} on a $\V$-category $Z$ comprises an object $p(z)$ of $\V$ for each object $z \in \ob{Z}$, together with a left-action of $Z$ on $p$, \ie{} a morphism $\c_{z', z} \colon Z(z', z) \tensor p(z) \to p(z')$ in $\V$ for each pair of objects $z, z' \in \ob{Z}$, compatible with identities and composition in $Z$ in the sense that the following diagrams commute.
  \[
\begin{tikzcd}[column sep=large]
	{\tensorI \tensor p(z)} & {Z(z, z) \tensor p(z)} \\
	& {p(z)}
	\arrow["{\circ_{z, z}}", from=1-2, to=2-2]
	\arrow["{\I_z \tensor p(z)}", from=1-1, to=1-2]
	\arrow["{\lambda_{p(z)}}"', from=1-1, to=2-2]
\end{tikzcd}
\quad
\begin{tikzcd}[column sep=huge]
	{Z(z'', z') \tensor Z(z', z) \tensor p(z)} & {Z(z'', z) \tensor p(z)} \\
	{Z(z'', z') \tensor p(z')} & {p(z'')}
	\arrow["{\circ_{z'', z}}", from=1-2, to=2-2]
	\arrow["{\circ_{z'', z', z} \tensor p(z)}", from=1-1, to=1-2]
	\arrow["{Z(z'', z') \tensor \circ_{z', z}}"', from=1-1, to=2-1]
	\arrow["{\circ_{z'', z'}}"', from=2-1, to=2-2]
\end{tikzcd}
  \]
  Let $p$ and $q$ be $\V$-presheaves on $Z$, and $v$ be an object of $\V$.
  A family of morphisms
  \[
    \{ \phi_z \colon p(z) \tensor v \to q(z) \}_{z \in \ob Z}
  \]
  is \emph{$\V$-natural} in $z \in \ob Z$ when the following diagram commutes for each pair of objects $z, z' \in \ob Z$.
\[\begin{tikzcd}
	{Z(z', z) \tensor p(z) \tensor v} && {Z(z', z) \tensor q(z)} \\
	{p(z') \tensor v} && {q(z')}
	\arrow["{\circ^p_{z', z} \tensor v}"', from=1-1, to=2-1]
	\arrow["{\circ^q_{z', z}}", from=1-3, to=2-3]
	\arrow["{Z(z', z) \tensor \phi_z}", from=1-1, to=1-3]
	\arrow["{\phi_{z'}}"', from=2-1, to=2-3]
\end{tikzcd}\]
  If there exists an object $\rfP{Z}{p}{q}$ of $\V$, equipped with a universal $\V$-natural family
  \[
    \{ \varpi_z \colon p(z) \tensor \rfP{Z}{p}{q} \to q(z) \}_{z \in \ob Z}
  \]
  then we call $\rfP{Z}{p}{q}$ the \emph{object of $\V$-natural transformations from $p$ to $q$}.
\end{definition}

Explicitly, the universal property of $\rfP{Z}{p}{q}$ states that, for every $\V$-natural family $\phi$ as above, there is a unique morphism $\tilde\phi \colon v \to \rfP{Z}{p}{q}$ in $\V$ such that, for each object $z \in \ob{Z}$, the morphism $\phi_z$ is equal to the following composite in $\V$.
\[
  p(z) \tensor v \xto{p(z) \tensor \tilde\phi} p(z) \tensor \rfP{Z}{p}{q} \xto{\varpi_z} q(z)
\]
This universal property of $\rfP{Z}{p}{q}$ is the same as the (second form of the) universal property given in \cite[\S3]{gordon1999gabriel}. When $\V$ is symmetric and closed, a presheaf $p$ on $Z$ is the same as a $\V$-functor $p \colon Z\op \to \V$, in which case $\rfP{Z}{p}{q}$ is exhibited by the hom-object $[Z\op, \V](p, q)$ of the $\V$-functor category $[Z\op, \V]$ when it exists (\cf{}~\cite[\S2]{kelly1982basic}).
In general, when $\rfP{Z}{p}{q}$ exists for all presheaves $p$ and $q$, these objects form the hom-objects of a $\V$-category $\P Z$.
In particular, this is the case for small $\V$-categories $Z$ when $\V$ is complete and (left- and right-) closed~\cite[\S3]{gordon1999gabriel}.
However, as is the usual convention, we use the notation $\rfP{Z}{p}{q}$ even when the $\V$-category $\P Z$ does not exist.

\begin{example}\label{nerve-presheaf}
  The \emph{Yoneda embedding} of an object $x \in \ob{Z}$ is the $\V$-presheaf $\yo_Z x \defeq Z({-}, x)$, with the action of $Z$ given by composition.
  For each presheaf $q$ on $Z$, the object $\rfP{Z}{Z({-}, x)}{q}$ is isomorphic to $q(x)$, since, for a fixed object $v$ of $\V$, the $\V$-natural families
  $
    \{ Z(z, x) \tensor v \to q(z) \}_{z \in \ob Z}
  $
  are in bijection with morphisms $v \to q(x)$:
  this is precisely the Yoneda lemma.
  When the $\V$-category $\P Z$ exists, the Yoneda embedding forms a $\V$-functor $\yo_Z \colon Z \to \P Z$.

  For every $\V$-functor $j \colon A \to E$ and object $x \in \ob{E}$, there is a $\V$-presheaf $n_j x \defeq E(j{-}, x)$ on $A$, the \emph{nerve of $j$ at $x$}.
  The action of $A$ on $n_j x$ is given by the following composite in $\V$.
  \[
    A(z', z) \tensor E(\ob j z, x)
    \xto{j_{z', z} \tensor E(\ob j z, x)}
    E(\ob{j}z', \ob{j}z) \tensor E(\ob{j}z, x)
    \xto{\circ_{\ob{j}z', \ob{j}z, x}}
    E(\ob{j}z', x)
  \]
  When the $\V$-category $\P A$ exists, the nerve of $j$ forms a $\V$-functor $n_j \colon E \to \P A$.
\end{example}

Denote by $\Icat$ the $\V$-category with a single object $\star$, and hom-object $\Icat(\star, \star) \defeq \tensorI$.
An object $z$ of a $\V$-category $Z$ is then equivalently a $\V$-functor $z \colon \Icat \to Z$, while a $\V$-presheaf $p$ on $Z$ is equivalently a $\V$-distributor $p \colon \Icat \lto Z$.
An object $v$ of $\V$ is equivalently a $\V$-distributor $v \colon \Icat \lto \Icat$, and a $\V$-natural transformation
$
  \phi \colon p, v_1, \dots, v_n \tto q
$
is equivalently a $\V$-natural family of morphisms.
It follows that the object $\rfP{Z}{p}{q}$ is then exactly the right lift $q \rf p \colon \Icat \lto \Icat$.
If $r \colon X \lto Y$ is a $\V$-distributor, then the component $r(y, x)$, viewed as a $\V$-distributor $r(y, x) \colon \Icat \lto \Icat$, is the restriction of $r$ along the $\V$-functors $x \colon \Icat \to X$ and $y \colon \Icat \to Y$.

\begin{lemma}
  Let $p \colon Y \lto Z$ and $q \colon X \lto Z$ be $\V$-distributors.
  If the objects $\rfP{Z}{p({-}, y)}{q({-}, x)}$ exist for every $x \in \ob{X}$ and $y \in \ob{Y}$ then they form the right lift $q \rf p \colon X \lto Y$ in $\VCat$.
  \[
    (q \rf p)(y, x) \defeq \rfP{Z}{p({-}, y)}{q({-}, x)}
  \]
  The actions of $X$ and $Y$ on $q \rf p$ are unique such that the universal $\V$-natural families $\varpi_z \colon p(z, y) \tensor \rfP{Z}{p({-}, y)}{q({-}, x)} \to q(z, x)$ constitute a $\V$-natural transformation $p, q \rf p \tto q$.

  The converse also holds: if $q \rf p$ exists, then $(q \rf p)(y, x)$ satisfies the universal property of $\rfP{Z}{p({-}, y)}{q({-}, x)}$.
\end{lemma}
\begin{proof}
  Suppose that the objects $\rfP{Z}{p({-}, y)}{q({-}, x)}$ exist, and define $(q \rf p)(y, x)$ as above.
  We first show that $q \rf p$ canonically forms a $\V$-distributor. The universal property of $\rfP{Z}{p({-}, y)}{q({-}, x)}$ defines unique morphisms for each $x, x' \in \ob X$ and $y, y' \in \ob Y$,
  \begin{align*}
    \circ_{y, x, x'} &\colon (q \rf p)(y, x) \tensor X(x, x') \to (q \rf p)(y, x')
    \\
    \circ_{y', y, x} &\colon Y(y', y) \tensor (q \rf p)(y, x) \to (q \rf p)(y', x)
  \end{align*}
  rendering the following diagrams, natural in $z \in Z$, commutative.
  \begin{gather*}
\begin{tikzcd}[ampersand replacement=\&, column sep=7em]
	{p(z, y) \tensor (q \rf p)(y, x) \tensor X(x, x')} \& {q(z, x) \tensor X(x, x')} \\
	{p(z, y) \tensor (q \rf p)(y, x')} \& {q(z, x')}
	\arrow["{\varpi_z \tensor X(x, x')}", from=1-1, to=1-2]
	\arrow["{\circ_{z, x, x'}}", from=1-2, to=2-2]
	\arrow["{\varpi_z}"', from=2-1, to=2-2]
	\arrow["{p(z, y) \tensor \circ_{y, x, x'}}"', from=1-1, to=2-1]
\end{tikzcd}
\\
\begin{tikzcd}[ampersand replacement=\&, column sep=7em]
	{p(z, y') \tensor Y(y', y) \tensor (q \rf p)(y, x)} \& {p(z, y) \tensor (q \rf p)(y, x)} \\
	{p(z, y') \tensor (q \rf p)(y', x)} \& {q(z, x)}
	\arrow["{\circ_{z, y', y} \tensor (q \rf p)(y, x)}", from=1-1, to=1-2]
	\arrow["{\varpi_z}", from=1-2, to=2-2]
	\arrow["{\varpi_z}"', from=2-1, to=2-2]
	\arrow["{p(z, y') \tensor \circ_{y', y, x}}"', from=1-1, to=2-1]
\end{tikzcd}
\end{gather*}
  These are compatible with identities and composition because $p$ and $q$ are; that they are compatible with each other is immediate from the definitions.
  Hence $q \rf p$ forms a $\V$-distributor.
  Moreover, commutativity of the diagrams above, together with $\V$-naturality in $z$, are precisely the laws required for $\varpi$ to be a $\V$-natural transformation. Thus $q \rf p$ uniquely forms a $\V$-distributor such that $\varpi$ is a $\V$-natural transformation.

  To show that $q \rf p$ is the right lift, we have to prove that every $\V$-natural transformation
  \[
    \phi \colon p, r_1, \dots, r_n \tto q
  \]
  factors uniquely through $\varpi$ as a $\V$-natural transformation
  \[
    \tilde\phi \colon r_1, \dots, r_n \tto q \rf p
  \]
  The universal property of $\rfP{Z}{p({-}, y)}{q({-}, x)}$ provides morphisms
  \[
    \tilde\phi_{y, w_1, \dots, w_{n-1}, x} \colon r_1(y, w_1) \tensor \cdots \tensor r_n(w_{n-1}, x) \tto (q \rf p)(y, x)
  \]
  that uniquely factor through $\varpi$: that these constitute a $\V$-natural transformation is immediate from the universal properties and the fact that $\phi$ is a $\V$-natural transformation.

  To show the converse, suppose that $q \rf p$ exists.
  Using the observations about the $\V$-category $\Icat$ above, \cref{right-lift-and-restriction} shows that the component $(q \rf p)(y, x)$ is the right lift $q({-}, x) \rf p({-}, y)$, which is precisely $\rfP{Z}{p({-}, y)}{q({-}, x)}$.
\end{proof}

It follows from this characterisation of right lifts in $\VCat$ that our notion of weighted colimit agrees with the usual one~\cite[(3.5)]{kelly1982basic}.
If $p \colon Y \lto Z$ is a $\V$-distributor and $f \colon Z \to X$ is a $\V$-functor, then a $\V$-functor $p \wc f \colon Y \to X$ forms the $p$-weighted colimit of $f$ exactly when there are isomorphisms
\[
  X(\ob{p \wc f}y,x) \iso \rfP{Z}{p({-}, y)}{X(f{-}, x)}
\]
$\V$-natural in $x \in \ob X$ and $y \in \ob Y$.
Since we define left extensions in terms of colimits (\cref{left-extension}), it also follows that our definition of left extension agrees with the usual one for $\X = \VCat$~\cite[(4.20)]{kelly1982basic} and, consequently, that our definition of density coincides with the usual one~\cite[Theorem~5.1(v)]{kelly1982basic}.

We end this subsection by characterising, for a $\V$-functor $\jAE$, the $j$-absolute colimits in $E$ as those colimits that are preserved by the nerve $\V$-functor $n_j \colon E \to \P A$ (\cf{}~\cite[\S5.4]{kelly1982basic}).
To do so, we first prove that colimits in presheaf $\V$-categories are $j$-absolute; we shall make use of the characterisation of colimits above in the proof.

\begin{lemma}
    \label{yoneda-absolute-is-trivial}
    Let $A$ be a $\V$-category such that the presheaf $\V$-category $\P A$ exists.
    Every colimit in $\cl P A$ is $\yo_A$-absolute, where
    $\yo_A \colon A \to \cl P A$ denotes the Yoneda embedding.
    Hence, for a $\V$-distributor $p \colon X \lto Y$ and $\V$-functor $f \colon Y \to \cl P A$, a $p$-cocone $(c, \lambda)$ for $f$ is colimiting exactly when the canonical $\V$-natural transformation
    \[
      \cl P A(\yo_A, f), p \tto \cl P A(\yo_A, c)
    \]
    is left-opcartesian.
\end{lemma}

\begin{proof}
    Let $p \colon X \lto Y$ be a $\V$-distributor and $f \colon Y \to \cl P A$ be a $\V$-functor such that the colimit $p \wc f \colon X \to \cl{P}A$ exists.
    To establish $\yo_A$-absoluteness of $p \wc f$, we must show that the canonical $\V$-natural transformation
    $\cl P A(\yo_A, f), p \tto \cl P A(\yo_A, p \wc f)$ is left-opcartesian.
    Explicitly, we show that composition with this $\V$-natural transformation induces a bijection between $\V$-natural families of morphisms
    \[
        \{ \cl P A(\yo_A, p \wc f)(a, x) \otimes v \to \phi(a) \}_{a \in \ob{A}}
    \]
    and $\V$-natural families of morphisms
    \[
        \{ \cl P A(\yo_A, f)(a, y) \otimes p(y, x) \otimes v \to \phi(a) \}_{a \in \ob{A}, y \in \ob{Y}}
    \]
    for each $\phi \in \ob{\cl P A}$, $x \in \ob{X}$, and $v \in \ob{\V}$.
    This is indeed the case, because we have the following bijections, by the Yoneda lemma, the definition of $\cl P A$, the universal property of $p \wc f$, the definitions of $\cl P Y$ and $\cl P A$, and finally the Yoneda lemma again.
    \begin{iffseq}
        \cl P A(\yo_A, p \wc f)(a, x) \otimes v \to \phi(a) & \text{($\V$-natural in $a$)} \\
        (\ob{p \wc f} x)(a) \otimes v \to \phi(a) & \text{($\V$-natural in $a$)} \\
        v \to \cl P A(\ob{p \wc f} x, \phi) \\
        v \to \cl{P}Y(p({-}, x), \cl P A(f{-}, \phi)) \\
        (\ob f y)(a) \otimes p(y, x) \otimes v \to \phi(a) & \text{($\V$-natural in $a, y$)} \\
        \cl P A(\yo_A, f)(a, y) \otimes p(y, x) \otimes v \to \phi(a) & \text{($\V$-natural in $a, y$)}
    \end{iffseq}

    For the second part of the lemma, $\yo_A$ is dense by the Yoneda lemma.
    Hence left-opcartesianness implies that the cocone $(c, \lambda)$ is colimiting by \cref{absolute-implies-colimit}.
\end{proof}

\begin{remark}
    In the above lemma, the universal property of objects of $\V$-natural transformations only permits us to show that the canonical $\V$-natural transformation is left-opcartesian, rather than opcartesian.
    This is why we only require left-opcartesianness in the definition of $j$-absolute (\cref{j-absolute}).
\end{remark}

Our characterisation of $j$-absolute colimits follows.

\begin{lemma}
    \label{j-absolute-is-preservation-by-n_j}
    Let $\jAE$ be a $\V$-functor, and assume that the presheaf $\V$-category $\P A$ exists.
    Given a $\V$-distributor $p \colon X \lto Y$ and $\V$-functor $f \colon Y \to E$, a colimit $p \wc f$ in $E$ is $j$-absolute exactly when it is preserved by the nerve $\V$-functor $n_j \colon E \to \cl P A$.
\end{lemma}

\begin{proof}
    Let $p \wc f \colon X \to E$ be the $p$-colimit of a $\V$-functor $f \colon Y \to E$.
    By definition, $p \wc f$ is $j$-absolute exactly when the canonical $\V$-natural transformation
    \[
      E(j, f), p \tto E(j, p \wc f)
    \]
    is left-opcartesian.
    By the Yoneda lemma, this is equivalent to left-opcartesianness of the canonical $\V$-natural transformation
    \[
      \cl P A(\yo_A, n_j f), p \tto \cl P A(\yo_A, n_j (p \wc f))
    \]
    By \cref{yoneda-absolute-is-trivial}, this holds exactly when the canonical $\V$-natural transformation $p \tto E(n_j f, n_j (p \wc f))$ exhibits $(p \wc f) \d n_j$ as the colimit $p \wc (f \d n_j)$, in other words, when $n_j$ preserves $p \wc f$.
\end{proof}

\subsection{Relative monads and relative adjunctions}
\label{enriched-relative-monads-and-relative-adjunctions}

It is immediate that \cref{relative-adjunction} specialises to the expected definition of enriched relative adjunction, comprising $\V$-functors $\jAE$, $\ell \colon A \to C$, $r \colon C \to E$ together with a family of isomorphisms in $\V$,
\[\sharp_{a, c} \colon C(\ob \ell a, c) \iso E(\ob j a, \ob r c) \cocolon \flat_{a, c}\]
$\V$-natural in $a \in \ob A$ and $c \in \ob C$.

\begin{example}
	Our notion of $\V$-enriched relative adjunction subsumes various prior definitions.
	\begin{enumerate}
		\item A \emph{relative adjunction} in the sense of \textcite[Definition~2.2]{ulmer1968properties} is precisely a $\Set$-enriched relative adjunction. In particular, this captures the notion of \emph{pre-coadjoint} in the sense of \textcite[\S3]{isbell1966structure} (\cf{}~\cite[footnote 10]{ulmer1968properties}).
		\item A \emph{family of universal arrows} in the sense of \textcite[Example~1.1.2]{walters1970categorical} is precisely a $\Set$-enriched relative adjunction whose root is \ff{}.
        \item A \emph{partial adjunction} with respect to a functor, 2-functor, or $\V$-functor (for $\V$ a symmetric closed monoidal category), in the sense of \textcite[\S0.2]{blackwell1976some} is precisely a $\Set$-enriched, $\b{Cat}$-enriched, or $\V$-enriched relative adjunction respectively.
		\item A \emph{$j$-comodel} in the sense of \textcite[\S5.12]{kelly1982basic}, for $\V$ a locally small complete and cocomplete symmetric closed monoidal category and $\jAE$ a dense $\V$-functor with small domain, is precisely a $\V$-enriched $j$-adjunction.
		\item A \emph{relative 2-adjunction} in the sense of \textcite[Definition~3.6]{fiore2018relative} is precisely a $\b{Cat}$-enriched relative adjunction.
		\item An \emph{enriched relative adjunction} in the sense of \textcite[\S2.1]{staton2020classical} is precisely an enriched relative adjunction in our sense.
		\item A \emph{$j$-relative adjunction} in the sense of \textcite[Definition~17]{mcdermott2022flexibly}, for $\V$ a small monoidal category and $j$ a functor between locally $\V$-graded categories, is precisely a $[\V, \Set]$-enriched $j$-adjunction. \qedhere
	\end{enumerate}
\end{example}

To show that \cref{relative-monad} specialises to notions of enriched relative monad in the literature requires a little more work. To do so, we show that the definition of relative monad may be simplified in the setting of $\VCat$. In particular, it is only necessary to specify the action of the underlying $\V$-functor on objects: functoriality then follows automatically from the relative monad laws.

\begin{theorem}
	\label{enriched-relative-monad}
    A relative monad in $\VCat$ is equivalently specified by
    \begin{enumerate}
        \item a $\V$-functor $\jAE$;
        \item a function $\ob t \colon \ob A \to \ob E$;
        \item a morphism $\dag_{x, y} \colon E(\ob j x, \ob t y) \to E(\ob t x, \ob t y)$ in $\V$ for each $x, y \in \ob A$;
        \item a morphism $\eta_x \colon \tensorI \to E(\ob j x, \ob t x)$ in $\V$ for each $x \in \ob A$,
    \end{enumerate}
	satisfying the following equations for each $x, y, z \in \ob A$.
	\[
	\begin{tikzcd}
		{E(\ob j x, \ob t y)} & {E(\ob{t}x, \ob{t}y)} \\
		& {E(\ob j x, \ob t y)}
		\arrow["{\dag_{x, y}}", from=1-1, to=1-2]
		\arrow["{E(\eta_x, \ob{t}y)}", from=1-2, to=2-2]
		\arrow[Rightarrow, no head, from=1-1, to=2-2]
	\end{tikzcd}
	\hspace{4em}
	\begin{tikzcd}
		\tensorI & {E(\ob j x, \ob t x)} \\
		& {E(\ob t x, \ob t x)}
		\arrow["{\eta_x}", from=1-1, to=1-2]
		\arrow["{\dag_{x, x}}", from=1-2, to=2-2]
		\arrow["{\I_{\ob t x}}"', from=1-1, to=2-2]
	\end{tikzcd}
	\]
	\[\begin{tikzcd}
		{E(\ob j x, \ob t y) \tensor E(\ob j y, \ob t z)} && {E(\ob j x, \ob t y) \tensor E(\ob t y, \ob t z)} \\
		{E(\ob t x, \ob t y) \tensor E(\ob t y, \ob t z)} && {E(\ob j x, \ob t z)} \\
		& {E(\ob t x, \ob t z)}
		\arrow["{\circ_{\ob t x, \ob t y, \ob t z}}"', from=2-1, to=3-2]
		\arrow["{\dag_{x, z}}", from=2-3, to=3-2]
		\arrow["{E(\ob j x, \ob t y) \tensor \dag_{y, z}}", from=1-1, to=1-3]
		\arrow["{\circ_{\ob j x, \ob t y, \ob t z}}", from=1-3, to=2-3]
		\arrow["{\dag_{x, y} \tensor \dag_{y, z}}"', from=1-1, to=2-1]
	\end{tikzcd}\]
	A $j$-monad morphism in $\VCat$ is equivalently specified by a morphism $\tau_x \colon \ob t x \to \ob{t'} x$ in $\V$ for each $x \in \ob A$, satisfying the following equations for each $x, y \in \ob A$.
	\[
	\begin{tikzcd}[column sep=small]
		& \tensorI \\
		{E(\ob j x, \ob t x)} && {E(\ob j x, \ob{t'} x)}
		\arrow["{\eta_x}"', from=1-2, to=2-1]
		\arrow["{\eta'_x}", from=1-2, to=2-3]
		\arrow["{E(\ob j x, \tau_x)}"', from=2-1, to=2-3]
	\end{tikzcd}
	\hspace{4em}
	\begin{tikzcd}[sep=small]
		{E(\ob j x, \ob t y)} && {E(\ob j x, \ob{t'} y)} \\
		{E(\ob t x, \ob t y)} && {E(\ob{t'} x, \ob{t'} y)} \\
		& {E(\ob t x, \ob{t'} y)}
		\arrow["{\dag_{x, y}}"', from=1-1, to=2-1]
		\arrow["{E(\ob t x, \tau_y)}"', from=2-1, to=3-2]
		\arrow["{E(\ob j x, \tau_y)}", from=1-1, to=1-3]
		\arrow["{\dag'_{x, y}}", from=1-3, to=2-3]
		\arrow["{E(\tau_x, \ob{t'} y)}", from=2-3, to=3-2]
	\end{tikzcd}
	\]
\end{theorem}

Before proceeding with the proof, we first note that $\V$-naturality of $\eta$ expresses commutativity of the following diagram in $\V$ for each $x, y \in \ob A$,
\[\begin{tikzcd}[column sep=large]
	{A(x, y)} & {E(\ob j x, \ob j y)} \\
	{E(\ob t x, \ob t y)} & {E(\ob j x, \ob t y)}
	\arrow["{j_{x, y}}", from=1-1, to=1-2]
	\arrow["{E(\ob j x, \eta_y)}", from=1-2, to=2-2]
	\arrow["{E(\eta_x, \ob t y)}"', from=2-1, to=2-2]
	\arrow["{t_{x, y}}"', from=1-1, to=2-1]
\end{tikzcd}\]
while $\V$-naturality of $\dag$ expresses commutativity of the following two diagrams in $\V$ for each $x, y, z \in \ob A$.
\[\begin{tikzcd}
	{A(x, y) \otimes E(\ob j y, \ob t z)} & {A(x, y) \otimes E(\ob t y, \ob t z)} \\
	{E(\ob j x, \ob j y) \otimes E(\ob j y, \ob t z)} & {E(\ob t x, \ob t y) \otimes E(\ob t y, \ob t z)} \\
	{E(\ob j x, \ob t z)} & {E(\ob t x, \ob t z)}
	\arrow["{A(x, y) \otimes \dag_{y, z}}", from=1-1, to=1-2]
	\arrow["{\dag_{x, z}}"', from=3-1, to=3-2]
	\arrow["{j_{x, y} \otimes E(\ob j y, \ob t z)}"', from=1-1, to=2-1]
	\arrow["{\circ_{\ob j x, \ob j y, \ob t z}}"', from=2-1, to=3-1]
	\arrow["{t_{x, y} \otimes E(\ob t y, \ob t z)}", from=1-2, to=2-2]
	\arrow["{\circ_{\ob t x, \ob t y, \ob t z}}", from=2-2, to=3-2]
\end{tikzcd}\]
\[\begin{tikzcd}
	{E(\ob j x, \ob t y) \otimes A(y, z)} & {E(\ob t x, \ob t y) \otimes A(y, z)} \\
	{E(\ob j x, \ob t y) \otimes E(\ob t y, \ob t z)} & {E(\ob t x, \ob t y) \otimes E(\ob t y, \ob t z)} \\
	{E(\ob j x, \ob t z)} & {E(\ob t x, \ob t z)}
	\arrow["{\dag_{x, y} \otimes A(y, z)}", from=1-1, to=1-2]
	\arrow["{\dag_{x, z}}"', from=3-1, to=3-2]
	\arrow["{E(\ob j x, \ob t y) \otimes t_{y, z}}"', from=1-1, to=2-1]
	\arrow["{\circ_{\ob j x, \ob t y, \ob t z}}"', from=2-1, to=3-1]
	\arrow["{E(\ob t x, \ob t y) \otimes t_{y, z}}", from=1-2, to=2-2]
	\arrow["{\circ_{\ob t x, \ob t y, \ob t z}}", from=2-2, to=3-2]
\end{tikzcd}\]

\begin{proof}
    Since it is trivial that every relative monad in $\VCat$ specifies the given data, it is enough to show that, given the specified data, $\ob t \colon \ob A \to \ob E$ extends to a $\V$-functor $t \colon A \to E$, for which $\{ \eta_x \}_{x \in \ob A}$ and $\{ \dag_{x, y} \}_{x, y \in \ob A}$ are $\V$-natural transformations, and furthermore that this extension is the unique such that the relative monad laws are satisfied.

	Supposing $\ob t$ thus extends, the following diagram must commute in $\V$ for each $x, y \in \ob A$, using $\V$-naturality of $\dag$, the second unit law, and the right unit law of composition in $E$.
	\[\begin{tikzcd}[column sep=huge]
		{A(x, y) \tensor \tensorI} && {A(x, y)} \\
		{A(x, y) \tensor E(\ob{j}y, \ob{t}y)} & {A(x, y) \tensor E(\ob{t}y, \ob{t}y)} \\
		{E(\ob{j}x, \ob{j}y) \tensor E(\ob{j}y, \ob{t}y)} & {E(\ob{t}x, \ob{t}y) \tensor E(\ob{t}y, \ob{t}y)} \\
		{E(\ob j x, \ob t y)} & {E(\ob t x, \ob t y)}
		\arrow["{\dag_{x, y}}"', from=4-1, to=4-2]
		\arrow["{A(x, y) \tensor \eta_y}"', from=1-1, to=2-1]
		\arrow["{\circ_{\ob{t}x, \ob{t}y, \ob{t}y}}"{description}, from=3-2, to=4-2]
		\arrow["{A(x, y) \tensor \dag_{y, y}}"{description}, from=2-1, to=2-2]
		\arrow["{A(x, y) \tensor \I_{\ob{t}y}}"{description}, from=1-1, to=2-2]
		\arrow["{t_{x, y} \tensor E(\ob{t}y, \ob{t}y)}"{description}, from=2-2, to=3-2]
		\arrow["{j_{x, y} \tensor E(\ob{j}y, \ob{t}y)}"', from=2-1, to=3-1]
		\arrow["{\circ_{\ob{j}x, \ob{j}y, \ob{t}y}}"', from=3-1, to=4-1]
		\arrow["{\rho_{A(x, y)}}"', from=1-3, to=1-1]
		\arrow["{t_{x, y}}", curve={height=-30pt}, from=1-3, to=4-2]
	\end{tikzcd}\]
	Thus any such extension is necessarily unique. Conversely, given the specified data, define
	\[t_{x, y} \defeq j_{x, y} \d E(\ob j x, \eta_y) \d \dag_{x, y}\]
	for each $x, y \in \ob A$.
	(Observe that this is precisely the outside composite in the diagram above, by commutativity of the following diagram.)
	\[\begin{tikzcd}[column sep=5.5em]
		{A(x, y)} & {E(\ob{j}x, \ob{j}y)} \\
		{A(x, y) \tensor \tensorI} & {E(\ob{j}x, \ob{j}y) \tensor \tensorI} \\
		{A(x, y) \tensor E(\ob{j}y, \ob{t}y)} & {E(\ob{j}x, \ob{j}y) \tensor E(\ob{j}y, \ob{t}y)} & {E(\ob j x, \ob t y)}
		\arrow["{A(x, y) \tensor \eta_y}"', from=2-1, to=3-1]
		\arrow["{j_{x, y} \tensor E(\ob{j}y, \ob{t}y)}"', from=3-1, to=3-2]
		\arrow["{\circ_{\ob{j}x, \ob{j}y, \ob{t}y}}"', from=3-2, to=3-3]
		\arrow["{\rho_{A(x, y)}}"', from=1-1, to=2-1]
		\arrow["{j_{x, y}}", from=1-1, to=1-2]
		\arrow["{\rho_{E(\ob{j}x, \ob{j}y)}}"{description}, from=1-2, to=2-2]
		\arrow["{E(\ob{j}x, \ob{j}y) \tensor \eta_y}"{description}, from=2-2, to=3-2]
		\arrow["{E(\ob{j}x, \eta_y)}", curve={height=-12pt}, from=1-2, to=3-3]
		\arrow["{j_{x, y} \tensor \tensorI}"{description}, from=2-1, to=2-2]
	\end{tikzcd}\]
	This is a $\V$-functor: preservation of identities follows by commutativity of
        \[
\begin{tikzcd}[column sep=huge]
	\tensorI & {A(x, x)} \\
	& {E(\ob j x, \ob j x)} \\
	& {E(\ob j x, \ob t x)} \\
	& {E(\ob t x, \ob t x)}
	\arrow["{\dag_{x, x}}", from=3-2, to=4-2]
	\arrow["{j_{x, x}}", from=1-2, to=2-2]
	\arrow["{E(\ob j x, \eta_x)}", from=2-2, to=3-2]
	\arrow["{\I_x}", from=1-1, to=1-2]
	\arrow["{\I_{\ob j x}}"{description}, from=1-1, to=2-2]
	\arrow["{\I_{\ob t x}}"', curve={height=18pt}, from=1-1, to=4-2]
	\arrow["{\eta_x}"{description}, from=1-1, to=3-2]
\end{tikzcd}
  \]
  using preservation of identities of $j$, and the second unit law; while preservation of composites follows by commutativity of
	\[\begin{tikzcd}[column sep=5.5em]
		{A(x, y) \otimes A(y, z)} && {A(x, z)} \\
		{E(\ob j x, \ob j y) \otimes E(\ob j y, \ob j z)} && {E(\ob j x, \ob j z)} \\
		{E(\ob j x, \ob t y) \otimes E(\ob j y, \ob t z)} & {E(\ob j x, \ob t y) \otimes E(\ob t y, \ob t z)} & {E(\ob j x, \ob t z)} \\
		{E(\ob t x, \ob t y) \otimes E(\ob t y, \ob t z)} && {E(\ob t x, \ob t z)}
		\arrow["{\circ_{\ob t x, \ob t y, \ob t z}}"', from=4-1, to=4-3]
		\arrow["{\circ_{x, y, z}}", from=1-1, to=1-3]
		\arrow["{j_{x, z}}", from=1-3, to=2-3]
		\arrow["{j_{x, y} \otimes j_{y, z}}"', from=1-1, to=2-1]
		\arrow["{\circ_{\ob j x, \ob j y, \ob j z}}"{description}, from=2-1, to=2-3]
		\arrow["{E(\ob j x, \eta_y) \otimes E(\ob j y, \eta_z)}"{description}, from=2-1, to=3-1]
		\arrow["{E(\ob j x, \eta_z)}", from=2-3, to=3-3]
		\arrow["{\dag_{x, y} \otimes \dag_{y, z}}"', from=3-1, to=4-1]
		\arrow["{\dag_{x, z}}", from=3-3, to=4-3]
		\arrow["{E(\ob j x, \ob t y) \otimes \dag_{y, z}}", from=3-1, to=3-2]
		\arrow["{\circ_{\ob j x, \ob t y, \ob t z}}"{description}, from=3-2, to=3-3]
	\end{tikzcd}\]
	using preservation of composites of $j$, the first unit law, and the associativity law. $\eta$ is then $\V$-natural, the following diagram commuting by the first unit law.
\[\begin{tikzcd}[column sep=large]
	{A(x, y)} & {E(\ob j x, \ob j y)} \\
	{E(\ob{j}x, \ob{j}y)} \\
	{E(\ob j x, \ob t y)} \\
	{E(\ob t x, \ob t y)} & {E(\ob j x, \ob t y)}
	\arrow["{j_{x, y}}", from=1-1, to=1-2]
	\arrow["{E(\ob j x, \eta_y)}", from=1-2, to=4-2]
	\arrow["{E(\eta_x, \ob t y)}"', from=4-1, to=4-2]
	\arrow["{\dag_{x, y}}"', from=3-1, to=4-1]
	\arrow[Rightarrow, no head, from=3-1, to=4-2]
	\arrow["{j_{x, y}}"', from=1-1, to=2-1]
	\arrow["{E(\ob j x, \eta_y)}"', from=2-1, to=3-1]
\end{tikzcd}\]
	$\dag$ is also $\V$-natural, the left-compatibility law following from commutativity of
\[\begin{tikzcd}[column sep=-3.2em]
	{A(x, y) \otimes E(\ob j y, \ob t z)} && {\hspace{11em}} & {A(x, y) \otimes E(\ob t y, \ob t z)} \\
	{E(\ob j x, \ob j y) \otimes E(\ob j y, \ob t z)} &&& {E(\ob j x, \ob j y) \otimes E(\ob t y, \ob t z)} \\
	& {E(\ob j x, \ob t y) \otimes E(\ob j y, \ob t z)} && {E(\ob j x, \ob t y) \otimes E(\ob t y, \ob t z)} \\
	& {E(\ob j x, \ob t y) \otimes E(\ob t y, \ob t z)} && {E(\ob t x, \ob t y) \otimes E(\ob t y, \ob t z)} \\
	{E(\ob j x, \ob t z)} &&& {E(\ob t x, \ob t z)}
	\arrow["{A(x, y) \otimes \dag_{y, z}}", from=1-1, to=1-4]
	\arrow["{\dag_{x, z}}"', from=5-1, to=5-4]
	\arrow["{j_{x, y} \otimes E(\ob j y, \ob t z)}"', from=1-1, to=2-1]
	\arrow["{\circ_{\ob j x, \ob j y, \ob t z}}"', from=2-1, to=5-1]
	\arrow["{j_{x, y} \otimes E(\ob t y, \ob t z)}", from=1-4, to=2-4]
	\arrow["{E(\ob j x, \ob j y) \otimes \dag_{y, z}}"{description}, from=2-1, to=2-4]
	\arrow["{\circ_{\ob t x, \ob t y, \ob t z}}", from=4-4, to=5-4]
	\arrow["{E(\ob j x, \eta_y) \otimes E(\ob t y, \ob t z)}", from=2-4, to=3-4]
	\arrow["{\dag_{x, y} \otimes E(\ob t y, \ob t z)}", from=3-4, to=4-4]
	\arrow["{E(\ob j x, \eta_y) \otimes E(\ob j y, \ob t z)}"{description}, from=2-1, to=3-2]
	\arrow["{E(\ob j x, \ob t y) \otimes \dag_{y, z}}", from=3-2, to=3-4]
	\arrow["{E(\ob j x, \ob t y) \otimes \dag_{y, z}}"{description}, from=3-2, to=4-2]
	\arrow["{\circ_{\ob j x, \ob t y, \ob t z}}"{description}, from=4-2, to=5-1]
\end{tikzcd}\]
	using the first unit law, and the associativity law; and the right-compatibility law following from commutativity of
\[\begin{tikzcd}[column sep=8em]
	{E(\ob j x, \ob t y) \otimes A(y, z)} & {E(\ob t x, \ob t y) \otimes A(y, z)} \\
	{E(\ob j x, \ob t y) \otimes E(\ob j y, \ob j z)} & {E(\ob t x, \ob t y) \otimes E(\ob j y, \ob j z)} \\
	{E(\ob j x, \ob t y) \otimes E(\ob j y, \ob t z)} & {E(\ob t x, \ob t y) \otimes E(\ob j y, \ob t z)} \\
	{E(\ob j x, \ob t y) \otimes E(\ob t y, \ob t z)} & {E(\ob t x, \ob t y) \otimes E(\ob t y, \ob t z)} \\
	{E(\ob j x, \ob t z)} & {E(\ob t x, \ob t z)}
	\arrow["{\dag_{x, y} \otimes A(y, z)}", from=1-1, to=1-2]
	\arrow["{\dag_{x, z}}"', from=5-1, to=5-2]
	\arrow["{E(\ob t x, \ob t y) \otimes j_{y, z}}", from=1-2, to=2-2]
	\arrow["{E(\ob j x, \ob t y) \otimes j_{y, z}}"', from=1-1, to=2-1]
	\arrow["{\circ_{\ob j x, \ob t y, \ob t z}}"', from=4-1, to=5-1]
	\arrow["{\circ_{\ob t x, \ob t y, \ob t z}}", from=4-2, to=5-2]
	\arrow["{E(\ob j x, \ob t y) \otimes E(\ob j y, \eta_z)}"', from=2-1, to=3-1]
	\arrow["{E(\ob t x, \ob t y) \otimes E(\ob j y, \eta_z)}", from=2-2, to=3-2]
	\arrow["{E(\ob t x, \ob t y) \otimes \dag_{y, z}}", from=3-2, to=4-2]
	\arrow["{E(\ob j x, \ob t y) \otimes \dag_{y, z}}"', from=3-1, to=4-1]
	\arrow["{\dag_{x, y} \otimes E(\ob j y, \ob j z)}"{description}, from=2-1, to=2-2]
	\arrow["{\dag_{x, y} \otimes E(\ob j y, \ob t z)}"{description}, from=3-1, to=3-2]
\end{tikzcd}\]
  using the associativity law.

	Given $j$-monads $T = (t, \dag, \eta)$ and $T' = (t', \dag', \eta')$ in $\VCat$, it is trivial that every relative monad morphism $T \to T'$ specifies the given data. Thus it is enough to show that, given the specified data, $\{ \tau_x \}_{x \in \ob A}$ forms a $\V$-natural transformation. This follows by commutativity of the following diagram,
	\[\begin{tikzcd}[column sep=large]
		{A(x, y)} & {E(\ob j x, \ob j y)} & {E(\ob j x, \ob t y)} & {E(\ob t x, \ob t y)} \\
		{E(\ob j x, \ob j y)} \\
		{E(\ob j x, \ob{t'} y)} \\
		{E(\ob{t'} x, \ob{t'} y)} &&& {E(\ob t x, \ob{t'} y)}
		\arrow["{j_{x, y}}", from=1-1, to=1-2]
		\arrow["{E(\ob j x, \eta_y)}", from=1-2, to=1-3]
		\arrow["{\dag_{x, y}}", from=1-3, to=1-4]
		\arrow["{j_{x, y}}"', from=1-1, to=2-1]
		\arrow[Rightarrow, no head, from=1-2, to=2-1]
		\arrow["{E(\ob j x, \eta'_y)}"', from=2-1, to=3-1]
		\arrow["{\dag'_{x, y}}"', from=3-1, to=4-1]
		\arrow["{E(\ob t x, \tau_y)}", from=1-4, to=4-4]
		\arrow["{E(\tau_x, \ob{t'} y)}"', from=4-1, to=4-4]
		\arrow["{E(\ob j x, \tau_y)}"{description}, from=1-3, to=3-1]
	\end{tikzcd}\]
	using the preservation of the unit and extension operators by $\tau$.
\end{proof}

\begin{remark}
	\Cref{enriched-relative-monad} is asserted without proof for $\Set$-enriched relative monads in \cites[300]{altenkirch2010monads}[4]{altenkirch2015monads}, and for $\V$-enriched relative monads with \ff{} roots in \cite[Remark~8.2]{lucyshyn2022diagrammatic}. \textcite[Theorems~1.4.1 \& 1.5.2]{walters1970categorical} gives a proof for $\Set$-enriched monads relative to the identity (there called \emph{full devices}~\cite[Definition~1.1.1]{walters1970categorical}).
\end{remark}

\begin{example}
	\label{examples-of-enriched-relative-monads}
    The explicit definition of $\V$-enriched relative monad stated in \cref{enriched-relative-monad} does not appear to have been given in complete generality in the literature, and subsumes various prior definitions.
    \begin{enumerate}
        \item \label{device} A \emph{device} in the sense of \textcite[\S1]{walters1969alternative} is precisely a $\Set$-enriched relative monad whose root is injective-on-objects and \ff{}. A \emph{device} in the sense of \cite[Definition~1.1.1]{walters1970categorical}, which is equivalent to a \emph{Kleisli structure} in the sense of \textcite[Definition~4]{altenkirch1999monadic}, is precisely a $\Set$-enriched relative monad whose root is \ff{}.
        \item An \emph{algebraic theory in extension form} in the sense of \textcite[Exercise~1.3.12]{manes1976algebraic} (called a \emph{full device} in \cite[Definition~1.1.1]{walters1970categorical}, a \emph{Kleisli triple} in \cite[Definition~1.2]{moggi1991notions}, and a \emph{monad in extension form} in \cite[Definition~2.13]{manes2003monads}) is precisely a $\Set$-enriched relative monad whose root is the identity.
        \item A \emph{finitary Kleisli triple} in the sense of \cite[Definition~1.1]{adamek2003some} is precisely a $\Set$-enriched relative monad whose root is the inclusion of the full subcategory of the finitely presentable objects in a \lfp{} category. \Cref{Xj1-is-skew-monoidal} therefore recovers the correspondence of \cite[Theorem~3.2]{adamek2003some} between finitary Kleisli triples and monoids.
        \item \label{ACU-relative-monad} A \emph{(Manes-style) relative monad} in the sense of \citeauthor{altenkirch2010monads}~\cites[Definition~1]{altenkirch2010monads}[Definition~2.1]{altenkirch2015monads} is precisely a $\Set$-enriched relative monad.
        \item A \emph{$\V$-enriched clone} in the sense of \textcite[Definition~4]{staton2013algebraic} is precisely a $\V$-enriched relative monad whose root is \ff{}.
        \item A \emph{$j$-abstract $\V$-clone} in the sense of \textcite[Definition~1.1]{fiore2017concrete}, for $j$ having codomain $\V$ a monoidal category with powers of objects in the image of $j$, is equivalent when $\V$ is left-closed to a $\V$-enriched $j$-monad whose root is \ff{} (\cf{}~\cite[Remark~1.2]{fiore2017concrete}).
        \item An \emph{enriched relative monad} in the sense of \textcite[\S2.1]{staton2020classical} is (the underlying functor of) an enriched relative monad in our sense (technically, the definition \ibid{} requires the relative monad to admit a resolution, but this follows from \cref{VCat-admits-opalgebra-objects}).
        \item A \emph{relative 2-monad} in the sense of \textcite[Definition~3.1]{fiore2018relative} is precisely a $\b{Cat}$-enriched relative monad.
        \item \label{LP-relative-monad} An \emph{$A$-relative $\V$-monad (on $E$)} in the sense of \textcite[Definition~8.1]{lucyshyn2022diagrammatic} is precisely a $\V$-enriched relative monad whose root is \ff{}.
        \item \label{MU-relative-monad} A \emph{$j$-relative monad} in the sense of \textcite[Definition~14]{mcdermott2022flexibly}, for $\V$ a small monoidal category and $j$ a functor between locally $\V$-graded categories, is precisely a $[\V, \Set]$-enriched $j$-monad.
    \end{enumerate}
	\Cref{relative-monads-are-loose-monads} also recovers several independent definitions of relative monad in the literature.
	\begin{enumerate}[resume]
		\item A \emph{$j$-monad} in the sense of \textcite[D\'efinitions~1.0]{diers1975jmonades} is precisely a $\Set$-enriched relative monad whose root is dense and \ff{}.
		\item Since loose-monads in $\Cat$ are equivalently cocontinuous monads on presheaf categories (as the bicategory of distributors is the Kleisli bicategory for the presheaf construction~\cite{fiore2018relative}), this characterisation also justifies the approach of \citeauthor{lee1977relative}, who represents relative monads by cocontinuous monads on presheaf categories~\cite[Chapter~2]{lee1977relative}. Precisely, a \emph{monad associated to a relative adjointness situation} in the sense of \cite[Chapter~2]{lee1977relative} is a $\Set$-enriched relative monad whose root is dense and \ff{}.
		\item A \emph{copresheaf-representable monad} in the sense of \textcite[Theorem~10.5]{lucyshyn2016enriched} is precisely a $\V$-enriched relative monad whose root is a cocompletion (\emph{copresheaf-representable} $\V$-distributors~\cite[Definition~9.2]{lucyshyn2016enriched} are precisely $j$-representable $\V$-distributors in the sense of \cref{representable-and-corepresentable}). \qedhere
	\end{enumerate}
\end{example}

\begin{example}
	The more general notion of relative monad enriched in a bicategory was defined in \cite[88]{arkor2022monadic}. While we shall not treat this case in detail, we note that \cref{VCat} may be generalised to a \vdc{} of categories enriched in a bicategory, relative monads in which coincide with those \loccit. By virtue of the virtual double categorical setting in which we work, the admissibility condition on the roots required \ibid{} may be dropped in our setting.
\end{example}

The presentation of an enriched relative monad given in \cref{enriched-relative-monad} is generally the most convenient form for applications: when $\V = \Set$, a relative monad in this form assigns to each morphism $f \colon \ob j x \to \ob t y$ a morphism $f^\dag \colon \ob t x \to \ob t y$. In contrast, a relative monad in monoid form (\cref{relative-monads-are-tight-monoids}) assigns to a pair of morphisms $f \colon x \to \ob t y$ and $g \colon \ob j y \to \ob t z$ a morphism $\mu(f, g) \colon x \to \ob t z$.

\subsection{Existence of algebra objects}

We show that enriched relative monads admit algebra objects, assuming the existence of enough structure in $\V$.
As preparation for this, we define a notion of \emph{\EM{} algebra} (after \textcite{eilenberg1965adjoint}) for a relative monad $T$ in $\VCat$.
These will be the objects of the \emph{\EM{} $\V$-category} $\EM(T)$, which forms the algebra object for $T$.

\begin{definition}\label{EM-algebra}
	Let $\jAE$ be a $\V$-functor, and let $T = (t, \dag, \eta)$ be a $j$-monad in $\VCat$.
	An \emph{\EM{} $T$-algebra} comprises
	\begin{enumerate}
		\item an object $e \in \ob E$, the \emph{carrier};
		\item a family of morphisms in $\V$, the \emph{extension operator}
		\[\{ \eaop_x \colon E(\ob{j}x, e) \to E(\ob{t}x, e) \}_{x \in \ob A}\]
	\end{enumerate}
	rendering the following diagrams in $\V$ commutative for all $x, y \in \ob A$.
	\[
	\begin{tikzcd}[column sep=small, row sep=3.6em]
		{E(\ob j x, e)} & {E(\ob{t}x, e)} \\
		& {E(\ob j x, e)}
		\arrow["{\eaop_{x}}", from=1-1, to=1-2]
		\arrow["{E(\eta_x, e)}", from=1-2, to=2-2]
		\arrow[Rightarrow, no head, from=1-1, to=2-2]
	\end{tikzcd}
	\quad
	\begin{tikzcd}[column sep=small, row sep=small]
		{E(\ob j x, \ob t y) \tensor E(\ob j y, e)} && {E(\ob j x, \ob t y) \tensor E(\ob t y, e)} \\
		{E(\ob t x, \ob t y) \tensor E(\ob t y, e)} && {E(\ob j x, e)} \\
		& {E(\ob t x, e)}
		\arrow["{\circ_{\ob t x, \ob t y, e}}"', from=2-1, to=3-2]
		\arrow["{\eaop_{x}}", from=2-3, to=3-2]
		\arrow["{E(\ob j x, \ob t y) \tensor \eaop_{y}}", from=1-1, to=1-3]
		\arrow["{\circ_{\ob j x, \ob t y, e}}", from=1-3, to=2-3]
		\arrow["{\dag_{x, y} \tensor \eaop_{y}}"', from=1-1, to=2-1]
	\end{tikzcd}
	\]
	Let $(e, \eaop)$ and $(e', \eaop')$ be \EM{} $T$-algebras, and let $v$ be an object of $\V$. A \emph{$v$-graded homomorphism} from $(e, \eaop)$ to $(e', \eaop')$ is a morphism $h \colon v \to E(e, e')$ in $\V$ rendering the following diagram in $\V$ commutative for each $x \in \ob A$.
	\[\begin{tikzcd}
		{E(\ob j x, e) \otimes v} & {E(\ob t x, e) \otimes v} \\
		{E(\ob j x, e) \otimes E(e, e')} & {E(\ob t x, e) \otimes E(e, e')} \\
		{E(\ob j x, e')} & {E(\ob t x, e')}
		\arrow["{\eaop_x \otimes v}", from=1-1, to=1-2]
		\arrow["{\eaop'_x}"', from=3-1, to=3-2]
		\arrow["{E(\ob j x, e) \otimes h}"', from=1-1, to=2-1]
		\arrow["{\circ_{\ob j x, e, e'}}"', from=2-1, to=3-1]
		\arrow["{E(\ob t x, e) \otimes h}", from=1-2, to=2-2]
		\arrow["{\circ_{\ob t x, e, e'}}", from=2-2, to=3-2]
	\end{tikzcd}\]
\end{definition}

The morphism $\I_e \colon \tensorI \to E(e, e)$ is an $\tensorI$-graded algebra homomorphism as it is the identity for composition.
Two graded algebra homomorphisms $h \colon v \to E(e, e')$ and $h' \colon v' \to E(e', e'')$ compose to give a graded algebra homomorphism
\[
  v \tensor v' \xto{h \tensor h'} E(e, e') \tensor E(e', e'')
  \xto{{\circ_{e, e', e''}}} E(e, e'')
\]
using associativity of composition in $E$. \EM{} $T$-algebras hence form a locally $\V$-graded category.

\begin{lemma}\label{algebras-in-VCat}
	Let $j \colon A \to E$ be a $\V$-functor, and let $T$ be a $j$-monad.
	A $T$-algebra $(e, \aop)$ is equivalently specified by a
	$\V$-functor $e \colon D \to E$, together with a family of morphisms
	\[
	\{ \aop_{x, z} \colon E(\ob{j}x, \ob{e}z) \to E(\ob{t}x, \ob{e}z) \}_{x \in A, z \in D}
	\]
	such that
	\begin{enumerate}
		\item $(\ob{e}z, \aop_{{-}, z})$ is an \EM{} $T$-algebra for all $z \in \ob{D}$;
		\item $e_{y, z}$ is a $D(y, z)$-graded homomorphism from $(\ob{e}y, \aop_{{-}, y})$ to $(\ob{e}z, \aop_{{-}, z})$ for all $y, z \in \ob{D}$.
	\end{enumerate}
	Moreover, a $(p_1, \dots, p_n)$-graded $T$-algebra morphism from $(e, \aop)$ to $(e', \aop')$ is equivalently a $\V$-natural transformation
	$
	\epsilon \colon p_1, \dots, p_n \tto E(e, e')
	$
	such that each morphism
	\[
	\phi_{z_0, \dots, z_n} \colon p_1(z_0, z_1) \tensor \cdots \tensor p_n(z_{n-1}, z_n) \to E(\ob{e}z_0, \ob{e'}z_n)
	\]
	is a graded homomorphism from $(\ob{e}z_0, \aop_{{-}, z_0})$ to $(\ob{e'}z_n, \aop'_{{-}, z_n})$.
\end{lemma}

\begin{proof}
  The two $T$-algebra laws are precisely the two laws required for (1), while one of the two laws required for $\aop$ to be a $\V$-natural transformation, namely $\V$-naturality in $z$, is (2).
  Hence, for the characterisation of $T$-algebras, it remains to show that (1) and (2) together imply the other $\V$-natural transformation law, namely $\V$-naturality of $\aop$ in $x$.
  This proof is analogous to that of $\V$-naturality of the extension operator $\dag$ of $T$ in its first component (\cref{enriched-relative-monad}).

  The characterisation of graded $T$-algebra morphisms is trivial from \cref{algebra-morphisms-equivalent}.
\end{proof}

\begin{remark}
	Consequently, \EM{} $T$-algebras and their $I$-graded morphisms subsume several notions in the literature.
	\begin{enumerate}
		\item A \emph{$T$-algebra} in the sense of \citeauthor{walters1969alternative}~\cites[\S1]{walters1969alternative}[Definitions~1.1.3]{walters1970categorical} is precisely an \EM{} $T$-algebra, for $T$ as in \cref{device}.
		\item An \emph{EM-algebra of $T$} in the sense of \citeauthor{altenkirch2010monads}~\cites[Definition~3]{altenkirch2010monads}[Definition~2.11]{altenkirch2015monads} is precisely an \EM{} $T$-algebra, for $T$ as in \cref{ACU-relative-monad}.
		\item A \emph{$T$-algebra} in the sense of \textcite[Definition~8.4]{lucyshyn2022diagrammatic} is precisely an \EM{} $T$-algebra, for $T$ as in \cref{LP-relative-monad}.
		\item A \emph{$T$-algebra} in the sense of \textcite[Definition~16]{mcdermott2022flexibly} is precisely an \EM{} $T$-algebra, for $T$ as in \cref{MU-relative-monad}. \qedhere
	\end{enumerate}
\end{remark}

\begin{theorem}\label{VCat-admits-algebra-objects}
  Let $\jAE$ be a $\V$-functor, and let $T$ be a $j$-monad.
  $T$ admits an algebra object exactly when, for all \EM{} $T$-algebras $(e, \eaop)$ and $(e', \eaop')$, there is a graded homomorphism
  \[
    (u_T)_{(e, \eaop), (e', \eaop')} \colon \EM(T)((e, \eaop), (e', \eaop')) \to E(e, e')
  \]
  universal in the sense that every graded homomorphism $v \to E(e, e')$ factors uniquely through $(u_T)_{(e, \eaop), (e', \eaop')}$ as a morphism $v \to \EM(T)((e, \eaop), (e', \eaop'))$.
\end{theorem}
\begin{proof}
  We first show that the algebra object exists assuming that $(u_T)_{(e, \eaop), (e', \eaop')}$ does.
  The \emph{\EM{} $\V$-category} $\EM(T)$ of $T$ has as objects \EM{} $T$-algebras, and as hom-objects the domains $\EM(T)((e, \eaop), (e', \eaop'))$ of the universal homomorphisms.
  Since graded homomorphisms compose, identities and composition in $\EM(T)$ are inherited from identities and composition in $E$ via the universal property of the hom-objects.
\[\begin{tikzcd}[column sep=large]
	\tensorI & {\EM(T)((e, \eaop), (e, \eaop))} \\
	& {E(e, e)}
	\arrow["{\I_{(e, \eaop)}}", from=1-1, to=1-2]
	\arrow["{(u_T)_{(e, \eaop), (e, \eaop)}}", from=1-2, to=2-2]
	\arrow["{\I_e}"', from=1-1, to=2-2]
\end{tikzcd}\]
\[\begin{tikzcd}[column sep=7em]
	{\EM(T)((e, \aop), (e', \aop')) \tensor \EM(T)((e', \aop'), (e'', \aop''))} & {\EM(T)((e, \aop), (e'', \aop''))} \\
	{E(e, e') \tensor E(e', e'')} & {E(e, e'')}
	\arrow["{\circ_{(e, \aop), (e', \aop'), (e'', \aop'')}}", from=1-1, to=1-2]
	\arrow["{\circ_{e, e', e''}}"', from=2-1, to=2-2]
	\arrow["{(u_T)_{(e, \aop), (e'', \aop'')}}", from=1-2, to=2-2]
	\arrow["{(u_T)_{(e, \aop), (e', \aop')} \tensor (u_T)_{(e', \aop'), (e'', \aop'')}}"{description}, from=1-1, to=2-1]
\end{tikzcd}\]
  Unitality and associativity of composition in $E$ clearly imply the corresponding properties for $\EM(T)$, so that $\EM(T)$ is a $\V$-category.
  The morphisms $(u_T)_{(e, \eaop), (e', \eaop')}$ form a $\V$-functor $u_T \colon \EM(T) \to E$, given on objects by $\ob{u_T}(e, \eaop) = e$.
  The extension operators of \EM{} $T$-algebras make $u_T$ into a $T$-algebra by \cref{algebras-in-VCat}.

  To show that this $T$-algebra satisfies the universal property of the algebra object, consider an arbitrary $T$-algebra $(e \colon D \to E, \aop)$.
  Using the characterisation of $T$-algebras in \cref{algebras-in-VCat}, we obtain a $\V$-functor $\unit_{(e, \aop)} \colon D \to \EM(T)$.
  This is given on objects by $\ob{\unit_{(e, \aop)}}z = (\ob{e}z, \aop_{{-}, z})$, and
  \[
    (\unit_{(e, \aop)})_{y, z} \colon D(y, z) \to \EM(T)(\ob{\unit_{(e, \aop)}}y, \ob{\unit_{(e, \aop)}}z)
  \]
  is given by the unique factorisation of $e_{y, z} \colon D(y, z) \to E(\ob{e}y, \ob{e}z)$ through $(u_T)_{\ob{\unit_{(e, \aop)}}y, \ob{\unit_{(e, \aop)}}z}$.
  This preserves identities and composition because $e$ does.
  Moreover, $\unit_{(e, \aop)}$ is clearly unique such that composing with $u_T$ recovers the $T$-algebra $(e, \aop)$.

  It remains to show that $T$-algebra morphisms factor uniquely though $u_T$.
  By \cref{algebras-in-VCat}, such a morphism is equivalently a $\V$-natural transformation
  \[
    \epsilon \colon p_1, \dots, p_n \tto E(e, e')
  \]
  each component of which is a graded homomorphism.
  Since $u_T$ consists of the universal graded homomorphisms, it is immediate that $\epsilon$ factors uniquely though $u_T$ as a $\V$-natural transformation
  \[
    p_1, \dots, p_n \tto \EM(T)(\unit_{(e, \aop)}, \unit_{(e', \aop')})
  \]

  For the converse, assume that the algebra object $(u_T \colon \Alg(T) \to E, \aop_T)$ exists.
  Each \EM{} $T$-algebra $(e, \eaop)$ can be viewed as a $T$-algebra $(e \colon \Icat \to E, \eaop)$, which, by the universal property of the algebra object, induces an object $\unit_{(e, \eaop)}$ of $\Alg(T)$.
  We show that
  \[
    (u_T)_{\unit_{(e, \eaop)}, \unit_{(e', \eaop')}}
    \colon \Alg(T)(\unit_{(e, \eaop)}, \unit_{(e', \eaop')})
    \to E(e, e')
  \]
  is the universal graded homomorphism.
  Each graded homomorphism $h \colon v \to E(e, e')$ can be viewed as a $(v)$-graded $T$-algebra morphism from $(e, \eaop)$ to $(e', \eaop')$, by viewing $v$ as a $\V$-distributor $v \colon \Icat \to \Icat$.
  The universal property of the algebra object implies that $h$ then factors uniquely through $(u_T)_{\unit_{(e, \eaop)}, \unit_{(e', \eaop')}}$, as required.
\end{proof}

\begin{corollary}\label{VCat-admits-algebra-objects-ii}
  Let $\jAE$ be a $\V$-functor.
  If $\V$ admits equalisers, and the object $\rfP{A}{E(\ob{j}{-}, e)}{q}$ of $\V$-natural transformations exists for each $e \in \ob{E}$ and $\V$-presheaf $q$ on $A$, then every $j$-relative monad admits an algebra object.
\end{corollary}

\begin{proof}
  By \cref{VCat-admits-algebra-objects}, it suffices to show that the universal graded homomorphisms
  \[
    (u_T)_{(e, \eaop), (e', \eaop')} \colon \EM(T)((e, \eaop), (e', \eaop')) \to E(e, e')
  \]
  exist.
  Observe that both of the following families of morphisms are $\V$-natural in $x \in \ob A$ because $\eaop$ and $\eaop'$ are (\cf{} \cref{algebras-in-VCat}).
  \begin{gather*}
    E(\ob{j}x, e) \tensor E(e, e')
    \xto{\eaop_x \tensor E(e, e')}
    E(\ob{t}x, e) \tensor E(e, e')
    \xto{\circ_{\ob{t}x, e, e'}}
    E(\ob{t}x, e')
    \\
    E(\ob{j}x, e) \tensor E(e, e')
    \xto{\circ_{\ob{j}x, e, e'}}
    E(\ob{j}x, e')
    \xto{\eaop'_x}
    E(\ob{t}x, e')
  \end{gather*}
  Hence there are corresponding morphisms,
  \[
    \zeta_1, \zeta_2 \colon E(e, e') \to \rfP{A}{E(\ob{j}{-}, e)}{E(\ob{t}{-}, e')}
  \]
  and a morphism $h \colon v \to E(e, e')$ is a graded homomorphism exactly when $h \d \zeta_1 = h \d \zeta_2$.
  It follows that the equaliser of $\zeta_1$ and $\zeta_2$ is the universal graded homomorphism.
\end{proof}

In particular, the assumptions of \cref{VCat-admits-algebra-objects-ii} hold for $\V$-functors $\jAE$ with small domain when $\V$ is complete and closed; and hold for the identity $\V$-functor $j = 1_E$ when $\V$ has equalisers.

\subsection{Existence of opalgebra objects}

The existence of opalgebra objects in $\VCat$ is simpler than that of algebra objects, and in particular requires no conditions on $\V$.

\begin{theorem}
	\label{VCat-admits-opalgebra-objects}
    Every relative monad in $\VCat$ admits an opalgebra object.
\end{theorem}

\begin{proof}
	Let $\jAE$ be a $\V$-functor and let $T = (t, \eta, \dag)$ be a $j$-monad. We define a $\V$-category $\Kl(T)$, the \emph{Kleisli $\V$-category of $T$} (after \cite{kleisli1965every}), as follows.
	\begin{align*}
		\ob{\Kl(T)} & \defeq \ob A &
		\Kl(T)(x, y) & \defeq E(\ob j x, \ob t y) \\
		\I_x & \defeq \eta_x &
		\circ_{x, y, z} & \defeq
	\end{align*}
	\[E(\ob j x, \ob t y) \otimes E(\ob j y, \ob t z) \xto{E(\ob j x, \ob t y) \otimes \dag_{y, z}} E(\ob j x, \ob t y) \otimes E(\ob t y, \ob t z) \xto{\circ_{\ob j x, \ob t y, \ob t z}} E(\ob j x, \ob t z)\]
	Unitality and associativity of composition follows from the unitality and associativity laws for $T$.

	We define an \ioo{} $\V$-functor $k_T \colon A \to \Kl(T)$, the \emph{Kleisli inclusion of $T$}, whose action on hom-objects is given by
        \[
        (k_T)_{x, y} \defeq~
        \begin{tikzcd}[column sep=large]
		{A(x, y)} & {E(\ob j x, \ob j y)} & {E(\ob j x, \ob t y)}
		\arrow["{j_{x, y}}", from=1-1, to=1-2]
		\arrow["{E(\ob j x, \eta_y)}", from=1-2, to=1-3]
	\end{tikzcd}
        \]
	with preservation of identities following from commutativity of the following diagram,
	\[\begin{tikzcd}
		\tensorI & {A(x, x)} \\
		& {E(\ob j x, \ob j x)} \\
		& {E(\ob j x, \ob t x)}
		\arrow["{j_{x, x}}", from=1-2, to=2-2]
		\arrow["{E(\ob j x, \eta_x)}", from=2-2, to=3-2]
		\arrow["{\I_x}", from=1-1, to=1-2]
		\arrow["{\I_{\ob jx}}"{description}, from=1-1, to=2-2]
		\arrow["{\eta_x}"', curve={height=18pt}, from=1-1, to=3-2]
	\end{tikzcd}\]
	and preservation of composites following from commutativity of the following diagram,
	\[\begin{tikzcd}[column sep=huge]
		{A(x, y) \otimes A(y, z)} && {A(x, z)} \\
		{E(\ob j x, \ob j y) \otimes E(\ob j y, \ob j z)} && {E(\ob j x, \ob j z)} \\
		& {E(\ob j x, \ob j y) \otimes E(\ob j y, \ob t z)} \\
		{E(\ob j x, \ob t y) \otimes E(\ob j y, \ob t z)} & {E(\ob j x, \ob t y) \otimes E(\ob t y, \ob t z)} & {E(\ob j x, \ob t z)}
		\arrow["{E(\ob j x, \ob t y) \otimes \dag_{y, z}}"', from=4-1, to=4-2]
		\arrow["{\circ_{x, y, z}}", from=1-1, to=1-3]
		\arrow["{\circ_{\ob j x, \ob t y, \ob t z}}"', from=4-2, to=4-3]
		\arrow["{j_{x, z}}", from=1-3, to=2-3]
		\arrow["{j_{x, y} \otimes j_{y, z}}"', from=1-1, to=2-1]
		\arrow["{E(\ob j x, \eta_y) \otimes E(\ob j y, \eta_z)}"{description}, from=2-1, to=4-1]
		\arrow["{E(\ob j x, \eta_z)}", from=2-3, to=4-3]
		\arrow["{\circ_{\ob j x, \ob j y, \ob j z}}"{description}, from=2-1, to=2-3]
		\arrow["{E(\ob j x, \ob j y) \otimes E(\ob j y, \eta_z)}"{description}, from=2-1, to=3-2]
		\arrow["{\circ_{\ob j x, \ob j y, \ob t z}}"{description}, from=3-2, to=4-3]
		\arrow["{E(\ob j x, \eta_y) \otimes E(\ob j y, \ob t z)}"{description}, from=3-2, to=4-1]
	\end{tikzcd}\]
	using the first unit law for $T$.

	We shall show that $k_T \colon A \to \Kl(T)$ together with the identity $\V$-natural transformation $E(j, t) = \Kl(T)(k_T, k_T)$ forms an opalgebra object for $T$. That $(k_T, 1)$ forms a $T$-opalgebra follows from the definition of identities and composition in $\Kl(T)$. Let $(a, \oop)$ be a $T$-opalgebra. The $\V$-natural transformation $\oop \colon E(j, t) \tto B(a, a)$ defines the action on hom-objects of a $\V$-functor $[]_{(a, \oop)} \colon \Kl(T) \to B$ with object-function $\ob{[]_{(a, \oop)}} \defeq \ob a$, preservation of identities and composites following from the unitality and extension laws of the opalgebra. We trivially have commutativity of
	\[\begin{tikzcd}
		{\Kl(T)} && B \\
		& A
		\arrow["{k_T}", from=2-2, to=1-1]
		\arrow["{[]_{(a, \oop)}}", from=1-1, to=1-3]
		\arrow["a"', from=2-2, to=1-3]
	\end{tikzcd}\]
	since $k_T$ is identity-on-objects. Let $\alpha$ be a $(p_1, \ldots, p_n)$-graded $T$-opalgebra morphism from $(a, \oop)$ to $(a', \oop')$, hence a family of morphisms
	\[\{ \alpha_{x_0, \ldots, x_n, y} \colon p_1(x_0, x_1) \otimes \cdots \otimes p_n(x_{n - 1}, x_n) \otimes B(x_n, \ob{a} y) \to B'(x_0, \ob{a'} y) \}_{x_0, \ldots, x_n, y}\]
	in $\V$. Since $k_T$ is identity-on-objects, this is equivalently a family of morphisms:
	\[\{ {[]_\alpha}_{x_0, \ldots, x_n, y} \colon p_1(x_0, x_1) \otimes \cdots \otimes p_n(x_{n - 1}, x_n) \otimes B(x_n, \ob{[]_{(a, \oop)}} y) \to B'(x_0, \ob{[]_{(a', \oop')}} y) \}_{x_0, \ldots, x_n, y}\]
	That this family forms a $\V$-natural transformation follows from the fact that $\alpha$ is a $\V$-natural transformation, together with the $T$-opalgebra morphism compatibility law. Hence $\alpha$ factors uniquely through $[]_\alpha$.
\end{proof}

\subsection{Existence of coalgebra and co\"opalgebra objects}

Given a monoidal category $\V$, we denote by $\V\rev$ the monoidal category with the same objects and unit as $\V$ and whose tensor product is defined by $x \otimes_{\V\rev} y \defeq y \otimes_{\V} x$.

To deduce sufficient conditions for the existence of co\"opalgebra and coalgebra objects for relative comonads, the following observation is useful.

\begin{proposition}
	\label{duality-of-VCat}
	There is an isomorphism of \vdcs{}:
	\[\VCat\co \iso \V\rev\h\Cat\]
\end{proposition}

\begin{proof}
	For each $\V$-category $C$, we may define a $\V\rev$-category $C\op$, its dual, by
	\begin{align*}
		\ob{C\op} & \defeq \ob C &
		C\op(x, y) & \defeq C(y, x) &
		{\I^{C\op}}_x & \defeq {\I^C}_x &
		{\circ^{C\op}}_{x, y, z} & \defeq {\circ^C}_{z, y, x}
	\end{align*}
	Unitality and associativity of composition in $C\op$ follows from that of $C$. For each $\V$-functor $f \colon C \to D$, we may define a $\V\rev$-functor $f\op \colon C\op \to D\op$ given by
	\begin{align*}
		\ob{f\op} & = \ob f &
		{f\op}_{x, y} & = f_{y, x}
	\end{align*}
	Preservation of identities and composites follows from that of $f$. For each $\V$-distributor $p \colon C \lto C'$, we may define a $\V\rev$-distributor $p\op \colon C'\op \lto C\op$ given by
	\begin{align*}
		p\op(x, y) & \defeq p(y, x) &
		{\circ^{p\op}}_{x, y, z} & \defeq {\circ^p}_{z, y, x}
	\end{align*}
	The compatibility laws for identities and composition follow from those of $p\op$. For each $\V$-natural transformation on the left below, we may define a $\V\rev$-natural transformation on the right below, given by
	\[\{ {\phi\op}_{x_0, \ldots, x_n} \defeq  \phi_{x_n, \ldots, x_0} \}_{x_0, \ldots, x_n}\]
	\[
	\begin{tikzcd}
		{A_0} & \cdots & {A_n} \\
		{B_0} && {B_n}
		\arrow["q", "\shortmid"{marking}, from=2-3, to=2-1]
		\arrow["{p_n}"', "\shortmid"{marking}, from=1-3, to=1-2]
		\arrow[""{name=0, anchor=center, inner sep=0}, "{f_n}", from=1-3, to=2-3]
		\arrow[""{name=1, anchor=center, inner sep=0}, "{f_0}"', from=1-1, to=2-1]
		\arrow["{p_1}"', "\shortmid"{marking}, from=1-2, to=1-1]
		\arrow["\phi"{description}, draw=none, from=0, to=1]
	\end{tikzcd}
	\hspace{4em}
	\begin{tikzcd}
		{A_n\op} & \cdots & {A_0\op} \\
		{B_n\op} && {B_0\op}
		\arrow["q\op", "\shortmid"{marking}, from=2-3, to=2-1]
		\arrow["{p_1\op}"', "\shortmid"{marking}, from=1-3, to=1-2]
		\arrow[""{name=0, anchor=center, inner sep=0}, "{f_0\op}", from=1-3, to=2-3]
		\arrow[""{name=1, anchor=center, inner sep=0}, "{f_n\op}"', from=1-1, to=2-1]
		\arrow["{p_n\op}"', "\shortmid"{marking}, from=1-2, to=1-1]
		\arrow["\phi\op"{description}, draw=none, from=0, to=1]
	\end{tikzcd}
	\]
	$\V\rev$-naturality of $\phi\op$ follows from $\V$-naturality of $\phi$.
	Dualisation is evidently self-inverse, and hence defines an isomorphism $\ph\op \colon \VCat\co \to \V\rev\h\Cat$.
\end{proof}

Consequently, since a relative comonad in $\VCat$ is precisely a relative monad in $\VCat\co$, \cref{duality-of-VCat} provides a convenient method to work with relative $\V$-comonads as relative $\V\rev$-monads. In particular, we can derive existence theorems for coalgebra and co\"opalgebra objects as follows.

By dualising \cref{presheaf} above, we obtain a notion of
\emph{$\V$-copresheaf} on a $\V$-category $Z$, comprising an object
$p(z)$ of $\V$ for each $z \in \ob{Z}$, and a family of morphisms
\[\{ \circ_{z, z'} \colon p(z) \tensor Z(z, z') \to p(z') \}_{z, z' \in Z}\] compatible with
identities and composition in $Z$.
We also have a notion of $\V$-natural family of morphisms between
copresheaves, and we can speak of the universal $\V$-natural family
$\{ \epsilon_z \colon \reP{Z}{p}{q} \tensor p(z) \to q(z) \}_{z \in \ob Z}$.
By dualising \cref{EM-algebra} above, given a $\V$-functor $i \colon Z \to U$, we obtain a notion of
\emph{\EM{} $D$-coalgebra} for an $i$-comonad $D$: such a coalgebra comprises an object $u \in \ob{U}$ and family of morphisms
$\{ \eaop_z \colon U(u, \ob{i}z) \to U(u, \ob{d}z) \}_{z \in \ob Z}$ satisfying counitality and compatibility laws.
A \emph{$v$-graded homomorphism} from a coalgebra $(u, \eaop)$
to a coalgebra $(u', \eaop')$ is then a morphism $h \colon v \to U(u, u')$ preserving the coextension operators.

\begin{theorem}
	\label{VCat-admits-coalgebra-objects}
	Let $i \colon Z \to U$ be a $\V$-functor, and let $D$ be an $i$-comonad. $D$ admits a coalgebra object exactly when there
	is a universal graded homomorphism between any two \EM{}
	$D$-coalgebras.
	In particular, every $i$-comonad admits a coalgebra object when
	$\V$ admits equalisers and $\reP{Z}{U(u, \ob{i}{-})}{q}$ exists for all
	objects $u \in \ob{U}$ and copresheaves $q$.
\end{theorem}

\begin{proof}
  By \cref{duality-of-VCat}, $D$ admits a coalgebra object if and only if the corresponding relative monad in $\VcoCat$ admits an algebra object, so the result follows from \cref{VCat-admits-algebra-objects} and \cref{VCat-admits-algebra-objects-ii}.
\end{proof}

In particular, the assumptions of \cref{VCat-admits-coalgebra-objects} hold for $\V$-functors $i \colon Z \to U$ with small domain when $\V$ is complete and closed.

\begin{theorem}
	\label{VCat-admits-coopalgebra-objects}
	Every relative comonad in $\VCat$ admits a co\"opalgebra object.
\end{theorem}

\begin{proof}
	By \cref{duality-of-VCat}, a relative comonad admits a co\"opalgebra objects if and only if the corresponding relative monad in $\VcoCat$ admits an opalgebra object, so the result follows from \cref{VCat-admits-opalgebra-objects}.
\end{proof}

\printbibliography

\end{document}